\documentclass[11pt,a4paper]{article}
\usepackage[utf8]{inputenc}
\usepackage{amsmath}
\usepackage{pifont}
\usepackage{bm}
\usepackage[american]{babel}
\usepackage[all]{xy}
\usepackage{enumitem}
\usepackage{anysize}
\usepackage[dvipsnames]{xcolor}
\usepackage{graphicx}
\usepackage{stmaryrd}
\usepackage{amsfonts,amstext,amssymb,amsmath,amsthm,pspicture,bigints}

\usepackage{fullpage}
\usepackage{moreverb}
\usepackage{pifont}
\usepackage{pst-all}
\usepackage[left=2cm,right=2cm,top=2cm,bottom=2cm]{geometry}
\usepackage{esint}
\usepackage{fourier-orns}
\usepackage{mathrsfs}
\usepackage{array}
\newcommand{\T}{\mathbb{T}}
\usepackage{hyperref}
\newcommand{\ii }{{\rm i} }
\usepackage{authblk}
\newcolumntype{C}[1]{>{\centering\arraybackslash}b{#1}}
\newcolumntype{R}[1]{>{\raggedleft\arraybackslash}b{#1}}
\newcolumntype{L}[1]{>{\raggedright\arraybackslash}b{#1}}
\newcolumntype{M}[1]{>{\centering}m{#1}}
\newtheorem{theo}{Theorem}[section]
\newtheorem{defin}{Definition}[section]
\newtheorem{lem}{Lemma}[section]
\newtheorem{prop}{Proposition}[section]

\newtheorem{remark}{Remark}[section]

\newtheorem{cor}{Corollary}[section]
\usepackage{bbm}
\numberwithin{equation}{section}
\DeclareMathOperator{\R}{\mathbb{R}}
\DeclareMathOperator{\N}{\mathbb{N}}
\DeclareMathOperator{\C}{\mathbb{C}}

\DeclareMathOperator{\Z}{\mathbb{Z}}
\DeclareMathOperator{\m}{\mathtt{m}}
\DeclareMathAlphabet{\mathsfit}{T1}{\sfdefault}{m}{sl}
\SetMathAlphabet{\mathsfit}{bold}{T1}{\sfdefault}{bx}{sl}

\DeclareFontEncoding{LS1}{}{}
\DeclareFontSubstitution{LS1}{stix}{m}{n}
\DeclareMathAlphabet{\mathscr}{LS1}{stixscr}{m}{n}
\SetMathAlphabet{\mathscr}{bold}{LS1}{stixscr}{b}{n}

\date{}
\title{Invariant KAM tori around annular vortex patches\\  for 2D Euler equations}
\author{Zineb Hassainia, Taoufik Hmidi and Emeric Roulley}
\begin{document}
	\maketitle
	\begin{abstract}
		We construct time quasi-periodic  vortex patch solutions with one hole  for the planar Euler equations. These structures are captured close to any  annulus  provided that its modulus belongs to  a massive Borel set. 		The proof is based on   Nash-Moser scheme  and  KAM theory applied with a  Hamiltonian system governing  the  radial deformations of the patch.  Compared to the scalar case discussed recently in \cite{HHM21,HR21-1,HR21,R22}, some technical issues emerge due  to the interaction between the interfaces. One of them is related to a new  small divisor problem  in the second order Melnikov non-resonances condition coming from the  transport equations advected with different velocities.
	\end{abstract}
	\tableofcontents
	\section{Introduction}	
	This work deals with some aspects on the vortex motion for the classical planar incompressible Euler equations that can be reformulated in the vorticity/velocity form as follows
		\begin{equation}\label{Euler equations}
		\left\lbrace\begin{array}{ll}
			\partial_{t}\boldsymbol{\omega}+\mathbf{v}\cdot\nabla\boldsymbol{\omega}=0, & \textnormal{in }\mathbb{R}_{+}\times\mathbb{R}^2,\\
			\mathbf{v}=\nabla^{\perp}\psi,\\
			\Delta\psi=\boldsymbol{\omega},\\
			\boldsymbol{\omega}(0,\cdot)=\boldsymbol{\omega}_0.
		\end{array}\right.
	\end{equation}
	The quantity $\mathbf{v}$ represents the {velocity field} of the fluid particles which is supposed to be solenoidal according to the second equation in \eqref{Euler equations} where the notation $\nabla^{\perp}\triangleq \begin{pmatrix}
-\partial_2\\
\partial_1
\end{pmatrix}$ is used. The scalar potential  $\boldsymbol{\omega}$ is  called the {vorticity} and  measures the local rotation effects inside  the fluid. It is related to the velocity field by the relation
	$$\boldsymbol{\omega}\triangleq\nabla^{\perp}\cdot\mathbf{v}.$$
	From  the third equation in \eqref{Euler equations}, we can recover the {stream function} $\psi$  from the vorticity $\boldsymbol{\omega}$ through the following integral operator with a logarithmic kernel
	\begin{equation}\label{def streamL1}
		\psi(t,z)=\frac{1}{2\pi}\int_{\mathbb{R}^{2}}\log(|z-\xi|)\boldsymbol{\omega}(t,\xi)dA(\xi),
	\end{equation}
	where $dA$ is the $2$-dimensional Lebesgue measure. It is well-known since the work of Yudovich \cite{Y63} that any bounded and integrable initial datum $\boldsymbol{\omega}_0$ generates a unique global in time weak solution of \eqref{Euler equations} which is Lagrangian, namely
	$$\boldsymbol{\omega}(t,x)=\boldsymbol{\omega}_0\big(\boldsymbol{\Phi}_{t}^{-1}(x)\big),\qquad\boldsymbol{\Phi}_t(x)=x+\int_{0}^{t}\mathbf{v}\big(t,\boldsymbol{\Phi}_s(x)\big)ds.$$
	In particular, if the initial datum is the characteristic function of some  bounded domain $D_0$ then
	$$\boldsymbol{\omega}(t,\cdot)=\mathbf{1}_{D_t},\qquad D_t\triangleq\boldsymbol{\Phi}_t(D_0)$$
	and the resulting solution is called a \textit{vortex patch}. The dynamics of these solutions is entirely described by the evolution of the boundary $\partial D_t.$ The global in time persistence of the boundary regularity of type $C^{1,\alpha}$, with $\alpha\in(0,1)$  was first proved  by  Chemin in \cite{C93,C95} and later by Bertozzi and Constantin in  \cite{BC93}. Notice that the boundary motion can be tracked from  the {\it  contour dynamics equation} of the patch.  Indeed,  for any parametrization $z(t):\mathbb{T}\rightarrow\partial D_t$ of the boundary, denoting $\mathbf{n}\big(t,z(t,\theta)\big)\triangleq\ii \partial_\theta z(t,\theta)$ a normal vector to the boundary at the point $z(t,\theta)$, one has
	\begin{equation}\label{CDE}
		\partial_{t}z(t,\theta)\cdot\mathbf{n}\big(t,z(t,\theta)\big)=\partial_{\theta}\Big[\psi\big(t,z(t,\theta)\big)\Big].
	\end{equation}
	We refer for instance   to \cite{HMV13} for a complete derivation of this equation. In 1858, Rankine observed that any radial initial domain $D_0$ (disc, annulus, etc...)  generates a stationary solution to \eqref{Euler equations}. Then from a dynamical system point of view  it is of important interest to explore  the local structure of the phase portrait and to know whether periodic solutions may exist around  these equilibrium states. This topic turns out to be highly rich leading to fruitful subjects connecting various areas in Mathematics.  The first result in this direction is due to Kirchhoff in 1874 \cite{K74} where he proved that an ellipse $\mathcal{E}_{a,b}$ with semi-axes $a$ and $b$ performs a uniform rotation about its center  with an angular velocity $\Omega$ if and only if
	$$\Omega=\frac{ab}{(a+b)^2}\cdot$$
	Actually, the ellipses form a subclass of   \textit{relative equilibria} or \textit{V-states} which are solutions keeping the same shape during the motion, from which we derive another subclass given by rotating patches where the domain $D_t$ rotates uniformly about its center (due to the space invariance, we can suppose without any  loss of generality that the center is the origin),
	$$D_t=e^{\ii\Omega t}D_0,\qquad\Omega\in\mathbb{R}.$$
	They form a family of rigid periodic solutions where the domain  is not deformed during the motion  and keeps its initial shape. Then, more recently in 1978,  Deem and Zabusky \cite{DZ78} discovered numerically 3-fold, 4-fold and 5-fold V-states living close to the unit disc. Few years after, Burbea \cite{B82} confirmed analytically these simulations  using  bifurcation theory. More precisely, he proved that for any integer  $\mathtt{m}\geqslant 1,$ one can find a branch of $\mathtt{m}$-fold simply-connected V-states bifurcating from the unit disc at the angular velocity
	$$\Omega_{\mathtt{m}}\triangleq\frac{\mathtt{m}-1}{2\mathtt{m}}\cdot$$
	Actually, the case $\mathtt{m}=1$ corresponds to a translation of the Rankine vortex whereas the branch associated with the mode $\mathtt{m}=2$ gives the Kirchhoff ellipses. Observe that for any $\mathtt{m}\geqslant2,$ the bifurcation frequency $\Omega_{\mathtt{m}}$ lives in the interval $(0,\tfrac{1}{2})$  and the series of works \cite{F00,GPSY20,H15} showed that, outside this interval and in the simply-connected case, the only relative equilibria are the radial ones. The boundary regularity of the V-states and the global bifurcation were analyzed in  \cite{CCG16,C-C-G16, HMW20,HMV13}. The second bifurcation from the ellipses has been discussed in  \cite{C-C-G16,HM16}. More precisely, if we consider an ellipse $\mathcal{E}_{a,b}$ ($a>b$) described by	$$\mathcal{E}_{a,b}=\Big\{\tfrac{a+b}{2}\big(w+\tfrac{\mathcal{Q}}{w}\big),\quad w\in\mathbb{T},\quad\mathcal{Q}\triangleq\tfrac{a-b}{a+b}\in(0,1)\Big\},$$
	then for any integer $\mathtt{m}\geqslant3,$ the bifurcation occurs at the angular velocity $\Omega=\frac{1-Q^2}{4}$, where  $\mathcal{Q}$ is a solution to the polynomial equation 
	\begin{equation}\label{constraint ellipse}
		f(\mathtt{m},\mathcal{Q})\triangleq1+\mathcal{Q}^{\mathtt{m}}-\tfrac{1-\mathcal{Q}^2}{2}\mathtt{m}=0.
	\end{equation}
	The boundary effects on the emergence of V-states have been explored recently in \cite{HHHM15} where  the authors proved the existence of V-states when the fluid evolves in the unit disc $\mathbb{D}$. 
	It was  shown that  for any integer $\mathtt{m}\geqslant 1$ a family of $\mathtt{m}$-fold implicit curves bifurcate  from the disc $b\mathbb{D},$ $b\in(0,1),$ at the angular velocity 
	$$\Omega_{\mathtt{m}}(b)\triangleq\frac{\mathtt{m}-1+b^{2\mathtt{m}}}{2\mathtt{m}}\cdot$$
	In contrast to the flat case $\R^2$, the one-fold curve is no longer trivial here and moreover the numerical simulations performed in   \cite{HHHM15} show that  in some regimes of $b$ the bifurcating  curves oscillate with respect to the angular velocity.  In the same spirit,  Hmidi, de la Hoz,  Mateu and Verdera discussed in \cite{HHMV16} the existence of rotating patches with one hole called doubly-connected V-states. They proved that for a fixed symmetry $\mathtt{m}\geqslant3$ and $b\in(0,1)$  two $\mathtt{m}$-fold curves of doubly-connected V-states  bifurcate from the annulus   
		\begin{equation}\label{annulus-Ab}
		A_b\triangleq\big\{z\in\mathbb{C}\quad\textnormal{s.t.}\quad b<|z|<1\big\},
	\end{equation}
	provided that the following constraint is satisfied
	$$f(\mathtt{m},b)<0,$$
	where $f$ is given by \eqref{constraint ellipse}. The bifurcation occurs at the angular velocities
		\begin{align}\label{Om-doub}
		\Omega_{\mathtt{m}}^{\pm}(b)\triangleq\frac{1-b^{2}}{4}\pm\frac{1}{2\mathtt{m}}\sqrt{\left(\frac{1-b^2}{2}\mathtt{m}-1\right)^2-b^{2\mathtt{m}}}.
		\end{align}
		It is worthy to point out that the role played by the same function $f$ in the two different  cases (bifurcation from the ellipses and the annulus) is quite mysterious and could be explained through  Joukowsky transformation. 
	As for the degenerate case where  
	\begin{equation}\label{deg bif}
		f(m,b)=0
	\end{equation}
	the situation turns out to be  more delicate to handle. The solutions to  \eqref{deg bif} can be ranged  in the form
	\begin{align}\label{bn-def11}
	\{(2,b),\quad b\in(0,1)\}\qquad\textnormal{or}\qquad\Big\{(n,\underline{b}_n),\quad n\geqslant 3,\quad\underline{b}_{n}\in(0,1)\Big\},
	\end{align}
	where the sequence $(\underline{b}_{n})_{n\geqslant 3}$ is increasing and tends to $1.$  This problem has  been explored by Hmidi and Mateu in  \cite{HM16-2}, where  they show that for $b\in(0,1)\setminus\{\underline{b}_{2p},\,p\geqslant2\}$ there is a trans-critical bifurcation of the 2-fold V-states. However, there is no bifurcation with the $\mathtt{m}$-fold symmetry  for  $b=\underline{b}_{\mathtt{m}},$ $\mathtt{m}\geqslant3.$ Very recently, Wang, Xu and Zhou extended in  \cite{WXZ22} the 2-fold trans-critical bifurcation to the cases  $b=\underline{b}_{4}$ and $b=\underline{b}_{6}.$
	We should also mention that over the past few years  there were a lot of rich  activities on   the construction of  V-states  around more general steady  shapes (multi-connected patches, Thomson polygons, von K\'arm\'an vortex streets, etc...) and for various  active scalar equations (generalized quasi-geostrophic equations, quasi-geostrophic shallow-water equations,  Euler-$\alpha$ equations). For more details, we refer  to \cite{ADPMW20,CLZ21,CQZZ21,CWWZ21,CJ22,EJ19,EJ20,HHH16,HHHM15,DHR19,G20,G21,GH22,GHS20,GGS20,GS19,HH15,HH21,HW21,HM17,HXX22,R17,R21,R22,T85}.
	
	\smallskip
	
	In the current work, we intend to explore  the existence of time quasi-periodic  vortex patches  for \eqref{Euler equations} close to the annulus. Recall that a \textit{quasi-periodic function} is any application $f:\mathbb{R}\rightarrow\mathbb{R}$ which can be written
	$$\forall\, t\in\mathbb{R},\quad f(t)=F(\omega t),$$
	with $F:\mathbb{T}^{d}\rightarrow\mathbb{R},$ where $\mathbb{T}^d$ denotes the flat torus of dimension $d\in\mathbb{N}^*$ and $\omega\in\mathbb{R}^{d}$ a non-degenerate frequency vector, namely
	\begin{equation}\label{nonresonnace omega}
		\forall\, l\in\mathbb{Z}^{d}\setminus\{0\},\quad \omega\cdot l\neq 0.
	\end{equation}
	Observe that the case $d=1$ corresponds to the definition of periodic functions with frequency $\omega\in\mathbb{R}^*.$ This type of functions are the natural solutions of finite dimensional integrable Hamiltonian systems where the phase space is foliated by Lagrangian invariant tori supporting quasi-periodic motion. The Kolmogorov-Arnold-Moser (KAM) theory \cite{A63,K54,M62} asserts that under suitable  regularity and non-degeneracy conditions, most of these invariant tori persist, up to a smooth deformation, under a small Hamiltonian  perturbation. A typical difficulty in the implementation of the KAM method is linked to the \textit{small divisors problems} preventing some intermediate  series to be convergent. The solution, proposed by Kolmogorov, is to introduce Diophantine conditions on the small denominators which lead to a fixed algebraic loss of regularity. This loss can be treated through a classical Newton method in the analytical regularity framework  as proved by Kolmogorov and Arnold. However, this approach turns out to be more involved   in the finitely many differentiable case (for example Sobolev or H\"older spaces). Indeed, to overcome this technical difficulty,  Moser used in \cite{M67}  a regularization of the  Newton method in the spirit of the ideas of Nash implemented in the isometric embedding problem \cite{N54}. Now, such a method is known as  Nash-Moser scheme.
	
	\smallskip
	
	 The search of lower dimensional invariant tori is so relevant not only for finite dimensional Hamiltonian systems but also  for Hamiltonian PDE where this query  is quite natural. Actually,  in the finite dimensional case, this problem has been explored for instance by Moser and P\"{o}schel \cite{M67,P89} leading to new Diophantine conditions called \textit{first and second order Melnikov conditions}. Later on,  the theory has been extended and refined for several Hamiltonian PDE. For example, it has been implemented for the 1D semi-linear wave and Schr\"odinger equations in the following papers \cite{B94,CY00,CW93,K87,P96,P96-1,W90}. Several results were also obtained for semi-linear perturbations of integrable PDE  \cite{BBP13,BBP14,B05,EK10,GYZ11,KP03,K98,K00,LY11}. However, the case of quasi-linear or fully nonlinear perturbations were only explored very recently  in a series of papers  \cite{BBM14,BBM16,BBM16-1,BB15,BKM21,FP14}. A typical example in this direction is given by the water-waves equations which have been the  subject of    rich and intensive activity over the past few years dealing with  the periodic and quasi-periodic solutions, see for instance  \cite{AB11,BBMH18, BFM21,BFM21-1,BM18,IPT05,PT01}.
	 
	 \smallskip
	 
	Concerning the emergence of quasi-periodic structures for   the 2D Euler equations which is known to be Hamiltonian, few results are known in the literature and some interesting developments have been made very recently opening new perspectives around the vortex motion.  
	One of the results on the  smooth case, supplemented with  periodic boundary conditions, goes back to Crouseilles and Faou in  \cite{CF13}. The construction of quasi-periodic  solutions  is founded on the superposition of localized  traveling   solutions without interaction. Notice that no sophisticated tools from KAM theory are required in their approach. Very recently, their result has been extended to higher dimensions  by Enciso, Peralta-Salas and Torres de Lizaur in  \cite{EPT22}.  For Euler equations on  the  3-dimensional torus, Baldi and Montalto \cite{BM21} were able to generate quasi-periodic solutions through  small quasi-periodic forcing terms.
	
	\smallskip
	
Another new  and promising  topic concerns the construction of quasi-periodic vortex patches  to  the system \eqref{Euler equations}  or to various active scalar equations  (generalized surface quasi-geostrophic equations,  quasi-geostrophic shallow-water equations and Euler-$\alpha$ equations) which has been partially explored in the recent papers \cite{HHM21,HR21,R22}. All of them deal with simply-connected quasi-periodic patches near Rankine vortices provided that  the suitable external parameter is selected in a massive Cantor set. We emphasize that for Euler model    there is no natural parameter anymore and one has to create an internal one. Two works have been performed in this direction. The first one is due to Berti, Hassainia and Masmoudi \cite{BHM21} who proved using KAM theory the existence of quasi-periodic patches close to  Kirchhoff ellipses provided that the aspect ratio of the ellipse belongs to a Cantor set.
The second one is obtained by Hassainia and Roulley in \cite{HR21-1}, where the fluid evolves in  the unit disc, and they proved the existence of quasi-periodic patches    close to Rankine vortices $\mathbf{1}_{b\mathbb{D}},$ when $b$ belongs to a suitable Cantor set in  $(0,1)$.

\smallskip

Our main task here is to investigate the emergence of  quasi-periodic  patches (denoted by (QPP)) near   the annulus $A_b$. The motivation behind that is the existence of   time periodic patches around the annulus as stated in  \cite{HHMV16} and one may get (QPP) at the linear level by mixing a finite number of frequencies. Note that the rigidity of the frequencies \eqref{Om-doub}  with respect to the modulus $b$ is an essential element to get the non-degeneracy of the linear torus.   One of  the difficulties in the construction of (QPP) at   the nonlinear level stems from the vectorial structure of the problem because we are dealing with two coupled  interfaces. As we shall see, this leads to more time-space  resonances coming in part from the interaction between  the transport equations advected by  two different speeds.\\
In what follows, we intend to carefully describe the situation around doubly-connected (QPP), then formulate the main result and sketch the principal  ideas of the  proof. First, we consider a modified polar parametrization of the two interfaces of the patch   close to the annulus $A_b$, namely for $k\in\{1,2\}$
	$$ z_k(t,\theta)\triangleq R_k(t,\theta)e^{\ii(\theta-\Omega t)},\quad R_{k}(t,\theta)\triangleq\sqrt{b_{k}^{2}+2r_{k}(t,\theta)},\quad b_1\triangleq 1,\quad b_2\triangleq b.$$
	The unknown is the pair of functions $r=(r_1,r_2)$ of small  radial deformations of the patch. It is worthy to point out that similarly to \cite{BHM21,HHM21,HR21,R22} our parametrization is written in a rotating frame with an angular velocity $\Omega>0.$ Nevertheless, we have multiple reasons here behind the introduction of this auxiliary parameter $\Omega$.    In the previous works, we make appeal to this parameter  to remedy to the degeneracy of the first equilibrium frequency leading to a trivial resonance. In our setting, this parameter is needed   to avoid an exponential accumulation towards a constant  of the unperturbed frequencies (eigenvalues) $\{\mathtt{m}\Omega_{\mathtt{m}}^{\pm}(b), \mathtt{m}\geqslant2\}$, see \eqref{Om-doub}.  This fact   induces a harmful effect especially  related to the second order Melnikov non-resonance condition. Therefore, thanks to the parameter $\Omega$ the eigenvalues will grow linearly with respect to the modes $\mathtt{m}$.  Another useful property induced by large values of $\Omega$ is the   monotonicity  of the eigenvalues, see Lemma \ref{lem-asym}, which allows in turn to get  R\"ussmann conditions on the diagonal part, see   Lemma \ref{lemma transversalityE}-(iv). 
	
	\smallskip
	One of the major difference  with \cite{HHM21,HR21-1,HR21,R22} is the vectorial structure of the system related to the interfaces coupling. Despite that, we are  able to check the Hamiltonian structure in terms of  the contour dynamics equations. In fact, we prove in Lemma \ref{lem eq EDC r} and Proposition \ref{prop HAM eq Edc} that the pair of radial deformations $r=(r_1,r_2)$  solves a system of two coupled nonlinear and nonlocal transport PDE admitting a Hamiltonian formulation in the form
	\begin{equation}\label{Ham eq EDC r1-r2 intro}
		\partial_{t}r=\mathcal{J}\nabla H(r),\qquad\mathcal{J}\triangleq\begin{pmatrix}
			\partial_{\theta} & 0\\
			0 & -\partial_{\theta}
		\end{pmatrix},
	\end{equation}
	where the Hamiltonian $H$ can be recovered from  the kinetic energy and the angular momentum. This Hamiltonian is reversible and invariant under space translations. 
	The linearized operator at a general state $r$ close to the annulus $A_b$ is described in Lemma \ref{lem lin op 1 DCE} and writes
	$$\partial_{t}\begin{pmatrix}
		\rho_{1}\\
		\rho_{2}
	\end{pmatrix}
	=\mathcal{J} \mathbf{M}_r\begin{pmatrix}
		\rho_{1}\\
		\rho_{2}
	\end{pmatrix},\qquad \mathbf{M}_r\triangleq \begin{pmatrix}
		-V_{1}(r)-L_{1,1}({r}) & L_{1,2}(r)\\
		L_{2,1}(r) & V_{2}(r)-{L}_{2,2}(r)
	\end{pmatrix},$$
	where $V_{k}(r)$ are scalar functions and ${L}_{k,n}({r})$ are nonlocal operators given by \eqref{def Vpm}--\eqref{def mathbfLkn}.
	The diagonal terms correspond to the self-induction of each interface. In particular, the operators $L_{k,k},$ for $k\in\{1,2\},$ are of zero order and reflect the planar simply-connected Euler action. For $k\neq n\in\{1,2\},$ the anti-diagonal operators $L_{k,n}$  describe  the interaction between the two boundaries and they are smoothing at any order. In Lemma \ref{lem lin op 2 DCE}, we shall  prove that at the equilibrium $r=(0,0)$, corresponding to the annulus patch, 
each entry of $\mathbf{M}_0$ is a Fourier multiplier and the operator $\mathcal{J}\mathbf{M}_0$ can be written in Fourier expansion  as a superposition of $2\times 2$ matrices,
	$$\mathcal{J}\mathbf{M}_0\begin{pmatrix}
		\rho_{1}\\
		\rho_2
	\end{pmatrix}=\sum_{j\in\mathbb{Z}^*}M_{j}(b,\Omega)\begin{pmatrix}
		\rho_{j,1}\\
		\rho_{j,2}
	\end{pmatrix}\mathbf{e}_{j},\qquad M_{j}(b,\Omega)\triangleq\frac{\ii j}{|j|} \begin{pmatrix}
		-|j|\big(\Omega+\tfrac{1-b^2}{2}\big)+\tfrac{1}{2}& -\tfrac{b^{|j|}}{2}\\
		\tfrac{b^{|j|}}{2} & -|j|\Omega-\tfrac{1}{2}
	\end{pmatrix},$$
	for all  $\rho_1$ and$\rho_2$ with Fourier expansion 
	$$
	\rho_k=\sum_{j\in\mathtt{m}\mathbb{Z}^*}\rho_{j,k}\mathbf{e}_j\quad\textnormal{s.t.}\quad \rho_{-j,k}=\overline{\rho_{j,k}},\qquad\; \mathbf{e}_{j}(\theta)\triangleq e^{\ii j\theta}.
	$$
	The spectrum of $M_{j}(b,\Omega)$ is 
	$$\sigma\big(M_{j}(b,\Omega)\big)=\big\{-\ii\Omega_{j,1}(b),-\ii\Omega_{j,2}(b)\big\},\qquad\Omega_{j,k}(b)\triangleq\frac{j}{|j|}\bigg[\big(\Omega+\tfrac{1-b^2}{4}\big)|j|-\ii^{\mathtt{H}\big(\Delta_{j}(b)\big)} \tfrac{(-1)^{k}}{2}\sqrt{|\Delta_{j}(b)|}\bigg],$$
	with $\mathtt{H}\triangleq \mathbf{1}_{[0,\infty)}$ the Heaviside function and 
	$$\Delta_{j}(b)\triangleq b^{2|j|}-\big(\tfrac{1-b^2}{2}|j|-1\big)^2.$$
	At this stage, we shall restrict the discussion to $\mathtt{m}$-fold symmetric structures  for some integer $\mathtt{m}\geqslant 3$ large enough. This is done for several reasons.  First, the mode $j=2$ corresponds to a double root  for any $b\in (0,1),$ because  $\Delta_{2}(b)=0,$ implying a nontrivial resonance that we cannot remove using the parameter $b$ but simply  by imposing higher symmetry for the (QPP). Second,  the hyperbolic spectrum, associated to non-zero real part for the eigenvalues, that could generate instabilities and time growth emerge only for lower symmetries. We believe that with this latter configuration, one can still hope to construct (QPP) by inserting the hyperbolic modes on the normal directions as it was recently performed in \cite{BHM21}. We refer for instance the reader to \cite{BB12,LS19,EGK16,G74,PP15,Z75,Z76} for an introduction to hyperbolic KAM theory in finite or infinite dimension.
	
	\smallskip
	
	Now, we fix $b^*\in(0,1)$ and set
	\begin{equation}\label{def m star intr}
	\mathtt{m}^*\triangleq\min\big\{n\geqslant3\quad\textnormal{s.t.}\quad\underline{b}_n>b^*\big\},
	\end{equation}
	where the sequence $\underline{b}_n$ defined in \eqref{bn-def11}.
	Then, for any integer $|j|\geqslant\mathtt{m}^*$, we have $\Delta_{j}(b)<0.$ Hence, the quantity $\Omega_{j,k}(b)$ is real and the matrix  $M_{j}(b,\Omega)$ has pure imaginary spectrum. The restriction of the Fourier modes to the lattice $\mathbb{Z}_{\mathtt{m}}^*\triangleq\mathtt{m}\mathbb{Z}\setminus\{0\}$ with $\mathtt{m}\geqslant\mathtt{m}^*$ allows to eliminate the hyperbolic modes.   
	At this stage, we find it  convenient to work with  new coordinates where the linearized operator at the equilibrium state  is governed by a diagonal matrix Fourier multiplier   operator. This can be done through the diagonalization of each the matrix $M_j(b,\Omega).$  To do that, we use the  symplectic transformation (with respect to $\mathcal{W}$)  $\mathbf{Q}$ taking the form
	$$\mathbf{Q}\begin{pmatrix}
		\rho_1\\
		\rho_2
	\end{pmatrix}=\sum_{j\in\mathbb{Z}_{\mathtt{m}}^*}\mathbf{Q}_j\begin{pmatrix}
		\rho_{j,1}\\
		\rho_{j,2}
	\end{pmatrix}\mathbf{e}_{j},\qquad\mathbf{Q}_j\triangleq \tfrac{-1}{\sqrt{1-a_{j}^2(b)}} \begin{pmatrix}
		-1 &  a_{j} (b) \\
		 a_{j}(b)& -1 
	\end{pmatrix}, 
	$$
	where
	\begin{equation}\label{def coef ajb}
	a_{j}(b)\triangleq\frac{b^{|j|}}{\tfrac{1-b^2}{2}|j|-1+ \sqrt{\big(\tfrac{1-b^2}{2}|j|-1\big)^2-b^{2|j|}}}\in(0,1),
	\end{equation}
	(see Corollary \ref{coro-equilib-freq} for more details on the bound of $a_{j}(b)$)  such that
	$$\mathbf{Q}^{-1}\mathcal{J} \mathbf{M}_0\mathbf{Q}= \mathcal{J} \mathbf{L}_0, \qquad  \mathbf{L}_0\begin{pmatrix}
		\rho_{1}\\
		\rho_{2}
	\end{pmatrix}
	\triangleq \sum_{j\in\mathbb{Z}_{\mathtt{m}}^*}\tfrac{1}{j}\begin{pmatrix}
		-\Omega_{j,1}(b) & 0\\
		0 & \Omega_{j,2}(b)
	\end{pmatrix} \begin{pmatrix}
		\rho_{j,1}\\
		\rho_{j,2}
	\end{pmatrix} \mathbf{e}_{j}.$$
		Notice that $\ii\Omega_{j,1}(b)$ and $\ii\Omega_{j,2}(b)$ are not complex conjugate, thus the dynamics cannot be reduced to one scalar equation associated with a complex variable unlike the water-waves \cite{BBMH18,BFM21,BFM21-1,BM18} situation. 
	The new Hamiltonian system through the symplectic transformation $r\mapsto\mathbf{Q}r$
	$$\partial_{t}r=\mathcal{J}\nabla K(r),\qquad K(r)\triangleq H(\mathbf{Q}r)$$
	whose linearization at the trivial solution has a good normal form
	\begin{equation}\label{KL0 intro}
		\partial_t \rho=\mathcal{J}\nabla K_{\mathbf{L}_0}(\rho),\qquad K_{\mathbf{L}_0}(\rho)\triangleq \tfrac{1}{2}\big\langle\mathbf{L}_0\rho,\rho\big\rangle_{L^2(\mathbb{T})\times L^2(\mathbb{T})}=-\sum_{j\in\mathbb{Z}_{\mathtt{m}}^*}\left(\tfrac{\Omega_{j,1}(b)}{2j}|\rho_{j,1}|^2-\tfrac{\Omega_{j,2}(b)}{2j}|\rho_{j,2}|^2\right).
	\end{equation}
Consider two disjoint finite sets of Fourier modes
	\begin{equation}\label{tang sets intro}
		\mathbb{S}_1,\mathbb{S}_2\subset\mathtt{m}\mathbb{N}^*,\qquad\textnormal{with}\qquad |\mathbb{S}_1|=d_1<\infty,\qquad|\mathbb{S}_{2}|=d_2<\infty\qquad\textnormal{and}\qquad \mathbb{S}_1\cap\mathbb{S}_2=\varnothing.
	\end{equation}
	Then, from Lemma \ref{lemma sol Eq}, we deduce that for any $0<b^*<1,$  $\Omega>0$ and $r_{j,1},r_{j,2}\in\mathbb{R}^*$, for almost all $b\in[0,b^*],$ any function in the form
	$$r(t,\theta)=\sum_{j\in{\mathbb{S}}_1}\tfrac{r_{j,1}}{\sqrt{1-a_{j}^2(b)}} \begin{pmatrix}
		1  \\
		- a_{j}(b) 
	\end{pmatrix}\cos\big(j\theta-\Omega_{j,1}(b)t\big)+\sum_{j\in{\mathbb{S}}_2}\tfrac{r_{j,2}}{\sqrt{1-a_{j}^2(b)}} \begin{pmatrix}
		-a_{j} (b) \\
		1 
	\end{pmatrix}\cos\big(j\theta-\Omega_{j,2}(b)t\big)$$
	 is a quasi-periodic solution with frequency 
	\begin{equation}\label{def omegaEq intro}
		\omega_{\textnormal{Eq}}(b)\triangleq\Big(\big(\Omega_{j,1}(b)\big)_{j\in\mathbb{S}_1},\big(\Omega_{j,2}(b)\big)_{j\in\mathbb{S}_2}\Big)
	\end{equation}
	of the original linearized equation $\partial_{t} r=\mathcal{J}\mathbf{M}_0r$ which is $\mathtt{m}$-fold and reversible, namely $r(-t,-\theta)=r(t,\theta)=r\big(t,\theta+\tfrac{2\pi}{\mathtt{m}}\big).$ 
	Our main result states that these structures persist at the non-linear level. More precisely, we have the following theorem.
	\begin{theo}\label{thm QPS E}
		Let $0<b_*<b^*<1$ and fix $\mathtt{m}\in\mathbb{N}$ with $\mathtt{m}\geqslant\mathtt{m}^*,$ where   $\mathtt{m}^*$ defined in \eqref{def m star intr}.	
		There exists $\Omega_{\mathtt{m}}^{*}\triangleq\Omega(b^*,\mathtt{m})>0$ satisfying
		$$\lim_{\mathtt{m}\to\infty}\Omega_{\mathtt{m}}^*=0$$
		such that for any $\Omega>\Omega_{\mathtt{m}}^*,$ there exists $\varepsilon_{0}\in(0,1)$ small enough with the following properties :  For every amplitudes 
		$$\mathtt{a}=\big((\mathtt{a}_{j,1})_{j\in\mathbb{S}_1},(\mathtt{a}_{j,2})_{j\in\mathbb{S}_2}\big)\in(\mathbb{R}_{+}^{*})^{d_1+d_2}\qquad\textnormal{satisfying}\qquad|{\mathtt{a}}|\leqslant\varepsilon_{0},$$ 
		there exists a Cantor-like set 
		$$\mathcal{C}_{\infty}^{\mathtt{a}}\subset(b_{*},b^{*}),\qquad\textnormal{with}\qquad\lim_{{\mathtt{a}}\rightarrow 0}|\mathcal{C}_{\infty}^{\mathtt{a}}|=b^{*}-b_{*},$$
		such that for any $b\in\mathcal{C}_{\infty}^{\mathtt{a}}$, the planar Euler equations \eqref{Euler equations} admit a $\mathtt{m}$-fold time quasi-periodic  doubly-connected vortex patch solution in the form
		$$\boldsymbol{\omega}(t,\cdot)=\mathbf{1}_{D_t},\qquad D_{t}=\Big\{\ell e^{\ii(\theta-\Omega t)},\quad\theta\in[0,2\pi],\quad \sqrt{b^2+2r_2(t,\theta)}<\ell<\sqrt{1+2r_1(t,\theta)}\Big\},$$
		where
		\begin{align*}
			\begin{pmatrix}
				r_1\\
				r_2
			\end{pmatrix}(t,\theta)&=\sum_{j\in{\mathbb{S}}_1}\tfrac{\mathtt{a}_{j,1}\cos(j\theta-\omega_{j,1}(b,\mathtt{a})t)}{\sqrt{1-a_{j}^2(b)}} \begin{pmatrix}
				1  \\
				- a_{j}(b) 
			\end{pmatrix}+\sum_{j\in{\mathbb{S}}_2}\tfrac{\mathtt{a}_{j,2}\cos(j\theta-\omega_{j,2}(b,\mathtt{a})t)}{\sqrt{1-a_{j}^2(b)}} \begin{pmatrix}
				-a_{j}(b) \\
				1 
			\end{pmatrix}+\mathtt{p}\big(\omega_{\textnormal{\tiny{pe}}}(b,\mathtt{a})t,\theta\big)
		\end{align*}
		and $a_{j}(b)$ are given by \eqref{def coef ajb}.
		This solution is associated with a non-resonant frequency vector $${\omega}_{\tiny{\textnormal{pe}}}(b,{\mathtt{a}})\triangleq \Big(\big(\omega_{j,1}(b,{\mathtt{a}})\big)_{j\in\mathbb{S}_1},\big(\omega_{j,2}(b,{\mathtt{a}})\big)_{j\in\mathbb{S}_2}\Big)\in\mathbb{R}^{d_1+d_2}$$
		satisfying the convergence
		$$\mathtt{\omega}_{\tiny{\textnormal{pe}}}(b,{\mathtt{a}})\underset{\mathtt{a}\rightarrow 0}{\longrightarrow}-\Big(\big(\Omega_{j,1}(b)\big)_{j\in\mathbb{S}_1},\big(\Omega_{j,2}(b)\big)_{j\in\mathbb{S}_2}\Big),$$
		where $\Omega_{j,1}(b)$ and $\Omega_{j,2}(b)$ are the equilibrium frequencies. 
		The perturbation $\mathtt{p}:\T^{d_1+d_2+1}\to\mathbb{R}^2$ is a function satisfying the symmetry properties 
		$$\forall(\varphi,\theta)\in\mathbb{T}^{d_1+d_2+1},\quad\mathtt{p}(-\varphi,-\theta)=\mathtt{p}(\varphi,\theta)=\mathtt{p}(\varphi,\theta+\tfrac{2\pi}{\mathtt{m}})$$
		and for some large index of Sobolev regularity $s$ it satisfies  the estimate
		$$\|\mathtt{p}\|_{H^{{s}}(\mathbb{T}^{d_1+d_2+1},\mathbb{R}^2)}\underset{{\mathtt{a}}\rightarrow 0}{=}o(|{\mathtt{a}}|).$$
	\end{theo}
\begin{remark}
	The lower bound restriction $b_*$ is required because the operators $L_{k,n}$, $V_{k,n}$ may become singular when $b=0,$ a situation which corresponds to the simply-connected case. 
\end{remark}
\begin{remark}
From this theorem,  we obtain global non trivial  solutions in the patch form which are confined around the annulus.  More studies on vortex confinement  can be found in \cite{ILN03,ISG99,M94}. 
\end{remark}

	We shall now briefly describe the main steps of the proof whose  general strategy  is borrowed from the Nash-Moser approach for KAM theory developed by Berti-Bolle \cite{BB15} and slightly modified in \cite[Sec. 6]{HHM21}. 
	Recall that the Nash-Moser scheme requires to invert the linearized operator in a neighborhood of the equilibrium state and the inverse operator must satisfy suitable tame estimates in the framework of Sobolev spaces. The first step that we intend to describe now is to reformulate the problem in terms of embedded tori. Remark that the Hamiltonian system associated with $K$ is a quasi-linear perturbation of its linearization at the equilibrium state, namely
	$$\partial_{t}r=\mathcal{J}\mathbf{L}_0r+X_{P}(r),\qquad\textnormal{with}\qquad X_{P}(r)\triangleq \mathcal{J}\nabla K({r})-\mathcal{J}\mathbf{L}_0r=\mathbf{Q}^{-1}X_{H\geqslant 3}(\mathbf{Q}{r}),$$
	where
	$$X_{H\geqslant 3}({r})\triangleq \mathcal{J}\big(\nabla H(r)-\mathbf{M}_0r\big).$$
	Under the rescaling $r\mapsto\varepsilon r$ and the quasi-periodic framework $\partial_{t}\leftrightarrow\omega\cdot\partial_{\varphi}$, the Hamiltonian system becomes
	$$\omega\cdot\partial_\varphi r=\mathcal{J}\mathbf{L}_0r+\varepsilon X_{P_{\varepsilon}}(r),$$
	where $X_{P_{\varepsilon}}$ is the rescaled Hamiltonian vector field defined by
	$X_{P_{\varepsilon}}(r)\triangleq \varepsilon^{-2}X_{P}(\varepsilon r).$
	Notice  that the previous equation is generated by
	the rescaled Hamiltonian  
	$$\mathcal{K}_{\varepsilon}(r)\triangleq\varepsilon^{-2}K(\varepsilon r)= K_{\mathbf{L}_0}(r)+\varepsilon P_{\varepsilon}(r),$$
	with $K_{\mathbf{L}_0}$ as in \eqref{KL0 intro} and 
	$\varepsilon P_{\varepsilon}(r)$ describes all the terms
	of higher order more than cubic.
	We consider two finite sets $\mathbb{S}_1$, $\mathbb{S}_2$ as in \eqref{tang sets intro} and we denote
	$d_1\triangleq|\mathbb{S}_1|,\, d_2\triangleq|\mathbb{S}_2|,\, d\triangleq d_1+d_2$
	and
	$$\overline{\mathbb{S}}_k\triangleq\mathbb{S}_k\cup(-\mathbb{S}_k),\qquad \overline{\mathbb{S}}_{0,k}\triangleq\overline{\mathbb{S}}_k\cup\{0\}.$$
	and set
	$$\mathbb{S}\triangleq \mathbb{S}_{1}\cup\mathbb{S}_{2},\quad  
		\overline{\mathbb{S}}\triangleq \mathbb{S}\cup(-\mathbb{S}), \quad \overline{\mathbb{S}}_0\triangleq \overline{\mathbb{S}}\cup \{0\}.$$
	Next, we decompose the phase space into the following orthogonal sum
	\begin{equation}\label{decomp prod l2}
	L_{\mathtt{m}}^2(\mathbb{T})\times L_{\mathtt{m}}^2(\mathbb{T})=\mathbf{H}_{\overline{\mathbb{S}}}\overset{\perp}{\oplus}\mathbf{H}_{\overline{\mathbb{S}_0}}^{\perp},\quad L_{\mathtt{m}}^2(\mathbb{T})\triangleq \bigg\{f=\sum_{j\in\mathbb{Z}_{\mathtt{m}}^*}f_{j}\mathbf{e}_j\quad\textnormal{s.t.}\quad f_{-j}=\overline{f_j},\quad\sum_{j\in\mathbb{Z}^*_{\mathtt{m}}}|f_j|^2<+\infty\bigg\},
	\end{equation}
	with 
	\begin{align*}
		\mathbf{H}_{\overline{\mathbb{S}}}&\triangleq \left\{\sum_{j\in\overline{\mathbb{S}}_1}v_{j,1}\begin{pmatrix}
			1\\
			0
		\end{pmatrix}\mathbf{e}_{j}+\sum_{j\in\overline{\mathbb{S}}_2}v_{j,2}\begin{pmatrix}
			0\\
			1
		\end{pmatrix}\mathbf{e}_{j},\quad v_{j,k}\in\mathbb{C},\,\quad\overline{v_{j,k}}=v_{-j,k}\right\},\\
		\mathbf{H}_{\overline{\mathbb{S}}_0}^{\perp}&\triangleq \left\{\sum_{j\in\mathbb{Z}_{\mathtt{m}}\setminus\overline{\mathbb{S}}_{0,1}}z_{j,1}\begin{pmatrix}
			1\\
			0
		\end{pmatrix}\mathbf{e}_{j}+\sum_{j\in\mathbb{Z}_{\mathtt{m}}\setminus\overline{\mathbb{S}}_{0,2}}z_{j,2}\begin{pmatrix}
			0\\
			1
		\end{pmatrix}\mathbf{e}_{j},\quad z_{j,k}\in\mathbb{C},\,\quad\overline{z_{j,k}}=z_{-j,k}\right\}.
	\end{align*}
	The sets $\mathbf{H}_{\overline{\mathbb{S}}}$ and $\mathbf{H}_{\overline{\mathbb{S}}_0}^{\perp}$ are called \textit{tangential} and \textit{normal subspaces}, respectively. The associated projections are defined through
	\begin{equation}\label{proj-nor1}
	\Pi_{\overline{\mathbb{S}}}\triangleq\begin{pmatrix}
		\Pi_{1} & 0\\
		0 & \Pi_{2}
	\end{pmatrix}\qquad\textnormal{and}\qquad\Pi_{\overline{\mathbb{S}}_0}^{\perp}\triangleq\begin{pmatrix}
		\Pi_{1}^{\perp} & 0\\
		0 & \Pi_{2}^{\perp}
	\end{pmatrix}
	\end{equation}
	namely,
	\begin{equation}\label{PI1-PI2}
	\forall k\in\{1,2\},\quad\Pi_k \sum_{j\in \mathbb{Z}_{\mathtt{m}}^*} v_{j,k}\mathbf{e}_{j}\triangleq\sum_{j\in \overline{\mathbb{S}}_k} v_{j,k}\mathbf{e}_{j}\qquad\textnormal{and}\qquad\Pi_{k}^{\perp}\triangleq\mathbb{I}_{\mathtt{m}}-\Pi_{k},
	\end{equation}
	where $\mathbb{I}_{\mathtt{m}}$ is the identity map of $L_{\mathtt{m}}^{2}(\mathbb{T}).$ On the tangential set $\mathbf{H}_{\overline{\mathbb{S}}}$, we introduce the action-angle variables 
	$$\vartheta\triangleq\big((\vartheta_{j,1})_{j\in{\mathbb{S}}_1},(\vartheta_{j,2})_{j\in{\mathbb{S}}_2}\big),\qquad I\triangleq\big((I_{j,1})_{j\in{\mathbb{S}}_1}, (I_{j,2})_{j\in{\mathbb{S}}_2}\big)$$
	as follows : Fix any amplitudes $(\mathtt{a}_{j,k})_{j\in \overline{\mathbb{S}}_k}$ such that for any  $ j\in\overline{\mathbb{S}}_k,\,\mathtt{a}_{-j,k}=\mathtt{a}_{j,k}>0$ and set 
	$$
	\forall k\in\{1,2\}, \quad \forall j\in\overline{\mathbb{S}}_k,\quad  v_{j,k}\triangleq \sqrt{(\mathtt{a}_{j,k})^2+|j|I_{j,k}}e^{\ii \vartheta_{j,k}},
	$$
	supplemented with the symmetry properties 
	$$\forall k\in\{1,2\},\quad\forall j\in\mathbb{S}_k,\quad I_{-j,k}=I_{j,k}\in\mathbb{R}\qquad\textnormal{and}\qquad\vartheta_{-j,k}=-\vartheta_{j,k}\in\mathbb{T}.$$
	Therefore, we have the following decomposition of $r=(r_1,r_2)$,
	\begin{equation}\label{aa-coord-00}
		r=\mathbf{A}(\vartheta,I,z)\triangleq v(\vartheta,I)+z,
	\end{equation}
	where $z\in \mathbf{H}_{\overline{\mathbb{S}}_0}^{\perp}$ and
	$$v (\vartheta, I)\triangleq \sum_{j\in\overline{\mathbb{S}}_1}   \sqrt{(\mathtt{a}_{j,1})^2+|j|I_{j,1}}e^{\ii \vartheta_{j,1}}\begin{pmatrix} 
		1\\
		0
	\end{pmatrix}\mathbf{e}_{j}+ \sum_{j\in\overline{\mathbb{S}}_2}\sqrt{(\mathtt{a}_{j,2})^2+|j|I_{j,2}}e^{\ii \vartheta_{j,2}} 
	\begin{pmatrix} 
		0\\
		1
	\end{pmatrix}\mathbf{e}_{j}\in \mathbf{H}_{\overline{\mathbb{S}}}.$$
	The transformation $\mathbf{A}$ is $\mathcal{W}$-symplectic and, in the new variables, the new Hamiltonian $K_{\varepsilon}\triangleq \mathcal{K}_{\varepsilon}\circ\mathbf{A}$ writes
	$$K_{\varepsilon} = 
	-\big({\mathtt J}\,{\omega}_{\textnormal{Eq}}(b)\big)\cdot I  + \tfrac12  \big\langle\mathbf{L}_0\, z, z \big\rangle_{L^2(\mathbb{T})\times L^2(\mathbb{T})} + \varepsilon \mathcal{ P}_{\varepsilon},\qquad{\mathtt J}\triangleq  \begin{pmatrix}
		{\rm I}_{d_1}& 0 \\
		0 & -{\rm I}_{d_2}  \end{pmatrix},  \qquad \mathcal{ P}_{\varepsilon} \triangleq    P_\varepsilon \circ  \mathbf{A}.$$
	Observe that the Poisson structure is associated with $\mathtt{J}$, and will be  needed later during the implementation of  Berti-Bolle approach. The corresponding Hamiltonian vector field is
	$$X_{K_{\varepsilon}} \triangleq  
	\big({\mathtt J}\partial_I K_{\varepsilon} , -{\mathtt J} \partial_\vartheta K_{\varepsilon} , \Pi_{\overline{\mathbb{S}}_0}^\bot
	\mathcal{J}\nabla_{z} K_{\varepsilon}\big).$$
	Therefore, the problem is reduced to looking for embedded invariant tori
	$$i :\mathbb{T}^d \rightarrow
	\mathbb{R}^d \times \mathbb{R}^d \times  \mathbf{H}^{\perp}_{\overline{\mathbb{S}}_0} 
	\,, \qquad \varphi \mapsto i(\varphi)\triangleq  \big(\vartheta(\varphi), I(\varphi), z(\varphi)\big),$$
	solution of the equation
	$$\omega\cdot\partial_{\varphi}i(\varphi)=X_{K_{\varepsilon}}\big(i(\varphi)\big).
	$$
	As observed in \cite{BB15,M67}, it turns out to be convenient along Nash-Moser scheme to work with one degree of freedom vector-valued parameter $\alpha\in\R^d$ which provides at the end of the scheme a solution for the original problem when it is fixed  to $-\mathtt{J}\omega_{\textnormal{Eq}}(b).$ Therefore, we shall consider the following $\alpha$-dependent family of Hamiltonians
	$$K_\varepsilon^\alpha \triangleq\alpha\cdot I+\tfrac12\big\langle\mathbf{L}_0\, z, z\big\rangle_{L^2(\mathbb{T})\times L^2(\mathbb{T})}+\varepsilon\mathcal{P}_{\varepsilon}$$
	and we search for the zeros of the following functional
	$$\mathcal{F}(i,\alpha,b,\omega,\varepsilon)\triangleq\omega\cdot\partial_{\varphi}i(\varphi)-X_{K_\varepsilon^\alpha}\big(i(\varphi)\big)=\left(\begin{array}{c}
		\omega\cdot\partial_{\varphi}\vartheta(\varphi)-\mathtt{J}\big(\alpha-\varepsilon\partial_{I}\mathcal{P}_{\varepsilon}(i(\varphi))\big)\\
		\omega\cdot\partial_{\varphi}I(\varphi)+\varepsilon\mathtt{J}\partial_{\vartheta}\mathcal{P}_{\varepsilon}(i(\varphi))\\
		\omega\cdot\partial_{\varphi}z(\varphi)-\mathcal{J}\big[\mathbf{L}_0z(\varphi)+\varepsilon\nabla_{z}\mathcal{P}_{\varepsilon}\big(i(\varphi)\big)\big]
	\end{array}\right).$$
	At each step of  Nash-Moser scheme, we have to linearize  this  functional at a small reversible embedded torus $i_0:\varphi\mapsto\big(\vartheta_0(\varphi),I_0(\varphi),z_0(\varphi)\big)$ and $\alpha_0\in\R^d,$ then we need  to construct an approximate right inverse of $d_{i,\alpha}\mathcal{F}(i_0,\alpha_0).$ The core of the Berti-Bolle theory is to conjugate the linearized operator by a linear diffeomorphism $G_0$ of the toroidal phase space $\mathbb{T}^d\times\mathbb{R}^d\times \mathbf{H}_{\overline{\mathbb{S}}_0}^{\perp}$ in order to obtain a triangular system in the action-angle-normal variables up to error terms.  Notice that in a similar way to \cite{HHM21}, we do not use isotropic tori. Therefore, in this framework, inverting the triangular system amounts to inverting  the linearized operator in the normal directions, denoted by  $\widehat{\mathcal{L}}.$ This latter fact is analyzed  in  Section \ref{reduction} and uses KAM reducibility techniques similarly to \cite{BBMH18,BM18,HHM21,HR21} that we shall now explain and extract the main new difficulties.
	According to Proposition \ref{lemma-normal-s}, we can write
	$$\widehat{\mathcal{L}}=\Pi_{\overline{\mathbb{S}}_0}^{\perp}\left(\mathcal{L}-\varepsilon\partial_\theta\mathcal{R}\right)\Pi_{\overline{\mathbb{S}}_0}^{\perp},\qquad \mathcal{L}\triangleq\omega\cdot\partial_{\varphi}\mathbf{I}_{\mathtt{m}}+\mathfrak{L}_{\varepsilon r},\qquad\mathbf{I}_{\mathtt{m}}\triangleq \begin{pmatrix}
		\mathbb{I}_{\mathtt{m}} &0\\
		0& \mathbb{I}_{\mathtt{m}}
	\end{pmatrix},\qquad \mathcal{R}\triangleq \begin{pmatrix} \mathcal{T}_{J_{1,1}}({r}) & \mathcal{T}_{J_{1,2}}({r})\\
		\mathcal{T}_{J_{2,1}}({r}) & \mathcal{T}_{J_{2,2}}({r})
	\end{pmatrix},
	$$
	where for any $(k,\ell)\in\{1,2\}^2$, $\mathcal{T}_{J_{k,\ell}}$ is a  $\mathtt{m}$-fold and reversibility preserving integral operator with smooth kernel $J_{k,\ell}$ and $\mathfrak{L}_{\varepsilon r}$ is the linearized operator 
	$$\begin{aligned}
		\mathfrak{L}_{r} =& \begin{pmatrix}
			\partial_\theta\big(\mathcal{V}_1(r)\, \cdot\big) +\frac{1}{2}\mathcal{H}+\partial_\theta\mathcal{Q}\ast\cdot&0\\
			0& \partial_\theta\big(\mathcal{V}_2(r)\, \cdot\big)-\tfrac{1}{2}\mathcal{H} -\partial_\theta \mathcal{Q}\ast\cdot
		\end{pmatrix} +\partial_\theta\begin{pmatrix}
			\mathcal{T}_{\mathscr{K}_{1,1}}(r)& \mathcal{T}_{\mathscr{K}_{1,2}}(r)\\
			\mathcal{T}_{\mathscr{K}_{2,1}}(r) & \mathcal{T}_{\mathscr{K}_{2,2}}(r)
		\end{pmatrix},
	\end{aligned}$$
	where we denote by $\mathcal{H}$ the $2\pi$-periodic Hilbert transform and $\mathcal{V}_{k}(r)$, $k\in\{1,2\}$, are scalar functions. The convolution operator $\mathcal{Q}\ast\cdot$ has even smooth kernel $\mathcal{Q}$. 
	For $k,n\in\{1,2\},$ the operator $\mathcal{T}_{\mathscr{K}_{k,n}}({r})$ is an integral  operator with smooth, $\mathtt{m}$-fold and reversibility preserving kernel $\mathscr{K}_{k,n}(r),$ see Proposition~\ref{prop:conjP}.

	First, following the KAM reducibility scheme in \cite{BM21,FGMP19,HR21}, we can reduce the transport part and the zero order part by conjugating by a quasi-periodic symplectic  change of variables in the form
	$$\mathscr{B}\triangleq\begin{pmatrix}
		\mathscr{B}_1 & 0\\
		0 & \mathscr{B}_2
	\end{pmatrix},\qquad\forall k\in\{1,2\},\quad\mathscr{B}_k\rho(\mu,\varphi,\theta)=\big(1+\partial_{\theta}\beta_k(\mu,\varphi,\theta)\big)\rho\big(\mu,\varphi,\theta+\beta_k(\mu,\varphi,\theta)\big).$$
	More precisely, as stated in Propositions \ref{prop strighten} and \ref{prop RTNL}, we can find two functions $\mathtt{c}_1\triangleq\mathtt{c}_1(b,\omega,i_0)$,  $\mathtt{c}_2\triangleq\mathtt{c}_2(b,\omega,i_0)$ and a Cantor set
	$$\mathcal{O}_{\infty,n}^{\gamma,\tau_1}(i_0)\triangleq\bigcap_{k\in\{1,2\}\atop\underset{|l|\leqslant N_{n}}{(l,j)\in\mathbb{Z}^{d}\times\mathbb{Z}\setminus\{(0,0)\}}}\Big\{(b,\omega)\in\mathcal{O}\quad\textnormal{s.t.}\quad\big|\omega\cdot l+j\mathtt{c}_{k}(b,\omega,i_0)\big|>\tfrac{4\gamma^{\upsilon}\langle j\rangle}{\langle l\rangle^{\tau_1}}\Big\},$$
	where $N_n\triangleq N_0^{(\frac{3}{2})^n}$ with $N_0\gg 1,$ in which the following decomposition holds
	$${\mathscr{B}}^{-1}\big(\omega\cdot\partial_{\varphi} {\mathbf{I}}_{\mathtt{m}}+\mathfrak{L}_{\varepsilon r}\big) {\mathscr{B}}=\omega\cdot\partial_{\varphi} \mathbf{I}_{\mathtt{m}}+{\mathscr{D}}+{\mathscr{R}}+{\mathscr{E}}_{n},$$
	where
	$$\mathscr{D}\triangleq \begin{pmatrix}
		\mathtt{c}_1\partial_{\theta}\, +\tfrac12\mathcal{H}+\partial_\theta  \mathcal{Q}\ast\cdot& 0\\
		0 &  \mathtt{c}_2\partial_{\theta}\, -\big(\tfrac12\mathcal{H}+\partial_\theta  \mathcal{Q}\ast\cdot)
	\end{pmatrix},$$
	and $\mathscr{R}\triangleq\mathscr{R}(\varepsilon r)$ is a real, $\mathtt{m}$-fold and reversibility preserving Toeplitz in time matrix integral operator enjoying good smallness properties. The operator $\mathscr{E}_n$ is of order one but with  small coefficients decaying faster in $n.$ 
	The next step deals with the localization effects  on  the normal modes. We first introduce  the operator
	$$\mathscr{B}_{\perp}\triangleq\Pi_{\overline{\mathbb{S}}_0}^{\perp}\mathscr{B}\Pi_{\overline{\mathbb{S}}_0}^{\perp}=\begin{pmatrix}
		\Pi_{1}^{\perp}\mathscr{B}_1\Pi_1^{\perp} & 0\\
		0 & \Pi_2^{\perp}\mathscr{B}_2\Pi_2^{\perp}
	\end{pmatrix}.$$
	Then, according to Proposition \ref{prop proj nor dir}, we prove by restricting the parameters to  the set  $\mathcal{O}_{\infty,n}^{\gamma,\tau_1}(i_{0})$ that  for any $n\in\mathbb{N}^{*}$ we have  the identity 
	$$\mathscr{B}_{\perp}^{-1}\widehat{{\mathcal{L}}}{\mathscr{B}}_{\perp}={\mathscr{L}}_{0}+{\mathscr{E}}_{n}^0,\qquad
	{\mathscr{L}}_{0}\triangleq \omega\cdot\partial_{\varphi}\mathbf{I}_{\mathtt{m},\perp}+{\mathscr{D}}_{0}+{\mathscr{R}}_{0},$$
	where $\mathbf{I}_{\mathtt{m},\perp}\triangleq \Pi_{\overline{\mathbb{S}}_{0}}^{\perp}\mathbf{I}_{\mathtt{m}}$ 
	and ${\mathscr{D}}_{0}=\Pi_{\overline{\mathbb{S}}_{0}}^{\perp} {\mathscr{D}}_{0}\Pi_{\overline{\mathbb{S}}_{0}}^{\perp}$ is an $\mathtt{m}$-fold preserving and reversible matrix Fourier multiplier operator in the form
	$${\mathscr{D}}_{0}\triangleq \begin{pmatrix}
		\mathscr{D}_{0,1} & 0\\
		0 & \mathscr{D}_{0,2}
	\end{pmatrix},$$
	with 
	$$\forall k\in\{1,2\},\quad\mathscr{D}_{0,k}\triangleq \left(\ii\mu_{j,k}^{(0)}\right)_{j\in\Z_{\m}\backslash\overline{\mathbb{S}}_{0,k}},\qquad\mu_{j,k}^{(0)}(b,\omega,i_{0})\triangleq \Omega_{j,k}(b)+j\big(\mathtt{c}_k(b,\omega)-\mathtt{v}_k(b)\big)$$
	and $\mathscr{R}_0=\Pi_{\overline{\mathbb{S}}_0}^{\perp}\mathscr{R}_0\Pi_{\overline{\mathbb{S}}_0}^{\perp}$ is a small real, $\mathtt{m}$-fold preserving and reversible Toeplitz in time matrix remainder whose entries are integral operators with smooth kernels. The error term $\mathscr{E}_n^0$ plays a similar role as the previous one  $\mathscr{E}_n$. The next goal is to implement a KAM reduction of the remainder term $\mathscr{R}_0.$ This is done in a new hybrid operator topology treating the diagonal and anti-diagonal terms differently. Along the scheme, the diagonal part is treated as in the scalar situation through the use of the off-diagonal Toeplitz norm, see for instance \cite[Prop. 6.5]{HR21}, whereas the anti-diagonal part, which is smoothing at any order in the spatial variable,  is studied in an isotropic topology. We refer to Section \ref{sec matrix op} for more details on this topological framework, in particular \eqref{hyb nor}.
	We point out that, thanks to the nice structure of the 2D-Euler equation, the diagonal and anti-diagonal terms of the remainder term $\mathscr{R}_0$ are smoothing at any order in the spatial variable  and therefore both  can be studied using the isotropic topology. However, this fact is not true for other transport models \cite{HHM21,HR21} where the remainders on the diagonal are not highly smoothing but of negative order. For this reason, we prefer to work in the most general framework.
	 The Proposition \ref{prop RR} states that we can find an operator $\Phi_{\infty}$ such that in the following Cantor set gathering both diagonal and anti-diagonal second order Melnikov conditions
	\begin{align*}
		&\mathscr{O}_{\infty,n}^{\gamma,\tau_{1},\tau_{2}}(i_{0})\triangleq \mathcal{O}_{\infty,n}^{\gamma,\tau_1}(i_{0})\\
		&\bigcap_{\underset{\,j, j_{0}\in\Z_{\m}\backslash\overline{\mathbb{S}}_{0,k}}{ {k\in\{1,2\}}}}\bigcap_{\underset{|l|\leqslant N_{n}}{l\in\mathbb{Z}^{d}\atop(l,j)\neq(0,j_{0})}}\Big\{(b,\omega)\in\mathcal{O}\quad\textnormal{s.t.}\quad\big|\omega\cdot l+\mu_{j,k}^{(\infty)}(b,\omega,i_{0})-\mu_{j_{0},k}^{(\infty)}(b,\omega,i_{0})\big|>\tfrac{2\gamma\langle j-j_0\rangle}{\langle l\rangle^{\tau_2}}\Big\}\\
		&\bigcap_{\underset{\,j_0\in\Z_{\m}\backslash\overline{\mathbb{S}}_{0,2}}{ j\in\Z_{\m}\backslash\overline{\mathbb{S}}_{0,1}}}\bigcap_{\underset{\langle l,j,j_0\rangle\leqslant N_{n}}{l\in\mathbb{Z}^{d}}}\Big\{(b,\omega)\in\mathcal{O}\quad\textnormal{s.t.}\quad\big|\omega\cdot l+\mu_{j,1}^{(\infty)}(b,\omega,i_{0})-\mu_{j_{0},2}^{(\infty)}(b,\omega,i_{0})\big|>\tfrac{2\gamma}{\langle l,j,j_0\rangle^{\tau_2}}\Big\}
	\end{align*}
	the following decomposition holds
	$$\Phi_{\infty}^{-1}\mathscr{L}_0\Phi_{\infty}=\omega\cdot\partial_{\varphi}\mathbf{I}_{\mathtt{m},\perp}+\mathscr{D}_{\infty}+\mathscr{E}_n^1\triangleq \mathscr{L}_{\infty}+\mathscr{E}_n^1,$$
	where $\mathscr{D}_{\infty}=\Pi_{\overline{\mathbb{S}}_{0}}^{\perp}\mathscr{D}_{\infty}\Pi_{\overline{\mathbb{S}}_{0}}^{\perp}=\mathscr{D}_{\infty}(b,\omega,i_{0})$ is a diagonal operator  with reversible Fourier multiplier entries, namely
	$$\mathscr{D}_{\infty}\triangleq\begin{pmatrix}
		\mathscr{D}_{\infty,1} & 0\\
		0 & \mathscr{D}_{\infty,2}
	\end{pmatrix},$$
	with
	$$\forall k\in\{1,2\},\quad\mathscr{D}_{\infty,k}\triangleq\left(\ii\mu_{j,k}^{(\infty)}\right)_{j\in\mathbb{Z}_{\mathtt{m}}\setminus\overline{\mathbb{S}}_{0,k}},\qquad\mu_{j,k}^{(\infty)}(b,\omega,i_{0})\triangleq\mu_{j,k}^{(0)}(b,\omega,i_{0})+r_{j,k}^{(\infty)}(b,\omega,i_{0})$$
	and 
	$$\sup_{j\in\mathbb{Z}_{\mathtt{m}}\setminus\overline{\mathbb{S}}_{0,k}}|j|\left\| r_{j,k}^{(\infty)}\right\|^{q,\gamma}\lesssim\varepsilon\gamma^{-1}.$$
	Notice that, according to the monotonicity of the eigenvalues  the difference $\mu_{j,k}^{(\infty)}-\mu_{j_0,k}^{(\infty)}$ is not vanishing for $j\neq j_0$ and grows like $|j-j_0|$. This is no longer true for the mixed difference $\mu_{j,1}^{(\infty)}-\mu_{j_0,2}^{(\infty)}$ (coming from the mutual interactions between the interfaces) due to the different transport speeds leading to a new  small divisor problem. Therefore, to handle this problem we should  adjust  the geometry of the Cantor sets  $\mathscr{O}_{\infty,n}^{\gamma,\tau_1,\tau_2}(i_0)$ with an isotropic decay on frequency. This explains in part the introduction of the  hybrid topology in \eqref{hyb nor} needed in the remainder reduction,  Another key observation is  that we have no resonances for the off-diagonal part at  $j=j_0$ and consequently the associated homological equations can be solved without any residual diagonal terms. Thus, at the end of the KAM scheme we get  a diagonal Fourier multiplier operator $\mathscr{D}_{\infty}.$
	  Now, the final  operator $\mathscr{L}_{\infty}$ can be easily inverted by  restricting the parameters to the following first order Melnikov conditions
	$$\Lambda_{\infty,n}^{\gamma,\tau_1}(i_0)\triangleq\bigcap_{k\in\{1,2\}\atop\underset{|l|\leqslant N_{n}}{(l,j)\in\mathbb{Z}^{d }\times(\mathbb{Z}_{\mathtt{m}}\setminus\overline{\mathbb{S}}_{0,k})} }\Big\{(b,\omega)\in\mathcal{O}\quad\textnormal{s.t.}\quad\left|\omega\cdot l+\mu_{j,k}^{(\infty)}(b,\omega,i_0)\right|>\tfrac{\gamma\langle j\rangle}{\langle l\rangle^{\tau_1}}\Big\}.$$
	As a consequence, we can construct an approximate right inverse of $\widehat{\mathcal{L}}$ provided that we choose $(b,\omega)$ in the set 
	$$\mathtt{G}_{n}^{\gamma}(i_0)\triangleq\mathcal{O}_{\infty,n}^{\gamma,\tau_1}(i_0)\cap\mathscr{O}_{\infty,n}^{\gamma,\tau_1,\tau_2}(i_0)\cap\Lambda_{\infty,n}^{\gamma,\tau_1}(i_0).$$
	Therefore, we can perform in Proposition \ref{Nash-Moser} and Corollary \ref{Corollary NM} a Nash-Moser scheme as in \cite{BM18,HHM21,HR21} with slight modifications due to our particular Poisson structure and the off-diagonal second order Melnikov conditions. Hence, we can find a non-trivial solution $(b,\omega)\mapsto(i_{\infty}(b,\omega),\alpha_{\infty}(b,\omega))$ to the equation $\mathcal{F}(i,\alpha,b,\omega,\varepsilon)=0$ provided that we restrict the parameters $(b,\omega)$ to a  Borel set $\mathtt{G}_{\infty}^{\gamma}$ constructed as the intersection of all the Cantor sets encountered along the different  schemes of the multiple reductions.   A  solution to the original problem is obtained by constructing a frequency curve $b\mapsto\omega(b,\varepsilon)$ solution to the implicit  equation 
	$$\alpha_{\infty}\big(b,\omega(b,\varepsilon)\big)=-\mathtt{J}\omega_{\textnormal{Eq}}(b).$$
	By this way we construct   a solution for any value of $b$ in 
	$$\mathcal{C}_{\infty}^{\varepsilon}\triangleq\Big\{b\in(b_*,b^*)\quad\textnormal{s.t.}\quad\big(b,\omega(b,\varepsilon)\big)\in\mathtt{G}_{\infty}^{\gamma}\Big\}.$$
	The last step is to check that this final set is non-empty and massive. Actually, we prove in Proposition \ref{lem-meas-es1} the following measure bound
	$$(b^*-b_*)-\varepsilon^{\delta}\leqslant|\mathcal{C}_{\infty}^{\varepsilon}|\leqslant(b^*-b_*)\quad\textnormal{for some }\delta\triangleq\delta(q_0,d,\tau_1,\tau_2)>0.$$
	The proof is quite standard and   based on the  perturbation of  R\"ussmann conditions, shown to be true at  the equilibrium state. We emphasize that  the restriction $\Omega>\Omega_{\mathtt{m}}^*$  is required by  Lemma \ref{lemma transversalityE}-(iv) and Lemma \ref{lem Ru-pert}-(iv), and the  value of $\Omega_{\mathtt{m}}^*$ given in \eqref{expr Omega star}  is not necessary  optimal.\\
	
	\noindent\textbf{Acknoledgments :} The work of Zineb Hassainia has been supported by Tamkeen under the NYU Abu Dhabi Research Institute grant of the center SITE. The work of Taoufik Hmidi has been supported by Tamkeen under the NYU Abu Dhabi Research Institute grant. The work of Emeric Roulley has been partially supported by PRIN 2020XB3EFL, "Hamiltonain and Dispersive PDEs".
	\section{Function and operator spaces}\label{sec funct set}

	This section is devoted to the presentation of the general topological framework for both functions and operators classes. In addition, we shall set some basic notations, definitions and give some technical results used in this work.
	
	\paragraph{Notations.}
	Along this paper we shall make use of the following set notations.
	\begin{enumerate}[label=\textbullet]
		\item The sets of numbers that will be frequently used are denoted as follows  
		$$
		\mathbb{N}= \{0,1,2,\ldots\},\quad\mathbb{N}^*= \mathbb{N}\setminus\{0\},\quad\mathbb{Z}=\mathbb{N}\cup(-\mathbb{N}),\qquad\mathbb{Z}^*= \mathbb{Z}\setminus\{0\}, \qquad \mathbb{T}= \mathbb{R}/2\pi\mathbb{Z}.
		$$
		For any $\mathtt{m}\in\mathbb{N}^*,$ we may denote
		$$
		\mathbb{Z}_{\mathtt{m}}=\mathtt{m}\mathbb{Z},\qquad\mathbb{N}_{\mathtt{m}}= \mathtt{m}\mathbb{N},\qquad \mathbb{Z}_{\mathtt{m}}^*= \mathtt{m}\mathbb{Z}^*, \qquad \mathbb{N}_{\mathtt{m}}^*= \mathtt{m}\mathbb{N}^*,\qquad \mathbb{T}^{\mathtt{m}}= \underbrace{\mathbb{T}\times\cdots \times\mathbb{T}}_{{\mathtt{m}} \textnormal{ times}},
		$$
		and for any $m,n\in\mathbb{Z},$ such that $m<n$,
		$$
		 \llbracket m,n\rrbracket\triangleq \{m,m+1,\ldots,n-1,n\}.
		 $$
		\item We fix two real numbers $b_*$ and $b^*$ such that
		$$0<b_*<b^*<1.$$
		The parameter $b$  lies in the interval $(b_{*},b^{*})$ and represents the radius of the annulus $A_b$ in \eqref{annulus-Ab}, corresponding to the equilibrium state and 
		\item Consider the following parameters, that will be used to construct the Cantor set as well as the regularity of the perturbations,
		\begin{align}
			d\in\mathbb{N}^*,&\qquad q\in\mathbb{N}^*,\label{setting q}\\
			0<\gamma\leqslant1,&\qquad \tau_{2}>\tau_{1}>d,\label{setting tau1 and tau2}\\
			S\geqslant s\geqslant &\,s_{0}>\tfrac{d+1}{2}+q+2.\label{init Sob cond}	
		\end{align}
		\item For any $n\in\mathbb{N}^*$ and any complex periodic function  $\rho:\T^n\to\R$, we denote
		\begin{equation}\label{average-notation}
			\int_{\mathbb{T}^{n}}\rho(\eta)d\eta\triangleq \frac{1}{(2\pi)^n}\int_{[0,2\pi]^n}\rho(\eta)d\eta.
		\end{equation} 
		\item Let $f:X\to Y$ be a map where $X$ is a set and $Y$ is a vector space.  For any $r_1,r_2\in X$, we denote
		$$\Delta_{12}f=f(r_1)-f(r_2).$$ 
	\end{enumerate}

	\subsection{Function spaces}\label{Functionspsaces}

	This section is devoted to  some functional tools frequently used along this paper. First, we shall  introduce the complex Sobolev space on the periodic setting  $H^{s}(\mathbb{T}^{d +1},\mathbb{C})$  with index regularity $s\in\R.$ It is  the set
	of all the complex periodic functions $\rho:\mathbb{T}^{d +1}\to \mathbb{C}$ with the Fourier expansion
	$$
	{\rho=\sum_{(l,j)\in\mathbb{Z}^{d+1 }}\rho_{l,j}\,\mathbf{e}_{l,j}},\qquad \mathbf{e}_{l,j}(\varphi,\theta)\triangleq e^{\ii(l\cdot\varphi+j\theta)},\qquad \rho_{l,j}\triangleq\big\langle \rho,\mathbf{e}_{l,j}\big\rangle_{L^{2}(\mathbb{T}^{d+1})}
	$$
	equipped with the scalar product
	$$
	\big\langle\rho,\widetilde{\rho}\big\rangle_{H^{s}}\triangleq \sum_{(l,j)\in\mathbb{Z}^{d+1 }}\langle l,j\rangle^{2s}\rho_{l,j}\overline{\widetilde\rho_{l,j}},\qquad\textnormal{with}\qquad \langle l,j\rangle\triangleq \max(1,|l|,|j|),
	$$
	where $|\cdot|$ denotes  the classical $\ell^{1}$ norm in $\mathbb{R}^{d }$. For $s=0$ this space coincides with the standard  $ L^{2}(\mathbb{T}^{d+1},\mathbb{C})$ space equipped with the scalar product
	$$
			\big\langle\rho_{1},\rho_{2}\big\rangle_{L^{2}(\mathbb{T}^{d+1})}\triangleq\bigintssss_{\mathbb{T}^{d+1}}\rho_{1}(\varphi,\theta)\overline{\rho_{2}(\varphi,\theta)}d\varphi d\theta.
	$$
We shall make use of {the product Sobolev space}
	\begin{equation}\label{def:sob-product}
	\mathbf{H}^{s}_{\mathtt{m}}(\mathbb{T}^{d+1},\mathbb{C})\triangleq H_{\mathtt{m}}^{s}(\mathbb{T}^{d+1},\mathbb{C})\times H_{\mathtt{m}}^{s}(\mathbb{T}^{d+1},\mathbb{C}),
	\end{equation}
	equipped with the scalar product 
	$$
	\big\langle(\rho_1,\rho_2),(\widetilde{\rho}_1,\widetilde{\rho}_1)\big\rangle_{\mathbf{H}^{s}_{\mathtt{m}}(\mathbb{T}^{d+1},\mathbb{C})}\triangleq \sum_{(l,j)\in\mathbb{Z}^{d+1 }}\langle l,j\rangle^{2s}\rho_{l,j}^1\overline{\widetilde\rho_{l,j}^1}+\sum_{(l,j)\in\mathbb{Z}^{d+1 }}\langle l,j\rangle^{2s}\rho_{l,j}^2\overline{\widetilde\rho_{l,j}^2}.
	$$
	We also simply denote the real space
	$$\mathbf{H}_{\mathtt{m}}^s\triangleq H_{\mathtt{m}}^s\times H_{\mathtt{m}}^{s}.$$		
		 As we shall see later, the main enemy in the construction of quasi-periodic solutions is the resonances and in particular the trivial ones which can be fortunately removed by imposing more symmetry on the solutions. For this aim we need  to work with the following subspace $H_{\mathtt{m}}^{s}(\mathbb{T}^{d +1},\mathbb{C})$, with $\mathtt{m}\in\N^*,$ whose elements enjoy the $\mathtt{m}$-fold symmetry in the variable $\theta,$ that is
	$$H_{\mathtt{m}}^{s}(\mathbb{T}^{d +1},\mathbb{C})\triangleq \Big\{\rho\in H^{s}(\mathbb{T}^{d +1},\mathbb{C})\quad\textnormal{s.t.}\quad\forall (\varphi,\theta)\in\T^{d+1},\quad\rho\left(\varphi,\theta+\tfrac{2\pi}{\mathtt{m}}\right)=\rho(\varphi,\theta)\Big\}.
	$$
	Notice  the $\m$-fold symmetry is equivalent to say that
	\begin{align*}
		\forall l\in\mathbb{Z}^{d },\quad \forall j\in \mathbb{Z}\setminus \Z_{\mathtt{m}},\quad  {\big\langle \rho,\mathbf{e}_{l,j}\big\rangle_{L^{2}(\mathbb{T}^{d+1})}=0.}
	\end{align*} 
	The real Sobolev space $H^{s}_{\mathtt{m}}(\mathbb{T}^{d+1},\mathbb{R})$  is simply denoted by $ H^{s}_{\mathtt{m}}$ and  we define the subspace $$H_{\m}^\infty\triangleq\bigcap_{s\in\R} H_{\m}^s.$$  
	For $N\in\mathbb{N}$, we define the cut-off frequency projectors $\Pi_{N}$ and its orthogonal $\Pi_{N}^\perp$ on  $H^{s}(\mathbb{T}^{d +1},\mathbb{C})$ as follows 
	\begin{equation}\label{def projectors PiN}
		\Pi_{N}\rho\triangleq\sum_{\underset{\langle l,j\rangle\leqslant N}{(l,j)\in\Z^{d+1}}}\rho_{l,j}\mathbf{e}_{l,j}\qquad\textnormal{and}\qquad \Pi^{\perp}_{N}\triangleq\textnormal{Id}-\Pi_{N}.
	\end{equation}
	We shall also make use of the following mixed weighted Sobolev spaces with respect to a given  parameter  $\gamma\in(0,1)$. 
	{Let $\mathcal{O}$ be an open bounded set of $\mathbb{R}^{d+1}$ and define the Banach spaces    }
	\begin{align*}
		W^{q,\infty,\gamma}(\mathcal{O},H^{s}_{\mathtt{m}})&\triangleq\Big\lbrace \rho:\mathcal{O}\rightarrow H^{s}_{\mathtt{m}}\quad\textnormal{s.t.}\quad\|\rho\|_{s}^{q,\gamma,\mathtt{m}}\triangleq\sum_{\underset{|\alpha|\leqslant q}{\alpha\in\mathbb{N}^{d+1}}}\gamma^{|\alpha|}\sup_{\mu\in{\mathcal{O}}}\|\partial_{\mu}^{\alpha}\rho(\mu,\cdot)\|_{H^{s-|\alpha|}_{\mathtt{m}}}<\infty\Big\rbrace,\\
		W^{q,\infty,\gamma}(\mathcal{O},\mathbb{C})&\triangleq\Big\lbrace\rho:\mathcal{O}\rightarrow\mathbb{C}\quad\textnormal{s.t.}\quad\|\rho\|^{q,\gamma}\triangleq\sum_{\underset{|\alpha|\leqslant q}{\alpha\in\mathbb{N}^{d+1}}}\gamma^{|\alpha|}\sup_{\mu\in{\mathcal{O}}}|\partial_{\mu}^{\alpha}\rho(\mu)|<\infty\Big\rbrace.
	\end{align*}
	Through this paper, we shall implicitly use the notation $\|\rho\|_{s}^{q,\gamma,\mathtt{m}}$, while the function $\rho$ depends on more variables such as with $(\varphi,\theta)\in\Z^{d}\times\Z_{\m}^{d^\prime}\mapsto \rho(\varphi,\theta)$, frequently encountered  when we have to estimate the kernels of some operators, in which case the variables can be doubled.\\
	
	In the  next lemma we shall collect   some useful classical results related {to various actions over  weighted Sobolev spaces. The proofs are standard and can be found for instance in \cite{BFM21,BFM21-1,BM18}.}
	\begin{lem}\label{lem funct prop}
		Let $q\in\N$, $\m\in\N^*$ and $(\gamma,d,s_{0},s)$ satisfy \eqref{setting tau1 and tau2}-\eqref{init Sob cond}, then the following assertions hold true.
			\begin{enumerate}[label=(\roman*)]
				\item Frequency growth/decay of projectors : Let $\rho\in W^{q,\infty,\gamma}(\mathcal{O},H^{s}_{\mathtt{m}}),$ then for all $N\in\mathbb{N}^{*}$ and $t>0$,
				$$\|\Pi_{N}\rho\|_{s+t}^{q,\gamma,{\mathtt{m}}}\leqslant N^{t}\|\rho\|_{s}^{q,\gamma,{\mathtt{m}}}\qquad\textnormal{and}\qquad\|\Pi_{N}^{\perp}\rho\|_{s}^{q,\gamma,{\mathtt{m}}}\leqslant N^{-t}\|\rho\|_{s+t}^{q,\gamma,{\mathtt{m}}},
				$$
				where the cut-off projectors are defined in \eqref{def projectors PiN}.
				\item Product law : 
				Let $\rho_{1},\rho_{2}\in W^{q,\infty,\gamma}(\mathcal{O},H^{s}_{\mathtt{m}}).$ Then $\rho_{1}\rho_{2}\in W^{q,\infty,\gamma}(\mathcal{O},H^{s}_{\mathtt{m}})$ and 
				$$\| \rho_{1}\rho_{2}\|_{s}^{q,\gamma,\mathtt{m}}\lesssim\| \rho_{1}\|_{s_{0}}^{q,\gamma,\mathtt{m}}\| \rho_{2}\|_{s}^{q,\gamma,\mathtt{m}}+\| \rho_{1}\|_{s}^{q,\gamma,\mathtt{m}}\| \rho_{2}\|_{s_{0}}^{q,\gamma,\mathtt{m}}.$$
				\item Composition law $1$ : Let $f\in C^{\infty}(\mathcal{O}\times\mathbb{R},\mathbb{R})$ and  $\rho_{1},\rho_{2}\in W^{q,\infty,\gamma}(\mathcal{O},H^{s}_{\mathtt{m}})$  such that $$\| \rho_{1}\|_{s}^{q,\gamma,\mathtt{m}},\|\rho_{2}\|_{s}^{q,\gamma,\mathtt{m}}\leqslant C_{0}$$ for an  arbitrary  constant  $C_{0}>0$ and define the pointwise composition $$\forall (\mu,\varphi,\theta)\in \mathcal{O}\times\mathbb{T}^{d+1},\quad f(\rho)(\mu,\varphi,\theta)\triangleq  f\big(\mu,\rho(\mu,\varphi,\theta)\big).$$
				Then 
				$$\| f(\rho_{1})-f(\rho_{2})\|_{s}^{q,\gamma,\mathtt{m}}\leqslant C(s,d,q,f,C_{0})\| \rho_{1}-\rho_{2}\|_{s}^{q,\gamma,\mathtt{m}}.$$
				\item Composition law $2$ : Let $f\in C^{\infty}(\mathbb{R},\mathbb{R})$ with bounded derivatives. Let $\rho\in W^{q,\infty,\gamma}(\mathcal{O},\mathbb{C}).$ Then 			$$\|f(\rho)-f(0)\|^{q,\gamma}\leqslant C(q,d,f)\|\rho\|^{q,\gamma}\left(1+\|\rho\|_{L^{\infty}(\mathcal{O})}^{q-1}\right).$$
				\item Interpolation inequality : Let $q<s_{1}\leqslant s_{3}\leqslant s_{2}$ and $\overline{\theta}\in[0,1],$ with  $s_{3}=\overline{\theta} s_{1}+(1-\overline{\theta})s_{2}.$\\
				If $\rho\in W^{q,\infty,\gamma}(\mathcal{O},H^{s_{2}}_{\mathtt{m}})$, then  $\rho\in W^{q,\infty,\gamma}(\mathcal{O},H^{s_{3}}_{\mathtt{m}})$ and
				$$\|\rho\|_{s_{3}}^{q,\gamma,\mathtt{m}}\lesssim\left(\|\rho\|_{s_{1}}^{q,\gamma,\mathtt{m}}\right)^{\overline{\theta}}\left(\|\rho\|_{s_{2}}^{q,\gamma,\mathtt{m}}\right)^{1-\overline{\theta}}.$$
		\end{enumerate}
	\end{lem}
	The next result is proved in \cite[Lem. 4.2]{HR21} and  will be useful later in the study of some regularity aspects for  the linearized operator.
	\begin{lem}\label{lem triche}
		Let $q\in\N$, $\m\in\N^*$, $(\gamma,d,s_{0},s)$ satisfy \eqref{setting tau1 and tau2}-\eqref{init Sob cond} and  $f\in W^{q,\infty,\gamma}(\mathcal{O},H_{\mathtt{m}}^{s}).$\\
		We consider the function $g:\mathcal{O}\times\mathbb{T}_{\varphi}^{d}\times\mathbb{T}_{\theta}\times\mathbb{T}_{\eta}\rightarrow\mathbb{C}$ defined by
		$$g(\mu,\varphi,\theta,\eta)=\left\lbrace\begin{array}{ll}
			\frac{f(\mu,\varphi,\eta)-f(\mu,\varphi,\theta)}{\sin\left(\tfrac{\eta-\theta}{2}\right)}, & \textnormal{if }\theta\neq \eta,\\
			2\partial_{\theta}f(\mu,\varphi,\theta),& \textnormal{if }\theta=\eta.
		\end{array}\right.$$
		Then 
		$$\|g\|_{s}^{q,\gamma,\mathtt{m}}\lesssim\|f\|_{s+1}^{q,\gamma,\mathtt{m}}.$$
	\end{lem}
	\subsection{Operators}\label{section-ope}
	We intend in this section to explore some algebraic and analytical aspects on the a large class of operators that fit with our context. Firstly, we shall classify them according to their Toeplitz in time structures, real and $\mathtt{m}$-fold symmetry, etc... Secondly, we shall fix some specific norms, such as the off-diagonal/isotropic  decay, and analyze some of their properties. This part is a crucial later in the reduction of the remainder of the linearized operator. Thirdly, a particular attention will be focused on operators with kernels by exploring the link between the different norms and the action of suitable quasi-periodic transformations. The last point concerns a short discussion on matrix operators.
	
	\subsubsection{Symmetry}
	Consider a smooth family of bounded linear operators acting on the Sobolev spaces $H^s(\T^{d+1}, \C)$, 
	$$T: \mu\in \mathcal{O}\mapsto T(\mu)\in\mathcal{L}\big(H^s(\T^{d+1},\mathbb{C})\big).$$ 
	The linear operator $T(\mu)$ can be identified to the infinite dimensional matrix $\left(T_{l_{0},j_{0}}^{l,j}(\mu)\right)_{\underset{j,j_{0}\in\mathbb{Z}}{l,l_{0}\in\mathbb{Z}^{d}}}$ with 
	$$T(\mu)\mathbf{e}_{l_{0},j_{0}}=\sum_{(l,j)\in\mathbb{Z}^{d}\times\Z}T_{l_{0},j_{0}}^{l,j}(\mu)\mathbf{e}_{l,j},\qquad\textnormal{where}\qquad T_{l_{0},j_{0}}^{l,j}(\mu)\triangleq\big\langle T(\mu)\mathbf{e}_{l_{0},j_{0}},\mathbf{e}_{l,j}\big\rangle_{L^{2}(\mathbb{T}^{d+1},\mathbb{C})}.$$
	Along this paper the operators and the test functions may depend on the same parameter $\mu$ and thus 
	the action of the operator $T(\mu)$ on a scalar function $\rho\in W^{q,\infty,\gamma}\big(\mathcal{O},H^{s}(\mathbb{T}^{d+1},\mathbb{C})\big)$  is by convention
	defined through
	$$(T\rho)(\mu,\varphi,\theta)\triangleq T(\mu)\rho(\mu,\varphi,\theta).$$
	We recall the following definitions of Toeplitz, real, reversible, reversibility preserving and $\mathtt{m}$-fold preserving operators, see for instance \cite[Def. 2.2]{BBM14}.
	\begin{defin}\label{Def-Rev}
		Let $\rho:\T^{d+1}\to\R$ be a periodic function.  Define the  involution
		$$(\mathscr{I}_{0}\rho)(\varphi,\theta)\triangleq\rho(-\varphi,-\theta)$$
		and for a given integer  $\mathtt{m}\geqslant1$  consider the transformation
		$$(\mathscr{I}_{\mathtt{m}}\rho)(\varphi,\theta)\triangleq\rho\left(\varphi,\theta+\tfrac{2\pi}{\mathtt{m}}\right).$$
		We say that an operator $T\in  \mathcal{L}\big(L^2(\T^{d+1},\mathbb{C})\big)$ is 
		\begin{enumerate}[label=\textbullet]
			\item  Toeplitz in time (actually in the variable $\varphi$) if its Fourier coefficients satisfy, 
			$$\forall(l,l_{0},j,j_{0})\in(\mathbb{Z}^{d})^{2}\times\mathbb{Z}^{2},\quad T_{l_{0},j_{0}}^{l,j}=T_{j_{0}}^{j}(l-l_{0}),\qquad\textnormal{with}\qquad T_{j_{0}}^{j}(l)\triangleq T_{0,j_{0}}^{l,j},$$
			\item real if for all $\rho\in L^{2}(\mathbb{T}^{d+1},\mathbb{R}),$ we have 
			$T\rho$ is real-valued, or equivalently 
			$$\forall(l,l_{0},j,j_{0})\in(\mathbb{Z}^{d})^{2}\times\mathbb{Z}^{2},\quad T_{-l_{0},-j_{0}}^{-l,-j}=\overline{T_{l_{0},j_{0}}^{l,j}},$$
			\item reversible if
			$T\circ\mathscr{I}_{0}=-\mathscr{I}_{0}\circ T,$ or equivalently,
			$$\forall(l,l_{0},j,j_{0})\in(\mathbb{Z}^{d})^{2}\times\mathbb{Z}^{2},\quad T_{-l_{0},-j_{0}}^{-l,-j}=-T_{l_{0},j_{0}}^{l,j},$$
			\item reversibility preserving if
			$T\circ\mathscr{I}_{0}=\mathscr{I}_{0}\circ T,$ or equivalently,
			$$\forall(l,l_{0},j,j_{0})\in(\mathbb{Z}^{d})^{2}\times\mathbb{Z}^{2},\quad T_{-l_{0},-j_{0}}^{-l,-j}=T_{l_{0},j_{0}}^{l,j},$$
			\item $\mathtt{m}$-fold preserving if
			$T\circ\mathscr{I}_{\mathtt{m}}=\mathscr{I}_{\mathtt{m}}\circ T,$ or equivalently,
			$$\forall(l,l_{0},j,j_{0})\in(\mathbb{Z}^{d})^{2}\times\mathbb{Z}^{2},\quad T_{l_{0},j_{0}}^{l,j}\neq0\quad\Rightarrow\quad j- j_0\in\mathbb{Z}_{\mathtt{m}}.$$
		\end{enumerate}
	\end{defin}
	\subsubsection{Operator topologies} 	
	We shall restrict ourselves to Toeplitz operators and fix different topologies whose use will be motivated later by different applications. Given $\m\in\N^*$, then any $\mathtt{m}$-fold preserving Toeplitz operator  $T(\mu)$ acting on $\m$-fold symmetric  functions $\rho=\displaystyle\sum_{l_{0}\in\mathbb{Z}^{d}\atop j_{0}\in\Z_{\m}}\rho_{l_{0},j_{0}}\mathbf{e}_{l_{0},j_{0}}$ is described by
	$$T(\mu)\rho=\sum_{l,l_{0}\in\mathbb{Z}^{d}\\\atop j,j_{0}\in\mathbb{Z}_{\mathtt{m}}}T_{j_{0}}^{j}(\mu,l-l_{0})\rho_{l_{0},j_{0}}\mathbf{e}_{l,j}.$$
	For $q\in\mathbb{N},$ $\gamma\in(0,1]$ and $s\in\mathbb{R},$ {we equip this set of operators} with the off-diagonal norm  given by,
	\begin{align}\label{Top-NormX}
		\| T\|_{\textnormal{\tiny{O-d}},s}^{q,\gamma,\mathtt{m}}\triangleq\max_{\underset{|\alpha|\leqslant q}{\alpha\in\mathbb{N}^{d+1}}}\gamma^{|\alpha|}\sup_{\mu \in{\mathcal{O}}}\|\partial_{\mu}^{\alpha}(T)(\mu)\|_{\textnormal{\tiny{O-d}},s-|\alpha|},
	\end{align}
	with
	$$\| T\|_{\textnormal{\tiny{O-d}},s}\triangleq\sup_{(l,m)\in\mathbb{Z}^{d+1}}\langle l,m\rangle^{s}\sup_{j_0-j=m}|T_{j_0}^{j}(\mu,l)|.$$
	We define the cut-off frequency operator
	$$\left(P_{N}^1T(\mu)\right)\mathbf{e}_{l_{0},j_{0}}\triangleq\sum_{\underset{\langle l-l_{0},j-j_{0}\rangle\leqslant N}{(l,j)\in\mathbb{Z}^{d+1}}}T_{l_{0},j_{0}}^{l,j}(\mu)\mathbf{e}_{l,j}\qquad\mbox{and}\qquad P_{N}^{1,\perp}T\triangleq T-P_{N}^1T.$$
	or equivalently
	\begin{equation}\label{definition of projections for operators}
		\left(P_{N}^1T(\mu)\right)_{j_0}^j(l)=\begin{cases}
			T_{j_{0}}^{j}(\mu,l),& \hbox{if}\quad \langle l,j-j_0\rangle\leqslant N,\\
			0,&\hbox{if not.}
		\end{cases}
	\end{equation}
	Another norm that will be used together with the previous one during the reduction process of the remainder of the linearized operator, is given by the isotropic frequency decay
	\begin{align}\label{Top-NormX2}
		\|T\|_{\textnormal{\tiny{I-D}},s}^{q,\gamma,\mathtt{m}}\triangleq\sup_{\underset{|\alpha|\leqslant q}{\alpha\in\mathbb{N}^{d+1}}}\gamma^{|\alpha|}\sup_{\mu\in{\mathcal{O}}}\|\partial_{\mu}^{\alpha}(T)(\mu)\|_{\textnormal{\tiny{I-D}},s-|\alpha|},
	\end{align}
	where 
	$$\|T(\mu)\|_{\textnormal{\tiny{I-D}},s}\triangleq \sup_{l\in\mathbb{Z}^{d}\atop j,j_0 \in\Z_{\m}}\langle l,j_0,j\rangle^{s}|T_{j_0}^{j}(\mu,l)|.$$
	The associated cut-off projectors $(P_N^2)_{N\in\mathbb{N}}$ are defined as follows 
	\begin{equation}\label{definition of projections for operators2}
		\left(P_{N}^2T(\mu)\right)\mathbf{e}_{l_{0},j_{0}}\triangleq\begin{cases}
			\displaystyle{\sum_{\underset{\langle l-l_{0},j\rangle\leqslant N}{(l,j)\in\mathbb{Z}^{d}\times\Z_{\mathtt{m}}}}}T_{j_{0}}^{j}(\mu,l-l_0)\mathbf{e}_{l,j},& \hbox{if}\quad  |j_0|\leqslant N,\\
			0,&\hbox{if}\quad  |j_0|>N.
		\end{cases}
	\end{equation}
	or equivalently
	\begin{equation}\label{definition of projections for operators3}
		\left(P_{N}^2T(\mu)\right)_{j_0}^j(l)=\begin{cases}
			T_{j_{0}}^{j}(\mu,l),& \hbox{if}\quad \langle l,j,j_0\rangle\leqslant N,\\
			0,&\hbox{if not.}
		\end{cases}
	\end{equation}
	We also define the orthogonal projector  $ P_{N}^{2,\perp}T\triangleq T-P_{N}^2T.$ The next lemma lists some elementary results related to the off-diagonal and the isotropic norms.
	\begin{lem}\label{properties of Toeplitz in time operators}
		Let $q\in\N$, $\m\in\N^*$, $(\gamma,d,s_{0},s)$ satisfy \eqref{setting tau1 and tau2}-\eqref{init Sob cond}, $T$ and $S$ be Toeplitz in time operators.
		\begin{enumerate}[label=(\roman*)]
			\item Frequency localization : Let $N\in\mathbb{N}^{*}$ and  $\mathtt{t}\in\mathbb{R}_{+}.$ Then
			$$\|P_{N}^{1}T\|_{\textnormal{\tiny{O-d}},s+\mathtt{t}}^{q,\gamma,\mathtt{m}}\leqslant N^{\mathtt{t}}\|T\|_{\textnormal{\tiny{O-d}},s}^{q,\gamma,\mathtt{m}},\qquad\|P_{N}^{1,\perp}T\|_{\textnormal{\tiny{O-d}},s}^{q,\gamma,\mathtt{m}}\leqslant N^{-\mathtt{t}}\|T\|_{\textnormal{\tiny{O-d}},s+\mathtt{t}}^{q,\gamma,\mathtt{m}}$$
			and
			$$\|P_{N}^{2}T\|_{\textnormal{\tiny{I-D}},s+\mathtt{t}}^{q,\gamma,\mathtt{m}}\leqslant N^{\mathtt{t}}\|T\|_{\textnormal{\tiny{I-D}},s}^{q,\gamma,\mathtt{m}},\qquad\|P_{N}^{2,\perp}T\|_{\textnormal{\tiny{I-D}},s}^{q,\gamma,\mathtt{m}}\leqslant N^{-\mathtt{t}}\|T\|_{\textnormal{\tiny{I-D}},s+\mathtt{t}}^{q,\gamma,\mathtt{m}}.$$
			\item Link with the classical operator norm :
			\begin{align*}
				\|T\rho\|_{s}^{q,\gamma,\mathtt{m}}&\lesssim\|T\|_{\textnormal{\tiny{O-d}},s_{0}}^{q,\gamma,\mathtt{m}}\|\rho\|_{s}^{q,\gamma,\mathtt{m}}+\|T\|_{\textnormal{\tiny{O-d}},s}^{q,\gamma,\mathtt{m}}\|\rho\|_{s_{0}}^{q,\gamma,\mathtt{m}},\\
				\|T\rho\|_{s}^{q,\gamma,\mathtt{m}}&\lesssim\|T\|_{\textnormal{\tiny{I-D}},s_0}^{q,\gamma,\mathtt{m}}\|\rho\|_{s}^{q,\gamma,\mathtt{m}}+\|T\|_{\textnormal{\tiny{I-D}},s}^{q,\gamma,\mathtt{m}}\|\rho\|_{s_{0}}^{q,\gamma,\mathtt{m}}.
			\end{align*}
			In particular,
			\begin{align*}
				\|T\rho\|_{s}^{q,\gamma,\mathtt{m}}&\lesssim\|T\|_{\textnormal{\tiny{O-d}},s}^{q,\gamma,\mathtt{m}}\|\rho\|_{s}^{q,\gamma,\mathtt{m}},\\
				\|T\rho\|_{s}^{q,\gamma,\mathtt{m}}&\lesssim\|T\|_{\textnormal{\tiny{I-D}},s}^{q,\gamma,\mathtt{m}}\|\rho\|_{s}^{q,\gamma,\mathtt{m}}.
			\end{align*} 
			\item We have the embedding : for any $s\geqslant 0$
			$$\|T\|_{\textnormal{\tiny{O-d}},s}^{q,\gamma,\mathtt{m}}\lesssim\|T\|_{\textnormal{\tiny{I-D}},s}^{q,\gamma,\mathtt{m}}.$$
			\item Composition law :
			\begin{align*}
				\|TS\|_{\textnormal{\tiny{O-d}},s}^{q,\gamma,\mathtt{m}}&\lesssim\|T\|_{\textnormal{\tiny{O-d}},s}^{q,\gamma,\mathtt{m}}\|S\|_{\textnormal{\tiny{O-d}},s_0}^{q,\gamma,\mathtt{m}}+\|T\|_{\textnormal{\tiny{O-d}},s_0}^{q,\gamma,\mathtt{m}}\|S\|_{\textnormal{\tiny{O-d}},s}^{q,\gamma,\mathtt{m}}
			\end{align*}
			and
			\begin{align*}
				\|TS\|_{\textnormal{\tiny{I-D}},s}^{q,\gamma,\mathtt{m}}+\|ST\|_{\textnormal{\tiny{I-D}},s}^{q,\gamma,\mathtt{m}}&\lesssim\|T\|_{\textnormal{\tiny{I-D}},s}^{q,\gamma,\mathtt{m}}\|S\|_{\textnormal{\tiny{O-d}},s_0}^{q,\gamma,\mathtt{m}}+\|T\|_{\textnormal{\tiny{I-D}},s_0}^{q,\gamma,\mathtt{m}}\|S\|_{\textnormal{\tiny{O-d}},s}^{q,\gamma,\mathtt{m}}.
			\end{align*}
		\end{enumerate}
	\end{lem}
	\begin{proof}
		\textbf{(i)} and \textbf{(ii)} can be easily obtained using \eqref{Top-NormX}-\eqref{definition of projections for operators3} in a similar way to  \cite{BM18}. \\
		\noindent \textbf{(iii)} We shall prove the embedding for $q=0$ { and the  the case $q\geqslant 1$ is similar}. We write by definition 
		$$|T_{j_0}^{j}(\mu,l)|\leqslant \langle l,j_0,j\rangle^{-s}\|T(\mu)\|_{\textnormal{\tiny{I-D}},s}.$$
		Hence
		\begin{align*}
			\sup_{j_0-j=m}|T_{j_0}^{j}(\mu,l)|&\leqslant \sup_{j}\langle l,j+m,j\rangle^{-s}\|T(\mu)\|_{\textnormal{\tiny{I-D}},s}.
		\end{align*}
		By direct computations we infer
		$$\inf_{x\in\R}\langle l, x+m,x\rangle=\langle l,\tfrac{m}{2}\rangle$$
		Therefore
		\begin{align*}
			\sup_{j_0-j=m}|T_{j_0}^{j}(\mu,l)|&\lesssim  \langle l,m\rangle^{-s}\|T(\mu)\|_{\textnormal{\tiny{I-D}},s}.
		\end{align*}
		It follows that
		$$\|T\|_{\textnormal{\tiny{O-d}},s}\lesssim\|T(\mu)\|_{\textnormal{\tiny{I-D}},s}.$$
		\textbf{(iv)} We shall prove these tame estimates for $q=0$. The general case $q\geqslant 1$ can be done in a similar way using Leibniz formula. One can check that 
		$$\left(TS\right)_{j_0}^j(l)=\sum_{l_1\in\Z^d\atop  j_1\in\Z_\mathtt{m}}T_{j_0}^{j_1}(l_1)S_{j_1}^j(l-l_1).$$
		Hence for $s\geqslant 0$ and using the norm definition and the triangle inequality we infer
		\begin{align}\label{ST-1}
			\langle l,j_0,j\rangle^{s}\big|(TS)_{j_0}^j(l)\big|&\lesssim\sum_{l_1\in\Z^d\atop  j_1\in\Z_\mathtt{m}}\langle l_1,j_0,j_1\rangle^{s}|\nonumber T_{j_0}^{j_1}(l_1)S_{j_1}^j(l-l_1)|\\
			&\quad+\sum_{l_1\in\Z^d\atop  j_1\in\Z_\mathtt{m}}\langle l-l_1, j-j_1\rangle^{s}|T_{j_0}^{j_1}(l_1)S_{j_1}^j(l-l_1)|.
		\end{align}
		By definition we get 
		\begin{align}\label{Ineq-N1}
			\langle l-l_1,j- j_1\rangle^{s}|S_{j_1}^j(l-l_1)|\lesssim\|S\|_{\textnormal{\tiny{O-d}},s}.
		\end{align}
		Consequently,
		\begin{align*}
			\langle l,j_0,j\rangle^{s}\big|\left(TS\right)_{j_0}^j(l)\big|&\lesssim\|T\|_{\textnormal{\tiny{I-D}},s}\sum_{l_1\in\Z^d\atop  j_1\in\Z_\mathtt{m}}|S_{j_1}^j(l-l_1)|+  \| S\|_{\textnormal{\tiny{O-d}},s}\sum_{l_1\in\Z^d\atop  j_1\in\Z_\mathtt{m}}|T_{j_0}^{j_1}(l_1)|.
		\end{align*}
		Using \eqref{Ineq-N1} we deduce for $s_0>\frac{d+1}{2}$ that
		\begin{align*}
			\sum_{l_1\in\Z^d\atop  j_1\in\Z_\mathtt{m}}|S_{j_1}^j(l-l_1)|&\leqslant\|S\|_{\textnormal{\tiny{O-d}},s_0}\sum_{l_1\in\Z^d\atop j_1\in\Z_\mathtt{m}}\langle l-l_1,j- j_1\rangle^{-s_0}\\
			&\lesssim\|S\|_{\textnormal{\tiny{O-d}},s_0}.
		\end{align*}
		We also have
		\begin{align}
			\nonumber\sum_{l_1\in\Z^d\atop  j_1\in\Z_\mathtt{m}}|T_{j_0}^{j_1}(l_1)|&\leqslant\|T\|_{\textnormal{\tiny{I-D}},s_0} \sum_{l_1\in\Z^d\atop  j_1\in\Z_\mathtt{m}}\langle l_1,j_1\rangle^{-s_0}\\
			&\lesssim\|T\|_{\textnormal{\tiny{I-D}},s_0}.\label{T-00}
		\end{align}
		Therefore, we obtain
		\begin{align*}
			\langle l,j_0,j\rangle^{s}\big|\left(TS\right)_{j_0}^j(l)\big|&\lesssim\|T\|_{\textnormal{\tiny{I-D}},s} \| S\|_{\textnormal{\tiny{O-d}},s_0}+ \| S\|_{\textnormal{\tiny{O-d}},s}\|T\|_{\textnormal{\tiny{I-D}},s_0},
		\end{align*}
		leading to
		\begin{align*}
			\|TS\|_{\textnormal{\tiny{I-D}},s}&\lesssim\|T\|_{\textnormal{\tiny{I-D}},s}\|S\|_{\textnormal{\tiny{O-d}},s_0}+\|T\|_{\textnormal{\tiny{I-D}},s_0}\| S\|_{\textnormal{\tiny{O-d}},s}.
		\end{align*}
		Let us now move to the estimate of $ST$. Proceeding as for \eqref{ST-1} we get 
		\begin{align*}
			\nonumber\langle l,j_0,j\rangle^{s}\big|(ST)_{j_0}^j(l)\big|&\lesssim \sum_{l_1\in\Z^d\atop  j_1\in\Z_\mathtt{m}}\langle l_1,j_0-j_1\rangle^{s}| S_{j_0}^{j_1}(l_1)T_{j_1}^j(l-l_1)|\\
			&+\sum_{l_1\in\Z^d\atop  j_1\in\Z_\mathtt{m}}\langle l-l_1, j_1,j\rangle^{s}|S_{j_0}^{j_1}(l_1)T_{j_1}^j(l-l_1)|.
		\end{align*}
		Applying \eqref{Ineq-N1} together with \eqref{T-00} yields
		\begin{align*}
			\sum_{l_1\in\Z^d\atop j_1\in\Z_\mathtt{m}}\langle l_1,j_0-j_1\rangle^{s}|S_{j_0}^{j_1}(l_1)T_{j_1}^j(l-l_1)|&\lesssim \|S\|_{\textnormal{\tiny{O-d}},s}\sum_{l_1\in\Z^d\atop j_1\in\Z_\mathtt{m}}|T_{j_1}^j(l-l_1)|\\
			&\lesssim\|S\|_{\textnormal{\tiny{O-d}},s}\|T\|_{\textnormal{\tiny{I-D}},s_0}.
		\end{align*}
		Similarly we get
		\begin{align*}
			\sum_{l_1\in\Z^d\atop  j_1\in\Z_\mathtt{m}}\langle l-l_1, j_1,j\rangle^{s}|S_{j_0}^{j_1}(l_1)T_{j_1}^j(l-l_1)|&\lesssim\|T\|_{\textnormal{\tiny{I-D}},s}\sum_{l_1\in\Z^d\atop  j_1\in\Z_\mathtt{m}}|S_{j_0}^{j_1}(l_1)|\\
			&\lesssim\|T\|_{\textnormal{\tiny{I-D}},s}\| S\|_{\textnormal{\tiny{O-d}},s_0} \sum_{l_1\in\Z^d\atop  j_1\in\Z_\mathtt{m}}\langle l_1,j_1-j_0\rangle^{-s_0}\\
			&\lesssim\|T\|_{\textnormal{\tiny{I-D}},s}\| S\|_{\textnormal{\tiny{O-d}},s_0} .
		\end{align*}
		Putting together the preceding estimates we get
		\begin{align*}
			\|ST\|_{\textnormal{\tiny{I-D}},s}&\lesssim \|T\|_{\textnormal{\tiny{I-D}},s} \| S\|_{\textnormal{\tiny{O-d}},s_0}+ \|T\|_{\textnormal{\tiny{I-D}},s_0}\|S\|_{\textnormal{\tiny{O-d}},s}.
		\end{align*}
		This concludes the proof of the lemma.
	\end{proof}
	
	\subsubsection{Integral operators}
	{The main goal in this part is to analyze  Toeplitz integral operators and connect the different norms introduced before to the regularity of the kernel.  Consider a Toeplitz integral operator taking the form }
	\begin{equation}\label{Top-op1}(\mathcal{T}_K\rho)(\mu,\varphi,\theta)\triangleq\int_{\mathbb{T}}K(\mu,\varphi,\theta,\eta)\rho(\mu,\varphi,\eta)d\eta,
	\end{equation}
	where the kernel function $K(\mu,\varphi,\theta,\eta)$ may be smooth or singular at the  diagonal set $\{\theta=\eta\}$. The kernel is called $\mathtt{m}$-fold preserving if 
	$$(\mathscr{I}_{\mathtt{m},2}K)(\mu,\varphi,\theta,\eta)\triangleq K\left(\mu,\varphi,\theta+\tfrac{2\pi}{\mathtt{m}},\eta+\tfrac{2\pi}{\mathtt{m}}\right)=K(\mu,\varphi,\theta,\eta).$$
	We shall need the following  lemma whose proof is a consequence of  \cite[Lem. 4.4]{HR21}.
	\begin{lem}\label{lem sym--rev}
		Let $q\in\N$, $\m\in\N^*$, $(\gamma,d,s_{0},s)$ satisfy \eqref{setting tau1 and tau2}-\eqref{init Sob cond} and $\mathcal{T}_K$ be an integral operator with a real-valued kernel $K$. Then the following assertions hold true.
		\begin{enumerate}[label=\textbullet]
			\item If $K$ is even in $(\varphi,\theta,\eta)$, then $\mathcal{T}_K$ is reversibility preserving.
			\item If $K$ is odd in $(\varphi,\theta,\eta)$,
			then $\mathcal{T}_K$ is reversible.
			\item If $K$ is $\mathtt{m}$-fold preserving, then $\mathcal{T}_K$ is  $\mathtt{m}$-fold preserving.
		\end{enumerate}
		In addition, 
		$$\|\mathcal{T}_K\|_{\textnormal{\tiny{I-D}},s}^{q,\gamma,\mathtt{m}}\lesssim\|K\|_{s}^{q,\gamma,\mathtt{m}}$$
		and
		\begin{align*}
			\| \mathcal{T}_K\rho\|_{s}^{q,\gamma,\mathtt{m}}&\lesssim \|\rho\|_{s_0}^{q,\gamma,\mathtt{m}}\|K\|_{s}^{q,\gamma,\mathtt{m}} +\|\rho\|_{s}^{q,\gamma,\mathtt{m}} \|K\|_{s_0}^{q,\gamma,\mathtt{m}}.
		\end{align*}
	\end{lem}
	\begin{proof}
		The reversibility properties have already been proved in \cite[Lem. 4.4]{HR21}. Now, let us prove the $\mathtt{m}$-fold property. Assume that $K$ is $\mathtt{m}$-fold preserving, then
		\begin{align*}
			\mathcal{T}_{K}(\mathscr{I}_{\mathtt{m}}\rho)(\varphi,\theta)&=\int_{\mathbb{T}}K(\varphi,\theta,\eta)\rho\big(\varphi,\eta+\tfrac{2\pi}{\mathtt{m}}\big)d\eta\\
			&=\int_{\mathbb{T}}K\big(\varphi,\theta+\tfrac{2\pi}{\mathtt{m}},\eta+\tfrac{2\pi}{\mathtt{m}}\big)\rho\big(\varphi,\eta+\tfrac{2\pi}{\mathtt{m}}\big)d\eta\\
			&=\int_{\mathbb{T}}K\big(\varphi,\theta+\tfrac{2\pi}{\mathtt{m}},\eta\big)\rho\big(\varphi,\eta\big)d\eta\\
			&=\mathscr{I}_{\mathtt{m}}(\mathcal{T}_{K}\rho)(\varphi,\theta).
		\end{align*}
		Hence $\mathcal{T}_{K}$ is $\mathtt{m}$-fold preserving. By duality $H_{\mathtt{m}}^{-s}-H_{\mathtt{m}}^{s}$, we have
		\begin{align*}
			\left|(\mathcal{T}_{K})_j^{j'}(l)\right|&=\left|\int_{\mathbb{T}^{d+2}}K(\varphi,\theta,\eta)e^{\ii(l\cdot\varphi+j\theta-j'\eta)}d\varphi d\theta d\eta\right|\\
			&\lesssim\langle l,j,j'\rangle^{-s}\|K\|_{s}^{q,\gamma,\m}
		\end{align*}
		proving the first estimate. The second one follows easily from Lemma \ref{properties of Toeplitz in time operators}-(ii)-(iii).
	\end{proof}
	The next task is to introduce some quasi-periodic symplectic change of variables needed later in the reduction of the transport part of the linearized operator. The following lemma is proved in the scalar case $d^\prime=1$ in \cite[Lem. 2.34]{BM18}. The vectorial case $d^\prime\geqslant 2$ can be obtained in a similar way, up to slight modifications.
	\begin{lem}\label{Compos1-lemm}
		Let $q\geqslant 0 $, $\m, d,d^\prime\geqslant 1$,
		$s\geqslant s_0> \tfrac{d+d^\prime}{2}+q+1$
		and $\beta_1,\cdots,\beta_{d^\prime}\in W^{q,\infty,\gamma}\big(\mathcal{O},H^{\infty}_{\m}(\T^{d+1})\big)$  such that 
		\begin{equation}\label{small beta lem}
			\max_{k\in\{1,\cdots,d^\prime\}}\|\beta_k \|_{2s_0}^{q,\gamma,\mathtt{m}}\leqslant \varepsilon_0,
		\end{equation}
		with $\varepsilon_{0}$ small enough. 		 Then the following assertions hold true.
		\begin{enumerate}[label=(\roman*)]
			\item The function $\widehat{\beta}$ defined by the inverse diffeomorphism
			$$y=x+\beta(\mu,\varphi,x)\qquad\Leftrightarrow\qquad x=y+\widehat\beta(\mu,\varphi,y),$$
			where
			\begin{align}\label{beta-new}
				\beta(\mu,\varphi,x)&\triangleq  \big(\beta_1(\mu,\varphi,x_1),\cdots, \beta_{d^\prime}(\mu,\varphi,x_{d^\prime})\big), \quad x=(x_1,\cdots,x_{d^\prime}),\\
				\widehat{\beta}(\mu,\varphi,y)&\triangleq  \big(\widehat{\beta}_1(\mu,\varphi,y_1),\cdots, \widehat{\beta}_{d^\prime}(\mu,\varphi,y_{d^\prime})\big), \,\quad y=(y_1,\cdots,y_{d^\prime}),\notag
			\end{align}
			satisfies
			\begin{equation}\label{beta-hat and beta norm}
				\forall s\geqslant s_0,\quad\|\widehat{\beta}\|_{s}^{q,\gamma,\mathtt{m}}\lesssim  \|\beta\|_{s}^{q,\gamma,\mathtt{m}}.
			\end{equation}
			\item The composition operator $\mathcal{B}:W^{q,\infty,\gamma}\big(\mathcal{O},H^{s}_{\m}(\T^{d+d^\prime})\big)\to W^{q,\infty,\gamma}\big(\mathcal{O},H^{s}_{\m}(\T^{d+d^\prime})\big)$, defined by
			\begin{equation}\label{def symplctik CVAR}
				\mathcal{B}\rho(\mu,\varphi,x)\triangleq \rho\big(\mu,\varphi,x+\beta(\mu,\varphi,x)\big),
			\end{equation}
			is continuous and invertible, with inverse
			$$
			\mathcal{B}^{-1} \rho(\mu,\varphi,y)=\rho\big(\mu,\varphi,y+\widehat{\beta}(\mu,\varphi,y)\big).
			$$
			Moreover, we have the estimates  
			\begin{align}
				\|\mathcal{B}^{\pm1}\rho\|_{s}^{q,\gamma,\mathtt{m}}&\leqslant \|\rho\|_{s}^{q,\gamma,\mathtt{m}}\left(1+C\|\beta\|_{s_{0}}^{q,\gamma,\mathtt{m}}\right)+C\|\beta\|_{s}^{q,\gamma,\mathtt{m}}\|\rho\|_{s_{0}}^{q,\gamma,\mathtt{m}},\nonumber
				\\
				\|\mathcal{B}^{\pm1}\rho-\rho\|_{s}^{q,\gamma,\mathtt{m}}&\leqslant C\left(\|\rho\|_{s+1}^{q,\gamma,\mathtt{m}}\|\beta\|_{s_0}^{q,\gamma,\mathtt{m}}+\|\rho\|_{s_0}^{q,\gamma,\mathtt{m}}\|\beta\|_{s}^{q,\gamma,\mathtt{m}}\right).\label{e-spe comp B}
			\end{align}
			\item Let $\beta^{[1]},\beta^{[2]}\in W^{q,\infty,\gamma}\big(\mathcal{O},H_{\m}^{\infty}(\mathbb{T}^{d+d^\prime})\big)$ as in \eqref{beta-new} and satisfying \eqref{small beta lem}. If we denote 
			$$\Delta_{12}\beta\triangleq\beta^{[1]}-\beta^{[2]}\qquad\textnormal{and}\qquad\Delta_{12}\widehat{\beta}\triangleq\widehat{\beta}^{[1]}-\widehat{\beta}^{[2]},$$
			then we have
			\begin{equation}\label{Delta12 bh vs Delta12 b}
				\forall s\geqslant s_0,\quad\|\Delta_{12}\widehat{\beta}\|_{s}^{q,\gamma,\mathtt{m}}\leqslant C\left(\|\Delta_{12}\beta\|_{s}^{q,\gamma,\mathtt{m}}+\|\Delta_{12}\beta\|_{s_{0}}^{q,\gamma,\mathtt{m}}\max_{\ell\in\{1,2\}}\|\beta^{[\ell]}\|_{s+1}^{q,\gamma,\mathtt{m}}\right).
			\end{equation}
		\end{enumerate}
	\end{lem}
	Next, we gather several results related to the action of the transformation \eqref{def symplctik CVAR} on Toeplitz  integral operators.
	\begin{lem}\label{lem CVAR kernel}
		Let $q\in\N$, $\m\in\N^*$ and $(\gamma,d,s_{0},s)$ satisfy \eqref{setting tau1 and tau2}-\eqref{init Sob cond}. Consider a smooth $\mathtt{m}$-fold preserving kernel
		$$K:(\mu,\varphi,\theta_1,\theta_2)\mapsto K(\mu,\varphi,\theta_1,\theta_2).$$
		Let $\beta_k: \mathcal{O}\times \T^{d+1}\to \T$, $k\in\{1,2\}$ be odd $\m$-fold symmetric functions and subject to the smallness condition
		\begin{equation}\label{assumption smallness beta lem}
			\max_{k\in\{1,2\}}\|\beta_k \|_{2s_0}^{q,\gamma,\mathtt{m}}\leqslant \varepsilon_0.
		\end{equation} 
		Consider the quasi-periodic change of variables
		$$\forall k\in\{1,2\},\quad \mathscr{B}_k\triangleq (1+\partial_{\theta}\beta_k)\mathcal{B}_k,\qquad\mathcal{B}_k\rho(\mu,\varphi,\theta)=\rho\big(\mu,\varphi,\theta+\beta_k(\mu,\varphi,\theta)\big),$$
		Then the following assertions hold true.
		\begin{enumerate}[label=(\roman*)]
			\item The operator $\mathscr{B}_1^{-1}\mathcal{T}_K\mathscr{B}_2$ is $\mathtt{m}$-fold  preserving  integral operator. Moreover, we have 
			\begin{equation}\label{e-odsBtB}
				\|\mathscr{B}_1^{-1}\mathcal{T}_K\mathscr{B}_2\|_{\textnormal{\tiny{I-D}},s}^{q,\gamma,\mathtt{m}}\lesssim\|K\|_{s}^{q,\gamma,\mathtt{m}}+\|K\|_{s_0}^{q,\gamma,\mathtt{m}}\,\max_{k\in\{1,2\}}\|\beta_{k}\|_{s+1}^{q,\gamma,\mathtt{m}}.
			\end{equation}
			and
			\begin{equation}\label{e-odsBtB-diff}
				\|\mathscr{B}_1^{-1}\mathcal{T}_K\mathscr{B}_2-\mathcal{T}_K\|_{\textnormal{\tiny{I-D}},s}^{q,\gamma,\mathtt{m}}\lesssim\|K\|_{s+1}^{q,\gamma,\mathtt{m}}\,\max_{k\in\{1,2\}}\|\beta_{k}\|_{s_0}^{q,\gamma,\mathtt{m}}+\|K\|_{s_0}^{q,\gamma,\mathtt{m}}\,\max_{k\in\{1,2\}}\|\beta_{k}\|_{s+1}^{q,\gamma,\mathtt{m}}.
			\end{equation}
			
			\item If $K$ is even in all the variables $(\varphi,\theta_1,\theta_2)$ (resp. odd), then $\mathscr{B}_1^{-1}\mathcal{T}_K\mathscr{B}_2$ is a reversibility preserving (resp.  reversible) integral operator.			
			\item Given smooth functionals  $r\in W^{q,\infty,\gamma}(\mathcal{O},H_{\mathtt{m}}^{s}\times H_{\mathtt{m}}^{s})\mapsto K(r), \beta_{k}(r)$, for $k\in \{1,2\}$.\\
			Consider $r^{[\ell]}=(r_{1}^{[\ell]},r_2^{[\ell]})\in W^{q,\infty,\gamma}(\mathcal{O},H^s_{\mathtt{m}}\times H_{\mathtt{m}}^{s})$, $\ell\in\{1,2\}$.  We denote
			$$\forall k\in\{1,2\},\quad f^{[k]}\triangleq f(r^{[k]})\qquad\textnormal{and}\qquad\Delta_{12}f\triangleq f^{[1]}-f^{[2]}.$$ 
			Assume that there exists $\varepsilon_0>0$ small enough such that 
			\begin{equation}\label{assumption smallness beta K lem}
				\max_{(k,\ell)\in\{1,2\}^2}\|\beta_{k}^{[\ell]}\|_{2s_0}^{q,\gamma,{\mathtt{m}}}+\max_{\ell\in\{1,2\}}\|K^{[\ell]}\|_{s_0+1}^{q,\gamma,{\mathtt{m}}}\leqslant\varepsilon_0.
			\end{equation}
			Then, the following estimate holds,
			\begin{align}\label{diff12BoxBtB}
				\|\Delta_{12}\mathscr{B}_{1}^{-1}\mathcal{T}_{K}&\mathscr{B}_{2}\|_{\textnormal{\tiny{I-D}},s}^{q,\gamma,\mathtt{m}}\lesssim\|\Delta_{12}K\|_{s}^{q,\gamma,{\mathtt{m}}}+\|\Delta_{12}K\|_{s_0}^{q,\gamma,{\mathtt{m}}}\max_{(k,\ell)=\{1,2\}^2}\|\beta_{k}^{[\ell]}\|_{s}^{q,\gamma,\mathtt{m}}\\
				&\quad+\Big(\max_{\ell\in\{1,2\}}\|K^{[\ell]}\|_{s+1}^{q,\gamma,\mathtt{m}}+\max_{(k,\ell)\in\{1,2\}^2}\|\beta_{k}^{[\ell]} \|_{s+1}^{q,\gamma,\mathtt{m}}\Big)\max_{k\in\{1,2\}}\|\Delta_{12}\beta_{k}\|_{s_0+1}^{q,\gamma,\mathtt{m}}\nonumber\\
				&\quad+ \max_{k\in\{1,2\}}\|\Delta_{12}\beta_{k}\|_{s+1}^{q,\gamma,\mathtt{m}}. \nonumber
			\end{align}
		\end{enumerate}
	\end{lem}
	\begin{proof}
		\textbf{(i)}  
		Straightforward computations lead to
		\begin{equation}\label{mathscrB1}
			\mathscr{B}_1^{-1}\rho(\mu,\varphi,y_1)=\Big(1+\partial_y\widehat{\beta}_1(\mu,\varphi,y_1)\Big) \rho\big(\mu,\varphi,y_1+\widehat{\beta}_1(\mu,\varphi,y_1)\big).
		\end{equation}
		Thus, the conjugation of the operator $\mathcal{T}_K$ writes
		\begin{align}\label{comp b-1tb}
			\mathscr{B}_1^{-1}\mathcal{T}_{K}\mathscr{B}_2\rho(\mu,\varphi,\theta_1)&=\int_{\mathbb{T}}\rho(\mu,\varphi,{\theta_2})\widehat{K}(\mu,\varphi,\theta_1,{\theta_2})d{\theta_2},
		\end{align}
		with
		$$\widehat{K}(\mu,\varphi,\theta_1,{\theta_2})\triangleq \big(1+\partial_{\theta_1}\widehat{\beta}_1(\mu,\varphi,\theta_1)\big)(\mathcal{B}^{-1} K)(\mu,\varphi,\theta_1,{\theta_2})$$
		and
		$$
		\mathcal{B}^{-1} K(\mu,\varphi,\theta_1,\theta_2)=K\big(\mu,\varphi,\theta_1+\widehat{\beta}_1(\mu,\varphi,\theta_1),{\theta_2}+\widehat{\beta}_2(\mu,\varphi,{\theta_2})\big).
		$$
		Using the product laws in Lemma \ref{lem funct prop}, Lemma \ref{Compos1-lemm} and \eqref{assumption smallness beta lem}, we get
		\begin{equation}\label{e-B1B2rho}
			\|\widehat{K}\|_{s}^{q,\gamma,\mathtt{m}}\lesssim\|K\|_{s}^{q,\gamma,\mathtt{m}}+\|K\|_{s_0}^{q,\gamma,\mathtt{m}}\max_{k\in\{1,2\}}\|\beta_k\|_{s+1}^{q,\gamma,\mathtt{m}}.
		\end{equation}
		Consequently, we obtain the estimate \eqref{e-odsBtB} by applying Lemma \ref{lem sym--rev}. As for the difference with the original operator, we can write
		$$\big(\mathscr{B}_{1}^{-1}\mathcal{T}_{K}\mathscr{B}_2-\mathcal{T}_{K}\big)\rho(\mu,\varphi,\theta_1)=\int_{\mathbb{T}}\rho(\mu,\varphi,{\theta_2})\widetilde{K}(\mu,\varphi,\theta_1,{\theta_2})d{\theta_2},$$
		with
		$$\widetilde{K}(\mu,\varphi,\theta_1,{\theta_2})\triangleq \partial_{\theta_1}\widehat{\beta}_1(\mu,\varphi,\theta_1)(\mathcal{B}^{-1}K)(\mu,\varphi,\theta_1,{\theta_2})+\big[\mathcal{B}^{-1}K-K\big](\mu,\varphi,\theta_1,{\theta_2}).$$
		Therefore, by the product laws in Lemma \ref{lem funct prop}, together with \eqref{e-spe comp B} and \eqref{assumption smallness beta lem}, we infer
		$$\|\widetilde{K}\|_{s}^{q,\gamma,\mathtt{m}}\lesssim\|K\|_{s+1}^{q,\gamma,\mathtt{m}}\max_{k\in\{1,2\}}\|\beta_k\|_{s_0}^{q,\gamma,\mathtt{m}}+\|K\|_{s_0}^{q,\gamma,\mathtt{m}}\max_{k\in\{1,2\}}\|\beta_k\|_{s+1}^{q,\gamma,\mathtt{m}}.$$
		Then, the estimate \eqref{e-odsBtB-diff} follows by applying Lemma \ref{lem sym--rev}.\\
		\noindent \textbf{(ii)} The symmetry properties follow immediately from Lemma \ref{lem sym--rev} and the symmetry assumptions.\\
		\noindent \textbf{(iii)} By definition and according to \eqref{comp b-1tb} we have
		\begin{align*}
			\Delta_{12}(\mathscr{B}_1^{-1}\mathcal{T}_K\mathscr{B}_{2})(\rho)(\mu,\varphi,\theta_1)&=\int_{\mathbb{T}}\rho(\mu,\varphi,{\theta_2})\mathbb{K}(\mu,\varphi,\theta_1,{\theta_2})d{\theta_2},
		\end{align*}
		with 
		\begin{align*}
			\mathbb{K}(\mu,\varphi,\theta_1,{\theta_2})&\triangleq \big(1+\partial_{\theta_1}\widehat{\beta}_1^{[1]}(\mu,\varphi,\theta_1)\big)\mathcal{B}_{[1]}^{-1}K^{[1]}\big(\mu,\varphi,\theta_1,{\theta_2}\big)\\ &\quad-\big(1+\partial_{\theta_1}\widehat{\beta}_1^{[2]}(\mu,\varphi,\theta_1)\big)\mathcal{B}_{[2]}^{-1}K^{[2]}\big(\mu,\varphi,\theta_1,{\theta_2}\big)
		\end{align*}	
		and
		$$
		\mathcal{B}^{-1}_{[\ell]} f(\mu,\varphi,y)=f\big(\mu,\varphi,\theta_1+\widehat{\beta}_1^{[\ell]}(\mu,\varphi,\theta_1),{\theta_2}+\widehat{\beta}_2^{[\ell]}(\mu,\varphi,{\theta_2})\big).
		$$
		This can also be written as
		\begin{align*}
			\mathbb{K}(\mu,\varphi,\theta_1,{\theta_2})&=\partial_{\theta_1}\Delta_{12}\widehat{\beta}_1(\mu,\varphi,\theta_1)\mathcal{B}_{[1]}^{-1}K^{[1]}\big(\mu,\varphi,\theta_1,{\theta_2}\big)\\
			&\quad+\big(1+\partial_{\theta_1}\widehat{\beta}_1^{[2]}(\mu,\varphi,\theta_1)\big)\mathcal{B}_{[1]}^{-1}(\Delta_{12}K)\big(\mu,\varphi,\theta_1,{\theta_2}\big)\\
			&\quad+\big(1+\partial_{\theta_1}\widehat{\beta}_1^{[2]}(\mu,\varphi,\theta_1)\big)\Big[\mathcal{B}_{[1]}^{-1}K^{[2]}\big(\mu,\varphi,\theta_1,{\theta_2}\big)
			-\mathcal{B}_{[2]}^{-1}K^{[2]}\big(\mu,\varphi,\theta_1,{\theta_2}\big)\Big].
		\end{align*}
		Then, Taylor Formula implies
		\begin{align*}
			\mathbb{K}(\mu,\varphi,\theta_1,{\theta_2})&=\partial_{\theta_1}\Delta_{12}\widehat{\beta}_1(\mu,\varphi,\theta_1)\mathcal{B}_{[1]}^{-1}K^{[1]}\big(\mu,\varphi,\theta_1,{\theta_2}\big)\\
			&\quad+\big(1+\partial_{\theta_1}\widehat{\beta}_1^{[2]}(\mu,\varphi,\theta_1)\big)\mathcal{B}_{[1]}^{-1}(\Delta_{12}K)\big(\mu,\varphi,\theta_1,{\theta_2}\big)\\
			&\quad+\big(1+\partial_{\theta_1}\widehat{\beta}_1^{[2]}(\mu,\varphi,\theta_1)\big)\Big[\Delta_{12}\widehat{\beta}_1(\mu,\varphi,\theta_1)\int_{0}^1\big(\mathcal{B}_{[1],[2]}^\tau(\partial_{\theta_1}K^{[2]})\big)\Big(\mu,\varphi,\theta_1,{\theta_2}\Big)d\tau\\
			&\quad+\Delta_{12}\widehat{\beta}_2(\mu,\varphi,{\theta_2})\int_{0}^1\big(\widetilde{\mathcal{B}}_{[1],[2]}^\tau(\partial_{{\theta_2}}K^{[2]})\big)\Big(\mu,\varphi,\theta_1,{\theta_2}\Big)d\tau\Big],
		\end{align*}
		where we have used the notations
		\begin{align*}
			\mathcal{B}_{[1],[2]}^\tau f(\mu,\varphi,\theta_1,{\theta_2})&=f\big(\mu,\varphi,\theta_1+\tau\widehat{\beta}_{1}^{[1]}(\mu,\varphi,\theta_1)+(1-\tau)\widehat{\beta}_{1}^{[2]}(\mu,\varphi,\theta_1),{\theta_2}+\widehat{\beta}_2^{[2]}(\mu,\varphi,{\theta_2})\big),\\
			\widetilde{\mathcal{B}}_{[1],[2]}^\tau f(\mu,\varphi,\theta_1,{\theta_2})&=f\big(\mu,\varphi,\theta_1+\widehat{\beta}_1^{[2]}(\mu,\varphi,\theta_1),{\theta_2}+\tau\widehat{\beta}_{2}^{[1]}(\mu,\varphi,{\theta_2})+(1-\tau)\widehat{\beta}_{2}^{[2]}(\mu,\varphi,{\theta_2})\big).
		\end{align*}
		By product laws, \eqref{Delta12 bh vs Delta12 b}, \eqref{e-B1B2rho}, \eqref{assumption smallness beta K lem}, we obtain
		\begin{align*}
			\|\mathbb{K}\|_{s}^{q,\gamma,\mathtt{m}}&\lesssim\|\Delta_{12}K\|_{s}^{q,\gamma,\mathtt{m}}+\|\Delta_{12}K\|_{s_0}^{q,\gamma,\mathtt{m}}\max_{(k,\ell)\in\{1,2\}^2}\|\beta_{k}^{[\ell]}\|_{s}^{q,\gamma,\mathtt{m}}+\max_{k\in\{1,2\}}\|\Delta_{12}\beta_{k}\|_{s+1}^{q,\gamma,\mathtt{m}}\\
			&\quad+\left(\max_{\ell\in\{1,2\}}\|K^{[\ell]}\|_{s+1}^{q,\gamma,\mathtt{m}}+\max_{(k,\ell)\in\{1,2\}^2}\|\beta_{k}^{[\ell]}\|_{s+1}^{q,\gamma,\mathtt{m}}\right)\max_{k\in\{1,2\}}\|\Delta_{12}\beta_k\|_{s_0+1}^{q,\gamma,\mathtt{m}}.
		\end{align*}
		Finally, combining this estimate with Lemma \ref{lem sym--rev} we conclude \eqref{diff12BoxBtB}. This achieves the proof of Lemma~\ref{lem CVAR kernel}.
	\end{proof}
	Now, we recall the  following result stated in  \cite[Lemma 2.36]{BM18} and dealing with  the conjugation of the Hilbert transform with the quasi-periodic change of coordinates introduced in \eqref{mathscrB1}. Here, and along the paper, $ {\mathcal H}$ denotes the standard Hilbert transform on the periodic setting acting only on the variable $\theta\in\T,$ namely,
	\begin{equation}\label{def Hilbert}
		\mathcal{H}(1)=0,\qquad\forall j\in\mathbb{Z}^*,\quad\mathcal{H}\mathbf{e}_j=-\ii\, \mathtt{sgn}(j)\mathbf{e}_j,
	\end{equation}
	where $\mathtt{sgn}$ denotes the usual sign function.
	\begin{lem} \label{lemma:conjug-Hilbert}
		Let $q\in\N$, $\m\in\N^*$, $(\gamma,d,s_{0},s)$ satisfy \eqref{setting tau1 and tau2}-\eqref{init Sob cond} and $\beta\in W^{q,\infty,\gamma}(\mathcal{O},H_{\mathtt{m}}^{\infty})$ odd in the variables $(\varphi,\theta)$. There exists $\varepsilon_0> 0$ such that, if $\|\beta\|_{2s_0 }^{q,\gamma,\mathtt{m}} \leqslant \varepsilon_0$, then
		$$
		( {\mathscr B}^{-1} {\cal H} {\mathscr B} - {\cal H}) \rho (\mu, \varphi, \theta) 
		= \int_{\mathbb{T}}  \,K(\mu, \varphi, \theta,\eta) \rho(\mu, \varphi, \eta)\,d\eta,
		$$
		defines a reversible and $\mathtt{m}$-fold preserving integral operator with the estimates : for all $s\geqslant s_0,$  
		$$\|K\|_{s}^{q,\gamma,\mathtt{m}}\leqslant C(s,q)\|\beta\|_{s+2}^{q,\gamma,\mathtt{m}},\qquad\|\mathscr{B}^{-1}\mathcal{H}\mathscr{B}-\mathcal{H}\|_{\textnormal{\tiny{I-D}},s}^{q,\gamma,\mathtt{m}}\leqslant C(s,q)\|\beta \|_{s+2}^{q,\gamma,\mathtt{m}}$$
		and
		$$\|\Delta_{12}(\mathscr{B}^{-1}\mathcal{H}\mathscr{B}-\mathcal{H})\|_{\textnormal{\tiny{I-D}},s}^{q,\gamma,\mathtt{m}}\leqslant C(s,q)\Big[\|\Delta_{12}\beta\|_{s+2}^{q,\gamma,\mathtt{m}}+\|\Delta_{12}\beta\|_{s_0+1}^{q,\gamma,\mathtt{m}}\max_{\ell\in\{1,2\}}\|\beta^{[\ell]}\|_{s+3}^{q,\gamma,\mathtt{m}}\Big].$$
	\end{lem}
	\begin{proof}
		We shall use the following classical integral representation of the Hilbert transform
		$$\mathcal{H}(\rho)(\theta)=\int_{\mathbb{T}}\rho(\eta)\cot\left(\frac{\theta-\eta}{2}\right)d\eta,$$
		where this integral is understood in the principal value sense. Therefore, we have
		$$K(\mu,\varphi,\theta,\eta)=\big(1+\partial_{\theta}\widehat{\beta}(\mu,\varphi,\theta)\big)\cot\left(\frac{\theta-\eta+\widehat{\beta}(\mu,\varphi,\theta)-\widehat{\beta}(\mu,\varphi,\eta)}{2}\right)-\cot\left(\frac{\theta-\eta}{2}\right).$$
		One can easily check that
		$$K(\mu,\varphi,\theta,\eta)=2\partial_{\theta}\left[\log\left(\frac{\sin\left(\frac{\theta-\eta+\widehat{\beta}(\mu,\varphi,\theta)-\widehat{\beta}(\mu,\varphi,\eta)}{2}\right)}{\sin\left(\frac{\theta-\eta}{2}\right)}\right)\right].$$
		This can also be written as
		$$K(\mu,\varphi,\theta,\eta)=2\partial_{\theta}\Big[\log\big(1+g(\mu,\varphi,\theta,\eta)\big)\Big],$$
		where
		$$g(\mu,\varphi,\theta,\eta)=\cos\left(\tfrac{\widehat{\beta}(\mu,\varphi,\theta)-\widehat{\beta}(\mu,\varphi,\eta)}{2}\right)-1+\cos\left(\tfrac{\theta-\eta}{2}\right)\frac{\sin\left(\frac{\widehat{\beta}(\mu,\varphi,\theta)-\widehat{\beta}(\mu,\varphi,\eta)}{2}\right)}{\sin\left(\frac{\theta-\eta}{2}\right)}\cdot$$
		The symmetry assumptions on $\beta$ (and thus $\widehat{\beta}$) implies
			$$g\big(\mu,\varphi,\theta+\tfrac{2\pi}{\mathtt{m}},\eta+\tfrac{2\pi}{\mathtt{m}}\big)=g(\mu,\varphi,\theta,\eta)=g(\mu,-\varphi,-\theta,-\eta),$$
			that is
			$$K\big(\mu,\varphi,\theta+\tfrac{2\pi}{\mathtt{m}},\eta+\tfrac{2\pi}{\mathtt{m}}\big)=K(\mu,\varphi,\theta,\eta)=-K(\mu,-\varphi,-\theta,-\eta),$$
			The Lemma \ref{lem sym--rev} implies that $\mathscr{B}^{-1}\mathcal{H}\mathscr{B}-\mathcal{H}$ is a reversible and $\mathtt{m}$-fold preserving integral operator.
		Using composition laws in Lemma \ref{lem funct prop}, Lemma \ref{lem triche} and \eqref{beta-hat and beta norm}, we get
		$$\|g\|_{s}^{q,\gamma,\mathtt{m}}\lesssim\|\beta\|_{s+1}^{q,\gamma,\mathtt{m}}.$$
		Hence, still by composition laws, we infer
		$$\|K\|_{s}^{q,\gamma,\mathtt{m}}\lesssim\|\beta\|_{s+2}^{q,\gamma,\mathtt{m}}$$
		and we conclude by applying Lemma \ref{lem sym--rev}. As for the difference, we have
		$$\|\Delta_{12}K\|_{s}^{q,\gamma,\mathtt{m}}\lesssim\|\Delta_{12}g\|_{s+1}^{q,\gamma,\mathtt{m}}.$$
		We set
		$$h(\mu,\varphi,\theta,\eta)\triangleq \frac{\widehat{\beta}(\mu,\varphi,\theta)-\widehat{\beta}(\mu,\varphi,\eta)}{2}\cdot$$
		Then, using Taylor formula, we can write
		\begin{align*}
			\Delta_{12}g(\mu,\varphi,\theta,\eta)=&-\Delta_{12}h(\mu,\varphi,\theta,\eta)\int_{0}^{1}\sin\big(h^{[2]}(\mu,\varphi,\theta,\eta)+t\Delta_{12}h(\mu,\varphi,\theta,\eta)\big)dt\\
			&+\cos\left(\tfrac{\theta-\eta}{2}\right)\frac{\Delta_{12}h(\mu,\varphi,\theta,\eta)}{\sin\left(\frac{\theta-\eta}{2}\right)}\int_{0}^{1}\cos\big(h^{[2]}(\mu,\varphi,\theta,\eta)+t\Delta_{12}h(\mu,\varphi,\theta,\eta)\big)dt.
		\end{align*}
		From the identity $\Delta_{12}h(\mu,\varphi,\theta,\eta)=\Delta_{12}\widehat{\beta}(\mu,\varphi,\theta)-\Delta_{12}\widehat{\beta}(\mu,\varphi,\eta)$ together with the  product/composition laws combined with Lemma \ref{lem triche} and \eqref{Delta12 bh vs Delta12 b}, we get
		$$\|\Delta_{12}g\|_{s}^{q,\gamma,\mathtt{m}}\lesssim\|\Delta_{12}\beta\|_{s+1}^{q,\gamma,\mathtt{m}}+\|\Delta_{12}\beta\|_{s_0+1}^{q,\gamma,\mathtt{m}}\max_{\ell\in\{1,2\}}\|\beta^{[\ell]}\|_{s+2}^{q,\gamma,\mathtt{m}}.$$
		Again we conclude by invoking Lemma \ref{lem sym--rev}.
	\end{proof}
	The following lemma deals with the kernel structure of iterative operators that will be useful later.
	\begin{lem} \label{iter-kerns}
		Let $q\in\N$, $\m\in\N^*$, $n\in \N^*$  and $(\gamma,d,s_{0},s)$ satisfy \eqref{setting tau1 and tau2}-\eqref{init Sob cond} and consider a  family of $\mathtt{m}$-fold preserving kernel operators $(\mathcal{T}_{K_i})_{i=1}^n$ as in \eqref{Top-op1}. Then there exists a kernel $K$ such that 
		$$
		\prod_{i=1}^n\mathcal{T}_{K_i}=\mathcal{T}_{K},\qquad\textnormal{with}\qquad 
		\|K\|_{s}^{q,\gamma,\mathtt{m}}\leqslant C\sum_{i=1}^n\|K_i\|_{s}^{q,\gamma,\mathtt{m}}\prod_{j\neq i}\|K_{j}\|_{s_0}^{q,\gamma,\mathtt{m}}.
		$$
		In addition, if for some $i_0$ we have $K_{i_0}(\varphi,\theta,\eta)=f(\varphi,\theta)\delta({\theta-\eta}),$ with $f\big(\varphi,\theta+\tfrac{2\pi}{\mathtt{m}}\big)=f(\varphi,\theta)$, then
		$$
		\|K\|_{s}^{q,\gamma,\mathtt{m}}\leqslant C\|f\|_{s}^{q,\gamma,\mathtt{m}}\prod_{i\neq i_0}\|K_{i}\|_{s_0}^{q,\gamma,\mathtt{m}}+C\|f\|_{s_0}^{q,\gamma,\mathtt{m}}\sum_{i=1,\atop  i\neq i_0}^n\|K_i\|_{s}^{q,\gamma,\mathtt{m}}\prod_{j\neq i, i_0}\|K_{j}\|_{s_0}^{q,\gamma,\mathtt{m}}.
		$$
	\end{lem}
	\begin{proof}
		The kernel $K$ is explicit and takes the form,
		$$
		K(\varphi,\theta,\eta)=\bigintsss_{\T^{n-1}}\prod_{i=1}^{n}K_i(\varphi,\eta_{i-1},\eta_{i}) \prod_{i=1}^{n-1}d\eta_i,
		$$
		with the convention $\eta_0=\theta$ and $\eta_{n}=\eta.$ The $\mathtt{m}$-fold preserving property of $K$ is inherited from the one of the $K_i$. Thus, to get the first result it suffices to use the products law in Lemma  \ref{lem funct prop}. In the second case, the kernel takes the form
		$$
		K(\varphi,\theta,\eta)=\bigintsss_{\T^{n-2}}f(\varphi,\theta,\eta_{i_0})K_{i_0-1}(\varphi,\eta_{i_0-2},\eta_{i_0}) \prod_{i=1\atop i\neq i_0,i_0-1}^{n}K_i(\varphi,\eta_{i-1},\eta_{i}) \prod_{i=1\atop i\neq i_0-1}^{n-1}d\eta_i
		$$
		and the desired estimate follows once again from the products law detailed in Lemma  \ref{lem funct prop}.
	\end{proof}
	
	\subsubsection{Matrix operators}\label{sec matrix op}
	For further purposes related to the  reduction of  the remainder in the transport linear parts subject to a vectorial structure, we need to introduce $2\times2$ matrices of scalar operators taking the form
	\begin{align}\label{Op-Vec1}
		\mathbf{T}=\begin{pmatrix}
			T_1 & T_3\\
			T_4 & T_2
		\end{pmatrix},
	\end{align}
	acting on {the product Hilbert space}  $\mathbf{H}^{s}_{\mathtt{m}}(\mathbb{T}^{d+1},\mathbb{C})$, defined in \eqref{def:sob-product} . 
	Notice that we shall restrict our discussion to the case where all the $T_i:\mathcal{O}\rightarrow\mathcal{L}\big(H_{\m}^s(\mathbb{T}^{d+1},\mathbb{C})\big)$  are $\mathtt{m}$-fold preserving Toeplitz kernel operators as in \eqref{Top-op1}. The matrix operator $\mathbf{T}$ is said to be real (resp. $\mathtt{m}$-fold preserving, reversible, reversibility-preserving)  if all the entries $T_i$ enjoy  this property.
	The diagonal part $\lfloor\mathbf{T}\rfloor$ of   $\mathbf{T}$ is defined as follows, 
	\begin{equation}\label{def diag-diag}
		\lfloor\mathbf{T}\rfloor\triangleq\begin{pmatrix}
			\lfloor T_1\rfloor & 0\\
			0 & \lfloor T_2\rfloor
		\end{pmatrix},
	\end{equation}
	where for any scalar operator $T$, the notation $\lfloor T\rfloor$ is its diagonal part defined by
	\begin{equation}\label{def diag opi}
		\forall(l_{0},j_{0})\in\mathbb{Z}^{d}\times\mathbb{Z}_{\mathtt{m}},\quad\lfloor T\rfloor {\bf e}_{l_0,j_0}\triangleq T_{l_0,j_0}^{l_0,j_0}{\bf e}_{l_0,j_0}=\big\langle T{\bf e}_{l_0,j_0}, {\bf e}_{l_0,j_0}\big\rangle_{L^2(\T^{d+1},\mathbb{C})}\,{\bf e}_{l_0,j_0}.
	\end{equation}
	The next goal is to equip the class of matrix operators with the following hybrid norm
	\begin{equation}\label{hyb nor}
		\interleave {\bf T}\interleave_{s}^{q,\gamma,\mathtt{m}}\triangleq \|T_1\|_{\textnormal{\tiny{O-d}},s}^{q,\gamma,\mathtt{m}}+\|T_2\|_{\textnormal{\tiny{O-d}},s}^{q,\gamma,\mathtt{m}}+\|T_3\|_{\textnormal{\tiny{I-D}},s}^{q,\gamma,\mathtt{m}}+\|T_4\|_{\textnormal{\tiny{I-D}},s}^{q,\gamma,\mathtt{m}}.
	\end{equation}
	The choice of this norm will be motivated later in the remainder reduction performed in Section \ref{Reduction-Remaind}. Actually, the off-diagonal norm used to measure  the diagonal terms $T_1$ and $T_2$ is compatible with the scalar case as in the papers \cite{BHM21,HHM21,HR21-1,HR21,R22}. However the isotropic norm used to measure the anti-diagonal terms $T_3$ and $T_4$ is compatible with the smoothing effects of the operators and it is introduced to remedy to a new space resonance phenomenon in the second order Melnikov condition due to the interaction between the diagonal eigenvalues. The cut-off projectors $({\bf P_N})_ {N\in\mathbb{N}}$ are defined as follows 
	\begin{align}\label{proj mat}
		{\bf P_N}\mathbf{T}\triangleq\begin{pmatrix}
			P_N^1T_1 & P_N^2T_3\vspace{0.1cm}\\
			P_N^2T_4 & P_N^1T_2
		\end{pmatrix}\qquad\hbox{and}\qquad{\bf P_N^\perp}\mathbf{T}\triangleq\begin{pmatrix}
			P_N^{1,\perp}T_1 & P_N^{2,\perp}T_3\vspace{0.1cm}\\
			P_N^{2,\perp}T_4& P_N^{1,\perp}T_2
		\end{pmatrix},
	\end{align}
	where $P_N^1$ is defined in \eqref{definition of projections for operators} and $P_N^2$ in \eqref{definition of projections for operators2}. We shall prove the following result.
	\begin{cor}\label{cor-hyb-nor}
		Let $q\in\N$, $\m\in\N^*$, $(\gamma,d,s_{0},s)$ satisfy \eqref{setting tau1 and tau2}-\eqref{init Sob cond} and ${\bf T}$, ${\bf S}$ two matrix operators as in \eqref{Op-Vec1}, then the following assertions hold true.
		\begin{enumerate}[label=(\roman*)]
			\item Projector property : for any $\mathtt{t}\geqslant 0$ 
			$$\interleave{\bf P_N T }\interleave_{s+\mathtt{t}}^{q,\gamma,\mathtt{m}}\leqslant N^{\mathtt{t}}\interleave{\bf T}\interleave_{s}^{q,\gamma,\mathtt{m}}\qquad\textnormal{and}\qquad\interleave{\bf P_N^\perp  T }\interleave_{s}^{q,\gamma,\mathtt{m}}\leqslant N^{-\mathtt{t}}\interleave{\bf T}\interleave_{s+\mathtt{t}}^{q,\gamma,\mathtt{m}}.$$
			\item Composition law :
			$$\interleave {\bf T S}\interleave_{s}^{q,\gamma,\mathtt{m}}\lesssim \interleave {\bf T}\interleave_{s}^{q,\gamma,\mathtt{m}}\interleave {\bf S}\interleave_{s_0}^{q,\gamma,\mathtt{m}}+\interleave {\bf T}\interleave_{s_0}^{q,\gamma,\mathtt{m}}\interleave {\bf S}\interleave_{s}^{q,\gamma,\mathtt{m}}.$$
			\item Link with the operator norm : for $\rho=(\rho_1,\rho_{2})\in\mathbf{H}_{\mathtt{m}}^{s},$
			$$\|\mathbf{T}\rho\|_{s}^{q,\gamma,\mathtt{m}}\lesssim\interleave\mathbf{T}\interleave_{s_0}^{q,\gamma,\mathtt{m}}\|\rho\|_{s}^{q,\gamma,\mathtt{m}}+\interleave\mathbf{T}\interleave_{s}^{q,\gamma,\mathtt{m}}\|\rho\|_{s_0}^{q,\gamma,\mathtt{m}}.$$
			In particular,
			$$\|\mathbf{T}\rho\|_{s}^{q,\gamma,\mathtt{m}}\lesssim\interleave\mathbf{T}\interleave_{s}^{q,\gamma,\mathtt{m}}\|\rho\|_{s}^{q,\gamma,\mathtt{m}}.$$
		\end{enumerate}
	\end{cor}
	\begin{proof}
		\textbf{(i)} It  follows immediately  from \eqref{proj mat}, \eqref{hyb nor} and Lemma \ref{properties of Toeplitz in time operators}-(i).\\
		\textbf{(ii)} One has
		$$\mathbf{TS}=\begin{pmatrix}
			T_1 S_1+T_3S_4& T_1S_3+T_3 S_2\\
			T_4S_1+T_2S_4 & T_2S_2+T_4S_3
		\end{pmatrix}\triangleq \begin{pmatrix}
			R_1 & R_3\\
			R_4 & R_2
		\end{pmatrix}.
		$$
		Let us estimate $R_1$. One has from the law products detailed in Lemma \ref{properties of Toeplitz in time operators}-$($iv$)$
		\begin{align*}
			\|R_1\|_{\textnormal{\tiny{O-d}},s}^{q,\gamma,\mathtt{m}}&\lesssim\|T_1\|_{\textnormal{\tiny{O-d}},s}^{q,\gamma,\mathtt{m}}\|S_1\|_{\textnormal{\tiny{O-d}},s_0}^{q,\gamma,\mathtt{m}}+\|T_1\|_{\textnormal{\tiny{O-d}},s_0}^{q,\gamma,\mathtt{m}}\|S_1\|_{\textnormal{\tiny{O-d}},s}^{q,\gamma,\mathtt{m}}\\
			&\quad+\|T_3\|_{\textnormal{\tiny{O-d}},s}^{q,\gamma,\mathtt{m}}\|S_4\|_{\textnormal{\tiny{O-d}},s_0}^{q,\gamma,\mathtt{m}}+\|T_3\|_{\textnormal{\tiny{O-d}},s_0}^{q,\gamma,\mathtt{m}}\|S_4\|_{\textnormal{\tiny{O-d}},s}^{q,\gamma,\mathtt{m}}.
		\end{align*}
		Then using the embedding estimate in Lemma \ref{properties of Toeplitz in time operators}-(iii) together with \eqref{hyb nor}, we get
		\begin{align*}
			\|R_1\|_{\textnormal{\tiny{O-d}},s}^{q,\gamma,\mathtt{m}}&\lesssim\interleave {\bf T }\interleave_{s}^{q,\gamma,\mathtt{m}}\interleave {\bf  S}\interleave_{s_0}^{q,\gamma,\mathtt{m}}+\interleave {\bf T }\interleave_{s_0}^{q,\gamma,\mathtt{m}}\interleave {\bf S}\interleave_{s}^{q,\gamma,\mathtt{m}}.
		\end{align*}
		Let us now estimate $R_3$. Using Lemma \ref{properties of Toeplitz in time operators}-(iv) and \eqref{hyb nor}, we infer
		\begin{align*}
			\|R_3\|_{\textnormal{\tiny{I-D}},s}^{q,\gamma,\mathtt{m}}&\lesssim\|S_3\|_{\textnormal{\tiny{I-D}},s}^{q,\gamma,\mathtt{m}}\|T_1\|_{\textnormal{\tiny{O-d}},s_0}^{q,\gamma,\mathtt{m}}+\|S_3\|_{\textnormal{\tiny{I-D}},s_0}^{q,\gamma,\mathtt{m}}\|T_1\|_{\textnormal{\tiny{O-d}},s}^{q,\gamma,\mathtt{m}}\\
			&\quad+\|T_3\|_{\textnormal{\tiny{I-D}},s}^{q,\gamma,\mathtt{m}}\|S_1\|_{\textnormal{\tiny{O-d}},s_0}^{q,\gamma,\mathtt{m}}+\|T_3\|_{\textnormal{\tiny{I-D}},s_0}^{q,\gamma,\mathtt{m}}\|S_1\|_{\textnormal{\tiny{O-d}},s}^{q,\gamma,\mathtt{m}}\\
			&\lesssim\interleave {\bf T }\interleave_{s}^{q,\gamma,\mathtt{m}}\interleave {\bf  S}\interleave_{s_0}^{q,\gamma,\mathtt{m}}+\interleave {\bf T }\interleave_{s_0}^{q,\gamma,\mathtt{m}}\interleave {\bf S}\interleave_{s}^{q,\gamma,\mathtt{m}}.
		\end{align*}
		The terms $R_2$ and $R_4$ can be treated in a similar way.\\
		\textbf{(iii)} This point is a direct consequence of \eqref{hyb nor} and Lemma \ref{properties of Toeplitz in time operators}-(ii).
	\end{proof}
	\section{Hamiltonian reformulation}
	Here, we  describe the contour dynamics by using polar parametrization for the two interfaces  of the patch near  the annulus. We end up with  a  system of coupled  nonlinear and nonlocal transport equations satisfied by the radial deformations and that can be recast  as a Hamiltonian system. This structure is crucial to establish quasi-periodic solutions near the stationary annulus patch.
	\subsection{Transport system for radial deformations}
	Let $D_0=D_1\backslash\overline{D_2}$ be a doubly-connected domain where $D_1$ and $D_2$ are two simply-connected domains with $\overline{D_{2}}$ strictly embedded in $ D_{1}.$ Consider the initial datum $\boldsymbol{\omega}_0={\bf{1}}_{D_0},$  then the Yudovich solution takes for any $t\geqslant 0$ the form $\boldsymbol{\omega}(t)={\bf{1}}_{D_t},$ with $D_t=D_{t,1}\backslash \overline{D_{t,2}}$ a doubly-connected domain. In addition, $D_{t,1}$ and $D_{t,2}$ are two simply-connected domains with $\overline{D_{t,2}}$ strictly embedded in $D_{t,1}$. For  fixed $b\in(0,1),$ we start with a domain $D_0$ close to the annulus $A_b$, defined in \eqref{annulus-Ab}, then for a short interval of time $[0,T]$,  the domain $D_t$ will be localized around the same annulus. Therefore, we may use on this time  interval  the following symplectic polar parametrization of the boundary.
	For $k\in\{1,2\},$
	\begin{equation}\label{def zk}
		z_{k}(t):\begin{array}[t]{rcl}
			\mathbb{T} & \mapsto & \partial D_{t,k}\\
			\theta & \mapsto & e^{-\ii\Omega t}w_k (t,\theta),
		\end{array}\quad \textnormal{where}\qquad\begin{array}{l}
			w_k (t,\theta)\triangleq \left(b_k^{2}+2r_k(t,\theta)\right)^{\frac{1}{2}}e^{\ii\theta},\\
			b_1\triangleq1, \;\; 
			b_2\triangleq b.
		\end{array}
	\end{equation} 
	Similarly to \cite{BHM21, HHM21,HR21}, the introduction of the angular velocity $\Omega>0$ is due to some technical issues and devised to circumvent the trivial resonances associated to the eigen-mode $n=1$ and used in the current configuration to remedy to a more delicate phenomenon related to the analytic accumulation  of a sequence of eigenvalues in the vectorial case, see Lemma \ref{lem lin op 2 DCE}.
	The radial deformations $r_1$ and $r_2$ are assumed to be small, namely 
	$$|r_1(t,\theta)|+|r_2(t,\theta)|\ll1.$$
	In the sequel, for more convenience, we denote 
	\begin{equation}\label{def Rk}
		\forall k\in\{1,2\},\quad R_k(t,\theta)\triangleq \left(b_k^{2}+2r_k(t,\theta)\right)^{\frac{1}{2}}.
	\end{equation}
	Remind that in this particular case the stream function defined through \eqref{def streamL1} takes the form
	\begin{equation}\label{def stream}
		\psi(t,z)=\frac{1}{2\pi}\int_{D_{t,1}}\log(|z-\xi|)dA(\xi)-\frac{1}{2\pi}\int_{D_{t,2}}\log(|z-\xi|)dA(\xi).
	\end{equation}
	The vortex patch equation \eqref{CDE} provides a system of coupled transport-type PDE satisfied by $r_1$ and $r_2$. This is described  by the following lemma.
	\begin{lem}\label{lem eq EDC r}
		For short time $T>0,$ the radial deformations $r_1$ and $r_2$ defined through \eqref{def zk} satisfy the following nonlinear coupled system: for all $k\in \{1,2\}$, $(t,\theta)\in[0,T]\times\mathbb{T}$, 
		\begin{align}
			\partial_{t}r_k(t,\theta)+\Omega \partial_\theta r_k (t,\theta)&=-\partial_{\theta}\Big[\psi\big(t,z_k(t,\theta)\big)\Big]\label{Edc eq rkPsi}\\
			&=(-1)^{k+1} F_{k,k}[r](t,\theta)+(-1)^{k}F_{k,3-k}[r](t,\theta), \label{Edc eq}
		\end{align}
		where $r=(r_1,r_2)$ and, for all $k,n\in\{1,2\}$,
		\begin{align}\label{Fkn Edc}
			F_{k,n}[r](t,\theta)&\triangleq \int_{\mathbb{T}}\log\big(A_{k,n}(t,\theta,\eta)\big)\partial_{\theta\eta}^2\Big(R_k(t,\theta)R_n(t,\eta)\sin(\eta-\theta)\Big)d\eta,\\
			A_{k,n}(t,\theta,\eta)&\triangleq \big|R_k(t,\theta)e^{\ii\theta}-R_n(t,\eta)e^{\ii\eta}\big|,\label{Akn}
		\end{align}
		where we have used the notation \eqref{average-notation}. 
	\end{lem}
	\begin{proof}
		For $k\in\{1,2\},$ we denote by $\mathbf{n}_k(t,\cdot)$ an inward normal vector to the boundary $\partial D_{t,k}$ of the patch. According to \cite[p. 174]{HMV13}, the vortex patch equation writes 
		$$\forall k\in\{1,2\},\quad\partial_{t}z_k(t,\theta)\cdot \mathbf{n}_k\big(t,z_k(t,\theta)\big)=\partial_{\theta}\Big[\psi\big(t,z_k(t,\theta)\big)\Big].$$
		Identifying $\mathbb{C}$ with $\mathbb{R}^{2}$ 
		and making the choice  $\mathbf{n}_k\big(t,z_k(t,\theta)\big)=\ii\partial_{\theta}z_k(t,\theta)$ we get from \eqref{def zk}
		\begin{align*}
			\partial_{t}z_k(t,\theta)\cdot \mathbf{n}_k\big(t,z_k(t,\theta)\big)&=\mbox{Im}\left(\partial_{t}z_k(t,\theta)\overline{\partial_{\theta}z_k(t,\theta)}\right)\\ &=-\partial_{t}r_k(t,\theta)-\Omega \partial_\theta r_k (t,\theta).
		\end{align*}
		Combining the last two identities we obtain \eqref{Edc eq rkPsi}. Next, we intend to use Stokes theorem in order to transform the integral \eqref{def stream} into an integration on the boundary. This theorem can be recast in the complex form,
		\begin{equation}\label{stokes}
			2\ii\int_{D}\partial_{\overline{\xi}}f(\xi,\overline{\xi})dA(\xi)=\int_{\partial D}f(\xi,\overline{\xi})d\xi,
		\end{equation}
		where $f:\overline{D}\to\C$ is a function of class $C^1$, $D$ is a simply-connected bounded domain and $\partial D$ is the boundary of $D$. To make the argument rigorous, we shall mollify the logarithmic kernel by setting,
		$$\epsilon>0,\quad f_{\epsilon}(\xi,\overline{\xi})\triangleq (\overline{\xi}-\overline{z})\Big[\log\big(|z-\xi|^2+\epsilon\big)-1\Big].$$
		Then, we have
		$$\partial_{\overline{\xi}}f_{\epsilon}(\xi,\overline{\xi})=\log\big(|z-\xi|^2+\epsilon\big)-\frac{\epsilon}{|z-\xi|^2+\epsilon}\cdot$$
		Applying \eqref{stokes} yields
		$$
		2\ii\int_{D_{t,k}}\log\big(|z-\xi|^2+\epsilon\big)dA(\xi)-2\ii\int_{D_{t,k}}\frac{\epsilon}{|z-\xi|^2+\epsilon}dA(\xi)=\int_{\partial D_{t,k}}f_\epsilon(\xi,\overline{\xi})d\xi
		$$
		and taking the limit $\epsilon\to0$ together with \eqref{def stream} allow to get
		$$\psi(t,z)=\frac{1}{8\ii\pi}\int_{\partial D_{t,1}}(\overline{\xi}-\overline{z})\Big[\log\left(|\xi-z|^{2}\right)-1\Big]d\xi-\frac{1}{8\ii\pi}\int_{\partial D_{t,2}}(\overline{\xi}-\overline{z})\Big[\log\left(|\xi-z|^{2}\right)-1\Big]d\xi.$$
		Parametrizing the boundaries with  \eqref{def zk} and using the notation \eqref{average-notation} we infer  	\begin{equation}\label{def Psii}
			\begin{aligned}
				\psi(t,z)&=\frac{1}{4\ii}\int_{\mathbb{T}}(\overline{z}_1(t,\eta)-\overline{z})\Big[\log\big(|z_1(t,\eta)-z|^2\big)-1\Big]\partial_{\eta}z_1(t,\eta)d\eta\\
				&\quad-\frac{1}{4\ii}\int_{\mathbb{T}}(\overline{z}_2(t,\eta)-\overline{z})\Big[\log\big(|z_2(t,\eta)-z|^2\big)-1\Big]\partial_{\eta}z_2(t,\eta)d\eta.
			\end{aligned}
		\end{equation}
		As a consequence, we get by differentiating inside the integral
		$$\partial_{\overline{z}}\psi(t,z)=-\frac{1}{4\ii}\int_{\mathbb{T}}\log\big(|z_1(t,\eta)-z|^2\big)\partial_{\eta}z_1(t,\eta)d\eta+\frac{1}{4\ii}\int_{\mathbb{T}}\log\big(|z_2(t,\eta)-z|^2\big)\partial_{\eta}z_2(t,\eta)d\eta.$$
		Therefore we find through elementary computations
		\begin{align}\label{partial-theta-psi}
			\nonumber\partial_{\theta}[\psi(t,z_k(t,\theta))]&=	2\textnormal{Re}\Big((\partial_{\overline{z}}\psi)(t,z_k(t,\theta))\partial_{\theta}\overline{z}_k(t,\theta)\Big)
			\\ &=
			\nonumber -\int_{\mathbb{T}}\log\Big(|z_k(t,\theta)-z_1(t,\eta)|\Big)\partial_{\theta\eta}^{2}\textnormal{Im}\big(z_1(t,\eta)\overline{z}_k(t,\theta)\big)d\eta
			\\ &\quad+
			\int_{\mathbb{T}}\log\Big(|z_k(t,\theta)-z_2(t,\eta)|\Big)\partial_{\theta\eta}^{2}\textnormal{Im}\big(z_2(t,\eta)\overline{z}_k(t,\theta)\big)d\eta.
		\end{align}
		Using \eqref{def zk} we obtain 
		\begin{align*}
			\forall\,k,n\in\{1,2\},\quad\textnormal{Im}\big(z_n(t,\eta)\overline{z}_k(t,\theta)\big)&=R_n(t,\eta)R_k(t,\theta)\sin(\eta-\theta),\\
			|z_k(t,\theta)-z_n(t,\eta)| &=\big|R_k(t,\theta)e^{\ii\theta}-R_n(t,\eta)e^{\ii\eta}\big|.
		\end{align*}
		By combining the last two identities with \eqref{Edc eq rkPsi}-\eqref{partial-theta-psi} we conclude the proof of Lemma \ref{lem eq EDC r}.
	\end{proof}
	\subsection{Hamiltonian structure}
	The main purpose is to explore the Hamiltonian structure beyond the equations described  in \mbox{Lemma \ref{lem eq EDC r}.} First, the kinetic energy associated to the  vortex  patch $\boldsymbol{\omega}(t,\cdot)={\bf{1}}_{D_t}=\mathbf{1}_{D_{t,1}\backslash \overline{D_{t,2}}}$ is given  by
	\begin{align}\label{def kinetic energy DCE}
		E(r)&\triangleq\frac{1}{2\pi}\int_{D_{t}}\mathbf{\psi}(t,z)dA(z)\nonumber\\
		&=\frac{1}{2\pi}\int_{D_{t,1}}\mathbf{\psi}(t,z)dA(z)-\frac{1}{2\pi}\int_{D_{t,2}}\mathbf{\psi}(t,z)dA(z)
	\end{align}
	and its angular impulse is defined by
	\begin{align}\label{def angular impulse DCE}
		J(r)&\triangleq\frac{1}{2\pi}\int_{D_{t}}|z|^{2}dA(z)\nonumber\\
		&=\frac{1}{2\pi}\int_{D_{t,1}}|z|^{2}dA(z)-\frac{1}{2\pi}\int_{D_{t,2}}|z|^{2}dA(z),
	\end{align}
	where the stream function $\psi$ is defined in \eqref{def stream}. The main result of this section reads as follows.  
	\begin{prop}\label{prop HAM eq Edc}
		The system \eqref{Edc eq} is Hamiltonian and takes the form 
		\begin{equation}\label{Hamilt form DCE}
			\partial_{t}r=\mathcal{J}\nabla H(r)
		\end{equation}
		where $r\triangleq(r_{1},r_{2})$,
		\begin{equation}\label{def calJ}
			\mathcal{J}\triangleq \begin{pmatrix}
				\partial_{\theta} & 0\\
				0 & -\partial_{\theta}
			\end{pmatrix}
		\end{equation}
		and $\nabla$ is the $L^{2}(\mathbb{T})\times L^{2}(\mathbb{T})$-gradient 
		and the hamiltonian $H$ is defined by 
		\begin{equation}\label{def H}
			H(r)\triangleq -\tfrac{1}{2}\big(E(r)+\Omega J(r)\big),
		\end{equation}
		where $E$ and $J$ are defined in \eqref{def kinetic energy DCE} and \eqref{def angular impulse DCE}.
	\end{prop}
	\begin{proof}
		We shall first compute the $L^{2}(\mathbb{T})\times L^{2}(\mathbb{T})$ gradient of the angular impulse $J$. For this aim, we need to  write its  expression  in terms of $r$. Using  \eqref{stokes} combined with \eqref{def angular impulse DCE} and \eqref{def zk} yields
		\begin{align*}
			J(r)&=\frac{1}{8\pi \ii}\int_{\partial D_{t,1}}|z|^{2}\overline{z}dz-\frac{1}{8\pi \ii}\int_{\partial D_{t,2}}|z|^{2}\overline{z}dz\\
			&=\frac{1}{4}\int_{\mathbb{T}}\big(1+2r_{1}(t,\theta)\big)^{2}d\theta-\frac{1}{4}\int_{\mathbb{T}}\left(b^{2}+2r_{2}(t,\theta)\right)^{2}d\theta.
		\end{align*}
		Differentiating in $r=(r_1,r_2)$ one gets for $\rho=(\rho_{1},\rho_{2})\in L^{2}(\mathbb{T})\times L^{2}(\mathbb{T})$,
		$$\big\langle\nabla J(r),\rho\big\rangle_{L^{2}(\mathbb{T})\times L^{2}(\mathbb{T})}=\int_{\mathbb{T}}\big(1+2r_{1}(t,\theta)\big)\rho_{1}(\theta)d\theta-\int_{\mathbb{T}}\big(b^{2}+2r_{2}(t,\theta)\big)\rho_{2}(\theta)d\theta.$$
		This implies that 
		\begin{equation}\label{link J and r DCE}
			\nabla J(r)=\begin{pmatrix}
				1+2r_{1}\\
				-b^{2}-2r_{2}
			\end{pmatrix}\qquad{\rm and}\qquad \tfrac{1}{2}\Omega\,\mathcal{J}\nabla J(r)=\Omega\,\partial_{\theta}r.
		\end{equation}
		The next task is to compute the $L^{2}(\mathbb{T})\times L^{2}(\mathbb{T})$ gradient of the kinetic energy $E$ defined in \eqref{def kinetic energy DCE}. Combining \eqref{def stream} with \eqref{def zk} and changing $\xi$ by $e^{-\ii t\Omega}\xi$ we find
		$${\psi}\big(t,e^{-\ii t \Omega} z\big)= \frac{1}{2\pi}\int_{\widetilde D_{t,1}}\log(|z-\xi|)dA(\xi)-\frac{1}{2\pi}\int_{\widetilde D_{t,2}}\log(|z-\xi|)dA(\xi),$$
		where $ \widetilde D_{t,k}$ are the domains with boundaries parametrized by
		$$w_{k}:\begin{array}[t]{rcl}
				\mathbb{T} & \mapsto & \partial \widetilde D_{t,k}\\
				\theta & \mapsto & R_k(t,\theta)e^{\ii\theta},\quad R_k(t,\theta)=\left(b_k^{2}+2r_k(t,\theta)\right)^{\frac{1}{2}}.
			\end{array}$$
		Using polar change of coordinates allows to get after straightforward computations,
		\begin{equation}\label{psipol}
			\psi\big(t,e^{-\ii \Omega t} z\big)=\int_{\mathbb{T}}\int_{R_2(t,\eta)}^{R_1(t,\eta)}G\big(z,\ell_2e^{\ii\eta}\big)\ell_2d\ell_2d\eta,\qquad G(z,\xi)\triangleq \log(|z-\xi|).
		\end{equation}
		Coming back to \eqref{def kinetic energy DCE} and using once again polar change of coordinates gives
		$$E(r)=\int_{\mathbb{T}}\int_{\mathbb{T}}\int_{R_2(t,\theta)}^{R_1(t,\theta)}\int_{R_2(t,\eta)}^{R_1(t,\eta)}G\big(\ell_1e^{\ii\theta},\ell_2e^{\ii\eta}\big)\ell_1\ell_2d\ell_1d\ell_2d\theta d\eta.$$
		Therefore, the G\^ateaux derivative of $E$ in a given direction $\rho=(\rho_1,\rho_2)$ takes the form
		\begin{align*}
			\frac{d E(r)\rho}{dr}&=\int_{\mathbb{T}}\int_{\mathbb{T}}\int_{R_2(t,\eta)}^{R_1(t,\eta)}G\big(R_1(t,\theta)e^{\ii\theta},\ell_2e^{\ii\eta}\big)\rho_1(\theta)\ell_2d\ell_2d\theta d\eta\\
			&\quad+\int_{\mathbb{T}}\int_{\mathbb{T}}\int_{R_2(t,\theta)}^{R_1(t,\theta)}G\big(\ell_1e^{\ii\theta},R_1(t,\eta)e^{\ii\eta}\big)\rho_1(\eta)\ell_1d\ell_1d\theta d\eta\\
			&\quad-\int_{\mathbb{T}}\int_{\mathbb{T}}\int_{R_2(t,\eta)}^{R_1(t,\eta)}G\big(R_2(t,\theta)e^{\ii\theta},\ell_2e^{\ii\eta}\big)\rho_2(\theta)\ell_2d\ell_2d\theta d\eta\\
			&\quad-\int_{\mathbb{T}}\int_{\mathbb{T}}\int_{R_2(t,\theta)}^{R_1(t,\theta)}G\big(\ell_1e^{\ii\theta},R_2(t,\eta)e^{\ii\eta}\big)\rho_2(\eta)\ell_1d\ell_1d\theta d\eta.
		\end{align*}
		Since $G(z,\xi)$ is symmetric in $(z,\xi)$ then we obtain, by exchanging $\theta\leftrightarrow \eta$ if necessary,
		\begin{align*}
			\frac{d E(r)\rho}{dr}&=2\int_{\mathbb{T}}\int_{\mathbb{T}}\int_{R_2(t,\eta)}^{R_1(t,\eta)}G\big(R_1(t,\theta)e^{\ii\theta},\ell_2e^{\ii\eta}\big)\rho_1(\theta)\ell_2d\ell_2d\theta d\eta\\
			&\quad-2\int_{\mathbb{T}}\int_{\mathbb{T}}\int_{R_2(t,\eta)}^{R_1(t,\eta)}G\big(R_2(t,\theta)e^{\ii\theta},\ell_2e^{\ii\eta}\big)\rho_2(\theta)\ell_2d\ell_2d\theta d\eta.
		\end{align*}
		It follows from \eqref{psipol} and  \eqref{def zk}
		\begin{equation}\label{link E and r}
			\nabla E(r)=\begin{pmatrix}
				\displaystyle 2\int_{\mathbb{T}}\int_{R_2(t,\eta)}^{R_1(t,\eta)}G\big(w_1(t,\theta),\ell_2e^{\ii\eta}\big)\ell_2d\ell_2d\eta\vspace{0.1cm}\\
				\displaystyle -2\int_{\mathbb{T}}\int_{R_2(t,\eta)}^{R_1(t,\eta)}G\big(w_2(t,\theta),\ell_2e^{\ii\eta}\big)\ell_2d\ell_2d\eta
			\end{pmatrix}=\begin{pmatrix}
				2\psi\big(t,z_1(t,\theta)\big)\\
				-2\psi\big(t,z_2(t,\theta)\big)
			\end{pmatrix}.
		\end{equation}
		Finally,  \eqref{link E and r}, \eqref{link J and r DCE} and  \eqref{Edc eq rkPsi} give the desired result.	This achieves the proof of Proposition \ref{prop HAM eq Edc}.
	\end{proof}
	\subsection{Symplectic structure and invariance}
	In this section, we intend to  discuss the symplectic structure behind the Hamiltonian formulation already seen in Proposition \ref{prop HAM eq Edc}. We shall also discuss some symmetry properties such as the reversibility and the $\mathtt{m}-$fold persistence.\\
	$\blacktriangleright${ \it{Symplectic structure.}} We shall present the symplectic structure associated with the Hamiltonian equation \eqref{Hamilt form DCE}. To do so, we need to fix the phase space but before that we shall use the following fact that can be derived from  \eqref{Hamilt form DCE},
	$$\frac{d}{dt}\int_{\mathbb{T}}r(t,\theta)d\theta=0.$$
	This means that  the  area enclosed by the boundaries is conserved in time. 
	Therefore, we shall work with the following phase space with zero space average $L_{*}^2(\mathbb{T})\times L_{*}^2(\mathbb{T})$ defined by 
	$$
	L_{*}^{2}(\mathbb{T})\triangleq \bigg\{f=\sum_{j\in\mathbb{Z}^*}f_{j}\mathbf{e}_j\quad\textnormal{s.t.}\quad f_{-j}=\overline{f_j},\quad\sum_{j\in\mathbb{Z}^*}|f_j|^2<+\infty\bigg\},\qquad\; \mathbf{e}_{j}(\theta)\triangleq e^{\ii j\theta}.$$
	The equation \eqref{Hamilt form DCE} induces on the phase space $L_{*}^2(\mathbb{T})\times L_{*}^2(\mathbb{T})$ a symplectic structure given by the symplectic $2$-form
	\begin{align*}
			\mathcal{W}(r,h)&\triangleq\big\langle \mathcal{J}^{-1}r,h\big\rangle_{L^{2}(\mathbb{T})\times L^{2}(\mathbb{T})}\\
			&=\int_{\mathbb{T}}\partial_{\theta}^{-1}r_1(\theta)h_1(\theta)d\theta-\int_{\mathbb{T}}\partial_{\theta}^{-1}r_2(\theta)h_2(\theta)d\theta,
	\end{align*}
	where
	$$
	\partial_{\theta}^{-1}f\triangleq\sum_{j\in\mathbb{Z}^{*}}\tfrac{f_{j}}{\ii j}\mathbf{e}_{j} \quad\quad\textnormal{for }\quad \quad f=\sum_{j\in\mathbb{Z}^{*}}f_{j}\mathbf{e}_{j}.$$
	The corresponding Hamiltonian vector field is $X_{H}(r)\triangleq\mathcal{J}\nabla H(r)$ (where $\nabla$ is the $L^{2}\times L^2$-gradient). It is defined as the symplectic gradient of the Hamiltonian $H$ with respect to the symplectic $2$-form $\mathcal{W}$, namely
	$$dH(r)[\cdot]=\mathcal{W}(X_{H}(r),\cdot).$$
	Decomposing into Fourier series
	$$r=(r_1,r_2),\qquad\forall k\in\{1,2\},\quad r_k=\sum_{j\in\mathbb{Z}^{*}}r_{j,k}\mathbf{e}_{j}\qquad\textnormal{with}\qquad r_{-j,k}=\overline{r_{j,k}},$$
	then the  symplectic form $\mathcal{W}$ becomes
	\begin{equation}\label{Symp-F}
		\mathcal{W}(r,h)=\sum_{j\in\mathbb{Z}^{*}}\frac{1}{\ii j}\Big[r_{j,1}h_{-j,1}-r_{j,2}h_{-j,2}\Big],	\end{equation}
	or equivalently,
	\begin{align}\label{sympl ref}
		\nonumber \mathcal{W}&=\tfrac{1}{2}\sum_{j\in\mathbb{Z}^{*}}\frac{1}{\ii j}\Big[dr_{j,1}\wedge dr_{-j,1}-dr_{j,2}\wedge dr_{-j,2}\Big]\\
		&=\sum_{j\in\mathbb{N}^{*}}\frac{1}{\ii j}\Big[dr_{j,1}\wedge dr_{-j,1}-dr_{j,2}\wedge dr_{-j,2}\Big].
	\end{align}
	\begin{defin} {\bf (Symplectic)}\label{def:sympl}
		A linear transformation $\Phi$ of the phase space $L_*^2(\mathbb{T})\times L_*^2(\mathbb{T})$ is symplectic, if $\Phi$ 
		preserves the symplectic $2$-form $\mathcal{W}$, i.e. 
		$${\mathcal W}(\Phi u,\Phi v)={\mathcal W}(u,v),$$ 
		or equivalently
		$$\Phi^\top\circ\mathcal{J}^{-1}\circ\Phi=\mathcal{J}^{-1}.$$
	\end{defin}
	This allows to establish the following result which is useful later and  whose proof is straightforward.
	\begin{lem}\label{lem: charac symp}
		Let $\Phi$ be a matrix  space-Fourier multiplier with the form
		$$\Phi\begin{pmatrix}
				\rho_{1}\\
				\rho_{2}
			\end{pmatrix}
			\triangleq \sum_{j\in \mathbb{Z}^*}\, \Phi_j \begin{pmatrix}
				\rho_{j,1}\\
				\rho_{j,2}
			\end{pmatrix} \mathbf{e}_j\, , \qquad 
			\Phi_j\in M_{2}(\mathbb{R}),$$
		and consider the symplectic $2$-form $\mathcal{W}$ defined in \eqref{Symp-F}. Then $\Phi$ is symplectic if and only if
		$$\forall j\in \mathbb{Z}^*,\quad \Phi_j^\top \begin{pmatrix}
			1 &  0 \\
			0 & -1 
		\end{pmatrix} \Phi_j=\begin{pmatrix}
			1 &  0 \\
			0 & -1 
		\end{pmatrix}.$$
		
	\end{lem}
	
	\noindent $\blacktriangleright$ {\it{Reversibility.}} We shall analyze  the reversibility property of the equation \eqref{Hamilt form DCE} which is crucial to reduce by  symmetry  the phase space and remove most of the trivial resonances.
	We consider the involution $\mathscr{S}$ defined on the phase space $L_{*}^{2}(\mathbb{T})\times L_{*}^2(\mathbb{T})$ by 
	\begin{equation}\label{defin inv scr S}
		(\mathscr{S}r)(\theta)\triangleq r(-\theta),
	\end{equation}
	which satisfies
	\begin{equation}\label{prop inv scr S}
		\mathscr{S}^{2}=\textnormal{Id}\qquad \mbox{and}\qquad \mathcal{J}\circ\mathscr{S}=-\mathscr{S}\circ\mathcal{J}.
	\end{equation}
	Using the change of variables $\eta\mapsto-\eta$ and parity arguments, one gets from \eqref{Fkn Edc}
	$$\forall\,k,n\in\{1,2\},\quad F_{k,n}\circ\mathscr{S}=-\mathscr{S}\circ F_{k,n}.$$
	Then we conclude by Lemma \ref{lem eq EDC r}, \eqref{Hamilt form DCE} and \eqref{prop inv scr S} that
	the Hamiltonian vector field $X_H$ satisfies
	$$X_H\circ\mathscr{S}=-\mathscr{S}\circ X_H.$$
	Therefore, we will focus on quasi-periodic solutions to \eqref{Hamilt form DCE} satisfying the reversibility condition
	\begin{equation}\label{reversibility condition r}
		r(-t,-\theta)=r(t,\theta).
	\end{equation}
	\noindent $\blacktriangleright$ {\it The {$\mathtt{m}$}-fold symmetry.} Let $\mathtt{m}\geqslant1$ be  an integer and consider the transformation $\mathscr{T}_{\mathtt{m}}$ on the phase space $L_{*}^2(\mathbb{T})\times L_{*}^2(\mathbb{T})$ defined by
	\begin{equation}\label{def scr Tm}
		(\mathscr{T}_{\mathtt{m}}r)(\theta)\triangleq r\left(\theta+\tfrac{2\pi}{\mathtt{m}}\right).
	\end{equation}
	Then it is an immediate fact that 
	$$\mathscr{T}_{\mathtt{m}}^{\mathtt{m}}=\textnormal{Id}\qquad\textnormal{and}\qquad\mathcal{J}\circ\mathscr{T}_{\mathtt{m}}=\mathscr{T}_{\mathtt{m}}\circ\mathcal{J}.$$
	Using the change of variables $\eta\mapsto\eta+\tfrac{2\pi}{\mathtt{m}}$ we easily obtain from \eqref{Fkn Edc} 
	$$\forall\, k,n\in\{1,2\},\quad F_{k,n}\circ\mathscr{T}_{\mathtt{m}}=\mathscr{T}_{\mathtt{m}}\circ F_{k,n}.$$
	Therefore,
	$$X_{H}\circ\mathscr{T}_{\mathtt{m}}=\mathscr{T}_{\mathtt{m}}\circ X_H.$$
	Thus, the solutions that we shall be interested in satisfy  the $\mathtt{m}$-fold symmetry 
	\begin{equation}\label{m-fold symmetry r}
		r\left(t,\theta+\tfrac{2\pi}{\mathtt{m}}\right)=r(t,\theta).
	\end{equation}
	Consequently, we shall work in the closed  subspace $L_{\mathtt{m}}^2(\mathbb{T})\times L_{\mathtt{m}}^2(\mathbb{T})$ defined by
	\begin{equation}\label{def L2m}
		L_{\mathtt{m}}^2(\mathbb{T})\triangleq \bigg\{f=\sum_{j\in\mathbb{Z}^*}f_{j}\mathbf{e}_j\in L_{*}^2(\mathbb{T})\quad\textnormal{s.t.}\quad f_{j}\neq 0\,\Rightarrow\,j\in\mathbb{Z}_{\mathtt{m}}\bigg\}.
	\end{equation}
	\section{Linearization and symplectic transformation}
	In this section we shall compute the linear Hamiltonian obtained through the linearization of the equation \eqref{Hamilt form DCE} at any state close to the equilibrium solution $r=0.$ It turns out that at the equilibrium state we find a matrix Fourier multiplier that can be diagonalized in a suitable basis using a linear symplectic change of coordinates, for more details we refer to Lemma \ref{lem: properties P}. However, this procedure requires to work with higher $\mathtt{m}$-fold symmetries to avoid the double eigenvalue corresponding to the mode $j=2$ as well as potential hyperbolic directions.
	\subsection{Linearized operator}
	The main purpose is to explore the structure of the linearized operator which takes the form of a transport system with variable coefficients and subject to compact perturbations.
	\begin{lem}\label{lem lin op 1 DCE}
		The linearized equation of \eqref{Hamilt form DCE} at a small state $r$ is given by the linear Hamiltonian equation,
		\begin{equation}\label{defLr}
			\partial_{t}\begin{pmatrix}
				\rho_{1}\\
				\rho_{2}
			\end{pmatrix}
			=\mathcal{J} \mathbf{M}_r\begin{pmatrix}
				\rho_{1}\\
				\rho_{2}
			\end{pmatrix},\qquad \mathbf{M}_r\triangleq \begin{pmatrix}
				-V_{1}(r)-L_{1,1}({r}) & L_{1,2}(r)\\
				L_{2,1}(r) & V_{2}(r)-{L}_{2,2}(r)
			\end{pmatrix},
		\end{equation}
		where $V_{k}(r)$ are scalar functions and  ${L}_{k,n}({r})$ are nonlocal operators  defined by
		\begin{align}
			V_{k}(r)(t,\theta)&\triangleq \Omega+(-1)^k\big[V_{k,k}(r)(t,\theta)-V_{k,3-k}(r)(t,\theta)\big], \label{def Vpm}\\
			V_{k,n}(r)(t,\theta)&\triangleq \int_{\mathbb{T}}\log\big(A_{k,n}({r})(t,\theta,\eta)\big)\partial_{\eta}\Big(\tfrac{R_{n}(t,\eta)}{R_{k}(t,\theta)}\sin(\eta-\theta)\Big)d\eta,\label{def Vkn}\\
			{L}_{k,n}(r)\rho(t,\theta)&\triangleq \int_{\mathbb{T}}\rho(t,\eta)\log\big(A_{k,n}({r})(t,\theta,\eta)\big)d\eta\label{def mathbfLkn}
		\end{align}
		and  $A_{k,n}({r})$ and $R_{k}$ are respectively defined by \eqref{Akn} and \eqref{def Rk}. Moreover, if $r$ satisfies \eqref{reversibility condition r} and \eqref{m-fold symmetry r} with $\m\geqslant1$, then  the operator $\mathbf{M}_r$ is $\mathtt{m}$-fold reversibility preserving. 
	\end{lem}
	\begin{proof} Throughout the proof, we shall alleviate the notation by removing the time dependence and keep $r$ when it is relevant. In view of \eqref{Edc eq rkPsi}, it suffices to linearize the term involving the stream function. All the computations are done at a formal level, but can be rigorously  justified in a classical way in the functional context introduced in Section \ref{sec funct set}. According to \eqref{def Psii} we can write
		\begin{align}
			\psi\big(z_k(\theta)\big)&=(-1)^{k+1}\Big[\widetilde{\psi}\big(r_k,z_k(\theta)\big)-\widetilde{\psi}\big(r_{3-k},z_k(\theta)\big)\Big],\label{Psikn recall0}\\
			\widetilde{\psi}(r_n,z)&\triangleq \tfrac{1}{4\ii}\int_{\mathbb{T}}(\overline{z}_n(\eta)-\overline{z})\Big[\log\left(|z-z_n(\eta)|^{2}\right)-1\Big]\partial_{\eta}z_n(\eta) d\eta.\label{Psikn recall}
		\end{align}
		Applying the chain rule yields
		\begin{equation}\label{dr psi k-k}
			\begin{aligned}
				d_{r_k}\Big(\psi\big(z_k(\theta)\big)\Big)[\rho_k](\theta)&=(-1)^{k+1}\Big[d_{r_k}\widetilde{\psi}\big(r_k,z_k(\theta)\big)[\rho_k](\theta)\\ &\quad +2\textnormal{Re}\Big((\partial_{\overline{z}}\widetilde{\psi})\big(r_k,z_k(\theta)\big)d_{r_k}\overline{z}_k(\theta)[\rho_k](\theta)\Big)\\
				&\quad -2\textnormal{Re}\Big((\partial_{\overline{z}}\widetilde{\psi})\big(r_{3-k},z_k(\theta)\big)d_{r_k}\overline{z}_k(\theta)[\rho_k](\theta)\Big)\Big],\\
				d_{r_{3-k}}\Big(\psi\big(z_k(\theta)\big)\Big)[\rho_{3-k}](\theta)&=(-1)^{k}d_{r_{3-k}}\widetilde{\psi}\big(r_{3-k},z_{k}(\theta)\big)[\rho_{3-k}](\theta).
			\end{aligned}
		\end{equation}
		From \eqref{psipol}, we have the following  expression of $\widetilde{\psi}$,
		$$\widetilde{\psi}\big(r_k,e^{-\ii\Omega t}z\big)=\int_{\mathbb{T}}\int_0^{R_k(\eta)}\log\big(|z-\ell_2e^{\ii\eta}|\big)\ell_2d\ell_2d\eta.$$
		Therefore, differentiating with respect to $r_k$ in the direction $\rho_k$, we obtain
		$$d_{r_k}\widetilde{\psi}\big(r_k,e^{-\ii\Omega t}z\big)[\rho_k](\theta)=\int_{\mathbb{T}}\rho_k(\eta)\log\big(|z-R_k(\eta)e^{\ii\eta}|\big)d\eta.$$
		It follows that, for any $k,n\in\{1,2\}$, we have by virtue of \eqref{def mathbfLkn}
		\begin{align}\label{lkn part}
			\nonumber	d_{r_k}\widetilde{\psi}\big(r_k,z_{n}(\theta)\big)[\rho_k](\theta)&=\int_{\mathbb{T}}\rho_k(\eta)\log\big(|R_n(\theta)e^{\ii\theta}-R_k(\eta)e^{\ii\eta}|\big)d\eta\\
			&={L}_{k,n}(r)\rho_k(\theta).
		\end{align}
		On the other hand, differentiating \eqref{def zk} leads to
		$$d_{r_k}\overline{z}_k(\theta)[\rho_k](\theta)=\tfrac{\rho_k(\theta)}{R_k(\theta)}e^{-\ii(\theta-\Omega t)}.$$
		In addition, by virtue of \eqref{Psikn recall}, we have
		$$\partial_{\overline{z}}\widetilde{\psi}(r_n,z)=-\tfrac{1}{4\ii}\int_{\mathbb{T}}\log\big(|z-z_n(\eta)|^2\big)\partial_{\eta}z_n(\eta)d\eta.$$
		Combining the last two identities we infer, for $k,n\in\{1,2\}$,
		\begin{align*}
			2\textnormal{Re}\Big((\partial_{\overline{z}}\widetilde{\psi})\big(r_n,z_k(\theta)\big)d_{r_k}\overline{z}_k(\theta)[\rho_k](\theta)\Big)&=-\tfrac{\rho_k(\theta)}{R_k(\theta)}\int_{\mathbb{T}}\log\big(|z_k(\theta)-z_n(\eta)|\big)\partial_{\eta}\textnormal{Im}\Big(z_n(\eta)e^{-\ii(\theta-\Omega t)}\Big)d\eta.
		\end{align*}
		From \eqref{def zk} we find the identity 
		$$\textnormal{Im}\Big(z_n(\eta)e^{-\ii(\theta-\Omega t)}\Big)=R_n(\eta)\sin(\eta-\theta).$$
		Then, by \eqref{def Vkn} we conclude that
		\begin{align}\label{Vkn part}
			2\textnormal{Re}\Big((\partial_{\overline{z}}\widetilde{\psi})\big(r_n,z_k(\theta)\big)d_{r_k}\overline{z}_k(\theta)[\rho_k](\theta)\Big)&=-
			V_{k,n}(r)(\theta)\rho_k(\theta).
		\end{align}
		Putting together \eqref{Psikn recall0}, \eqref{dr psi k-k}, \eqref{lkn part} and \eqref{Vkn part} yields
		\begin{align*}
			d_{r}\Big(\psi\big(z_k(\theta)\big)\Big)[\rho](\theta)&=d_{r_k}\Big(\psi\big(z_k(\theta)\big)\Big)[\rho_k](\theta)+d_{r_{3-k}}\Big(\psi\big(z_k(\theta)\big)\Big)[\rho_{3-k}](\theta)\\ &=(-1)^{k+1}\Big[{L}_{k,k}(r)\rho_k(\theta)-V_{k,k}(r)(\theta)\rho_k(\theta)\\
			&\qquad\qquad\qquad +V_{k,3-k}(r)(\theta)\rho_{k}(\theta)-{L}_{k,3-k}(r)\rho_{3-k}(\theta)\Big].
		\end{align*}
		This gives the expression of $\mathbf{M}_r$ in \eqref{defLr}. Next, assume that $r$ satisfies \eqref{reversibility condition r} and \eqref{m-fold symmetry r}. Then, from  \eqref{def Rk}, \eqref{def Vpm} and \eqref{def Vkn}, we get the following symmetry properties
		\begin{equation}\label{sym-m Vpm}
			V_{k}(r)(-t,-\theta)=V_{k}(r)(t,\theta)=V_{k}(r)\big(t,\theta+\tfrac{2\pi}{\mathtt{m}}\big).
		\end{equation}
		Similarly, from \eqref{def Rk} and \eqref{Akn}, one has
		\begin{equation}\label{sym Akn}
			A_{k,n}(r)(-t,-\theta,-\eta)=A_{k,n}(r)(t,\theta,\eta)=A_{k,n}(r)\big(t,\theta+\tfrac{2\pi}{\mathtt{m}},\eta+\tfrac{2\pi}{\mathtt{m}}\big).
		\end{equation}
		Thus, the symmetry properties of the operator $\mathbf{M}_r$  are immediate consequences of Lemma \ref{lem sym--rev}. The proof of Lemma \ref{lem lin op 1 DCE} is now complete.
	\end{proof}
	The next goal is to derive the explicit structure of the linearized operator at the equilibrium \mbox{state $r=0$.}
	\begin{lem}\label{lem lin op 2 DCE}
		The linearized  equation of \eqref{Hamilt form DCE} at $r=0$ writes,
		\begin{equation}\label{Ham eq-eq DCE}
			\partial_{t}\begin{pmatrix}
				\rho_{1}\\
				\rho_{2}
			\end{pmatrix}=\mathcal{J} \mathbf{M}_0\begin{pmatrix}
				\rho_{1}\\
				\rho_{2}
			\end{pmatrix}, \qquad \mathbf{M}_0\triangleq \begin{pmatrix}
				-V_{1}(0)-\mathcal{K}_1\ast\cdot & \mathcal{K}_b\ast\cdot\\
				\mathcal{K}_b\ast\cdot & V_{2}(0)-\mathcal{K}_1\ast\cdot
			\end{pmatrix},
		\end{equation}
		where  
		\begin{align}
			\forall k\in\{1,2\}, \quad 	\mathtt{v}_k(b)&\triangleq V_k({0})=\Omega+(2-k)\tfrac{1-b^2}{2}, 
			\label{def V10 V20}\\
			\label{def mathcalKkn}
			\forall x\in(0,1],	\quad \mathcal{K}_x(\theta)&\triangleq\log\big|1-xe^{\ii\theta}\big|.
		\end{align}
		The convolution is understood in the following sense
		$$\mathcal{K}_{x}\ast\rho(\theta)=\int_{\mathbb{T}}\mathcal{K}_{x}(\theta-\eta)\rho(\eta)d\eta.$$
		Given the space Fourier expansion of the real solutions 
		$$\forall k\in\{1,2\},\quad \rho_{k}(t,\theta)=\displaystyle\sum_{j\in\mathbb{Z}^{*}}\rho_{j,k}(t)e^{\ii j\theta},\qquad \textnormal{with}\qquad  \rho_{-j,k}=\overline{\rho_{j,k}},$$  the system \eqref{Ham eq-eq DCE} is equivalent to the following countable family of linear differential systems 
		\begin{equation}\label{def MjbO}
			\forall j\in\mathbb{Z}^*,\quad\begin{pmatrix}
				\dot{\rho}_{j,1}\vspace{0.1cm}\\
				\dot{\rho}_{j,2}
			\end{pmatrix}= M_{j}(b,\Omega)\begin{pmatrix}
				\rho_{j,1}\vspace{0.1cm}\\
				\rho_{j,2}
			\end{pmatrix}, 
			\quad M_{j}(b,\Omega)\triangleq \frac{\ii j}{|j|} \begin{pmatrix}
				-|j|\big(\Omega+\tfrac{1-b^2}{2}\big)+\tfrac{1}{2}& -\tfrac{b^{|j|}}{2}\\
				\tfrac{b^{|j|}}{2} & -|j|\Omega-\tfrac{1}{2}
			\end{pmatrix}. 
		\end{equation}
	\end{lem}
	\begin{proof}
		We shall make use of the following formula which can be found in \cite[Lem. A.3]{CCG16} and \cite[Lem. 3.2]{R21}.
		\begin{align}
			\forall j\in\mathbb{Z}^*,\quad\forall x\in (0,1],\quad\int_{\mathbb{T}}\log\big(\big|1-xe^{\ii\theta}\big|\big)\cos(j\theta)d\theta&=-\frac{x^{|j|}}{2|j|}\cdot\label{int-2}
		\end{align}
		First observe that from \eqref{Akn}, one deduces for $r=0$ that  
		$$A_{k,n}({0})(\theta,\eta)=\left|b_{k}e^{i\theta}-b_{n}e^{i\eta}\right|=\left(b_{k}^{2}+b_{n}^{2}-2b_{k}b_{n}\cos(\eta-\theta)\right)^{\frac{1}{2}}$$
		leading in particular to 
		$$
		A_{1,2}({0})(\theta,\eta)=A_{2,1}({0})(\theta,\eta)=\left|1-be^{\ii(\eta-\theta)}\right|.
		$$
		Taking $r=0$ in \eqref{def Vkn} and  using the change of variables $\eta\mapsto\eta+\theta$ together with  \eqref{int-2} imply	\begin{align*}
			\forall\, k\in\{1,2\},\quad V_{k,k}(0)(\theta) &= \bigintssss_{\mathbb{T}}\log\left(\left|1-e^{\ii\eta}\right|\right)\cos(\eta)d\eta= -\frac{1}{2},
			\\
			V_{k,3-k}(0)(\theta)
			& =  \frac{b_{3-k}}{b_{k}}\bigintssss_{\mathbb{T}}\log\left(\left|1-be^{\ii\eta}\right|\right)\cos(\eta)d\eta= -\frac{b_{3-k}b }{2b_{k}}\cdot
		\end{align*}
		Combining the last identity with  \eqref{def Vpm} yields \eqref{def V10 V20}.
		Substituting $r=0$ into \eqref{def mathbfLkn} gives, since the space average of $\rho$ is zero,
		\begin{align*}
			\forall\, k\in\{1,2\},\qquad L_{k,k}({0})\rho
			&=\mathcal{K}_1\ast\rho\qquad\textnormal{and}\qquad L_{k,3-k}({0})\rho={\mathcal{K}}_b\ast\rho,
		\end{align*}
		where the kernels $\mathcal{K}_1$ and $\mathcal{K}_b$ are defined by \eqref{def mathcalKkn}. Applying once again  \eqref{int-2} allows to get 
		$$\forall j\in\mathbb{Z}^*,\qquad\mathcal{K}_1\ast\mathbf{e}_{j}=-\frac{1}{2|j|}\mathbf{e}_{j}\qquad\textnormal{and}\qquad{\mathcal{K}}_b\ast\mathbf{e}_{j}=-\frac{b^{|j|}}{2|j|}\mathbf{e}_j,\qquad \mathbf{e}_{j}(\theta)=e^{\ii j\theta}.$$
		Finally, gathering the previous computations leads to \eqref{def MjbO}. This achieves the proof of Lemma \ref{lem lin op 2 DCE}. 
	\end{proof}
	\subsection{Diagonalization at the equilibrium state}
	In this subsection we shall diagonalize the equilibrium matrix operator appearing in Lemma \ref{lem lin op 2 DCE}. This provides a new Hamiltonian system more adapted for the action-angles reformulation. Before that, we will establish the following result on the spectral structure of the matrix $M_{j}$ introduced in \eqref{def MjbO}.
	\begin{lem}\label{lem lin op 3 DCE}
		Let $\Omega>0$,
		$j\in\mathbb{Z}^*$ and $b\in(0,1)$. Then the eigenvalues of the matrix $ M_{j}(b,\Omega)$, defined in \eqref{def MjbO}, are given by $-\ii \Omega_{j,k}(b)$, $k\in\{1,2\}$, where
		\begin{equation}\label{omega jk b}
			\Omega_{j,k}(b)\triangleq \frac{j}{|j|}\bigg[\big(\Omega+\tfrac{1-b^2}{4}\big)|j|-\ii^{\mathtt{H}\big(\Delta_{j}(b)\big)} \tfrac{(-1)^{k}}{2}\sqrt{|\Delta_{j}(b)|}\bigg],
		\end{equation}
		with $\mathtt{H}\triangleq \mathbf{1}_{[0,\infty)}$ the Heaviside function and 
		\begin{equation}\label{def delta j}
			\Delta_{j}(b)\triangleq b^{2|j|}-\big(\tfrac{1-b^2}{2}|j|-1\big)^2.
		\end{equation}
		The corresponding  eigenspaces are  one dimensional and generated by  the vectors
		$$v_{j,1}(b)\triangleq 
		\begin{pmatrix}
			1 \\
			-a_{j}(b)
		\end{pmatrix}, \quad v_{j,2}(b)\triangleq  
		\begin{pmatrix}
			-a_{j}(b)\\
			1
		\end{pmatrix}, \quad  a_{j}(b)\triangleq \frac{b^{|j|}}{\tfrac{1-b^2}{2}|j|-1+\ii^{\mathtt{H}\big(\Delta_{j}(b)\big)} \sqrt{|\Delta_{j}(b)|}}\cdot$$
	\end{lem}
	\begin{proof}
		According to   \eqref{def MjbO} we have
		$$
		\forall j\in \mathbb{Z}^*,\quad M_{-j}(b,\Omega)=\overline{M_{j}(b,\Omega)}.
		$$
		Thus, it suffices to consider the case $j\in \mathbb{N}^*$. 
		The eigenvalues of the matrix  $M_{j}(b,\Omega)$ are solutions of the following second order polynomial equation 
		\begin{equation}\label{poly2}
			X^{2}+\ii\left[\mu_j+\delta_j\right]X-\left[\mu_j\delta_j+\tfrac{b^{2j}}{4}\right]=0, \quad\textnormal{with}\quad\mu_j\triangleq j\big(\Omega+\tfrac{1-b^2}{2}\big)-\tfrac{1}{2}\quad\textnormal{and}\quad\delta_j\triangleq j\Omega+\tfrac{1}{2}.
		\end{equation}
		The discriminant of the last equation is given by
		\begin{align*}
			-(\mu_j+\delta_j)^2+b^{2j}+4\mu_j\delta_j &=b^{2j}-\big(\mu_j-\delta_j\big)^2\\
			&=\Delta_j(b)
		\end{align*}
		and the solutions to the equation are $-\ii \Omega_{j,k}(b)$, $k\in\{1,2\}$, where $\Omega_{j,k}(b)$ are given by \eqref{omega jk b}. The expression of the eigenvectors $v_{j,k}(b)$ follows by direct computations. Notice that, for all $j\in \mathbb{N}^*$,
		$$a_j(0)=0,\qquad a_j(1)=-1$$
		and for all $b\in (0,1)$, $a_j(b)$ is well-defined if $\Delta_j(b)>0$. We shall prove that $a_j(b)$ is still well-defined even when $\Delta_j(b)\leqslant 0$. In view of \eqref{def delta j}, we may write for all $j\in\mathbb{N}^*$ 
		\begin{equation}\label{deltaj}
			\Delta_j(b)=\Big(b^{j}-1+\tfrac{1-b^2}{2}j\Big)\Big(b^{j}+1-\tfrac{1-b^2}{2}j\Big).
		\end{equation}
		In particular we have
		$$\forall b\in (0,1),\quad \Delta_1(b)=-\frac14(1-b)^2(1+b)^2<0,\, \qquad \Delta_2(b)=0\qquad\textnormal{and}\qquad a_2(b)=-1.$$
		For $j\geqslant 3$, we can easily check that
		\begin{align}\label{deltaj1}
			\nonumber\forall b\in (0,1),\quad \tfrac{1-b^2}{2}j +b^{j}-1&=(1-b)\Big(\tfrac{j}{2}(1+b)-(1+b)-b^2\big[1+b+\cdots +b^{j-3}\big]\Big)\\
			\nonumber&\geqslant (1-b)\Big(\tfrac{j-2}{2}(1+b)-b^2(j-2)\Big)\\
			\nonumber&\geqslant  \tfrac{j-2}{2}(1-b)^2(1+2b)\\ &>0.
		\end{align}
		It follows that  $\Delta_j(b)\leqslant 0$ if and only if $b^{j}+1-\tfrac{1-b^2}{2}j\leqslant 0$. In this case the denominator of   $a_j(b)$ satisfies, for all $b\in (0,1)$, 
		$$
		\tfrac{1-b^2}{2}|j|-1+ \sqrt{-\Delta_{j}(b)}\geqslant \tfrac{1-b^2}{2}|j|-1\geqslant b^{j}>0.
		$$ 
		This ends the proof of Lemma \ref{lem lin op 3 DCE}.
	\end{proof}
	The next task is devoted to the study of the sign of the  discriminant $\Delta_j(b)$ defined in \eqref{def delta j}.
	\begin{lem}\label{lem:disper-relation}
		There exists a strictly increasing sequence $(\underline{b}_{j})_{j\geqslant 3}\subset(0,1)$ converging to $1$ such that
		$$
		\Big\{b\in(0,1)\quad\textnormal{s.t.}\quad\exists\, j\in\mathbb{N}\setminus\{0,1,2\}, \quad \Delta_{j}(b)=0\Big\}=\Big\{\underline{b}_{j},\,j\geqslant 3\Big\},
		$$
		with
		\begin{equation}\label{b3b4}
			\underline{b}_{3}=\frac12\qquad\textnormal{and}\qquad \underline{b}_{4}=\sqrt{\sqrt{2}-1}.
		\end{equation}
		Moreover, for any fixed $j_0\geqslant 3$ we have
		$$\forall b\in(\underline{b}_{j_0},\underline{b}_{j_0+1}),\quad  \begin{cases}
				\Delta_{j}(b)>0, & \textnormal{if } j\leqslant j_0, \\
				\Delta_{j}(b)<0, & \textnormal{if } j\geqslant j_0+1.
			\end{cases}$$
	\end{lem}
	\begin{proof}
		For $j\geqslant 3$, one gets in view of \eqref{deltaj} and \eqref{deltaj1} that  the zeros of the function $b\mapsto \Delta_j(b)$ are the zeros of the function $b\mapsto b^{j}+1-\tfrac{1-b^2}{2}j$. To study the zeros of the latter discrete function let us consider its continuous version
		\begin{equation}\label{def:f}
			\forall (b,x)\in(0,1)\times[3,\infty),\quad f(b,x)\triangleq b^{x}+1-\tfrac{1-b^2}{2}x.
		\end{equation}
		Then, for fixed $x\in[3,\infty)$, one has $f(0,x)=1-\tfrac{x}{2}<0,\quad f(1,x)=2\quad$ and 
		\begin{equation}\label{variation-b}
			\forall b\in(0,1),\quad \partial_b f(b,x)=x\big(b^{x-1}+b\big)>0. 
		\end{equation}
		Consequently, by the intermediate value theorem, there exists a unique  $\underline{b}_x\in (0,1)$ satisfying 
		$$
		f(\underline{b}_x,x)=0,\qquad \forall b< \underline{b}_x,\quad f(b,x)<0\qquad  {\rm and}\qquad  \forall b>\underline{b}_x,\quad  f(b,x)>0.
		$$
		Moreover, since 
		\begin{equation}\label{variation-x}
			\forall b\in(0,1),\quad \partial_x f(b,x)={b}^{x}\log (b)-\tfrac{1-{b}^2}{2}<0.
		\end{equation}
		Then the function $x\mapsto  f(b,x)$ is strictly decreasing on $[3,\infty)$, which implies that $x\mapsto \underline{b}_x$ is strictly increasing on $[3,\infty)$.
		It follows that for any fixed integer  $j_0\geqslant 3$ we have
		\begin{equation}\label{deltaj2}
			f(\underline{b}_{j_0},{j_0})=0 \qquad 	{\rm and }\qquad \forall b\in(\underline{b}_{j_0},\underline{b}_{{j_0}+1}),\quad \begin{cases}
				f(b,j)>0, & \textnormal{if } j\leqslant j_0, \\
				f(b,j)<0, & \textnormal{if } j\geqslant j_0+1.
			\end{cases}
		\end{equation}
		Combining \eqref{deltaj}, \eqref{deltaj1}, \eqref{def:f} and \eqref{deltaj2} we conclude the desired result. Finally, \eqref{b3b4} follows from the identities 
		\begin{align*}
			f(b,3)&=b^{3}+\tfrac{3}{2}b^2-\tfrac{1}{2}\\
			&=(b+1)^2\big(b-\tfrac{1}{2}\big)
		\end{align*}
		and
		\begin{align*}
			f(b,4)&=b^{4}+2b^2-1\\
			&=\big({b^2+1+\sqrt{2}}\big)\left(b+\sqrt{\sqrt{2}-1}\right)\left(b-\sqrt{\sqrt{2}-1}\right).
		\end{align*}
		This ends the proof of Lemma \ref{lem:disper-relation}.
	\end{proof}
	We shall now focus on the conditions that guarantee the ellipticity of the eigenvalues based on Lemma \ref{lem:disper-relation}.
	\begin{cor}\label{coro-equilib-freq}
		Let $\Omega>0$,  $b^*\in(0,1)\setminus \big\{\underline{b}_{n},\, n\geqslant 3\big\} $ and set
		\begin{equation}\label{define-m}
			\mathtt{m}^*\triangleq \mathtt{m}^*(b^*)\triangleq \min\big\{n\geqslant 3 \quad\textnormal{s.t.}\quad \underline{b}_{n}>b^*\big\}.
		\end{equation}
		Then,  for all $|j|\geqslant\mathtt{m}^*$ and $b\in[0, b^*]$, the eigenvalues of the matrix $ M_{j}(b,\Omega)$, defined in \eqref{def MjbO}, are simple and  pure imaginary   $-\ii \Omega_{j,k}(b)$, $k\in\{1,2\}$,  with
		\begin{align}
			\Omega_{j,k}(b)&=\tfrac{j}{|j|}\bigg[\big(\Omega+\tfrac{1-b^2}{4}\big)|j|+ \tfrac{(-1)^{k+1}}{2}\sqrt{\big(\tfrac{1-b^2}{2}|j|-1\big)^2-b^{2|j|}}\bigg]\label{omegajk}\\
			&=\tfrac{j}{|j|}\bigg[\Big(\Omega+(2-k)\tfrac{1-b^2}{2}\Big)|j|+ \tfrac{(-1)^k}{2}+(-1)^{k+1} \mathtt{r}_{j}(b)\bigg], \label{ASYFR1+}
		\end{align}
		where 
		\begin{align}\label{ASYFR1-}
			\forall (n,m)\in\mathbb{N}^2,\quad  \forall\alpha\in\mathbb{N}^*,\quad  \sup_{{b \in [0,b^*]}\atop |j| \geqslant \mathtt{m}^*}|j|^\alpha|\partial_j^m\partial_ b^n \mathtt{r}_{j}(b)| \leqslant C_{n,m,\alpha}.
		\end{align}
		The corresponding  eigenspaces are real and generated by
		\begin{equation}\label{def-aj} 
			v_{j,1}(b)=
			\begin{pmatrix}
				1 \\
				-a_{j}(b)
			\end{pmatrix}, \quad v_{j,2}(b)= 
			\begin{pmatrix}
				-a_{j}(b)\\
				1
			\end{pmatrix},   \quad  a_{j}(b)=\frac{b^{|j|}}{\tfrac{1-b^2}{2}|j|-1+ \sqrt{\big(\tfrac{1-b^2}{2}|j|-1\big)^2-b^{2|j|}}}\cdot
		\end{equation}
		Moreover, there exists $\delta\triangleq \delta(b^*)>0$ such that for all $|j|\geqslant\mathtt{m}^*$ and $b\in[0, b^*]$,
		\begin{equation}\label{bound aj}
			0\leqslant a_{j}(b)<1-\delta
		\end{equation}
		and
		\begin{equation}\label{estimate-aj}
			\forall n,\alpha\in\mathbb{N}^*,\quad  \sup_{{b \in [0,b^*]}\atop |j| \geqslant \mathtt{m}^*}|j|^\alpha|\partial_{b}^{n}a_j(b)|<\infty.
		\end{equation}
	\end{cor}
	\begin{proof}
		In view of \eqref{define-m}, \eqref{variation-b}, \eqref{variation-x} and \eqref{deltaj2}, for all $b\in[0,b^*]$ and $|j|\geqslant \mathtt{m}^*$ one has
		\begin{equation}\label{deltaj4}
			f(b,|j|)\leqslant   f(b^*,\mathtt{m}^*)<f(\underline{b}_{\mathtt{m}^*},\mathtt{m}^*)=0.
		\end{equation}
		Combining \eqref{deltaj}, \eqref{deltaj1}, \eqref{def:f} and \eqref{deltaj4} we find
		$$
		\forall b\in[0,b^*],\quad\forall |j|\geqslant\mathtt{m}^*,\quad  \Delta_j(b)<0.
		$$
		Then, by Lemma \ref{lem lin op 3 DCE}, we conclude \eqref{omegajk} and \eqref{def-aj}. On the other hand, the inequality  \eqref{deltaj4} also implies 
		\begin{equation}\label{ineq-b-star}
			f(b^*,\mathtt{m}^*)=\big(b^*\big)^{\mathtt{m}^*}+1-\tfrac{1-(b^*)^2}{2}\mathtt{m}^*<0.
		\end{equation}
		This gives in turn, for any $|j|\geqslant\mathtt{m}^*$ and $b\in[0,b^*]$, 
		\begin{equation}\label{nOme3}
			\tfrac{1-b^2}{2}|j|-1 \geqslant\tfrac{1-(b^*)^2}{2}\mathtt{m}^*-1> 0.
		\end{equation}
		Consequently, for any  $|j|\geqslant\mathtt{m}^*$ and $b\in[0,b^*]$ we may write, by \eqref{omegajk},   
		\begin{align*}
			\Omega_{j,k}(b)		&=\tfrac{j}{|j|}\Bigg[\Big(\Omega+\tfrac{1-b^2}{4}\Big)|j|+\tfrac{(-1)^{k+1}}{2} \Big(
			\tfrac{1-b^2}{2}|j|-1\Big) \sqrt{1-b^{2|j|}\Big(\tfrac{1-b^2}{2}|j|-1\Big)^{-2}}\Bigg]\\ &=\tfrac{j}{|j|}\bigg[\Big(\Omega+\tfrac{1-b^2}{4}\Big)|j|+\tfrac{(-1)^{k+1}}{2} \Big(\tfrac{1-b^2}{2}|j|-1\Big)+(-1)^{k+1} \mathtt{r}_{j}(b)\bigg],
		\end{align*}
		with
		\begin{align}\label{rngpm}
			\mathtt{r}_{j}(b)&\triangleq \tfrac{1}{2} \Big(\tfrac{1-b^2}{2}|j|-1\Big)\Bigg[ \sqrt{1-b^{2|j|}\Big(\tfrac{1-b^2}{2}|j|-1\Big)^{-2}}-1\Bigg].
		\end{align}
		By virtue of \eqref{ineq-b-star} and  \eqref{nOme3},  one has for all $|j|\geqslant\mathtt{m}^*$ and $b\in [0,b^*]\subset[0,1],$
		\begin{align}\label{est dec}
			b^{|j|}\Big(\tfrac{1-b^2}{2}|j|-1\Big)^{-1}\leqslant (b^*)^{|j|}\Big(\tfrac{1-(b^*)^2}{2}\mathtt{m}^*-1\Big)^{-1}\leqslant (b^*)^{\mathtt{m}^*}\Big(\tfrac{1-(b^*)^2}{2}\mathtt{m}^*-1\Big)^{-1}<1.
		\end{align}
		Thus expanding in power series the square root and using Leibniz rule  we get  after straightforward computations the bounds for $\mathtt{r}_{j}(b)$ claimed in 
		\eqref{ASYFR1-}. Next, we shall check the inequalities  \eqref{bound aj}.
		Using \eqref{def-aj}, \eqref{nOme3} and \eqref{est dec} we conclude the existence of $\delta=\delta(b^*)\in(0,1)$ such that for all $|j|\geqslant\mathtt{m}^*$ and $b\in[0, b^*]$,
		$$
		0\leqslant a_{j}(b)\leqslant b^{|j|}\Big(\tfrac{1-b^2}{2}|j|-1\Big)^{-1} <1-\delta.$$
		Therefore the estimate \eqref{estimate-aj} follows from \eqref{def-aj} and Leibniz rule.
		This achieves the proof of Corollary~\ref{coro-equilib-freq}. 
	\end{proof}
	As a consequence of Corollary \ref{coro-equilib-freq}, we may restrict the Fourier modes to the lattice $\mathbb{Z}_{\mathtt{m}}$ with $\mathtt{m}\geqslant\mathtt{m}^*$ in order to avoid the hyperbolic spectrum. Therefore, we shall work in the phase space $L^2_{\mathtt{m}}(\mathbb{T})\times L^2_{\mathtt{m}}(\mathbb{T})$ introduced in \eqref{def L2m}. In what follows, we introduce a suitable symplectic transformation $\mathbf{Q}$ used in the diagonalization of the linearized operator at the equilibrium state described by Lemma \ref{lem lin op 2 DCE}. This diagonalization is required latter in order to perform the reduction of the remainder term, see Proposition \ref{prop RR}. The linear transformation $\mathbf{Q}$ is defined by its action on any element $(\rho_1,\rho_2)\in L^2_{\mathtt{m}}(\mathbb{T})\times L^2_{\mathtt{m}}(\mathbb{T})$ with the Fourier expansions
	$$\forall k\in\{1,2\},\quad \rho_{k}=\displaystyle\sum_{j\in\mathbb{Z}_{\mathtt{m}}^*}\rho_{j,k}\mathbf{e}_{j},\qquad \textnormal{with}\qquad  \rho_{-j,k}=\overline{\rho_{j,k}}, \qquad \mathbf{e}_{j}(\theta)=e^{\ii j\theta}
	$$ 
	as follows 
	\begin{equation}\label{def-P}
		\mathbf{Q}\begin{pmatrix}
			\rho_{1}\\
			\rho_{2}
		\end{pmatrix}
		\triangleq \sum_{j\in \mathbb{Z}_{\mathtt{m}}^*}\, \mathbf{Q}_j \begin{pmatrix}
			\rho_{j,1}\\
			\rho_{j,2}
		\end{pmatrix} \mathbf{e}_{j}\, , \qquad 
		\mathbf{Q}_j\triangleq \tfrac{1}{\sqrt{1-a_{j}^2(b)}} \begin{pmatrix}
			1 &  -a_{j} (b) \\
			- a_{j}(b)& 1 
		\end{pmatrix} \, , 
	\end{equation}
	where $a_{j}(b)$ is given by \eqref{def-aj}. 
	We have the following properties.
	\begin{lem}\label{lem: properties P}
		Let $\mathtt{m}\geqslant \mathtt{m}^*,$  $b\in[0,b^*]$, where $\mathtt{m}^*$ and $b^*$ are defined in Corollary $\ref{coro-equilib-freq}.$ Then the following assertions hold true.
		\begin{enumerate}
			\item  $\mathbf{Q}:L^2_{\mathtt{m}}(\mathbb{T})\times L^2_{\mathtt{m}}(\mathbb{T})\to L^2_{\mathtt{m}}(\mathbb{T})\times L^2_{\mathtt{m}}(\mathbb{T})$ is symplectic with respect to the symplectic form \eqref{sympl ref}.\\ In addition $\mathbf{Q}^{\top}=\mathbf{Q}.$
			\item $\mathbf{Q}$ is invertible and its inverse is given by
			\begin{equation}\label{def P-1}
				\mathbf{Q}^{-1}\begin{pmatrix}
					\rho_{1}\\
					\rho_{2}
				\end{pmatrix}
				\triangleq \sum_{j\in \mathbb{Z}_{\mathtt{m}}^*}\, \mathbf{Q}_j^{-1} \begin{pmatrix}
					\rho_{j,1}\\
					\rho_{j,2}
				\end{pmatrix} \mathbf{e}_{j}\, , \qquad 
				\mathbf{Q}_j^{-1}=\tfrac{1}{\sqrt{1-a_{j}^2(b)}}\begin{pmatrix}
					1 &  a_{j}(b) \\
					a_{j}(b)& 1 
				\end{pmatrix}.
			\end{equation}
			\item The transformations $\mathbf{Q}^{+1}\triangleq \mathbf{Q}$ and $\mathbf{Q}^{-1}$ write
			\begin{equation}\label{PP-1}
				\mathbf{Q}^{\pm 1}=\begin{pmatrix}
					\mathbb{I}_{\mathtt{m}} &0\\
					0& \mathbb{I}_{\mathtt{m}}
				\end{pmatrix}+\begin{pmatrix}
					{P}_{1}\ast \cdot & \mp{P}_{2}\ast \cdot\\
					\mp{P}_{2}\ast \cdot& {P}_{1}\ast \cdot
				\end{pmatrix},
				\qquad {P}_{k}\triangleq\sum_{j\in\mathbb{Z}_{\mathtt{m}}^*}\tfrac{(2-k)+(-1)^k\sqrt{(2-k)+(-1)^k a_j^2(b)}}{\sqrt{1-a_j^2(b)}}\mathbf{e}_{j}.
			\end{equation}
			For any $k\in\{1,2\},$ the kernel $P_{k}$ satisfies the symmetry properties
			\begin{equation}\label{sym Pk}
				P_{k}(-\theta)=P_{k}(\theta)=P_{k}\big(\theta+\tfrac{2\pi}{\mathtt{m}}\big)
			\end{equation}
			and the estimate
			\begin{equation}\label{e-P1P2}
				\forall n\in\mathbb{N},\quad\|\partial_{\theta}^n{P}_{k}\ast\rho\|_{s}^{q,\gamma,\m}\lesssim \|\rho\|_{s}^{q,\gamma,\m}.
			\end{equation}
			\item The transformation $\mathbf{Q}$ diagonalizes the operator $\mathcal{J} \mathbf{M}_0$,
			where $\mathbf{M}_0$ is introduced in  \eqref{Ham eq-eq DCE}, namely
			\begin{equation}\label{def Lbmat}
				\mathbf{Q}^{-1}\mathcal{J} \mathbf{M}_0\mathbf{Q}= \mathcal{J} \mathbf{L}_0, \qquad  \mathbf{L}_0\begin{pmatrix}
					\rho_{1}\\
					\rho_{2}
				\end{pmatrix}
				\triangleq \sum_{j\in\mathbb{Z}_{\mathtt{m}}^*}\tfrac{1}{j}\begin{pmatrix}
					-\Omega_{j,1}(b) & 0\\
					0 & \Omega_{j,2}(b)
				\end{pmatrix} \begin{pmatrix}
					\rho_{j,1}\\
					\rho_{j,2}
				\end{pmatrix} \mathbf{e}_{j}.
			\end{equation}
			\item All the real-valued solutions of the linearized contour dynamics equation \eqref{Ham eq-eq DCE} have the form   
			\begin{equation}\label{lin-sol}
				\rho(t,\theta)=\sum_{j\in\mathbb{Z}_\mathtt{m}^*}\tfrac{A_j}{\sqrt{1-a_{j}^2(b)}} \begin{pmatrix}
					1  \\
					- a_{j}(b) 
				\end{pmatrix}e^{-\ii \left(\Omega_{1,j}(b)t- j\theta\right)}+\tfrac{B_j}{\sqrt{1-a_{j}^2(b)}} \begin{pmatrix}
					-a_{j} (b) \\
					1 
				\end{pmatrix}e^{-\ii \left(\Omega_{2,j}(b)t-j\theta\right)}
			\end{equation}
			with $\overline{A_j}=A_{-j},\,\overline{B_j}=B_{-j}.$
		\end{enumerate}
	\end{lem}
	\begin{proof}
		\textbf{1.} Straightforward computations based on the definition \eqref{def-P} lead to 
		$$\forall j\in \mathbb{Z}_{\mathtt{m}}^*,\quad \mathbf{Q}_j^\top \begin{pmatrix}
			1 &  0 \\
			0 & -1 
		\end{pmatrix} \mathbf{Q}_j=\begin{pmatrix}
			1 &  0 \\
			0 & -1 
		\end{pmatrix}.$$
		Then, using Lemma \ref{lem: charac symp} we conclude the first point. Notice that one also has $$\forall j\in\mathbb{Z}_{\mathtt{m}}^*,\quad\mathbf{Q}_{j}^{\top}=\mathbf{Q}_{j},$$
		which implies $\mathbf{Q}^{\top}=\mathbf{Q}.$\\
		\textbf{2.} The second point follows easily by direct computations.\\
		\textbf{3.} In view of \eqref{def-P} and \eqref{def P-1} we can write
		$$
		\mathbf{Q}_j^{\pm 1}= \begin{pmatrix}
			1 &  0\\
			0&  1 
		\end{pmatrix}+\tfrac{1}{\sqrt{1-a_j^2(b)}} \begin{pmatrix}
			1-\sqrt{1-a_j^2(b)} &  \mp a_j (b) \\
			\mp a_j(b)& 1-\sqrt{1-a_j^2(b)} 
		\end{pmatrix} \, , 
		$$
		leading to \eqref{PP-1}. The symmetry properties \eqref{sym Pk} are obtained either by the fact that $a_{-j}(b)=a_{j}(b)$ (see \eqref{def-aj}) or by the restriction of the Fourier  modes in the definition of $P_k.$ The estimate \eqref{e-P1P2} is obtained by applying the Leibniz and the chain rules with \eqref{PP-1} and \eqref{estimate-aj}.
		\\
		\textbf{4.} Notice, from \eqref{def-P} and \eqref{def-aj}, that 
		$$
		\mathbf{Q}_j=\tfrac{1}{\sqrt{1-a_{j}^2(b)}}\Big(v_{j,1}(b)\quad v_{j,2}(b)\Big).
		$$
		Then, according to Corollary  \ref{coro-equilib-freq}, the matrices $\mathbf{Q}_j$ diagonalize the matrices $M_j(b,\Omega)$, defined in \eqref{def MjbO}, namely
		$$\forall j\in\mathbb{Z}_{\mathtt{m}}^*,\quad \mathbf{Q}_j^{-1}M_j(b,\Omega)\mathbf{Q}_j=-\ii\begin{pmatrix}
			\Omega_{j,1}(b) & 0\\
			0 & \Omega_{j,2}(b)
		\end{pmatrix}.$$
		Therefore we deduce from Lemma \ref{lem lin op 2 DCE} 
		\begin{align}\label{linerized-op-diag}
			\big(\mathbf{Q}^{-1}\mathcal{J}\mathbf{M}_0\mathbf{Q}\big)\begin{pmatrix}
				\rho_{1}\\
				\rho_{2}
			\end{pmatrix}
			&=\sum_{j\in\mathbb{Z}_{\mathtt{m}}^*}\begin{pmatrix}
				-\ii\Omega_{1,j}(b) & 0\\
				0 & -\ii\Omega_{2,j}(b)
			\end{pmatrix} \begin{pmatrix}
				\rho_{j,1}\\
				\rho_{j,2}
			\end{pmatrix} \mathbf{e}_{j}\\ &=\sum_{j\in\mathbb{Z}_{\mathtt{m}}^*}\ii j\begin{pmatrix}
				1& 0\\
				0 & -1
			\end{pmatrix}\begin{pmatrix}
				-\frac{1}{j}\Omega_{1,j}(b) & 0\\
				0 & \frac{1}{j} \Omega_{2,j}(b)
			\end{pmatrix} \begin{pmatrix}
				\rho_{j,1}\\
				\rho_{j,2}
			\end{pmatrix} \mathbf{e}_{j},\nonumber
		\end{align}
		which gives in turn \eqref{def Lbmat}.\\
		\textbf{5.} It follows immediately from the fourth point when solving the linear differential system \eqref{def MjbO}. This completes the proof of Lemma \ref{lem: properties P}.
	\end{proof}
	\subsection{Symplectic change of coordinates}
	In this section we intend to conjugate the nonlinear Hamiltonian  system \eqref{Hamilt form DCE}  with respect to the symplectic linear  transformation $\mathbf{Q}$ introduced in \eqref{def-P}. Notice that this transformation does not depend on the unknown $r$, it depends only on the parameter $b.$\\
	
	Let us consider the symplectic unknown  $\tilde{r}\triangleq \mathbf{Q}^{-1}r$. Then the Hamiltonian system \eqref{Hamilt form DCE} writes
	\begin{equation}\label{nonlinear-func0}
		\partial_t \tilde{r}=X_K(\tilde{r})=\mathcal{J}\nabla K(\tilde{r}), \qquad K(\tilde{r})\triangleq H(\mathbf{Q}\tilde{r}),\qquad \tilde{r}=(\tilde{r}_1,\tilde{r}_2)\in L^2_{\mathtt{m}}(\mathbb{T})\times L^2_{\mathtt{m}}(\mathbb{T}).
	\end{equation}
	Indeed, on one hand, we have
	\begin{equation}\label{symp1}
		\partial_{t}\widetilde{r}=\mathbf{Q}^{-1}\partial_{t}r=\mathbf{Q}^{-1}\mathcal{J}\nabla H(r)=\mathbf{Q}^{-1}\mathcal{J}(\nabla H)(\mathbf{Q}\widetilde{r}).
	\end{equation}
	On the other hand, 
	\begin{equation}\label{symp2}
		\mathcal{J}\nabla K(\widetilde{r})=\mathcal{J}\nabla\big(H(\mathbf{Q}\widetilde{r})\big)=\mathcal{J}\mathbf{Q}(\nabla H)(\mathbf{Q}\widetilde{r}).
	\end{equation}
	Therefore, if $ \mathbf{Q}\mathcal{J}\mathbf{Q}=\mathcal{J}$ then we have the equivalence
	\begin{equation}\label{symp3}
		\Big(\partial_{t}r=\mathcal{J}\nabla H(r)\quad\Leftrightarrow\quad\partial_{t}\widetilde{r}=\mathcal{J}\nabla K(\widetilde{r})\Big)
	\end{equation}
	and this last condition is true since Lemma \ref{lem: properties P}-1 implies that $\mathbf{Q}$ is symplectic and $\mathbf{Q}^{\top}=\mathbf{Q}.$\\
	
	We shall look for time quasi-periodic solutions of \eqref{nonlinear-func0}  in the form
	$$\tilde{r}(t,\theta)=\hat{r}(\omega t,\theta),$$
	where $\hat{r}:(\varphi,\theta)\in \mathbb{T}^{d+1}\to \mathbb{R}^2$ and $\omega\in \mathbb{R}^d$ 
	is a non-resonant vector frequency. In
	this setting, the equation \eqref{nonlinear-func0} becomes
	$$
	\omega\cdot\partial_{\varphi} \hat{r} =\mathcal{J}\nabla K(\hat{r}).
	$$
	In the sequel, we shall alleviate the notation and denote $\hat{r}$ simply by $r$. Hence, the foregoing equation becomes
	\begin{equation}\label{nonlinear-func}
		\omega\cdot\partial_{\varphi} {r} =\mathcal{J}\nabla K(r),  \qquad K({r})\triangleq H(\mathbf{Q}r),\qquad r=(r_1,r_2)\in L^2_{\mathtt{m}}(\mathbb{T})\times L^2_{\mathtt{m}}(\mathbb{T}).
	\end{equation}
	The main result of this section reads as follows.
	\begin{prop}\label{prop:conjP}
		The linearized equation of \eqref{nonlinear-func} at a given small state $r$  takes the form 
		$$
		\omega\cdot\partial_{\varphi}\rho +\mathfrak{L}_{r}\rho=0,\qquad\mathfrak{L}_r\triangleq -d_r\big(\mathcal{J}\nabla K({r})\big),
		$$ with 
		\begin{equation}\label{defLr2}
			\begin{aligned}
				\mathfrak{L}_{r} =& \begin{pmatrix}
					\partial_\theta\big(\mathcal{V}_1(r)\, \cdot\big) +\frac{1}{2}\mathcal{H}+\partial_\theta\mathcal{Q}\ast\cdot&0\\
					0& \partial_\theta\big(\mathcal{V}_2(r)\, \cdot\big)-\tfrac{1}{2}\mathcal{H} -\partial_\theta \mathcal{Q}\ast\cdot
				\end{pmatrix}
				\\
				&\qquad +\partial_\theta\begin{pmatrix}
					\mathcal{T}_{\mathscr{K}_{1,1}}(r)& \mathcal{T}_{\mathscr{K}_{1,2}}(r)\\
					\mathcal{T}_{\mathscr{K}_{2,1}}(r) & \mathcal{T}_{\mathscr{K}_{2,2}}(r)
				\end{pmatrix},
			\end{aligned}		
		\end{equation}
		where
		\begin{enumerate}
			\item the scalar functions $\mathcal{V}_{k}(r)\triangleq V_k(\mathbf{Q}r)$, $k\in\{1,2\}$ satisfy 
			\begin{equation}\label{es-f0}
				\|\mathcal{V}_{k}(r)-V_{k}(0)\|_{s}^{q,\gamma,\mathtt{m}}\lesssim\|r\|_{s+1}^{q,\gamma,\mathtt{m}},
			\end{equation}
			with $V_k(r)$ and $V_k(0)$ described in \eqref{def Vpm} and \eqref{def V10 V20}, respectively,
			\item the convolution operator $\mathcal{Q}\ast\cdot$ with even kernel $\mathcal{Q}$ is defined through
			\begin{align}\label{dcp calDeb0}
				\forall\, j\in\mathbb{Z}_{\mathtt{m}}^*,\quad \mathcal{Q}\ast \mathbf{e}_{j}\triangleq \tfrac{\mathtt{r}_{j}(b)}{|j|}\mathbf{e}_{j},
			\end{align}
			with $\mathtt{r}_{j}(b)$ being introduced in Corollary \ref{coro-equilib-freq},
			\item  for  $k,n\in\{1,2\},$ the operator $\mathcal{T}_{\mathscr{K}_{k,n}}({r})$ is an integral  operator in the form \eqref{Top-op1}  whose kernel
			$\mathscr{K}_{k,n}(r)$ is $\mathtt{m}$-fold reversibility preserving and satisfies the estimates: for any $s\geqslant s_0$
			\begin{equation*}
				\|\mathscr{K}_{k,n}(r)\|_{s}^{q,\gamma,\mathtt{m}}\lesssim\|r\|_{s+1}^{q,\gamma,\mathtt{m}}
			\end{equation*}
			and
			\begin{equation*}
				\|\Delta_{12}\mathscr{K}_{k,n}(r)\|_{s}^{q,\gamma,\m}\lesssim\|\Delta_{12}r\|_{s+1}^{q,\gamma,\m}+\|\Delta_{12}r\|_{s_0+1}^{q,\gamma,\m}\max_{\ell\in\{1,2\}}\|r_\ell\|_{s+1}^{q,\gamma,\m}.
			\end{equation*}
		\end{enumerate}
	\end{prop}
	
	\begin{proof}	 	
		According to \eqref{symp1}, \eqref{symp2} and \eqref{symp3}, one has 
		\begin{equation}\label{id:XKXH}
			\mathcal{J}\nabla K({r})=\mathbf{Q}^{-1}\mathcal{J}(\nabla H)(\mathbf{Q}{r}).
		\end{equation}
		Differentiating this identity with respect to $r$ in the direction $\rho$  and using Lemma \ref{lem lin op 1 DCE} lead to
		\begin{align}\label{drjk}
			\nonumber d_r \big(\mathcal{J}\nabla K({r})\big)\rho&=\mathbf{Q}^{-1}\big(d_r \big(\mathcal{J}\nabla H\big)(\mathbf{Q}{r})\big)\mathbf{Q}\rho
			\\ \nonumber &=\mathbf{Q}^{-1}\mathcal{J} \mathbf{M}_{\mathbf{Q}{r}}\mathbf{Q}\rho
			\\ &=\mathbf{Q}^{-1}\mathcal{J} \mathbf{M}_0\mathbf{Q}\rho +\mathbf{Q}^{-1}\mathcal{J}\big( \mathbf{M}_{\mathbf{Q}{r}}-\mathbf{M}_0\big)\mathbf{Q}\rho.
		\end{align}
		By virtue of \eqref{linerized-op-diag}, \eqref{ASYFR1+}, \eqref{def V10 V20}  and recalling \eqref{def Hilbert} and \eqref{dcp calDeb0} we may write
		\begin{equation}\label{p-1jl0p}
			\mathbf{Q}^{-1}\mathcal{J} \mathbf{M}_0\mathbf{Q}=-\begin{pmatrix}
				V_1(0)\partial_\theta+\tfrac{1}{2}\mathcal{H}+\partial_\theta\mathcal{Q}\ast\cdot  & 0\\
				0 & V_2(0)\partial_\theta-\tfrac{1}{2}\mathcal{H}-\partial_\theta\mathcal{Q}\ast\cdot  
			\end{pmatrix}.
		\end{equation}
		On the other hand, from Lemma \ref{lem lin op 1 DCE} and Lemma \ref{lem lin op 2 DCE} we deduce that
		\begin{equation}\label{lr-l0}
			\mathbf{M}_{\mathbf{Q}{r}}-\mathbf{M}_{0}=\begin{pmatrix}
				-f_{1}(r) & 0\\
				0 & f_{2}(r)
			\end{pmatrix}+\begin{pmatrix}
				-L_{1,1}(\mathbf{Q}{r})+\mathcal{K}_1\ast\cdot & L_{1,2}(\mathbf{Q}r)-\mathcal{K}_b\ast\cdot \\
				L_{2,1}(\mathbf{Q}r)-\mathcal{K}_b\ast\cdot  & -L_{2,2}(\mathbf{Q}r)+\mathcal{K}_1\ast\cdot 
			\end{pmatrix}
		\end{equation}
		with 
		\begin{align}\label{f0+-}
			\forall k\in\{1,2\},\quad f_{k}(r)\triangleq V_{k}(\mathbf{Q}r)-V_k(0).
		\end{align}
		Notice that if $r$ satisfies \eqref{reversibility condition r} and \eqref{m-fold symmetry r} then,
		by virtue of  \eqref{sym-m Vpm}, one gets that $f_k(r)$ itself  satisfies the same symmetries,
		\begin{align}
			f_k(r)(-\varphi,-\theta)=f_k(r)(\varphi,\theta)= f_k(r)\big(\varphi,\theta+\tfrac{2\pi}{\m}\big).\label{sym f0}
		\end{align} 
		We shall now turn to the quantitative estimates. For this aim, we shall first give the following decompositions. According to \eqref{Akn}, we can write
		\begin{align*}
			A_{k,k}({r})(\varphi,\theta,\eta)&=\left((R_{k}(\varphi,\theta)-R_{k}(\varphi,\eta))^{2}+4R_{k}(\varphi,\theta)R_{k}(\varphi,\eta)\sin^2\left(\tfrac{\eta-\theta}{2}\right)\right)^{\frac{1}{2}}\\
			&=2b_k\left|\sin\Big(\tfrac{\eta-\theta}{2}\Big)\right|\Bigg(\bigg(\frac{R_{k}(\varphi,\theta)-R_{k}(\varphi,\eta)}{2b_k\sin\big(\tfrac{\eta-\theta}{2}\big)}\bigg)^{2}+\frac{1}{b_k^2}R_{k}(\varphi,\theta)R_{k}(\varphi,\eta)\Bigg)^{\frac{1}{2}}\\
			&\triangleq 2b_k\left|\sin\left(\tfrac{\eta-\theta}{2}\right)\right|v_{k}({r})(\varphi,\theta,\eta).
		\end{align*}
		the function $v_k(r)$ is smooth with respect to each variables and with respect to $r$. In addition $v_k(0)=1.$ An application of Lemma \ref{lem triche} and Lemma \ref{lem funct prop}-(iii) gives
		\begin{equation}\label{vk-e}
			\|v_{k}(r)-1\|_{s}^{q,\gamma,\mathtt{m}}\lesssim\|r\|_{s+1}^{q,\gamma,\mathtt{m}},\qquad\|\Delta_{12}v_{k}\|_{s}^{q,\gamma,\mathtt{m}}\lesssim\|\Delta_{12}r\|_{s+1}^{q,\gamma,\mathtt{m}}+\|\Delta_{12}r\|_{s_0+1}^{q,\gamma,\mathtt{m}}\max_{\ell\in\{1,2\}}\|r_{\ell}\|_{s+1}^{q,\gamma,\mathtt{m}}.
		\end{equation}
		We can also write
		\begin{align*}
			A_{k,3-k}^2(r)(\varphi,\theta,\eta)&=R_{k}^2(\varphi,\theta)+R_{3-k}^2(\eta)-2R_{k}(\varphi,\theta)R_{3-k}(\varphi,\eta)\cos(\eta-\theta)\\
			&=A_{k,3-k}^2(0)(\varphi,\theta,\eta)\big(1+h_k(r)(\varphi,\theta,\eta)\big),
		\end{align*}
		where
		$$h_k(r)(\varphi,\theta,\eta)\triangleq 2\frac{r_{k}(\varphi,\theta)+r_{3-k}(\varphi,\eta)-\cos(\eta-\theta)\big(R_{k}(\varphi,\theta)R_{3-k}(\varphi,\eta)-b_kb_{3-k}\big)}{b_{k}^2+b_{3-k}^2-2b_kb_{3-k}\cos(\eta-\theta)}\cdot$$
		The function $h_k$ satisfies by composition laws in Lemma \ref{lem funct prop}-(iii)
		\begin{equation}\label{h-e}
			\|h_k(r)\|_{s}^{q,\gamma,\mathtt{m}}\lesssim\|r\|_{s}^{q,\gamma,\mathtt{m}},\qquad\|\Delta_{12}h_k\|_{s}^{q,\gamma,\mathtt{m}}\lesssim\|\Delta_{12}r\|_{s}^{q,\gamma,\mathtt{m}}.
		\end{equation}
		According to the foregoing decompositions and the estimates \eqref{vk-e}, \eqref{h-e} together with \eqref{f0+-}, \eqref{def Vpm}, \eqref{def Vkn}, the product and composition laws in Lemma \ref{lem funct prop}-(iii) imply
		\begin{equation}\label{es-f0-00}
			\|f_k(r)\|_{s}^{q,\gamma,\m}\lesssim\|r\|_{s+1}^{q,\gamma,\m}
		\end{equation}
		and
		\begin{equation}\label{es-f0-diff}
			\|\Delta_{12}f_k\|_{s}^{q,\gamma,\m}\lesssim\|\Delta_{12}r\|_{s+1}^{q,\gamma,\m}+\|\Delta_{12}r\|_{s_0+1}^{q,\gamma,\m}\max_{\ell\in\{1,2\}}\|r_\ell\|_{s+1}^{q,\gamma,\m}.
		\end{equation}
		In view of \eqref{def mathbfLkn} and \eqref{def mathcalKkn}, we also have the following decompositions for $k\in\{1,2\}$
		\begin{equation}\label{lkk-lk3-k}
			L_{k,k}(\mathbf{Q}r)=\mathcal{K}_1\ast\cdot+\mathcal{T}_{\mathbb{K}_{k,k}}(r)\qquad\textnormal{and}\qquad L_{k,3-k}(\mathbf{Q}r)=\mathcal{K}_b\ast\cdot+\mathcal{T}_{\mathbb{K}_{k,3-k}}(r),
		\end{equation}
		where $\mathcal{T}_{\mathbb{K}_{k,k}}(r)$ and  $\mathcal{T}_{\mathbb{K}_{k,3-k}}(r)$ are integral operators with smooth kernels
		$$\mathbb{K}_{k,k}(r)(\varphi,\theta,\eta)\triangleq \log\big(v_{k}(\mathbf{Q}r)(\varphi,\theta,\eta)\big)$$
		and 
		$$\mathbb{K}_{k,3-k}(r)(\varphi,\theta,\eta)\triangleq \tfrac{1}{2}\log\big(1+h_k(\mathbf{Q}r)(\varphi,\theta,\eta)\big).$$
		Moreover, if $r$ satisfies \eqref{reversibility condition r} and \eqref{m-fold symmetry r} then, for all $k,n\in\{1,2\}$, one can easily check, using in particular \eqref{sym Akn}, that 
		\begin{align}
			\mathbb{K}_{k,n}(r)(-\varphi,-\theta,-\eta)=\mathbb{K}_{k,n}(r)(\varphi,\theta,\eta)= \mathbb{K}_{k,n}(r)\big(\varphi,\theta+\tfrac{2\pi}{\mathtt{m}},\eta+\tfrac{2\pi}{\mathtt{m}}\big).\label{sym-m Kkn}
		\end{align} 
		It is clear that \eqref{PP-1}-\eqref{e-P1P2} imply the continuity of $\mathbf{Q}$ on $\mathbf{H}_{\mathtt{m}}^s.$ Thus, applying the composition laws in Lemma \ref{lem funct prop}-(iii) together with \eqref{vk-e} and \eqref{h-e}, we infer
		\begin{equation}\label{est-Kkn}
			\forall (k,n)\in\{1,2\}^2,\quad\|\mathbb{K}_{k,n}(r)\|_{s}^{q,\gamma,\m}\lesssim\|r\|_{s+1}^{q,\gamma,\m}
		\end{equation}
		and
		\begin{equation}\label{est-Kkn-diff}
			\forall (k,n)\in\{1,2\}^2,\quad\|\Delta_{12}\mathbb{K}_{k,n}(r)\|_{s}^{q,\gamma,\m}\lesssim\|\Delta_{12}r\|_{s+1}^{q,\gamma,\m}+\|\Delta_{12}r\|_{s_0+1}^{q,\gamma,\m}\max_{\ell\in\{1,2\}}\|r_\ell\|_{s+1}^{q,\gamma,\m}.
		\end{equation}
		Putting together \eqref{lr-l0} and \eqref{lkk-lk3-k} we deduce that
		$$\mathbf{Q}^{-1}\mathcal{J}\big( \mathbf{M}_{\mathbf{Q}{r}}-\mathbf{M}_0\big)\mathbf{Q}=-\mathbf{Q}^{-1}\partial_\theta\begin{pmatrix}
			f_1(r)& 0\\
			0 & f_2(r)
		\end{pmatrix}\mathbf{Q}-\mathbf{Q}^{-1}\partial_\theta\begin{pmatrix}
			\mathcal{T}_{\mathbb{K}_{1,1}}(r) & -\mathcal{T}_{\mathbb{K}_{1,2}}(r) \\
			\mathcal{T}_{\mathbb{K}_{2,1}}(r) & -\mathcal{T}_{\mathbb{K}_{2,2}}(r) 
		\end{pmatrix}\mathbf{Q}.
		$$
		Combining the last identity with \eqref{drjk} and \eqref{p-1jl0p} we find \eqref{defLr2} with
		\begin{align*}
			\begin{pmatrix}
				\mathcal{T}_{\mathscr{K}_{1,1}}(r)& \mathcal{T}_{\mathscr{K}_{1,2}}(r)\\
				\mathcal{T}_{\mathscr{K}_{2,1}}(r) & \mathcal{T}_{\mathscr{K}_{2,2}}(r)
			\end{pmatrix}&\triangleq \mathbf{Q}^{-1}\begin{pmatrix}
				f_1(r)& 0\\
				0 & f_2(r)
			\end{pmatrix}\mathbf{Q}-\begin{pmatrix}
				f_1(r) & 0\\
				0 & f_2(r)
			\end{pmatrix}+\mathbf{Q}^{-1}\begin{pmatrix}
				\mathcal{T}_{\mathbb{K}_{1,1}}(r) & -\mathcal{T}_{\mathbb{K}_{1,2}}(r) \\
				\mathcal{T}_{\mathbb{K}_{2,1}}(r) & -\mathcal{T}_{\mathbb{K}_{2,2}}(r) 
			\end{pmatrix}\mathbf{Q}.
		\end{align*}
		By virtue of the preceding estimate, \eqref{PP-1}, \eqref{e-P1P2}, \eqref{es-f0-00} and \eqref{est-Kkn}, together with Lemma \ref{iter-kerns} we find through straightforward computations that the kernels $\mathscr{K}_{k,n}(r)$ satisfy the following estimate 
		\begin{equation*}
			\forall (k,n)\in\{1,2\}^2,\quad\|\mathscr{K}_{k,n}(r)\|_{s}^{q,\gamma,\m}\lesssim\|r\|_{s+1}^{q,\gamma,\m}.
		\end{equation*}
		Similar argument as before  using in particular \eqref{est-Kkn-diff} and \eqref{es-f0-diff} implies
		\begin{equation*}
			\forall (k,n)\in\{1,2\}^2,\quad\|\Delta_{12}\mathscr{K}_{k,n}(r)\|_{s}^{q,\gamma,\m}\lesssim\|\Delta_{12}r\|_{s+1}^{q,\gamma,\m}+\|\Delta_{12}r\|_{s_0+1}^{q,\gamma,\m}\max_{\ell\in\{1,2\}}\|r_\ell\|_{s+1}^{q,\gamma,\m}.
		\end{equation*}
		As for the symmetry property,  it can be obtained easily from the structure of the kernel. Actually, one may check  from   \eqref{sym-m Kkn}, \eqref{sym f0} and \eqref{sym Pk} that
		\begin{align}
			r(-\varphi,-\theta)=r(\varphi,\theta)\quad&\Longrightarrow\quad\mathscr{K}_{k,n}(r)(-\varphi,-\theta,-\eta)=\mathscr{K}_{k,n}(r)(\varphi,\theta,\eta),\label{sym scrKkn}\\
			r\big(\varphi,\theta+\tfrac{2\pi}{\mathtt{m}}\big)=r(\varphi,\theta)\quad&\Longrightarrow\quad\mathscr{K}_{k,n}(r)\big(\varphi,\theta+\tfrac{2\pi}{\mathtt{m}},\eta+\tfrac{2\pi}{\mathtt{m}}\big)=\mathscr{K}_{k,n}(r)(\varphi,\theta,\eta).\label{sym-m scrKkn}
		\end{align}
		The proof of the desired results is now complete.
	\end{proof}
	\begin{remark}\label{remark-lin-eq-eq}
		The linearized equation of \eqref{nonlinear-func0} at $r=0$ takes the form 
		\begin{equation}\label{Edc Ham eq0}
			\partial_t \rho=\mathcal{J}\nabla K_{\mathbf{L}_0}(\rho),\qquad \rho=(\rho_1,\rho_2)\in L^2_{\mathtt{m}}(\mathbb{T})\times L^2_{\mathtt{m}}(\mathbb{T}), 
		\end{equation}
		where $\mathcal{J}$ is defined in \eqref{def calJ},  
		$K_{\mathbf{L}_0}$ is the quadratic Hamiltonian
		\begin{equation}\label{def KL}
			K_{\mathbf{L}_0}(\rho)\triangleq \tfrac{1}{2}\big\langle\mathbf{L}_0\rho,\rho\big\rangle_{L^2(\mathbb{T})\times L^2(\mathbb{T})}=-\sum_{j\in\mathbb{Z}_{\mathtt{m}}^*}\left(\tfrac{\Omega_{j,1}(b)}{2j}|\rho_{j,1}|^2-\tfrac{\Omega_{j,2}(b)}{2j}|\rho_{j,2}|^2\right)
		\end{equation}
		and $\mathbf{L}_0$ is the operator defined by \eqref{def Lbmat}. Thus,  real-valued oscillating solution of the linearized contour dynamics equation \eqref{Edc Ham eq0} are given by  
		$$\rho(t,\theta)=\sum_{j\in\mathbb{Z}_\mathtt{m}^*}A_j \begin{pmatrix}
				1  \\
				0
			\end{pmatrix}e^{-\ii \left(\Omega_{1,j}(b)t- j\theta\right)}+B_j\begin{pmatrix}
				0 \\
				1 
			\end{pmatrix}e^{-\ii \left(\Omega_{2,j}(b)t-j\theta\right)},$$
		with $\overline{A_j}=A_{-j},\,\overline{B_j}=B_{-j}$.
	\end{remark}

	\subsection{Geometric structure of the equilibrium frequencies}
	This section is devoted to some useful properties of the equilibrium frequencies. We shall first discuss their monotonicity and prove some useful bounds. Then, we shall be concerned with their non-degeneracy through the study of the transversality conditions. Those latter are crucial in the measure estimates of the final Cantor set giving rise to quasi-periodic solutions emerging at the linear and nonlinear levels. We have the following lemma.
	\begin{lem}\label{lem-asym}
		Let $\Omega>0$ and   $ \mathtt{m}^*,  b^*$ be defined as in Corollary $\ref{coro-equilib-freq}$. Then the following holds true.
		\begin{enumerate}
			\item  For all   $|j|\geqslant \mathtt{m}^*$ and $k\in\{1,2\}$,   
			$
			\Omega_{-j,k}(b)=-\Omega_{j,k}(b). 
			$
			\item  The sequence $\Big(-\tfrac{\Delta_{j}(b)}{j^2}\Big)_{j\geqslant \mathtt{m}^*}$ is positive increasing. Recall that $\Delta_{j}(b)$ was defined in \eqref{def delta j}.
			\item The sequence $\Big(\tfrac{\Omega_{j,1}(b)}{j}\Big)_{j\geqslant\mathtt{m}^*}$ is positive  increasing and the sequence $\Big(\tfrac{\Omega_{j,2}(b)}{j}\Big)_{j\geqslant \mathtt{m}^*}$ is positive decreasing. Moreover, for all $|j|\geqslant \mathtt{m}^*$ and $k\in\{1,2\}$ we have 
			\begin{equation}\label{lim omega jk}
				\lim_{n\to\infty}\tfrac{\Omega_{j,k}(b)}{j}= \Omega +(2-k)\tfrac{1-b^2}{2}\cdot
			\end{equation}
			and
			\begin{equation*}
				\big|\Omega_{j,k}(b)\big|\geqslant \Omega |j|.
			\end{equation*}
			\item  For all $\mathtt{m}\geqslant \mathtt{m}^*$, there exists $\Omega_{\mathtt{m}}^*=\Omega^*(b^*,\mathtt{m})>0$ satisfying 
			\begin{equation}\label{lim omega-star}
				\lim_{\mathtt{m}\to \infty}\Omega_{\mathtt{m}}^*=0
			\end{equation}
		such that for all $\Omega>\Omega^*_{\mathtt{m}}$ the sequence  $\big(\Omega_{j,2}(b)\big)_{j\geqslant \mathtt{m}}$ is   increasing.
			\item There exists $c>0$ such that,  for all $k\in\{1,2\}$,
			\begin{equation*}
				\forall\, \Omega>\Omega^*_{\mathtt{m}},\quad \forall\, b\in[0,b^*],\quad \forall\, |j|\geqslant \mathtt{m}^*,\quad \forall\, |j'|\geqslant \mathtt{m}^*,\quad \big|\Omega_{j,k}(b)-\Omega_{j',k}(b)\big|\geqslant c |j- j'|.
			\end{equation*}
			\item Given  ${q}_0\in\N$, there exists $C>0$ such that, for all $k\in\{1,2\}$,
			\begin{equation*}
				\forall \, |j|\geqslant \mathtt{m}^*,\quad \forall \, |j'|\geqslant \mathtt{m}^*,\quad \max_{q\in\llbracket 0,{q}_0\rrbracket} \sup_{b\in [0,b^*]}\Big|\partial_b^q\big(\Omega_{j,k}(b)-\Omega_{j',k}(b)\big)\Big|\leqslant C |j-j'|.
			\end{equation*}
		\end{enumerate}
	\end{lem} 
	\begin{proof}
		{\bf 1.} It follows immediately from \eqref{omegajk}.\\
		{\bf 2.} In order to  study the discrete function $j\mapsto -\tfrac{\Delta_{j}(b)}{j^2}$ we shall  
		consider its continuous version 
		$$\forall x\geqslant \mathtt{m}^*,\quad  g(x)\triangleq -\tfrac{\Delta_{x}(b)}{x^2}=\tfrac{1}{x^2}\big(\tfrac{1-b^2}{2}x-1\big)^2-\tfrac{b^{2x}}{x^2}\cdot$$
		Differentiating with respect to $x$ and using \eqref{nOme3}, we conclude that
		\begin{align*}
			g'(x)&= \tfrac{2}{x^3}\big(\tfrac{1-b^2}{2}x-1+b^{2x}\big)-\tfrac{2b^{2x}}{x^2}\log (b)>0.
		\end{align*}
		Thus, the mapping $j\mapsto -\Delta_{j}(b)/j^2$ is strictly increasing. \\
		{\bf 3.} The monotonicity of the sequences $\Big(  \tfrac{\Omega_{j,k}(b)}{j}\Big)_{j\geqslant\mathtt{m}^*} $ follows from the identity 
		$$
		\tfrac{\Omega_{j,k}(b)}{j}=\big(\Omega+\tfrac{1-b^2}{4}\big)+ \tfrac{(-1)^{k+1}}{2}\sqrt{\tfrac{-\Delta_{j}(b)}{j^2}}
		$$
		and the second point. Moreover, from the last identity we also conclude that
		$$
		\tfrac{\Omega_{j,1}(b)}{j}\geqslant \Omega.
		$$
		Next, from \eqref{ASYFR1+}-\eqref{ASYFR1-} we obtain \eqref{lim omega jk}.
		Since $\Big(  \tfrac{\Omega_{j,2}(b)}{j}\Big)_{j\geqslant\mathtt{m}^*} $ is decreasing, then from \eqref{lim omega jk} we infer that
		$$\tfrac{\Omega_{j,2}(b)}{j}\geqslant \lim_{j\to\infty} \tfrac{\Omega_{j,2}(b)}{j}=\Omega>0.$$
		This ends the proof of the third point.\\ 
		{\bf 4.} Consider the continuous extension  of the discrete mapping $j\mapsto \Omega_{j,2}(b)$,
		\begin{align*}
			\forall (b,x)\in[0,b^*]\times(\mathtt{m}^*,\infty),\quad h(b,x)&\triangleq \Omega x+\tfrac{1-b^2}{4}x-\tfrac{1}{2}\sqrt{-\Delta_{x}(b)}.
		\end{align*}
		Differentiating with respect to $x$ and using \eqref{def delta j} and \eqref{nOme3} lead to
		\begin{align*}
			\partial_x h(b,x)
			&=\tfrac{b^{2x}\log (b)-\tfrac{1-b^2}{2}\big(\tfrac{1-b^2}{2}x-1-\sqrt{-\Delta_{x}(b)}\big)+2\Omega\sqrt{-\Delta_{x}(b)}}{2\sqrt{-\Delta_{x}(b)}}\\
			&=\tfrac{b^{2x}\log (b)+\big(\tfrac{1-b^2}{2}x-1\big)\Big[2\Omega \sqrt{1-b^{2x}\big(\tfrac{1-b^2}{2}x-1\big)^{-2}}+\tfrac{1-b^2}{2}\Big(\sqrt{1-b^{2x}\big(\tfrac{1-b^2}{2}x-1\big)^{-2}}-1\Big)\Big]}{2\sqrt{-\Delta_{x}(b)}}\cdot
		\end{align*}
		According to \eqref{est dec} we have, for all $b<b^*$ and $x\geqslant\mathtt{m}\geqslant\mathtt{m}^*$, 
		\begin{align}\label{est dec2}
			0<\sqrt{1-(b^*)^{2\mathtt{m}}\Big(\tfrac{1-(b^*)^2}{2}\mathtt{m}-1\Big)^{-2}}\leqslant\sqrt{1-b^{2x}\Big(\tfrac{1-b^2}{2}x-1\Big)^{-2}}<1.
		\end{align} 
		Thus, in view of \eqref{nOme3} 	and \eqref{est dec2} we get
		\begin{align*}
			2\sqrt{-\Delta_{x}(b)} \partial_x h(b,x)
			\geqslant b^{2x}\log (b)+\big(\tfrac{1-b^2}{2}x-1\big)\bigg[&2\Omega \sqrt{1-(b^*)^{2\mathtt{m}}\Big(\tfrac{1-(b^*)^2}{2}\mathtt{m}-1\Big)^{-2}}\\ &+\tfrac{1-(b^*)^2}{2}\Big(\sqrt{1-(b^*)^{2\mathtt{m}}\Big(\tfrac{1-(b^*)^2}{2}\mathtt{m}-1\Big)^{-2}}-1\Big)\bigg].
		\end{align*}
		Setting 
		\begin{equation}\label{expr Omega star}
			\Omega^*_{\mathtt{m}}\triangleq \tfrac{\tfrac{1-(b^*)^2}{2}\Big(1-\sqrt{1-(b^*)^{2\mathtt{m}}\big(\tfrac{1-(b^*)^2}{2}\mathtt{m}-1\big)^{-2}}\Big)-\displaystyle\min_{b\in [0,b^*]}\big(b^{\mathtt{m}}\log (b)\big)}{2 \sqrt{1-(b^*)^{2\mathtt{m}}\big(\tfrac{1-(b^*)^2}{2}\mathtt{m}-1\big)^{-2}}}>0
		\end{equation}
		gives
		\begin{align*}
			\forall \Omega>\Omega^*_{\mathtt{m}}, \quad 2\sqrt{-\Delta_{x}(b)} \partial_x h(b,x)>\big(b^{x}+1-\tfrac{1-b^2}{2}x\big)\min_{b\in [0,b^*]} \big(b^{\mathtt{m}}\log (b)\big)\stackrel{\eqref{def:f}}=f(b,x)\min_{b\in [0,b^*]} \big(b^{\mathtt{m}}\log (b)\big) .
		\end{align*}
		Then by \eqref{deltaj4} we conclude that
		\begin{align*}
			\forall \Omega>\Omega^*_{\mathtt{m}}, \quad \partial_x h(b,x)>0 .
		\end{align*}
		Taking the limit $\mathtt{m}\to\infty$ in \eqref{expr Omega star} gives immediately \eqref{lim omega-star}. This ends the proof of the fourth point.\\
		{\bf 5.} Since $j\mapsto \Omega_{j,k}(b)$, $k\in\{1,2\}$, are odd then it is enough to check the result for $j\in\mathbb{N}^*$. The estimate on the sum $|\Omega_{j,k}(b)+\Omega_{j',k}(b)|$ easily follows from the positivity of the sequences $\big(\Omega_{j,k}(b)\big)_{j\geqslant \mathtt{m}^*}$, $k\in\{1,2\}$, and the third point, namely,
		$$
		\forall b\in[0,b^*],\quad\big|\Omega_{j,k}(b)+\Omega_{j',k}(b)\big|=\Omega_{j,k}(b)+\Omega_{j',k}(b)\geqslant \Omega(j+j').
		$$
		Next we shall prove the estimate on the difference. In view of \eqref{ASYFR1+}, for all $j, j'\geqslant\mathtt{m}^*$ with $j\neq j'$, one has
		\begin{align}\label{omej-jp}
			\Omega_{j,k}(b)- \Omega_{j',k}(b)&=\Big(\Omega+(2-k)\tfrac{1-b^2}{2}\Big)(j-j')+(-1)^{k+1} \big(\mathtt{r}_{j}(b)-\mathtt{r}_{j'}(b)\big). 
		\end{align}
		It follows that
		$$\big|\Omega_{j,k}(b)- \Omega_{j',k}(b)\big|\geqslant\Omega |j-j'|- \sup_{b\in[0,b^*]}\big|\mathtt{r}_{j}(b)-\mathtt{r}_{j'}(b)\big|.$$
		Using Taylor formula combined with   \eqref{ASYFR1-} gives, for any $b\in[0,b^*],$
		\begin{align*}
			\big| \mathtt{r}_{j}(b)-\mathtt{r}_{j'}(b)\big| &\leqslant C_0\Big| \int_{j'}^{j}\frac{dx}{x^{2}}\Big|\\ &\leqslant C_0 \tfrac{|j-j'|}{j j'}\cdot
		\end{align*}
		This implies that
		\begin{align}\label{diff-oj-ojp}
			\big| \Omega_{j,k}(b)- \Omega_{j',k}(b)\big| 
			&\geqslant \Big(\Omega-\tfrac{C_0}{jj'}\Big)|j-j'|.
		\end{align}
		Therefore, there exists $N$ such that if $jj'> N$ the desired inequality holds. For $jj'\leqslant  N$ we shall use the one-to-one property of $j\mapsto \Omega_{j,k}(b)$ combined with the continuity of $b\in[0,b^*]\mapsto \Omega_{j,k}(b)-\Omega_{j^\prime,k}(b)$ to get, for any $j\neq j^\prime\in\llbracket \mathtt{m}^*,N\rrbracket$,
		$$
		\forall \Omega>\Omega^*_{\mathtt{m}},\quad \inf_{b\in[0,b^*]}\big|\Omega_{j,k}(b)-\Omega_{j^\prime,k}(b)\big|\triangleq  c_{jj^\prime}^k> 0.
		$$
		Consequently 
		$$
		\inf_{j\neq j^\prime\in\llbracket \mathtt{m}^*,N\rrbracket\\
			\atop b\in[0,b^*]}\big|\Omega_{j,k}(b)-\Omega_{j^\prime,k}(b)\big|=\inf_{j\neq j^\prime\in\llbracket \mathtt{m}^*,N\rrbracket}c_{jj^\prime}^k>0.
		$$
		Taking 
		$$c\triangleq\frac1N\min\left({\inf_{j\neq j^\prime\in\llbracket \mathtt{m}^*,N\rrbracket}c_{jj^\prime}^k}\,\,\,,\,\,\, N\Omega-{C_0}\right)$$
		and combining the last inequality with \eqref{diff-oj-ojp} we get the desired result.\\
		{\bf 6.} Differentiating \eqref{omej-jp} gives 
		\begin{align*}
			\partial_b\big(\Omega_{j,k}(b)- \Omega_{j',k}(b)\big)&=-(2-k)b(j-j')+(-1)^{k+1} \partial_b\big(\mathtt{r}_{j}(b)-\mathtt{r}_{j'}(b)\big),\\
			\partial_b^2\big( \Omega_{j,k}(b)- \Omega_{j',k}(b)\big)&=-(2-k)(j-j')+(-1)^{k+1}\partial_b^2 \big(\mathtt{r}_{j}(b)-\mathtt{r}_{j'}(b)\big),\\
			\forall q\geqslant 3,\quad 	\partial_b^q\big(  \Omega_{j,k}(b)- \Omega_{j',k}(b)\big)&=(-1)^{k+1} \big(\partial_b^q\mathtt{r}_{j}(b)-\partial_b^q \mathtt{r}_{j'}(b)\big).
		\end{align*}
		By the mean value theorem combined with \eqref{ASYFR1-} we conclude the proof of Lemma \ref{lem-asym}.
	\end{proof}
		\paragraph{Non-degeneracy and transversality.} Through the rest of this section we  shall follow the approach developed in \cite{BBM11,BM18} to  discuss the non-degeneracy and the transversality properties of the linear frequencies. Let us first recall the definition of the non-degeneracy for  vector-valued functions.
	\begin{defin}\label{def-deg} 
		Let  $N\in\mathbb{N}^*$. A function $f \triangleq  (f_1, \ldots , f_N ) : [\alpha_1,\alpha_2] \to \mathbb{R}^N$, with $\alpha_1<\alpha_2$, is called non-degenerate if, for any vector $c \triangleq  (c_1,\ldots,c_N) \in  \mathbb{R}^N \backslash \{0\}$, the scalar function $f \cdot c = f_1c_1 + \cdots+ f_Nc_N$ is not identically zero on the whole interval $[\alpha_1,\alpha_2]$.
	\end{defin}
	We have the following result.
	
	\begin{lem}\label{lem non-deg}
		Let $\Omega>0$ and   $ \mathtt{m}^*,  b^*$ be defined as in Corollary $\ref{coro-equilib-freq}.$
		Fix an integer $\mathtt{m}\geqslant \mathtt{m}^*$ and consider the finite subsets 
		\begin{align*}
			\forall k\in\{1,2\},\quad S_{k}\subset\mathbb{Z}_{\mathtt{m}}\cap\mathbb{N}^*\quad\textnormal{with}\quad |{S}_{k}|<\infty. 
		\end{align*}
		Then the following hold true.
		\begin{enumerate}
			\item If $|S_1\cap S_2|\leqslant 1$ then the  vector valued function 
			$$[0,b^*] \ni b\mapsto	\left(\big(\Omega_{j,1}(b)\big)_{j\in{S}_{1}},\big(\Omega_{j,2}(b)\big)_{j\in{S}_{2}}\right)$$
			is non-degenerate.
			\item If $|S_1\cap S_2|=0$ then the  vector valued functions 
			\begin{align*}
				[0,b^*] \ni b\mapsto	&\left(\big(\Omega_{j,1}(b)\big)_{j\in S_{1}},\big(\Omega_{j,2}(b)\big)_{j\in S_{2}},\mathtt{v}_1(b),\mathtt{v}_2(b)\right),
				\\
				[0,b^*] \ni b\mapsto &\left(\big(\Omega_{j,1}(b)\big)_{j\in S_{1}},\big(\Omega_{j,2}(b)\big)_{j\in S_{2}},\mathtt{v}_k(b)\right), \quad k\in\{1,2\}
			\end{align*}
			are non-degenerate, where the $\mathtt{v}_k$  are defined in \eqref{def V10 V20}.
		\end{enumerate}
	\end{lem}
	
	\begin{proof} We point out that the linear frequencies \eqref{omegajk} are very similar to the linear frequencies close to the Kirchhoff ellipses, studied in \cite{BHM21}. Thus we shall use the same arguments developed in \cite[Lemma 5.2]{BHM21} with slight modifications. According to \eqref{omegajk} the functions  $b\mapsto\Omega_{j,k}(b)$, $k\in\{1,2\}$, are well defined and analytic in a full neighborhood of $b=0$. Moreover, by \eqref{ASYFR1+} and \eqref{rngpm} the frequencies $\Omega_{j,k}(b)$ write 
		\begin{align}
			& \Omega_{j,k}(b) =A_{j,k}(z)+{(-1)^{k+1}} B_j(z)\triangleq \widetilde{\Omega}_{j,k}(z), \label{Taylor-fre}\\
			\nonumber &z\triangleq b^2,\quad
			A_{j,k}(z)\triangleq \Omega j+\frac{2-k}{2}j(1-z)+ \tfrac{(-1)^{k}}{2},\quad
			B_j(z)\triangleq \mathtt{r}_j(b)\underset{z\to0}{=}-\frac{z^{j} }{2(j-2)}+O(z^{j+1}).
		\end{align}
		\textbf{1.} In view of Definition \ref{def-deg} one has to prove that, for all $c\triangleq \big((c_{j,1})_{j\in S_1},(c_{j,2})_{j\in S_2}\big)\in  \mathbb{R}^{|S_1|+|S_2|}\backslash\{0\}$, the function
		\begin{align*}
			z &\mapsto
			\sum_{j\in {S}_1\setminus (S_1\cap S_2)} c_{j,1} \widetilde{\Omega}_{j,1}(z) +\sum_{j\in S_2\setminus (S_1\cap S_2)}c_{j,2} \widetilde{\Omega}_{j,2}(z)+\sum_{j\in S_1\cap S_2} \big(c_{j,1}\widetilde{\Omega}_{j,1}(z)+c_{j,2} \widetilde{\Omega}_{j,2}(z)\big)
		\end{align*}
		is not identically zero on the interval $[0,(b^*)^2]$.
		By contradiction, suppose that there exists $c\triangleq \big((c_{j,1})_{j\in S_1},(c_{j,2})_{j\in S_2}\big)\in\mathbb{R}^{|S_1|+|S_2|}\backslash\{0\}$ such that for any $|z|\leqslant(b^*)^2,$
		\begin{equation}\label{rel lin 1}
			\sum_{j\in S_1\setminus(S_1\cap S_2)} c_{j,1} \widetilde{\Omega}_{j,1}(z)+\sum_{j\in {S}_2\setminus(S_1\cap S_2)} c_{j,2}\widetilde{\Omega}_{j,2}(z)+\sum_{j\in S_1\cap S_2} c_{j,1}\widetilde{\Omega}_{j,1}(z)+c_{j,2}\widetilde{\Omega}_{j,2}(z)=0.
		\end{equation}
		Writing  
		$$
		S_1\cup S_2=\{j_1,j_2,\cdots,j_d\},\qquad \textnormal{with}\qquad \mathtt{m}\leqslant j_1<j_2<\cdots<j_d,
		$$
		then differentiating with respect to $z$ the identity in \eqref{rel lin 1}, we find, since $\mathtt{m}\geqslant3,$
		$$
		\begin{cases}
			\widetilde{c}_{j_1} D_z^{(j_{1})}B_{j_{1}}(z)
			+ \ldots + \widetilde{c}_{j_{d}}D_z^{(j_{1})} B_{j_{d}}(z)= 0, \cr 
			\ldots  \ldots \ldots \cr
			\widetilde{c}_{j_1} D_z^{(j_{d})}B_{j_{1}}(z)
			+ \ldots + \widetilde{c}_{j_{d}}D_z^{(j_{d})} B_{j_{d}}(z)= 0,
		\end{cases}
		$$
		where
		$$\widetilde{c}_{j}\triangleq \begin{cases}
			c_{j,1},& \textnormal{if}\quad j\in {S}_1\setminus ({S}_1\cap S_2),\cr
			-c_{j,2},& \textnormal{if}\quad j\in {S}_2\setminus ({S}_1\cap S_2),\cr
			c_{j,1}-c_{j,2},& \textnormal{if}\quad j\in {S}_1\cap S_2.
		\end{cases}$$
		The latter is a linear system that can be recast as in a matricial form as
		$$\mathcal{M}(z)\widetilde{c}=0,\qquad
		{\mathcal M}(z)\triangleq\begin{pmatrix}
			D_z^{(j_{1})}B_{j_{1}}(z) & \dots & D_z^{(j_{1})}B_{j_{d}}(z)\\
			\vdots & \ddots & \vdots\\
			D_z^{(j_{d})}B_{j_{1}}(z)  & \dots & D_z^{(j_{d})}B_{j_{d}}(z) 
		\end{pmatrix},\qquad\widetilde{c}\triangleq \begin{pmatrix}
			\widetilde{c}_{j_1}\\
			\vdots\\
			\widetilde{c}_{j_d}
		\end{pmatrix}.
		$$ 
		Note, from \eqref{Taylor-fre}, that 
		for all $j\geqslant \mathtt{m}$ we have
		\begin{equation*}
			D^{(j)}_z  B_j(0) = -\frac{ j!}{2(j-2)}\qquad\textnormal{and}\qquad \forall\, 2\leqslant m< j, \quad D^{(m)}_z  B_{j}(0) =0 .
		\end{equation*} 
		It follows that, for some real constants $ \alpha_{i,j} $, we have that
		$$ 
		{\mathcal M}(0)  =
		\begin{pmatrix}
			-\frac{j_{1}!}{2(j_{1}-2)} & 0 & 0 & \dots & 0 \\
			\alpha_{2,1} &-\frac{j_2!}{2(j_2-2)}   & 0 & \dots & 0 \\
			\vdots & \ddots & \ddots & \ddots  & \vdots  \\
			\alpha_{d-1,1} & \ldots & \alpha_{d-1,d-2}  &  -\frac{j_{d-1}!}{2(j_{d-1}-2)} & 0 \\
			\alpha_{d,1} & \ldots & \alpha_{d,d-2} & \alpha_{d,d-1} & -\frac{j_{d}!}{2(j_{d}-2)}  
		\end{pmatrix} 
		$$
		which is a triangular matrix whose determinant  is given by
		$$ 
		\det {\mathcal M}(0) = (-1)^{d}\prod_{i=1}^{d}   \frac{ j_{i}!}{2(j_{i}-2)}  \neq 0  \, . 
		$$
		It follows that $\widetilde{c}=0$, i.e.
		\begin{equation}\label{cj=0}
			\forall j\in ({S}_1\cup {S}_2)\setminus ({S}_1\cap {S}_2), \quad  c_{j,k}=0\qquad \textnormal{and}\qquad \forall j\in {S}_1\cap {S}_2, \quad  c_{j,1}=c_{j,2}.
		\end{equation}
		Inserting \eqref{cj=0} into \eqref{rel lin 1} evaluated at $z=0$, we get from \eqref{Taylor-fre} that
		$$
		\big(2\Omega +\tfrac{1}{2}\big)\sum_{j\in {S}_1\cap S_2} jc_{j,1}=0.
		$$
		Using the fact that $\Omega>0,$ if ${S}_1\cap S_2=\{j_0\}$ then we have 
		$$
		c_{j_0,1}=0.
		$$
		This with \eqref{cj=0} lead to a contradiction proving  the first point.\\
		\textbf{2.} Next, we shall prove that the function 
		$$
		[0,b^*] \ni b\mapsto	\Big(\big(\Omega_{j,1}(b)\big)_{j\in S_{1}},\big(\Omega_{j,2}(b)\big)_{j\in S_{2}},\mathtt{v}_1(b),\mathtt{v}_2(b)\Big)
		$$
		is non-degenerate according to the Definition~\ref{def-deg} provided that $S_1\cap S_2=\varnothing$.
		Suppose, by contradiction, that there exists 
		$$c\triangleq \big((c_{j,1})_{j\in S_1},(c_{j,2})_{j\in S_2},c_{0,1},c_{0,2}\big)\in  \mathbb{R}^{|S_1|+|S_2|+2} \backslash \{0\}$$ such that for any $|z| \leqslant  (b^*)^2,$  
		\begin{equation}\label{ident3}
			c_{0,1}\widetilde{\mathtt{v}}_1(z)+c_{0,2}\widetilde{\mathtt{v}}_2(z)+\sum_{j\in {S}_1}c_{j,1}\widetilde{\Omega}_{j,1}(z) +\sum_{j\in{S}_2}c_{j,2}\widetilde{\Omega}_{j,2}(z)=0, 
		\end{equation}
		where we denote
		$$\forall k\in\{1,2\},\quad\widetilde{\mathtt{v}}_{k}(z)\triangleq \mathtt{v}_{k}(b)=\Omega+\frac{2-k}{2}(1-z).$$
		Arguing in a similar way to the first case we conclude by a differentiation argument that
		$$
		\forall j\in {S}_1\cup {S}_2,\quad\forall k\in \{1,2\},\quad  c_{j,k}=0. 
		$$
		Plugging these identities into \eqref{ident3} we find
		$$
		c_{0,1}\widetilde{\mathtt{v}}_1(z)+c_{0,2}\widetilde{\mathtt{v}}_2(z)=0.
		$$
		That is 
		$$
		\Omega\big(c_{0,1}+c_{0,2}\big)+\frac{1-z}{2}c_{0,1}=0.
		$$
		The last expression being true for any $|z|\leqslant (b^*)^2,$ then using the fact that $\Omega\neq 0$ we infer
		$$
		c_{0,1}=c_{0,2}=0.
		$$
		Thus,  the vector $c$ is vanishing and this contradicts the assumption. The proof of the non-degeneracy of the function 
		$$
		[0,b^*] \ni b\mapsto \left(\big(\Omega_{j,1}(b)\big)_{j\in S_{1}},\big(\Omega_{j,2}(b)\big)_{j\in S_{2}},\mathtt{v}_k(b)\right), \qquad k\in\{1,2\}
		$$
		can be obtained from the previous case by  choosing
		$$
		c\triangleq \big((c_{j,1})_{j\in S_1},(c_{j,2})_{j\in S_2},c_{0,k},c_{0,3-k}\big)=\big((c_{j,1})_{j\in S_1},(c_{j,2})_{j\in S_2},c_{0,k},0\big).
		$$
		This ends the proof of Lemma \ref{lem non-deg}.	
	\end{proof}
	Let $\Omega>0$ and  $ \mathtt{m}^*,  b^*$ be defined as in Corollary \ref{coro-equilib-freq}.
	Fix an integer $\mathtt{m}\geqslant \mathtt{m}^*$ and consider the finite subsets 
	\begin{align}
		\forall k\in\{1,2\},\quad 	\mathbb{S}_{k}\subset\mathbb{Z}_{\mathtt{m}}\cap\mathbb{N}^*,\qquad\textnormal{with}\qquad d_k\triangleq |\mathbb{S}_{k}|<\infty \qquad\textnormal{and}\qquad \mathbb{S}_{1}\cap \mathbb{S}_{2}=\varnothing.\label{S+}
	\end{align}
	For all $b\in [0,b^*]$ define the tangential equilibrium frequency vector by
	\begin{equation}\label{Eq freq vec Edc}
		\omega_{\textnormal{Eq}}(b)\triangleq \big(\omega_{\textnormal{Eq},1}(b),\omega_{\textnormal{Eq},2}(b)\big)\in\mathbb{R}^{d},\quad\textnormal{with}\quad \omega_{\textnormal{Eq},k}(b)\triangleq \big(\Omega_{j,k}(b)\big)_{j\in\mathbb{S}_{k}}\in\mathbb{R}^{d_{k}},\quad d\triangleq d_1+d_2 
	\end{equation}
	and set
	$$\mathbb{S}\triangleq \mathbb{S}_{1}\cup\mathbb{S}_{2},\quad  
		\overline{\mathbb{S}}\triangleq \mathbb{S}\cup(-\mathbb{S}), \quad \overline{\mathbb{S}}_0\triangleq \overline{\mathbb{S}}\cup \{0\},\quad \overline{\mathbb{S}}_k=\mathbb{S}_k \cup(-\mathbb{S}_k)\quad\textnormal{and}\quad\overline{\mathbb{S}}_{0,k}=\overline{\mathbb{S}}_k\cup\{0\}.$$
	In the next proposition we deduce some quantitative bounds from the qualitative non-degeneracy condition of Lemma \ref{lem non-deg}, the analyticity of the linear frequencies and their asymptotics.
	
	\begin{lem}{\textnormal{[Transversality]}}\label{lemma transversalityE}
		There exist $q_0\in\mathbb{N}$ and $\rho_{0}>0$ such that the following results hold true. Recall that $\mathtt{v}_k(b)$, $\Omega_{j,k}$ and  $\omega_{\textnormal{Eq}}$  are defined in \eqref{def V10 V20},  \eqref{omegajk} and  \eqref{Eq freq vec Edc} and respectively.
		\begin{enumerate}
			\item For any $l\in\mathbb{Z}^{d}\setminus\{0\},$ we have
			$$
			\inf_{b\in[0,b^{*}]}\max_{q\in\llbracket 0, q_{0}\rrbracket}\Big|\partial_{b}^{q}\omega_{\textnormal{Eq}}(b)\cdot l\Big|\geqslant\rho_{0}\langle l\rangle.
			$$
			\item For any $k\in\{1,2\}$ and $ (l,j)\in(\mathbb{Z}^{d}\times\mathbb{N}_{\mathtt{m}})\setminus\{(0,0)\}$ 
			$$
			\quad\inf_{b\in[0,b^{*}]}\max_{q\in\llbracket 0, q_{0}\rrbracket}\Big|\partial_{b}^{q}\big(\omega_{\textnormal{Eq}}(b)\cdot l+ j\mathtt{v}_k(b)\big)\Big|\geqslant\rho_{0}\langle l\rangle.
			$$
			\item  For any $k\in\{1,2\}$ and $ (l,j)\in\mathbb{Z}^{d}\times (\mathbb{N}_{\mathtt{m}}^*\setminus\mathbb{S}_{k})$ 
			$$
			\quad\inf_{b\in[0,b^{*}]}\max_{q\in\llbracket 0, q_{0}\rrbracket}\Big|\partial_{b}^{q}\big(\omega_{\textnormal{Eq}}(b)\cdot l+\Omega_{j,k}(b)\big)\Big|\geqslant\rho_{0}\langle l\rangle.
			$$
			\item We assume the additional constraint $\Omega>\Omega_{\mathtt{m}}^*$, see Lemma $\ref{lem-asym}$-4-5. For any $k\in\{1,2\}$ and $ l\in\mathbb{Z}^{d}, j,j^\prime\in\mathbb{N}_{\mathtt{m}}^*\setminus\mathbb{S}_k,$ satisfying the additional condition $(l,j)\neq(0,j^\prime)$, we have
			$$\,\quad\inf_{b\in[0,b^*]}\max_{q\in\llbracket 0, q_{0}\rrbracket}\Big|\partial_{b}^{q}\big(\omega_{\textnormal{Eq}}(b)\cdot l+\Omega_{j,k}(b)\pm\Omega_{j^\prime,k}(b)\big)\Big|\geqslant\rho_{0}\langle l\rangle.$$
			
			\item For any  $ l\in\mathbb{Z}^{d}, j\in\mathbb{N}^*\setminus\mathbb{S}_1,j^\prime\in\mathbb{N}^*\setminus\mathbb{S}_2,$  we have
			$$\,\quad\inf_{b\in[0,b^{*}]}\max_{q\in\llbracket 0, q_{0}\rrbracket}\Big|\partial_{b}^{q}\big(\omega_{\textnormal{Eq}}(b)\cdot l+\Omega_{j,1}(b)\pm\Omega_{j^\prime,2}(b)\big)\Big|\geqslant\rho_{0}\langle l,j,j^\prime\rangle.$$	
		\end{enumerate}
	\end{lem}
	\begin{proof}
		${\bf{1.}}$
		Suppose, by contradiction, that 
		for all $n\in\N$ there exist  $b_n\in [0,b^{*}]$ and $l_n\in\Z^d\backslash \{0\}$ such that 
		\begin{equation}\label{alphan}
			\max_{q\in\llbracket 0, n\rrbracket}\Big|\partial_b^q\Big({\omega}_{\textnormal{Eq}}(b)\cdot \tfrac{l_n}{\langle l_n\rangle}\Big)_{|{b=b_n}}\Big|< \tfrac{1}{n+1}.\cdot
		\end{equation}
		The sequences $(b_n)_n\subset [0,b^{*}]$ and $(c_n)_n\triangleq \big(\frac{l_n}{\langle l_n\rangle}\big)_n\subset \mathbb{R}^d\backslash\{0\}$ are bounded. Up to an extraction we may assume that 
		$$\lim_{n\to\infty}\tfrac{l_{n}}{\langle l_n\rangle}=\widetilde{c}\neq 0\qquad\hbox{and}\qquad \lim_{n\to\infty}b_{n}=\widetilde{b}.
		$$
		Taking to the limit in \eqref{alphan} for $n\to\infty$ we deduce that 
		$$\forall q \in \N,\quad \partial_b^q\big({\omega}_{\textnormal{Eq}}( b)\cdot \widetilde{c}\big)_{|{b=\widetilde b}}=0,\qquad{\rm with}\qquad \widetilde{c}\neq 0.$$  Therefore, the real analytic  function $b  \to {\omega}_{\textnormal{Eq}}(b)\cdot \widetilde{c}\,$ is identically zero. This contradicts Lemma~\ref{lem non-deg}.\\
		${\bf{2.}}$ In the case $l=0$ and $j\in\mathbb{N}^*_{\mathtt{m}}$ we obviously have from \eqref{def V10 V20}, 		
		\begin{align*}
			\quad\inf_{b\in[0,b^{*}]}\max_{q\in\llbracket 0, q_{0}\rrbracket}\big|\partial_{{b}}^{q}\big( j\mathtt{v}_k(b)\big)\big|&\geqslant \inf_{b\in[0,b^{*}]}\big| \mathtt{v}_k(b)\big|\geqslant \Omega\geqslant\rho_{0}\langle l\rangle,
		\end{align*}
		for some $\rho_0>0.$ Next, we shall consider the case  $j\in\mathbb{N}^*_{\mathtt{m}}$, $l\in\mathbb{Z}^d\setminus\{0\}$.
		By the triangle inequality combined with the boundedness of  ${\omega}_{\textnormal{Eq}}$ and $\mathtt{v}_k(b)$   we get
		$$\big|\omega_{\textnormal{Eq}}({b})\cdot l+j\mathtt{v}_k(b)\big|\geqslant|j|\big|\mathtt{v}_k(b)\big|-\big|\omega_{\textnormal{Eq}}({b})\cdot l\big|\geqslant c|j|-C|l|\geqslant |l|$$
		provided that  $|j|\geqslant C_{0}|l|$ for some $C_{0}>0.$ Hence, we shall only consider   indices  $j$ and $l$ satisfying
		\begin{equation}\label{parameter condition 10}
			|j|\leqslant C_{0}|l|, \qquad j\in\mathbb{N}_{\mathtt{m}}, \qquad l\in\mathbb{Z}^d\setminus\{0\}.
		\end{equation}
		By contradiction,  assume the existence of sequences $\{l_{n}\}\subset \Z^d\backslash\{0\}$, $\{j_n \}\subset {\mathbb{N}}_{\mathtt{m}}$ satisfying \eqref{parameter condition 10} and $\{{b}_{n}\}\subset[0,b^*]$ such that 
		\begin{equation}\label{Ross-01}
			\forall q\in\mathbb{N},\quad\forall n\geqslant q,\quad\left|\partial_{{b}}^{q}\left({\omega}_{\textnormal{Eq}}({b})\cdot\tfrac{l_{n}}{\langle l_{n}\rangle}+\tfrac{{j_{n}}\mathtt{v}_k(b)}{\langle l_{n}\rangle}\right)_{|{b=b_n}}\right|<\tfrac{1}{1+n}\cdot
		\end{equation}
		The sequences $\{{b}_n\}$, $\{d_n\}\triangleq \big\{\frac{j_n}{\langle l_n\rangle}\big\}$  and $\{c_n\}\triangleq \big\{\frac{l_n}{\langle l_n\rangle}\big\}$ are bounded. Thus,  up to an extraction, we may assume that
		$$
		\lim_{n\to\infty}{b}_{n}=\widetilde{{b}},\qquad \lim_{n\to\infty}d_{n}=\widetilde{d}\geqslant 0\qquad\hbox{and}\qquad \lim_{n\to\infty}c_{n}=\widetilde{c}\neq 0.
		$$
		Hence, letting  $n\rightarrow+\infty$ in \eqref{Ross-01} and  using the fact that  ${b}\mapsto \mathtt{v}_k(b)$ is smooth   we obtain
		$$\forall q\in\mathbb{N},\quad\partial_{{b}}^{q}\left({\omega}_{\textnormal{Eq}}({{b}})\cdot\widetilde{ c}+{\widetilde{d}}\,\mathtt{v}_k(b)\right)_{|{b}=\widetilde{{b}}}=0.$$
		Consequently, the real analytic  function ${b}\mapsto {\omega}_{\textnormal{Eq}}({{b}})\cdot\widetilde{ c}+{\widetilde{d}}\,\mathtt{v}_k(b)$ with $(\widetilde{ c},{\widetilde{d}})\neq (0,0)$ is identically zero
		and this is in contradiction with Lemma  \ref{lem non-deg}.
		\\
		${\bf{3.}}$ Let $k\in\{1,2\}$ and consider $(l,j)\in\mathbb{Z}^{d }\times (\mathbb{N}_{\mathtt{m}}^*\setminus\mathbb{S}_k)$. By  the  triangle inequality and Lemma \ref{lem-asym}-${\rm{3}}$, we get
		$$\big|\omega_{\textnormal{Eq}}(b)\cdot l+\Omega_{j,k}(b)\big|\geqslant\big|\Omega_{j}(b)\big|-\big|\omega_{\textnormal{Eq}}(b)\cdot l\big|\geqslant \Omega j-C|l|\geqslant \tfrac{\Omega}{2}\langle l\rangle$$
		provided that  $j\geqslant C_{0}| l|$ for some $C_{0}>0.$ Therefore, we shall restrict the proof to integers  $j$ with 
		\begin{equation}\label{parameter condition 1}
			0\leqslant j< C_{0} | l|,\qquad j\in\mathbb{N}_{\mathtt{m}}^*\setminus\mathbb{S}_k\qquad\hbox{and}\qquad l\in\mathbb{Z}^d\backslash\{0\}.
		\end{equation}
		By contradiction,  for all $n\in\mathbb{N}$, we assume  the existence of sequences  $\{l_{n}\}\subset \Z^d\backslash\{0\}, \{j_n\} \subset\mathbb{N}_{\mathtt{m}}^*\setminus\mathbb{S}_k$ and $\{b_{n}\}\subset [0,{b}^*]$ such that 
		\begin{equation}\label{Ross-1}
			\forall q\in\mathbb{N},\quad\forall n\geqslant q,\quad\Big|\partial_{b}^{q}\Big({\omega}_{\textnormal{Eq}}(b)\cdot\tfrac{l_{n}}{\langle l_{n}\rangle}+\tfrac{\Omega_{j_{n},k}(b)}{\langle l_{n}\rangle}\Big)_{|b={b}_n}\Big|<\tfrac{1}{1+n}\cdot
		\end{equation}
		Since the sequences $\{b_n\}$  and $\{c_n\}\triangleq \big\{\frac{l_n}{\langle l_n\rangle}\big\}$ are bounded,  then  by compactness  we can assume that
		$$
		\lim_{n\to\infty}b_{n}=\widetilde{b}\qquad\hbox{and}\qquad \lim_{n\to\infty}c_{n}=\widetilde{c}\neq 0.
		$$
		We shall distinguish two cases.\\
		$\bullet$ {\it Case $1$}: $\{l_{n}\}$ is bounded. From \eqref{parameter condition 1} and up to an extraction the sequences $\{l_{n}\}$ and $\{j_{n}\}$ are stationary. Thus, we can assume that for any $n\in\N$, we have $l_n=\widetilde{l}\in  \Z^d\backslash\{0\}$ and $j_n=\widetilde{\jmath}\in \mathbb{N}_{\mathtt{m}}^*\setminus\mathbb{S}_k$. 
		Taking the limit as $n\rightarrow+\infty$ in \eqref{Ross-1} yields
		$$\forall q\in\mathbb{N},\quad\partial_{b}^{q}\left({\omega}_{\textnormal{Eq}}({b})\cdot\widetilde{l}+\Omega_{\widetilde{\jmath},k}({b})\right)_{|_{b=\widetilde{b}}}=0.$$
		Consequently, the real analytic function $b\mapsto\omega_{\textnormal{Eq}}(b)\cdot\widetilde{l}+\Omega_{\widetilde{\jmath},k}(b)$ with $(\widetilde{l},1)\neq (0,0)$ is identically zero and this contradicts
		Lemma \ref{lem non-deg}.\\
		$\bullet$ {\it Case $2$}:  $\{l_{n}\}$ is unbounded. Up to a subsequence, we assume that $ \displaystyle\lim_{n\to\infty} | l_{n}|=\infty$ and 
		$\displaystyle \lim_{n\to\infty}\frac{ l_{n}}{\langle  l_{n}\rangle} =\widetilde{c}\in \R^d\backslash\{0\}.$
		We shall distinguish two sub-cases.\\
		$\bullet$ Sub-case \ding{172}. The sequence $\{j_n\}$ is bounded.
		Up to an extraction we may assume that this sequence of integers  is stationary.  Taking the limit $n\rightarrow+\infty$ in \eqref{Ross-1}, we get
		$$\forall q\in\mathbb{N},\quad \partial_{b}^{q}{\omega}_{\textnormal{Eq}}({b})_{|_{b=\widetilde{b}}}\cdot\widetilde{c}=0.$$
		Thus, the real analytic function $b\mapsto{\omega}_{\textnormal{Eq}}(b)\cdot\widetilde{c}$, with $\widetilde{c}\neq 0$,   is identically zero  and this is a contradiction with the Lemma  \ref{lem non-deg}.\\
		$\bullet$ Sub-case \ding{173}. The sequence $\{j_{n}\}$ is unbounded. Then up to an extraction we can assume that   $\displaystyle \lim_{n\to\infty} j_{n}=\infty$. According to \eqref{ASYFR1+} we have
		\begin{equation}\label{quo1}
			\tfrac{\Omega_{j_{n},k}(b)}{\langle  l_{n}\rangle}=\tfrac{j_n}{\langle  l_{n}\rangle}\Big(\Omega+(2-k)\tfrac{1-b^2}{2}\Big)+ \tfrac{(-1)^k}{2\langle  l_{n}\rangle}+(-1)^{k+1} \tfrac{\mathtt{r}_{j_n}(b)}{\langle  l_{n}\rangle}\cdot
		\end{equation}
		By \eqref{parameter condition 1}, the sequence $\Big\{\frac{j_{n}}{\langle  l_{n}\rangle}\Big\}$ is bounded. Up to a subsequence,  it converges to $\widetilde{d}.$ Differentiating then taking the limit in \eqref{quo1} we obtain 
		$$\lim_{n\to+\infty}\tfrac{\partial_{b}^q\Omega_{j_{n},k}(b_{n})}{\langle  l_{n}\rangle}=\partial_{b}^q\big(\widetilde{d}\,\mathtt{v}_k(b)\big)_{|_{b=\widetilde{b}}},
		$$
		having used in the last identity the estimate \eqref{ASYFR1-}.
		Hence, taking the limit $j\rightarrow+\infty$ in \eqref{Ross-1} gives
		$$\forall q\in\mathbb{N},\quad \partial_{b}^{q}\left({\omega}_{\textnormal{Eq}}({b})\cdot\widetilde{c}+\widetilde{d}\,\mathtt{v}_k(b)\right)_{|_{b=\widetilde{b}}}=0.$$
		Thus, the real analytic function $b\mapsto{\omega}_{\textnormal{Eq}}(b)\cdot\widetilde{c}+\widetilde{d}\mathtt{v}_k(b)$  is identically zero. This contradicts Lemma~\ref{lem non-deg} as  $(\widetilde{c},\widetilde{d})\neq 0$.
		\\
		\textbf{4.} Let $ l\in\mathbb{Z}^{d }, j,j^\prime\in \mathbb{N}_{\mathtt{m}}^*\setminus\mathbb{S}_k$  with $(l,j)\neq(0,j^\prime).$ By the  triangle inequality and  Lemma \ref{lem-asym}-${5}$, since $\Omega>\Omega_{\mathtt{m}}^*,$ we infer that
		$$\big|\omega_{\textnormal{Eq}}(b)\cdot  l+\Omega_{j,k}(b)\pm\Omega_{j',k}(b)\big|\geqslant\big|\Omega_{j,k}(b)\pm\Omega_{j',k}(b)\big|-\big|\omega_{\textnormal{Eq}}(b)\cdot  l\big|\geqslant c|j\pm j'|-C|l|\geqslant \langle l\rangle$$
		provided $|j-j'|\geqslant C_{0}| l|$ for some $C_{0}>0.$ In this case the desired estimate is trivial. So we shall restrict the proof to integers  such that
		\begin{equation}\label{parameter condition 2}
			|j\pm j'|< C_{0}\langle l\rangle,\qquad  l\in\mathbb{Z}^{d }\backslash\{0\},\qquad  j,j^\prime\in\mathbb{N}_{\mathtt{m}}^*\setminus\mathbb{S}_k.
		\end{equation}
		Arguing by contradiction,  assume that for all $n\in\mathbb{N}$, there exists $( l_n,j_n)\neq(0,j_n^\prime)\in \Z^{d+1}$ satisfying \eqref{parameter condition 2} and $b_{n}\in[0,b^{*}]$ such that 
		\begin{equation}\label{Ross-2}
			\forall q\in\mathbb{N},\quad\forall n\geqslant q,\quad\left|\partial_{b}^{q}\left({\omega}_{\textnormal{Eq}}(b)\cdot\tfrac{ l_{n}}{\langle  l_{n}\rangle}+\tfrac{\Omega_{j_{n},k}(b)\pm\Omega_{j'_{n},k}(b)}{\langle  l_{n}\rangle}\right)_{|_{b=b_n}}\right|<\frac{1}{1+n}\cdot
		\end{equation}
		Since the sequences $\left\{\frac{ l_{n}}{\langle  l_{n}\rangle}\right\}_{n}$ and $\{b_{n}\}_{n}$ are bounded, then up to an extraction we can assume that $\displaystyle \lim_{n\to\infty}\frac{ l_{n}}{\langle  l_{n}\rangle}=\widetilde{c}\neq 0$ and $\displaystyle \lim_{n\to\infty}b_{n}=\widetilde{b}.$ We distinguish two cases :\\
		$\bullet$ {\it Case $1$}: $( l_{n})_{n}$ is bounded. We shall only focus on the most delicate case associated to the difference $\Omega_{j_n,k}-\Omega_{j^\prime_n,k}$. Up to an extraction we may assume  that this sequence of integers  is stationary, that is, $ l_{n}=\widetilde l.$ Looking at \eqref{parameter condition 2} we have two sub-cases.
		\\
		$\bullet$ Sub-case \ding{172} : $(j_{n})_{n}$ and $(j'_{n})_{n}$ are bounded. Up to an extraction we can assume that they are stationary, that is, $j_n=\widetilde{\jmath}, j_n^\prime=\widetilde{\jmath}^\prime$. Moreover, by assumption we also have $(\widetilde l, \widetilde{\jmath})\neq(0,\widetilde{\jmath}^\prime)$ and $\widetilde{\jmath},\widetilde{\jmath}^\prime\notin\mathbb{S}_k$.
		Hence taking the limit $n\rightarrow+\infty$ in \eqref{Ross-2}, we get
		$$\forall q\in\mathbb{N},\quad\partial_{b}^{q}\left({\omega}_{\textnormal{Eq}}({b})\cdot\widetilde{ l}+\Omega_{\widetilde{\jmath},k}({b})-\Omega_{\widetilde{\jmath}^\prime,k}({b})\right)_{|_{b=\widetilde{b}}}=0.$$
		Therefore, the real analytic function $b\mapsto\omega_{\textnormal{Eq}}(b)\cdot\widetilde{ l}+\Omega_{\widetilde{\jmath},k}(b)-\Omega_{\widetilde{\jmath}^\prime,k}(b)$  is identically zero.  If $\widetilde{\jmath}=\widetilde{j^\prime}$ then this  contradicts Lemma \ref{lem non-deg} since $\widetilde{l}\neq 0.$ In the case  $\widetilde{\jmath}\neq \widetilde{j^\prime}\in \mathbb{N}_{\mathtt{m}}^*\setminus\mathbb{S}_k$ this still contradicts Lemma~\ref{lem non-deg}, applied with the vector frequency $(\omega_{\textnormal{Eq}},\Omega_{\widetilde{\jmath},k},\Omega_{\widetilde{j'},k})$. 
		\\
		$\bullet$ Sub-case \ding{173} :   $(j_n)_{n}$ and $(j'_{n})_{n}$ are unbounded. Up to an extraction, we  assume that  $\displaystyle \lim_{n\to\infty}j_{n}= \lim_{m\to\infty}j'_{n}=\infty$. 
		Assume, without loss of generality,  that for a given $n$ we have $j_n\geqslant j_n^\prime$. In view of  \eqref{ASYFR1+} we may write
		\begin{align}\label{split}
			\tfrac{\partial_{b}^{q}\left(\Omega_{j_{n},k}(b)-\Omega_{j'_{n},k}(b)\right)}{\langle  l_{n}\rangle}
			&=\partial_{b}^{q}\mathtt{v}_k(b)\tfrac{j_n-j_n^\prime}{\langle  l_n\rangle}
			+\tfrac{(-1)^{k+1}}{\langle  l_n\rangle}\partial_{b}^{k}\big(\mathtt{r}_{j_n}(b)-\mathtt{r}_{j_n^\prime}(b)\big).
		\end{align}
		According to \eqref{parameter condition 2},   up to an extraction, we can assume that $\displaystyle\lim_{n\to\infty}\frac{j'_{n}-j_{n}}{\langle l_{n}\rangle}=\widetilde{d}$. Therefore, combining \eqref{split} and \eqref{ASYFR1-}, we find
		$$\lim_{n\to\infty}\partial_{b}^{q}\left(\frac{\Omega_{j_{n},k}(b)-\Omega_{j'_{n},k}(b)}{\langle  l_{n}\rangle}\right)_{|_{b=b_n}}=\widetilde{d}\,\partial_{b}^{q}\big(\mathtt{v}_k(b)\big)_{|_{b=\widetilde{b}}}.$$
		Taking the limit $n\rightarrow+\infty$ in \eqref{Ross-2} gives
		$$\forall q\in\mathbb{N},\quad\partial_{b}^{q}\left({\omega}_{\textnormal{Eq}}({b})\cdot\widetilde{c}+\widetilde{d}\,\mathtt{v}_k(b)\right)_{|_{b=\widetilde{b}}}=0.$$
		Then, the real analytic function $b\mapsto{\omega}_{\textnormal{Eq}}({b})\cdot\widetilde{c}+\widetilde{d}\,\mathtt{v}_k(b)$ with $(\widetilde{c},\widetilde{d})\neq (0,0)$ is identically zero. This contradicts  Lemma \ref{lem non-deg}.
		\\
		$\bullet$ {\it Case} $2$: $( l_{n})_{n}$ is unbounded. Up to an extraction  we can assume that $\displaystyle \lim_{n\to\infty}|l_{n}|=\infty.$ We shall distinguish three sub-cases.\\
		$\bullet$ Sub-case \ding{172}. The sequences  $(j_{n})_{n}$ and $(j'_{n})_{n}$ are bounded. Thus,  up to an extraction  they will  converge. Taking the limit in \eqref{Ross-2} leads to
		$$
		\forall q\in\mathbb{N},\quad\partial_{b}^{q}{\omega}_{\textnormal{Eq}}(\bar{b})\cdot\widetilde{c}=0.
		$$
		which gives a contradiction with Lemma \ref{lem non-deg}. \\
		$\bullet$ Sub-case \ding{173}. The sequences  $(j_{n})_{n}$ and $(j'_{n})_{n}$ are both unbounded. This case is similar to  the sub-case \ding{173} of the case 1.\\
		$\bullet$ Sub-case \ding{174}. The sequence $(j_{n})_{n}$ is unbounded and $(j'_{n})_{n}$ is bounded.  Without loss of generality, we can assume that $\displaystyle \lim_{n\to\infty}j_n=\infty$ and $  j_{n}^\prime=\widetilde{\jmath}.$ By \eqref{parameter condition 2} and  up to an extraction one gets  $\displaystyle \lim_{n\to\infty}\frac{j_{n}\pm j'_{n}}{| l_{n}|}=\widetilde{d}.$
		Using  Taylor formula combined with   \eqref{ASYFR1-} gives for any $b\in[0,b^*],$
		\begin{align}\label{estm:dif22}
			\big| \partial_{b}^{q}\mathtt{r}_{j_n^\prime}(b)- \partial_{b}^{q}\mathtt{r}_{j_n}(b)\big| &\leqslant C\Big| \int_{j'_n}^{j_n}\frac{dx}{x^{2}}\Big|\nonumber\\ &\leqslant C {|j_n-j_n'|}{(j_n j_n^\prime)^{-1}}\cdot
		\end{align}
		Using \eqref{ASYFR1+}
		combined with  \eqref{estm:dif22} and \eqref{ASYFR1-} we get, for any $q\in\mathbb{N},$
		\begin{align*}
			\lim_{n\to\infty}\langle  l_{n}\rangle^{-1}
			\partial_b^q\Big(\Omega_{j_n,k}(b)\pm\Omega_{j_{n}^\prime,k}(b)-(j_n\pm j^\prime_{n})\mathtt{v}_k(b)\Big)_{|b=b_n}&=\\
			(-1)^{k}\ \lim_{n\to\infty}
			\partial_b^q\left(\tfrac{1\pm 1}{2\langle  l_n\rangle}-\tfrac{\mathtt{r}_{j_n}(b)\pm \mathtt{r}_{j_n^\prime}(b)}{\langle  l_n\rangle}\right)_{|b=b_n}&=0.
		\end{align*}
		Hence, taking the limit  in \eqref{Ross-2} implies 
		$$\forall q\in\mathbb{N},\quad \partial_{b}^{q}\left({\omega}_{\textnormal{Eq}}({b})\cdot\widetilde{c}+\widetilde{d}\mathtt{v}_k(b)\right)_{b=\widetilde{b}}=0.$$
		Thus, the real analytic function ${b}\mapsto{\omega}_{\textnormal{Eq}}({b})\cdot\widetilde{c}+\widetilde{d}\mathtt{v}_k(b)$ is identically zero with $(\widetilde{c},\widetilde{d})\neq0$ leading to a contradiction with Lemma \ref{lem non-deg}. \\
		\textbf{5.} 
		Arguing by contradiction, suppose  that for all $n\in\mathbb{N}$, there exist $b_n\in[0,b^*]$ and  $( l_n,j_n,j_n^\prime)\in \Z^{d+2}\setminus\{0\}$, with $j_n\in (\mathbb{N}^*\cap \mathbb{Z}_{\mathtt{m}})\setminus\mathbb{S}_1,$ and $j^\prime_n\in (\mathbb{N}^*\cap \mathbb{Z}_{\mathtt{m}})\setminus\mathbb{S}_2$, such that 
		$$
		\max_{q\in\llbracket 0, n\rrbracket}\left|\partial_{b}^{q}\left({\omega}_{\textnormal{Eq}}(b)\cdot\tfrac{ l_{n}}{\langle  l_{n},j_n,j_n^\prime\rangle}+\tfrac{\Omega_{j_{n},1}(b)\pm\Omega_{j'_{n},2}(b)}{\langle  l_{n},j_n,j_n^\prime\rangle}\right)_{|_{b=b_n}}\right|<\frac{1}{1+n}$$ 
		and therefore
		\begin{equation}\label{Rossemann 2-dif0}
			\forall q\in\mathbb{N},\quad\forall n\geqslant q,\quad\left|\partial_{b}^{q}\left({\omega}_{\textnormal{Eq}}(b)\cdot\tfrac{ l_{n}}{\langle  l_{n},j_n,j_n^\prime\rangle}+\tfrac{\Omega_{j_{n},1}(b)\pm\Omega_{j'_{n},2}(b)}{\langle  l_{n},j_n,j_n^\prime\rangle}\right)_{|_{b=b_n}}\right|<\frac{1}{1+n}\cdot
		\end{equation}
		The sequence $(b_n)_n\subset [0,b^{*}]$ 		is bounded. Up to an extraction we may assume that 
		$$
				\lim_{n\to\infty}b_{n}=\widetilde{b}\in [0,b^{*}].
		$$
		We distinguish two cases.\\
		$\bullet$ {\it Case $1$}:  The sequence $\{\langle  l_{n},j_n,j_n^\prime\rangle\}_n$ is bounded. Then up to an extraction we may assume that 
		$$\lim_{n\to\infty}l_n=\widetilde{c}\in \mathbb{Z}^d,\qquad \lim_{n\to\infty}j_n=\widetilde{\jmath} \in(\mathbb{N}^*\cap \mathbb{Z}_{\mathtt{m}})\setminus\mathbb{S}_1\qquad\hbox{and}\qquad \lim_{n\to\infty}j'_n=\widetilde{\jmath}^\prime \in(\mathbb{N}^*\cap \mathbb{Z}_{\mathtt{m}})\setminus\mathbb{S}_2 .
		$$
		Taking the limit in \eqref{Rossemann 2-dif0} we find
		$$\forall q\in\mathbb{N},\quad \partial_{b}^{q}\left({\omega}_{\textnormal{Eq}}(b)\cdot\widetilde{c}+\Omega_{\widetilde{\jmath},1}(b)\pm\Omega_{\widetilde{\jmath}^\prime,2}(b)\right)_{b=\widetilde{b}}=0.$$
		Thus, the real analytic function ${b}\mapsto {\omega}_{\textnormal{Eq}}(b)\cdot\widetilde{c}+\Omega_{\widetilde{\jmath},1}(b)\pm\Omega_{\widetilde{\jmath}^\prime,2}(b)$ is identically zero on the interval $[0,b^*]$. This contradicts Lemma \ref{lem non-deg} if one of the following holds: $$\widetilde{\jmath} \not\in \mathbb{S}_2\qquad\hbox{and}\qquad \widetilde{\jmath}^\prime \not\in \mathbb{S}_1,$$ or 
		$$\widetilde{\jmath} \in \mathbb{S}_2\qquad\hbox{and}\qquad \widetilde{\jmath}^\prime \not\in \mathbb{S}_1,$$
		or
		$$\widetilde{\jmath} \not\in \mathbb{S}_2\qquad\hbox{and}\qquad \widetilde{\jmath}^\prime \in \mathbb{S}_1.$$ Thus, it remain to check the case where
		\begin{equation}\label{last-case}
			\widetilde{\jmath} \in \mathbb{S}_2\qquad\hbox{and}\qquad \widetilde{\jmath}^\prime \in \mathbb{S}_1. 
		\end{equation}
		Denoting $\widetilde{c}=\Big(\big(\widetilde{c}_{j,1}\big)_{j\in\mathbb{S}_1},\big(\widetilde{c}_{j,1}\big)_{j\in\mathbb{S}_2}\Big),$ then we have for any $z\in [0,(b^*)^2],$
		\begin{equation}\label{comb-aj2}
			\sum_{j\in \mathbb{S}_1\setminus\{\widetilde{\jmath}^\prime\}} \widetilde{c}_{j,1}\widetilde{\Omega}_{j,1}(z) +\sum_{j\in \mathbb{S}_2\setminus\{\widetilde{\jmath}\}} \widetilde{c}_{j,2} \widetilde{\Omega}_{j,2}(z)+\widetilde{c}_{\widetilde{\jmath}^\prime,1} \widetilde{\Omega}_{\widetilde{\jmath}^\prime,1}(z)\pm\widetilde{\Omega}_{\widetilde{\jmath}^\prime,2}(z) +\widetilde{c}_{\widetilde{\jmath},2} \widetilde{\Omega}_{\widetilde{\jmath},2}(z)+ \widetilde{\Omega}_{\widetilde{\jmath},1}(z)=0. 
		\end{equation}
		Arguing as in the proof of Lemma \ref{lem non-deg} we conclude by a differentiation argument that  
		$$\forall j\in (\mathbb{S}_1\cup \mathbb{S}_2)\setminus \{\widetilde{\jmath}^\prime, \widetilde{\jmath}\}, \quad \forall k\in\{1,2\},\quad c_{j,k}=0, \qquad c_{\widetilde{\jmath},2}=1\qquad \textnormal{and}\qquad   c_{\widetilde{\jmath}^\prime,1}=\pm 1.$$
		Substituting these identities into \eqref{comb-aj2} evaluated at $z=0$ and using \eqref{Taylor-fre} we get 
		$$\big(2\Omega +\tfrac{1}{2}\big)(\widetilde{\jmath}^\prime\pm \widetilde{\jmath})=0.$$
		This implies that $\widetilde{\jmath}^\prime=\mp \widetilde{\jmath} $ contradicting \eqref{last-case} and \eqref{S+}.\\			
		$\bullet$ {\it Case $2$}: The sequence $\{\langle  l_{n},j_n,j_n^\prime\rangle\}_n$ is unbounded. 
		Using \eqref{ASYFR1+} we may write 		
		\begin{equation}\label{Rossemann 2-dif}
			\forall q\in\mathbb{N},\quad\forall n\geqslant q,\quad\left|\partial_{b}^{q}\left(\big({\omega}_{\textnormal{Eq}}(b),\Omega+\tfrac{1-b^2}{2},\pm\Omega\big)\cdot\tfrac{ (l_{n},j_n,j_n^\prime)}{\langle  l_{n},j_n,j_n^\prime\rangle}+\tfrac{- \tfrac{1}{2}\pm  \tfrac{1}{2}+ \mathtt{r}_{j_n}(b)\mp  \mathtt{r}_{j_n^\prime}(b)}{\langle  l_{n},j_n,j_n^\prime\rangle}\right)_{|_{b=b_n}}\right|<\frac{1}{1+n}\cdot
		\end{equation}	
		The sequence $(c_n)_n\triangleq \big(\tfrac{ (l_{n},j_n,j_n^\prime)}{\langle  l_{n},j_n,j_n^\prime\rangle}\big)_n\subset \mathbb{R}^d\backslash\{0\}$ is bounded. 
		By compactness and up to an extraction we may assume that 
		$$\lim_{n\to\infty}\tfrac{ (l_{n},j_n,j_n^\prime)}{\langle  l_{n},j_n,j_n^\prime\rangle}=\widetilde{c}\neq 0.
		$$			
		Taking the limit in \eqref{Rossemann 2-dif} and using \eqref{ASYFR1-}  we get	
		$$\forall q\in\mathbb{N},\quad \partial_{b}^{q}\left(\big({\omega}_{\textnormal{Eq}}(b),\Omega+\tfrac{1-b^2}{2},\pm\Omega\big)\cdot\widetilde{c}\right)_{b=\widetilde{b}}=0.$$
		Thus, the real analytic function ${b}\mapsto \big({\omega}_{\textnormal{Eq}}(b),\Omega+\tfrac{1-b^2}{2},\pm\Omega\big)\cdot\widetilde{c}$ is identically zero with $\widetilde{c}\neq0$ which contradicts Lemma \ref{lem non-deg}. 
		This completes the proof of the lemma.
	\end{proof}
	
	\paragraph{Linear quasi-periodic solution.} Notice  that by selecting  only  a finite number of frequencies, the sum  in \eqref{lin-sol}  gives rise to  quasi-periodic solutions of the linearized equation \eqref{Ham eq-eq DCE}, provided that the parameter $b$ belongs to a suitable  Cantor-like set of full measure.
	The following result follows in a similar way to \cite[Lem 3.3]{HHM21}, based on Lemma \ref{lemma transversalityE}-(i) and Lemma \ref{lemma Russmann book}.
	\begin{lem}\label{lemma sol Eq}
		Let $\Omega>0,$   $\mathbb{S}_1,\mathbb{S}_2\subset\mathbb{N}^*$, as in \eqref{S+} and $b^*$ as in Corollary $\ref{coro-equilib-freq}.$ Then, there exists a Cantor-like set $\mathcal{C}\subset[0,b^*]$ satisfying $|\mathcal{C}|=b^*$ and such that for all $b\in\mathcal{C}$, every function in the form
		$$\rho(t,\theta)=\sum_{j\in{\mathbb{S}}_1}\tfrac{\rho_{j,1}}{\sqrt{1-a_{j}^2(b)}} \begin{pmatrix}
				1  \\
				- a_{j}(b) 
			\end{pmatrix}\cos\big(j\theta-\Omega_{j,1}(b)t\big)+\sum_{j\in{\mathbb{S}}_2}\tfrac{\rho_{j,2}}{\sqrt{1-a_{j}^2(b)}} \begin{pmatrix}
				-a_{j} (b) \\
				1 
			\end{pmatrix}\cos\big(j\theta-\Omega_{j,2}(b)t\big),$$
		with $\rho_{j,1},\rho_{j,2}\in\mathbb{R}^*$, 	is a time quasi-periodic reversible solution to the equation \eqref{Ham eq-eq DCE} with the vector frequency $\omega_{\textnormal{Eq}}(b)$, defined in \eqref{Eq freq vec Edc}. 
	\end{lem}
	
	\section{Hamiltonian toolkit}
	The main scope of this section is to relate the existence of quasi-periodic solutions to the Hamiltonian equation \eqref{nonlinear-func0}  to the construction of invariant tori in a suitable phase space. More precisely, we shall reformulate the problem in terms of embedded tori through the introduction of action-angle variables.  Note that, according to Remark \ref{remark-lin-eq-eq}, \eqref{id:XKXH} and \eqref{def Lbmat}, the equation \eqref{nonlinear-func0} can be seen as a quasilinear perturbation of its linear part at the equilibrium state, namely,
	\begin{equation}\label{def XP}
		\partial_{t}r=\mathcal{J}\mathbf{L}_0r+X_{P}(r),\qquad\textnormal{with}\qquad X_{P}(r)\triangleq \mathcal{J}\nabla K({r})-\mathcal{J}\mathbf{L}_0r=\mathbf{Q}^{-1}X_{H\geqslant 3}(\mathbf{Q}{r}),
	\end{equation}
	where
	\begin{align*}
		X_{H\geqslant 3}({r})\triangleq \mathcal{J}\big(\nabla H(r)-\mathbf{M}_0r\big)
	\end{align*}
	and $\mathbf{Q},$ $H$, $\mathbf{M}_0,$ $\mathbf{L}_0$ are defined in \eqref{def-P}, \eqref{def H}, \eqref{Ham eq-eq DCE}, \eqref{def Lbmat}, respectively. The following lemma summarizes some tame estimates satisfied by the  vector field $X_P$. Notice that the structure of the two components of  vector  field $X_{H\geqslant 3}({r})$ are very similar to the one obtained in the setting of Euler equations in the unit disc  \cite[eq. (5.1)]{HR21-1}. Moreover, the symplectic change of variables $\mathbf{Q}^{\pm 1}$ depends only on the parameter $b$ (and not on  $r$) and  it acts continuously  from $\mathbf{H}_{\mathtt{m}}^s$ into itself for any $s.$ Therefore, one gets in a similar way to \cite[Lemma 5.2]{HR21-1}  the following estimates.
	\begin{lem}\label{tame XP}
		Let $b^*, \mathtt{m}^*$ as in Corollary $\ref{coro-equilib-freq},$   $\mathtt{m}\geqslant \mathtt{m}^*$, and  $(\gamma,q,s_{0},s)$ satisfy \eqref{setting q}, \eqref{setting tau1 and tau2} and \eqref{init Sob cond}. There exists $\varepsilon_{0}\in(0,1]$ such that if
		$$\| r\|_{s_{0}+2}^{q,\gamma,\mathtt{m}}\leqslant\varepsilon_{0},$$
		then the vector field $X_{P}$, defined in \eqref{def XP} satisfies the following estimates
		\begin{enumerate}[label=(\roman*)]
			\item $\| X_{P}(r)\|_{s}^{q,\gamma,\mathtt{m}}\lesssim \| r\|_{s+2}^{q,\gamma,\mathtt{m}}\| r\|_{s_0+1}^{q,\gamma,\mathtt{m}}.$
			\item $\| d_{r}X_{P}(r)[\rho]\|_{s}^{q,\gamma,\mathtt{m}}\lesssim\|\rho\|_{s+2}^{q,\gamma,\mathtt{m}}\| r\|_{s_0+1}^{q,\gamma,\mathtt{m}}+\| r\|_{s+2}^{q,\gamma,\mathtt{m}}\|\rho\|_{s_{0}+1}^{q,\gamma,\mathtt{m}}.$
			\item 
			$\| d_r^{2}X_{P}(r)[\rho_{1},\rho_{2}]\|_{s}^{q,\gamma,\mathtt{m}}\lesssim\|\rho_{1}\|_{s_{0}+1}^{q,\gamma,\mathtt{m}}\|\rho_{2}\|_{s+2}^{q,\gamma,\mathtt{m}}+\big(\|\rho_{1}\|_{s+2}^{q,\gamma,\mathtt{m}}+\| r\|_{s+2}^{q,\gamma,\mathtt{m}}\|\rho_{1}\|_{s_{0}+1}^{q,\gamma,\mathtt{m}}\big)\|\rho_{2}\|_{s_{0}+1}^{q,\gamma,\mathtt{m}}$.
		\end{enumerate}
	\end{lem}
	Since we shall  look for small amplitude quasi-periodic solutions then it is  more convenient  to 
	rescale the solution as follows  $r\mapsto\varepsilon r$ with $r$ bounded in  a suitable functions space. Hence, the Hamiltonian equation \eqref{nonlinear-func} takes the form
	\begin{equation}\label{perturbed hamiltonian}
		\omega\cdot\partial_\varphi r=\partial_{\theta}\mathbf{L}_0r+\varepsilon X_{P_{\varepsilon}}(r),
	\end{equation}
	where $\mathbf{L}_0$ is the operator defined by \eqref{def Lbmat} and $X_{P_{\varepsilon}}$ is the rescaled  Hamiltonian vector field defined by
	$X_{P_{\varepsilon}}(r)\triangleq \varepsilon^{-2}X_{P}(\varepsilon r).$
	Notice  that \eqref{perturbed hamiltonian} is the Hamiltonian system generated by
	the rescaled Hamiltonian  
	\begin{align}\label{def calKe}
		\nonumber \mathcal{K}_{\varepsilon}(r)&\triangleq\varepsilon^{-2}K(\varepsilon r)\\
		&= K_{\mathbf{L}_0}(r)+\varepsilon P_{\varepsilon}(r),
	\end{align}
	with  $K_{\mathbf{L}_0}$ the quadratic Hamiltonian defined in Remark \ref{remark-lin-eq-eq} and 
	$\varepsilon P_{\varepsilon}(r)$ describes all the terms
	of higher order more than cubic.
	\paragraph{Action-angle-normal variables}\label{subsec act-angl}
	Recalling the notations introduced in \eqref{def L2m}, \eqref{S+}--\eqref{Eq freq vec Edc}. Given  the decomposition \eqref{decomp prod l2} of the phase space $L_{\mathtt{m}}^2(\mathbb{T})\times L_{\mathtt{m}}^2(\mathbb{T})$
	and the decomposition in action-angle-normal variables \eqref{aa-coord-00},
	the symplectic  $2$-form in \eqref{sympl ref} becomes
	\begin{equation}\label{sympl_form3}
		{\mathcal W}=\sum_{j\in\mathbb{S}_1} d\vartheta_{j,1}\wedge dI_{j,2}-\sum_{j\in \mathbb{S}_2}d\vartheta_{j,2}\wedge dI_{j,2}   +\frac{1}{2\ii}\sum_{j\in\mathbb{Z}_{\mathtt{m}}\setminus\overline{\mathbb{S}}_{0,1}}\frac{1}{j}dr_{j,1}\wedge dr_{-j,1}-\frac{1}{2\ii}\sum_{j\in \mathbb{Z}_{\mathtt{m}}\setminus\overline{\mathbb{S}}_{0,2}}\frac{1}{j}dr_{j,2}\wedge dr_{-j,2}.
	\end{equation}
	The Poisson bracket is given by
	\begin{equation}\label{poisson-bracket}
		\{F,G\}\triangleq\mathcal{W}(X_F,X_G)=\big\langle\nabla F, \mathbf{J}\nabla G\big\rangle,
	\end{equation}
	where $\langle \cdot,\cdot\rangle$ is  the inner product, defined by
	$$\big\langle (\vartheta,I,z),(\underline \vartheta,\underline  I, \underline  z)\big\rangle\triangleq\vartheta\cdot \underline \vartheta+I\cdot\underline  I+\big\langle z,\underline  z \big\rangle_{L^2 (\mathbb{T})\times L^2 (\mathbb{T})}.$$
	The Poisson structure $\mathbf{J}$ corresponding to $\mathcal{W}$, defined by the identity \eqref{poisson-bracket} , is the unbounded operator
	$$\mathbf{J}:(\vartheta,I,z)\mapsto({\mathtt J}I,-{\mathtt J}\vartheta,\mathcal{J}z),$$ 
	where $\mathcal{J}$ is given by \eqref{def calJ} and
	$${\mathtt J}\triangleq  \begin{pmatrix}
		{\rm I}_{d_1}& 0 \\
		0 & -{\rm I}_{d_2}  \end{pmatrix},
	$$
	with  ${\rm I}_{d_k}$  the identity matrix of size ${d_k}$.
	Now we shall  study the Hamiltonian system  generated by the Hamiltonian $ {\mathcal K}_\varepsilon   $ in \eqref{def calKe}, 
	in  the action-angle-normal variables 
	$ 
	(\vartheta, I, z) \in  \mathbb{T}^d \times \mathbb{R}^d \times \mathbf{H}^{\perp}_{\overline{\mathbb{S}}_0} \, . 
	$ 
	We consider the Hamiltonian $ K_{\varepsilon} (\vartheta, I, z )$ defined by 
	\begin{equation}\label{K eps}
		K_{\varepsilon} \triangleq\mathcal{K}_{\varepsilon} \circ   \mathbf{A}
	\end{equation}
	where  
	$\mathbf{A} $ is the map defined in \eqref{aa-coord-00}. Since $\mathbf{L}_0 $ in \eqref{def Lbmat} stabilizes the subspace $ \mathbf{H}^{\perp}_{\overline{\mathbb{S}}_0}$ then the quadratic Hamiltonian $ K_{\mathbf{L}_0}$ in \eqref{def KL} in the variables $ (\vartheta, I, z) $ reads, up to a constant,
	\begin{align}\label{QHAM}
		K_{\mathbf{L}_0} \circ  \mathbf{A} &=  -\sum_{j\in\mathbb{S}_1} \, \Omega_{j,1}(b)I_{j,1} +\sum_{j\in\mathbb{S}_2} \,  \Omega_{j,2}(b)I_{j,2}+ \frac12 \big\langle \mathbf{L}_0\, z, z\big\rangle_{L^2(\T)\times L^2(\T)}\nonumber\\ & =   -\big({\mathtt J}\,{\omega}_{\textnormal{Eq}}(b)\big)\cdot I
		+ \frac12 \big\langle\mathbf{L}_0\, z, z \big\rangle_{L^2(\T)\times L^2(\T)}, 
	\end{align}
	where $ {\omega}_{\textnormal{Eq}}(b) \in \mathbb{R}^d $ is the unperturbed 
	tangential frequency vector.
	By \eqref{def calKe} and \eqref{QHAM}, 
	the Hamiltonian $K_{\varepsilon} $ in \eqref{K eps} reads
	\begin{equation}\label{cNP-K}
		\begin{aligned}
			&  K_{\varepsilon} = 
			{\mathcal N} + \varepsilon \mathcal{ P}_{\varepsilon},  \qquad \textnormal{with}  \\
			&
			{\mathcal N} \triangleq    -\big({\mathtt J}\,{\omega}_{\textnormal{Eq}}(b)\big)\cdot I  + \frac12  \big\langle\mathbf{L}_0\, z, z \big\rangle_{L^2(\mathbb{T})\times L^2(\mathbb{T})},  
			\quad 
			\quad \mathcal{ P}_{\varepsilon} \triangleq    P_\varepsilon \circ  \mathbf{A}.  
		\end{aligned}
	\end{equation}
	We look for an embedded invariant torus
	$$i :\mathbb{T}^d \rightarrow\mathbb{R}^d \times \mathbb{R}^d \times  \mathbf{H}^{\perp}_{\overline{\mathbb{S}}_0}\,, \qquad \varphi \mapsto i(\varphi)\triangleq  \big(\vartheta(\varphi), I(\varphi), z(\varphi)\big),$$
	where $\vartheta (\varphi)-\varphi $ is a  $ (2 \pi)^d $-periodic function, of the Hamiltonian vector field 
	$$ 
	X_{K_{\varepsilon}} \triangleq  
	\big({\mathtt J}\partial_I K_{\varepsilon} , -{\mathtt J} \partial_\vartheta K_{\varepsilon} , \Pi_{\overline{\mathbb{S}}_0}^\bot
	\mathcal{J}\nabla_{z} K_{\varepsilon}\big) 
	$$ 
	filled by quasi-periodic solutions with Diophantine frequency 
	vector $\omega$.
	Note that for the  value   $\varepsilon=0$, the Hamiltonian system 
	\begin{equation}\label{HS-T}
		\omega\cdot\partial_\varphi i (\varphi) = (X_{{\mathcal N}}  +\varepsilon   X_{\mathcal{P}_{\varepsilon}})  (i(\varphi) )
	\end{equation}   possesses, for any value of the parameter $b\in [0,b^*]$ , the invariant torus  $$i_{\textnormal{flat}}(\varphi)\triangleq(\varphi,0,0),
	$$
	provided that $\omega=-{\omega}_{\textnormal{Eq}}(b).$
	Now, in order to construct  an invariant torus to the  Hamiltonian system \eqref{HS-T} which supports a quasi-periodic motion with frequency vector $\omega$, close to $-{\omega}_{\textnormal{Eq}}(b)$, we shall formulate the
	problem as a "Nash-Moser Theorem of hypothetical conjugation" established in \cite{BM18}. It consists in using the frequencies $\omega\in\mathbb{R}^d$ as parameters and introducing “counter-terms” $\alpha\in\mathbb{R}^d$ in the family of Hamiltonians
	\begin{equation}\label{K alpha}
		\begin{aligned}
			K_\varepsilon^\alpha \triangleq  {\mathcal N}_\alpha +\varepsilon  {\mathcal P}_{\varepsilon} \, , \qquad  {\mathcal N}_\alpha \triangleq   \alpha \cdot I 
			+ \tfrac12\big\langle\mathbf{L}_0\, z, z\big\rangle_{L^2(\mathbb{T})\times L^2(\mathbb{T})}. 
		\end{aligned}
	\end{equation}
	The value of $\alpha$ will be adjusted along the iteration in order to control the average of the $I$-component at the linear level of the Hamiltonian equation
	\begin{align}
		{\mathcal F} (i, \alpha ) 
		& \triangleq    {\mathcal F} (i, \alpha, \omega,b,  \varepsilon )  \triangleq  \omega\cdot\partial_\varphi i (\varphi) - X_{K_\varepsilon^\alpha} ( i (\varphi))
		=  \omega\cdot\partial_\varphi i (\varphi) -  (X_{{\mathcal N}_\alpha}  +\varepsilon   X_{\mathcal{P}_{\varepsilon}})  (i(\varphi) ) \nonumber \\
		&  =   \left(
		\begin{array}{c}
			\omega\cdot\partial_\varphi \vartheta (\varphi) - \mathtt{J}\big(\alpha -  \varepsilon \partial_I \mathcal{P}_{\varepsilon} ( i(\varphi)   ) \big) \\
			\omega\cdot\partial_\varphi I (\varphi) + \varepsilon  {\mathtt J}\partial_\vartheta \mathcal{P}_{\varepsilon}( i(\varphi)  )  \\
			\omega\cdot\partial_\varphi z (\varphi) 
			-  \mathcal{J} \mathbf{L}_0 z(\varphi) - \varepsilon  \mathcal{J} \nabla_z \mathcal{P}_{\varepsilon} ( i(\varphi) )   
		\end{array}
		\right) =0. \label{operatorF} 
	\end{align}
	This degree of freedom through a parameter $\alpha$ will provides at the end of the scheme a solution $(\omega,i)$ for the original problem when it is fixed to $\alpha=-\mathtt{J}{\omega}_{\textnormal{Eq}}(b)$ for any value of $b$ in a suitable  Cantor set. Note that the involution $\mathscr{S}$, described in \eqref{defin inv scr S}, becomes 
	\begin{equation}\label{rev th I z}
		\mathfrak{S}:(\vartheta,I,z)\mapsto(-\vartheta,I,\mathscr{S}z)
	\end{equation}
{and the operator $\mathscr{T}_{\mathtt{m}}$ in \eqref{def scr Tm} becomes
\begin{equation}\label{mfold th I z}
	\mathfrak{T}_{\mathtt{m}}:(\vartheta,I,z)\mapsto(\vartheta,I,\mathscr{T}_{\mathtt{m}}z).
\end{equation}}
	Moreover, we can easily check that the Hamiltonian vector field $X_{K_\varepsilon^\alpha}$ is reversible with respect  $\mathfrak{S} $ {and $\mathtt{m}$-fold preserving with respect to $\mathfrak{T}_{\mathtt{m}}$}.
	Thus, it is natural to  look for $\mathtt{m}$-fold reversible solutions of $ {\mathcal F}(i, \alpha) = 0 $,  namely satisfying
	\begin{equation}\label{parity solution}
		\vartheta(-\varphi) = - \vartheta (\varphi),\qquad\
		I(-\varphi) = I(\varphi),\qquad z(-\varphi)=\big(\mathscr{S} z(\varphi)\big){\qquad\textnormal{and}\qquad\mathscr{T}_{\mathtt{m}}z(\varphi)=z(\varphi).}
	\end{equation}
	In the sequel, we shall denote by 
	$$\mathfrak{I}(\varphi)\triangleq  i(\varphi)-(\varphi,0,0)=\big(\vartheta (\varphi)-\varphi, I(\varphi), z(\varphi)\big)$$
	the periodic component of the torus $\varphi\mapsto i(\varphi)$.
	We end this section by  summarizing  some  tame estimates satisfied by the Hamiltonian vector field 
	$$X_{\mathcal{P}_{\varepsilon}}\triangleq\big(\mathtt{J}\partial_{I}\mathcal{P}_{\varepsilon},-\mathtt{J}\partial_{\vartheta}\mathcal{P}_{\varepsilon},\Pi_{\overline{\mathbb{S}}_0}^{\perp}\mathcal{J}\nabla_{z}\mathcal{P}_{\varepsilon}\big),$$
	where $\mathcal{P}_{\varepsilon}$ is  defined in \eqref{cNP-K}. The proof of the next lemma follows in a similar way  to  \cite[Lem. 5.1]{BM18} using Lemma \ref{tame XP}.
	\begin{lem}\label{tame X per}
		Let $b^*, \mathtt{m}^*$ as in Corollary $\ref{coro-equilib-freq},$   $\mathtt{m}\geqslant \mathtt{m}^*$ and  $(\gamma,q,s_{0},s)$ satisfy \eqref{setting q}, \eqref{setting tau1 and tau2} and \eqref{init Sob cond}.
		There exists $\varepsilon_0\in(0,1)$ such that if 
		$$\varepsilon\leqslant\varepsilon_0\qquad\textnormal{and}\qquad\|\mathfrak{I}\|_{s_{0}+2}^{q,\gamma,\mathtt{m}}\leqslant 1,$$ 
		then  the  perturbed Hamiltonian vector field $X_{\mathcal{P}_{\varepsilon}}$ satisfies the following tame estimates,
		\begin{enumerate}[label=(\roman*)]
			\item $\| X_{\mathcal{P}_{\varepsilon}}(i)\|_{s}^{q,\gamma,\mathtt{m}}\lesssim 1+\|\mathfrak{I}\|_{s+2}^{q,\gamma,\mathtt{m}}.$
			\item $\big\| d_{i}X_{\mathcal{P}_{\varepsilon}}(i)[\,\widehat{i}\,]\big\|_{s}^{q,\gamma,\mathtt{m}}\lesssim \|\,\widehat{i}\,\|_{s+2}^{q,\gamma,\mathtt{m}}+\|\mathfrak{I}\|_{s+2}^{q,\gamma,\mathtt{m}}\|\,\widehat{i}\,\|_{s_{0}+1}^{q,\gamma,\mathtt{m}}.$
			\item $\big\| d_{i}^{2}X_{\mathcal{P}_{\varepsilon}}(i)[\,\widehat{i},\widehat{i}\,]\big\|_{s}^{q,\gamma,\mathtt{m}}\lesssim \|\,\widehat{i}\,\|_{s+2}^{q,\gamma,\mathtt{m}}\|\,\widehat{i}\,\|_{s_{0}+1}^{q,\gamma,\mathtt{m}}+\|\mathfrak{I}\|_{s+2}^{q,\gamma,\mathtt{m}}\big(\|\,\widehat{i}\,\|_{s_{0}+1}^{q,\gamma,\mathtt{m}}\big)^{2}.$
		\end{enumerate}
	\end{lem}
	\section{Approximate inverse}\label{sec:Approximate-inverse}
	In order to prove the Theorem \ref{thm QPS E} using a  Nash-Moser scheme, we have to construct an approximate right inverse of the linearized operator associated to the functional $\mathcal{F}$, defined in \eqref{operatorF}, at any $\mathtt{m}$-fold and reversible state $(i_0,\alpha_0)$ close to the flat torus, 
	\begin{equation}\label{Linearized-op-F-DC}
		d_{(i,\alpha)}\mathcal{F}(i_0,\alpha_0)=\omega\cdot\partial_\varphi i_0  - d_i X_{K_\varepsilon^{\alpha_0}} ( i_0 )- \left(
		\begin{array}{c}
			\mathtt{J}\widehat\alpha  \\
			0
			\\
			0  
		\end{array}
		\right).
	\end{equation}
	For this aim, we shall use the Berti-Bolle approach for the approximate inverse developed  in \cite{BBP14} and which "approximately" decouples the linearized equations through a triangular system in the action-angle components and the normal ones. This strategy was slightly simplified in \cite[Section 6]{HHM21} bypassing the introduction of an intermediate  isotropic torus and directly working with the original one $i_0$. Here, we shall closely follow this latter procedure with giving close attention to the difference in the Hamiltonian structure, which is due to the vectorial framework. Thus, for completeness sake, we shall reproduce all the algebraic computations and refer the reader to  \cite[Section 6]{HHM21} for more details on the analysis, which is very similar.
	 
	We first introduce  the  diffeomorpshim 
	$ G_0 : (\phi, y, w) \mapsto (\vartheta, I, z)$ of the phase space $\mathbb{R}^d \times \mathbb{R}^d \times  \mathbf{H}_{\overline{\mathbb{S}}_0}^{\perp}$ given  by
	\begin{equation}\label{trasformation G}
		\begin{pmatrix}
			\vartheta \\
			I \\
			z
		\end{pmatrix} \triangleq G_0 \begin{pmatrix}
			\phi \\
			y \\
			w
		\end{pmatrix} \triangleq 
		\begin{pmatrix}
			\!\!\!\!\!\!\!\!\!\!\!\!\!\!\!\!\!\!\!\!\!\!\!\!\!\!\!\!\!\!\!\!\!
			\!\!\!\!\!\!\!\!\!\!\!\!\!\!\!\!\!\!\!\!
			\vartheta_0(\phi) \\
			I_0 (\phi) + L_1(\phi) y + L_2(\phi) w \\
			\!\!\!\!\!\!\!\!\!\!\!\!\!\!\!\!\!\!\!\!\!\!\!\!\!\!\!\!\!
			\!\!\!\!\!\!\!\!\!\!\!\!
			z_0(\phi) + w
		\end{pmatrix}, 
	\end{equation}
	where
	\begin{equation}\label{L1-L2}
		L_1(\phi)\triangleq  \mathtt{J}[\partial_\varphi \vartheta_0(\phi)]^{-\top},\qquad L_2(\phi) \triangleq \mathtt{J}[(\partial_\vartheta \widetilde{z}_0)(\vartheta_0(\phi))]^\top \mathcal{J}^{-1},\qquad \widetilde{z}_0 (\vartheta) \triangleq z_0 (\vartheta_0^{-1} (\vartheta))
	\end{equation}
	and the transposed operator is defined through the following duality relation :
	Given a Hilbert space $\mathtt{H}$ equipped with the inner product $\langle \cdot,\cdot \rangle_{\mathtt{H}}$ and a linear operator $A\in  {\mathcal L}({\mathbb R}^d, \mathtt{H})$,
	\begin{align*}
		\forall \, u\in {\mathtt{H}}\, , \qquad\forall  v\in\mathbb{R}^d,\qquad \big\langle A^{\top}u,v\big\rangle_{\mathbb{R}^d} \triangleq\big\langle u,  Av\big\rangle_{\mathtt{H}} .
	\end{align*}
	Note that, in the new coordinates, $i_0$ becomes the trivial
	embedded torus $(\phi,y,w)=(\varphi,0,0)$, namely
	$$
	G_0(\varphi,0,0)=i_0(\varphi).
	$$
	In what follows we shall use the following notations
	\begin{itemize}
		\item We denote by ${\mathtt u}=(\phi, y,w)$ the coordinates induced by $G_0$ in \eqref{trasformation G}.
		\item The mapping $${\mathtt u}_0(\varphi)\triangleq G_0^{-1}
		(i_0)(\varphi)=(\varphi,0,0)$$ refers to the trivial torus 
		\item  We shall denote by 
		$$
		\widetilde{G}_0 ({\mathtt u}, \alpha) \triangleq  \big( G_0 ({\mathtt u}), \alpha \big) 
		$$ 
		the diffeomorphism with the identity on the $ \alpha $-component.
		\item We quantify how an embedded torus $i_0(\mathbb{T})$ is approximately invariant for the Hamiltonian vector field $X_{K_\varepsilon^{\alpha_0}}$ in terms of the "error function"     
		\begin{equation}\label{def Z}
			Z(\varphi) \triangleq  (Z_1, Z_2, Z_3) (\varphi) \triangleq {\mathcal F}(i_0, \alpha_0) (\varphi) =
			\omega \cdot \partial_\varphi i_0(\varphi) - X_{K^{\alpha_0}_\varepsilon}\big(i_0(\varphi)\big).
		\end{equation}
	\end{itemize}
	
	\subsection{Linear change of variables and defect of the symplectic structure}
	In this subsection we shall  conjugate the linearized operator $d_{i,\alpha} {\mathcal F} (i_0,{\alpha}_0)$ in \eqref{Linearized-op-F-DC}, via the linear change of variables
	\begin{equation}\label{Differential G0}
		D G_0({\mathtt u}_0(\varphi))
		\begin{pmatrix}
			\widehat \phi \, \\
			\widehat y \\
			\widehat w
		\end{pmatrix} 
		=
		\begin{pmatrix}
			\partial_\varphi \vartheta_0(\varphi) & 0 & 0 \\
			\partial_\varphi I_0(\varphi) &L_1(\varphi) & 
			L_2(\varphi) \\
			\partial_\varphi z_0(\varphi) & 0 & I
		\end{pmatrix}
		\begin{pmatrix}
			\widehat \phi \, \\
			\widehat y \\
			\widehat w
		\end{pmatrix} ,
	\end{equation}
	to a triangular system with small errors of size $Z=\mathcal{F}(i_0,\alpha_0)$.  Our main result is the following.
	\begin{prop}\label{Proposition-Conjugation}
		Under the linear change of variables $D G_0({\mathtt u}_0)$ the linearized operator $d_{i,\alpha} {\mathcal F} (i_0,\alpha_0)$ is transformed into
		\begin{align}\label{Id-conj}
			& [D G_0({\mathtt u}_0)]^{-1}  d_{(i,\alpha)} {\mathcal F} (i_0,\alpha_0) D\widetilde G_0({\mathtt u}_0)
			[\widehat \phi, \widehat y, \widehat w, \widehat \alpha ]
			= \mathbb{D} [\widehat \phi, \widehat y, \widehat w, \widehat \alpha ]+\mathbb{E} [\widehat \phi, \widehat y, \widehat w] 
		\end{align}
		where
		\begin{enumerate} 
			\item the operator $ \mathbb{D}$ has the triangular form
			\begin{align*}
				\mathbb{D} [\widehat \phi, \widehat y, \widehat w, \widehat \alpha ]\triangleq\left(
				\begin{array}{c}
					\omega\cdot\partial_\varphi  \widehat\phi-\big[K_{20}(\varphi) \widehat y+K_{11}^\top(\varphi) \widehat w+L_1^\top (\varphi)\widehat \alpha\big]
					\\
					\omega\cdot\partial_\varphi  \widehat y+\mathcal{B}(\varphi) \widehat \alpha \\
					\omega\cdot\partial_\varphi \widehat w-\mathcal{J}\big[K_{11}(\varphi) \widehat y+K_{02}(\varphi)\widehat w +L_2^{\top}(\varphi) \widehat\alpha   \big] 
				\end{array}
				\right),
			\end{align*}
			$\mathcal{B}(\varphi)$ and   $K_{20}(\varphi) $ are   $d \times d$ real matrices, 
			\begin{align*}
				\mathcal{B}(\varphi)& \triangleq L_1^{-1} (\varphi)\partial_\varphi I_0(\varphi) L_1^\top (\varphi)+[\partial_\varphi z_0(\varphi)]^{\top} L_2^\top (\varphi) ,
				\\
				K_{20}(\varphi)&\triangleq \varepsilon L_1^\top(\varphi) ( \partial_{II} \mathcal{P}_\varepsilon)(i_0(\varphi))  L_1(\varphi) \, ,
			\end{align*}
			$K_{02}(\varphi)$ is a linear self-adjoint operator of $  \mathbf{H}_{\overline{\mathbb{S}}_0}^{\perp}$, given by 
			\begin{align}\label{def K02}
				K_{02}(\varphi)& \triangleq ( \partial_{z}\nabla_z K_\varepsilon^{\alpha_0}) (i_0(\varphi))  +\varepsilon L_2^\top(\varphi) ( \partial_{II} \mathcal{P}_\varepsilon)(i_0(\varphi)) L_2(\varphi) 
				\\ &\quad+\varepsilon L_2^\top(\varphi)( \partial_{zI} \mathcal{P}_\varepsilon) (i_0(\varphi))     + \varepsilon  (\partial_I\nabla_z \mathcal{P}_\varepsilon) (i_0(\varphi))  L_2(\varphi),\nonumber
			\end{align}
			and $K_{11}(\varphi)  \in {\mathcal L}({\mathbb R}^d,   \mathbf{H}_{\overline{\mathbb{S}}_0}^{\perp})$, 
			\begin{align*}
				K_{11}(\varphi)&\triangleq \varepsilon  L_2^\top(\varphi)  ( \partial_{II} \mathcal{P}_\varepsilon)(i_0(\varphi))L_1(\varphi) +\varepsilon ( \partial_I\nabla_z \mathcal{P}_\varepsilon) (i_0(\varphi)) L_1(\varphi) ,
			\end{align*}
			\item the remainder $\mathbb{E} $ is given by
			\begin{align*}
				\mathbb{E} [\widehat \phi, \widehat y, \widehat w] &\triangleq
				[D G_0({\mathtt u}_0)]^{-1}    \partial_\varphi Z(\varphi) \widehat\phi  \\  &\quad + \left(
				\begin{array}{c}
					0 
					\\
					\mathcal{A}(\varphi)\big[K_{20}(\varphi) \widehat y+K_{11}^\top(\varphi) \widehat w\big]-R_{10}(\varphi) \widehat y -R_{01}(\varphi) \widehat w 
					\\
					0   
				\end{array}
				\right)
			\end{align*}
			where $\mathcal{A}(\varphi)$ and   $R_{10}(\varphi) $ are   $d \times d$ real matrices, 
			\begin{align*}
				\mathcal{A}(\varphi)& \triangleq[\partial_\varphi \vartheta_0(\varphi)]^\top{\mathtt J} \partial_\varphi I_0(\varphi)-[\partial_\varphi I_0(\varphi)]^\top {\mathtt J}\partial_\varphi \vartheta_0(\varphi)  -[\partial_\varphi z_0(\varphi)]^{\top} \mathcal{J}^{-1} \partial_\varphi z_0(\varphi),
				\\
				R_{10}(\varphi)&\triangleq  [\partial_\varphi Z_1(\varphi)]^{\top}  [\partial_\varphi \vartheta_0(\varphi)]^{-\top}, 
			\end{align*}
			and $R_{01}(\varphi)\in {\mathcal L}( \mathbf{H}_{\overline{\mathbb{S}}_0}^{\perp},{\mathbb R}^d)$, 
			\begin{align*}
				R_{01}(\varphi)&\triangleq [\partial_\varphi Z_1(\varphi)]^{\top} [(\partial_\vartheta \widetilde{z}_0)(\vartheta_0(\varphi))]^\top \mathcal{J}^{-1}- [\partial_\varphi  Z_3(\varphi)]^{\top} \mathcal{J}^{-1}
				\, .
			\end{align*}
		\end{enumerate}
	\end{prop}
	\begin{proof}
		The composition of the nonlinear operator $\mathcal{F}$,  in \eqref{operatorF}, with the map  $G_0$ is given by
		\begin{align}\label{Composition F G0}
			{\mathcal F} (G_0({\mathtt u}(\varphi)),\alpha) 
			= \omega\cdot\partial_\varphi \big( G_0({\mathtt u}(\varphi))\big)- X_{K_\varepsilon^\alpha} \big(G_0({\mathtt u}(\varphi))\big).  
		\end{align}
		Then, by differentiating \eqref{Composition F G0}  at $({\mathtt u}_0,\alpha_0)$  in the direction  $(\widehat {\mathtt u}, \widehat \alpha)$ we obtain
		\begin{align}\label{differential-composition}
			&d_{({\mathtt u},\alpha)} ({\mathcal F} \circ G_0)({\mathtt u}_0,\alpha_0)
			[(\widehat {\mathtt u}, \widehat \alpha )](\varphi)
			=  \omega\cdot\partial_\varphi \big( DG_0({\mathtt u}_0)\widehat {\mathtt u} \big)- \partial_\phi\big[ X_{K_\varepsilon^{\alpha_0}} \big(G_0({\mathtt u}(\varphi))\big)\big]_{{\mathtt u}={\mathtt u}_0} \widehat\phi
			\\ &\qquad\qquad\qquad\quad -  \partial_y\big[ X_{K_\varepsilon^{\alpha_0}} \big(G_0({\mathtt u}(\varphi))\big)\big]_{{\mathtt u}={\mathtt u}_0} \widehat y- \partial_w\big[ X_{K_\varepsilon^{\alpha_0}} \big(G_0({\mathtt u}(\varphi))\big) \big]_{{\mathtt u}={\mathtt u}_0} \widehat w
			- \left(
			\begin{array}{c}
				\mathtt{J}\widehat\alpha  \\
				0
				\\
				0  
			\end{array}
			\right).\nonumber
		\end{align}
		In view of  \eqref{Differential G0},  one has
		\begin{equation}\label{omg-d-phi-DG}
			\omega\cdot\partial_\varphi \big( DG_0({\mathtt u}_0)[\widehat {\mathtt u} ](\varphi)\big)
			=DG_0({\mathtt u}_0) \, \omega\cdot\partial_\varphi \widehat {\mathtt u} + \partial_\varphi\big(\omega\cdot\partial_\varphi  i_0 \big)\widehat\phi +\left(
			\begin{array}{c}
				0
				\\
				(\omega\cdot\partial_\varphi  L_1(\varphi))\widehat y+\big(\omega\cdot\partial_\varphi  L_2(\varphi)\big) \widehat w
				\\
				0
			\end{array}
			\right),
		\end{equation} 
		and from \eqref{L1-L2}--\eqref{def Z} we get
		\begin{equation}\label{omg-d-phi-L1}
			\begin{aligned}
				\omega\cdot\partial_\varphi  L_1(\varphi)&=-\mathtt{J}[\partial_\varphi \vartheta_0(\varphi)]^{-\top}(\omega\cdot\partial_\varphi  [\partial_\varphi \vartheta_0(\varphi)]^{\top}) [\partial_\varphi \vartheta_0(\varphi)]^{-\top}\\
				&=-\mathtt{J}[\partial_\varphi \vartheta_0(\varphi)]^{-\top} \Big(\big[\partial_\varphi Z_1(\varphi)\big]^{\top} + \big[ \partial_\varphi\big( ( \partial_{I} K_\varepsilon^{\alpha_0}) (i_0(\varphi))\big) \big]^{\top}{\mathtt J}\Big) [\partial_\varphi \vartheta_0(\varphi)]^{-\top}.
			\end{aligned}
		\end{equation}
		Observe, from \eqref{L1-L2}, that  we have the identity 
		\begin{equation}\label{d-phi-z}
			\partial_\varphi z_0(\varphi)=(\partial_\vartheta \widetilde{z}_0)(\vartheta_0(\varphi))\partial_\varphi \vartheta_0(\varphi).
		\end{equation}
		Therefore  the operator $L_2(\varphi)$ can be written as 
		\begin{equation}\label{rewriting L2}
			L_2(\varphi)=\mathtt{J}[\partial_\varphi \vartheta_0(\varphi)]^{-\top}[\partial_\varphi z_0(\varphi)]^{\top}\mathcal{J}^{-1}=L_1(\varphi)[\partial_\varphi z_0(\varphi)]^{\top}\mathcal{J}^{-1}.
		\end{equation}
		From the last two identities we find
		\begin{align}
			\omega\cdot\partial_\varphi  L_2(\varphi)
			&=-\mathtt{J}[\partial_\varphi \vartheta_0(\varphi)]^{-\top}(\omega\cdot\partial_\varphi  [\partial_\varphi \vartheta_0(\varphi)]^{\top}) [\partial_\varphi \vartheta_0(\varphi)]^{-\top}[\partial_\varphi z_0(\varphi)]^{\top}\mathcal{J}^{-1}\notag
			\\
			& \quad +\mathtt{J}[\partial_\varphi \vartheta_0(\varphi)]^{-\top}  [\partial_\varphi    (\omega\cdot\partial_\varphi z_0)(\varphi)]^{\top} \mathcal{J}^{-1}\notag
		\end{align}
		and by \eqref{def Z}  we obtain
		\begin{align}\label{omg-d-phi-L2}
			\omega\cdot\partial_\varphi  L_2(\varphi)
			&=-\mathtt{J}[\partial_\varphi \vartheta_0(\varphi)]^{-\top}\Big(\big[\partial_\varphi Z_1(\varphi)\big]^{\top} + \big[ \partial_\varphi\big( ( \partial_{I} K_\varepsilon^{\alpha_0}) (i_0(\varphi))\big) \big]^{\top}\Big) L_2(\varphi)  \notag
			\\ &\quad  +\mathtt{J}[\partial_\varphi \vartheta_0(\varphi)]^{-\top}   \Big( \big[\partial_\varphi  Z_3(\varphi)\big]^{\top}-\big[ \partial_\varphi\big(( \nabla_{z} K_\varepsilon^{\alpha_0}) (i_0(\varphi))\big) \big]^{\top}\Big) .
		\end{align}
		Gathering  \eqref{omg-d-phi-DG}, \eqref{omg-d-phi-L1} and  \eqref{omg-d-phi-L2} gives
		\begin{align}\label{omdphdg}
			&\omega\cdot\partial_\varphi \big( DG_0({\mathtt u}_0)[\widehat {\mathtt u}](\varphi) \big)=DG_0({\mathtt u}_0) \, \omega\cdot\partial_\varphi \widehat {\mathtt u} + \partial_\varphi\big(\omega\cdot\partial_\varphi  i_0 \big)\widehat\phi \nonumber
			\\ &\quad  -  \left(
			\begin{array}{c}
				0
				\\
				~ \mathtt{J}[\partial_\varphi \vartheta_0(\varphi)]^{-\top}\Big(\big[\mathcal{C}_I(\varphi) L_1(\varphi)  +R_{10}(\varphi)  \big] \widehat y+\big[\mathcal{C}_I(\varphi) L_2(\varphi)+\mathcal{C}_z(\varphi)     + R_{01}(\varphi) \big] \widehat w\Big]
				\\
				0
			\end{array}
			\right),
					\end{align} 
		where $R_{10}(\varphi)$ and  $R_{01}(\varphi)$ are given by {\rm (ii)} 
		and
		\begin{align}
			\mathcal{C}_I(\varphi) &\triangleq \big[ \partial_\varphi\big( ( \partial_{I} K_\varepsilon^{\alpha_0}) (i_0(\varphi))\big) \big]^{\top}\label{def CI}\\ 
			&=[\partial_\varphi I_0(\varphi)]^\top(  \partial_{II} K_\varepsilon^{\alpha_0})(i_0(\varphi))+ [\partial_\varphi \vartheta_0(\varphi)]^\top (  \partial_{\vartheta I} K_\varepsilon^{\alpha_0})(i_0(\varphi))+ [\partial_\varphi z_0(\varphi)]^\top (  \partial_{I} \nabla_z K_\varepsilon^{\alpha_0})(i_0(\varphi)),\notag
			\\ 
			\mathcal{C}_z(\varphi) & \triangleq \big[ \partial_\varphi\big( ( \nabla_{z} K_\varepsilon^{\alpha_0}) (i_0(\varphi))\big) \big]^{\top}\label{def Cz}\\
			&=[\partial_\varphi I_0(\varphi)]^\top( \partial_{zI} K_\varepsilon^{\alpha_0})(i_0(\varphi))+ [\partial_\varphi \vartheta_0(\varphi)]^\top ( \partial_{z\vartheta } K_\varepsilon^{\alpha_0})(i_0(\varphi)) + [\partial_\varphi z_0(\varphi)]^\top ( \partial_{z} \nabla_z K_\varepsilon^{\alpha_0})(i_0(\varphi)).\notag
		\end{align}
		According  \eqref{operatorF} and  \eqref{trasformation G}, one may writes 
		\begin{align*}
			\partial_\phi\big[ X_{K_\varepsilon^{\alpha_0}} \big(G_0({\mathtt u}(\varphi))\big)\big]_{{\mathtt u}={\mathtt u}_0} \widehat\phi &
			= \partial_\varphi\big[ X_{K_\varepsilon^{\alpha_0}} (i_0(\varphi)))\big] \widehat\phi,
			\\
			\partial_y\big[ X_{K_\varepsilon^{\alpha_0}} \big(G_0({\mathtt u}(\varphi))\big)\big]_{{\mathtt u}={\mathtt u}_0} \widehat y &
			= \left(
			\begin{array}{c}
				{\mathtt J}( \partial_{II} K_\varepsilon^{\alpha_0})(i_0(\varphi))  L_1(\varphi) \widehat y 
				\\
				-{\mathtt J}( \partial_{I\vartheta} K_\varepsilon^{\alpha_0}) (i_0(\varphi))  L_1(\varphi) \widehat y 
				\\
				\mathcal{J} \big[ ( \partial_I\nabla_z K_\varepsilon^{\alpha_0}) (i_0(\varphi))  L_1(\varphi) \widehat y \big]     
			\end{array}
			\right), 
			\\
			\partial_w\big[ X_{K_\varepsilon^{\alpha_0}} \big(G_0({\mathtt u}(\varphi))\big) \big]_{{\mathtt u}={\mathtt u}_0} \widehat w &
			=\left(
			\begin{array}{c}
				({\mathtt J} \partial_{II} K_\varepsilon^{\alpha_0})(i_0(\varphi))L_2(\varphi) \widehat w+{\mathtt J}( \partial_{zI} K_\varepsilon^{\alpha_0}) (i_0(\varphi))  \widehat w 
				\\
				-{\mathtt J}( \partial_{I\vartheta} K_\varepsilon^{\alpha_0}) (i_0(\varphi)) L_2(\varphi) \widehat w- {\mathtt J}( \partial_{z\vartheta} K_\varepsilon^{\alpha_0}) (i_0(\varphi)) \widehat w
				\\
				\mathcal{J} \big[( \partial_I\nabla_z K_\varepsilon^{\alpha_0}) (i_0(\varphi)) L_2(\varphi) \widehat w  + ( \partial_{z}\nabla_z K_\varepsilon^{\alpha_0}) (i_0(\varphi)) \widehat w   \big]  
			\end{array}
			\right).
		\end{align*}
		Therefore inserting  \eqref{omdphdg} and the last three identities into \eqref{differential-composition} we get
		\begin{align}\label{dfcirg}
			&d_{({\mathtt u},\alpha)} ({\mathcal F} \circ G_0)({\mathtt u}_0,\alpha_0)
			[(\widehat {\mathtt u}, \widehat \alpha )]= DG_0({\mathtt u}_0) \, \omega\cdot\partial_\varphi \widehat {\mathtt u} + \partial_\varphi\big[{\mathcal F}(i_0(\varphi)) \big] \widehat\phi  \nonumber
			\\ 
			&   + \left(
			\begin{array}{c}
				-{\mathtt J}( \partial_{II} K_\varepsilon^{\alpha_0})(i_0(\varphi))  L_1(\varphi) \widehat y 
				\\
				{\mathtt J}( \partial_{I\vartheta} K_\varepsilon^{\alpha_0}) (i_0(\varphi))  L_1(\varphi)\widehat y  -{\mathtt J}[\partial_\varphi \vartheta_0(\varphi)]^{-\top} [\mathcal{C}_I(\varphi) L_1(\varphi) +R_{10}(\varphi) ]\widehat y 
				\\
				- \mathcal{J} ( \partial_I\nabla_z K_\varepsilon^{\alpha_0}) (i_0(\varphi))  L_1(\varphi) \widehat y      
			\end{array}
			\right) 
			\nonumber
			\\ 
			&+\left(
			\begin{array}{c}
				-{\mathtt J}( \partial_{II} K_\varepsilon^{\alpha_0})(i_0(\varphi))L_2(\varphi) \widehat w-{\mathtt J}( \partial_{zI} K_\varepsilon^{\alpha_0}) (i_0(\varphi))  \widehat w 
				\\
				\big[{\mathtt J}( \partial_{I\vartheta} K_\varepsilon^{\alpha_0}) (i_0(\varphi)) L_2(\varphi) + {\mathtt J}( \partial_{z\vartheta} K_\varepsilon^{\alpha_0}) (i_0(\varphi))\big] \widehat w
				\\
				-  \mathcal{J} \big[( \partial_I\nabla_z K_\varepsilon^{\alpha_0}) (i_0(\varphi)) L_2(\varphi) \widehat w  + ( \partial_{z}\nabla_z K_\varepsilon^{\alpha_0}) (i_0(\varphi)) \widehat w \big]     
			\end{array}
			\right)\nonumber
			\\ &-  \left(
			\begin{array}{c}
				0
				\\
				~ \mathtt{J}[\partial_\varphi \vartheta_0(\varphi)]^{-\top}\big[\mathcal{C}_I(\varphi) L_2(\varphi)+ \mathcal{C}_z(\varphi) +R_{01}(\varphi) \big] \widehat w
				\\
				0
			\end{array}
			\right)- \left(
			\begin{array}{c}
				\mathtt{J}\widehat\alpha  \\
				0
				\\
				0  
			\end{array}
			\right).
		\end{align}
		From  \eqref{Differential G0}, \eqref{d-phi-z} and \eqref{rewriting L2},  one may easily check that
		\begin{align*}
			& [D G_0({\mathtt u}_0)]^{-1} 
			= 
			\begin{pmatrix}
				[\partial_\varphi \vartheta_0(\varphi)]^{-1} & 0 & 0 \\
				-\mathcal{B}(\varphi){\mathtt J} & [\partial_\varphi \vartheta_0(\varphi)]^\top {\mathtt J}& 
				-[\partial_\varphi z_0(\varphi)]^\top \mathcal{J}^{-1}   \\
				-(\partial_\vartheta \widetilde{z}_0)(\vartheta_0(\varphi))    & 0 & I
			\end{pmatrix}
		\end{align*}
		where $\mathcal{B}(\varphi)$  is given by {\rm(i)}.
		Finally, applying $[D G_0({\mathtt u}_0)]^{-1}$ to \eqref{dfcirg} and using  \eqref{def CI}, \eqref{def Cz} 
		we obtain
		\begin{align*}
			& [D G_0({\mathtt u}_0)]^{-1}  d_{({\mathtt u},\alpha)} ({\mathcal F} \circ G_0)({\mathtt u}_0,\alpha_0)
			[\widehat {\mathtt u}, \widehat \alpha ]
			=  \omega\cdot\partial_\varphi \widehat {\mathtt u}+  [D G_0({\mathtt u}_0)]^{-1}   \partial_\varphi\big[{\mathcal F}(i_0(\varphi)) \big]  \widehat\phi 
			\\
			&  + \left(
			\begin{array}{c}
				-K_{20}(\varphi)\widehat y
				\\
				\mathcal{A}(\varphi) K_{20}(\varphi)  \widehat y -R_{10}(\varphi) \widehat y  \\
				-\mathcal{J} K_{11}(\varphi) \widehat y       
			\end{array}
			\right)
			+ \left(
			\begin{array}{c}
				-K_{11}^\top(\varphi) \widehat w 
				\\
				\mathcal{A}(\varphi) K_{11}^\top(\varphi) \widehat w -R_{01}(\varphi) \widehat w  \\
				-\mathcal{J} K_{02}(\varphi)\widehat w   
			\end{array}
			\right)
			+ \left(
			\begin{array}{c}
				- [\partial_\varphi \vartheta_0(\varphi)]^{-1}\mathtt{J}\widehat \alpha
				\\
				\mathcal{B}(\varphi) \widehat\alpha  \\
				~ [(\partial_\vartheta \widetilde{z}_0)(\vartheta_0(\varphi))  ]\mathtt{J}\widehat\alpha      
			\end{array}
			\right),
		\end{align*}
		where $\mathcal{A}(\varphi)$ is defined in {\rm (ii)} and satisfies 
		$\mathcal{B}(\varphi)=\mathcal{A}(\varphi) L_1^\top(\varphi) +[\partial_\varphi I_0(\varphi)]^\top,$
		and 
		\begin{align*}
			K_{20}(\varphi)&\triangleq L_1^\top(\varphi) ( \partial_{II} K_\varepsilon^{\alpha_0})(i_0(\varphi))  L_1(\varphi) \, ,
			\\
			K_{11}(\varphi)&\triangleq L_2^\top(\varphi)( \partial_{II} K_\varepsilon^{\alpha_0})(i_0(\varphi))L_1(\varphi) + ( \partial_I\nabla_z K_\varepsilon^{\alpha_0}) (i_0(\varphi)) L_1(\varphi)\, , 
			\\ 
			K_{02}(\varphi)& \triangleq ( \partial_{z}\nabla_z K_\varepsilon^{\alpha_0}) (i_0(\varphi))  +L_2^\top(\varphi) ( \partial_{II} K_\varepsilon^{\alpha_0})(i_0(\varphi)) L_2(\varphi)
			+L_2^\top(\varphi)( \partial_{zI} K_\varepsilon^{\alpha_0}) (i_0(\varphi))     \\ &\quad +   (\partial_I\nabla_z K_\varepsilon^{\alpha_0}) (i_0(\varphi))  L_2(\varphi). \nonumber 
		\end{align*}
		This together with \eqref{K alpha} give the desired identity, concluding the proof of Proposition~\ref{Proposition-Conjugation}.
	\end{proof}
	
	Next, in order to prove that the remainder $\mathbb{E} $ is of size $Z$, we shall prove that the matrix  $\mathcal{A}$, defined in  Proposition \ref{Proposition-Conjugation}-{\rm (ii)},  is zero at an exact solution on some Cantor like set, up to an exponentially small remainder. In particular, we shall prove the following lemma.
	\begin{lem}\label{lemma1est} 
		The coefficients of the matrix $\mathcal{A}$, given by 
		\begin{align*}
			\mathcal{A}_{kj}(\varphi)
			\triangleq [\mathtt{J} \partial_{\varphi_j} I_0(\varphi)]\cdot\partial_{\varphi_k} \vartheta_0(\varphi)
			-[\mathtt{J} \partial_{\varphi_j} \vartheta_0(\varphi)]\cdot\partial_{\varphi_k} I_0(\varphi)  
			-\big\langle\mathcal{J}^{-1} \partial_{\varphi_j}  z_0(\varphi) ,\partial_{\varphi_k}  z_0(\varphi)\big\rangle_{L^2(\mathbb{T})\times L^2(\mathbb{T})}.
		\end{align*}
		satisfy for all $\varphi\in\mathbb{T}^d$, are
		\begin{align*}
			\omega\cdot\partial_\varphi \mathcal{A}_{kj}(\varphi)
			&=[{\mathtt J} \partial_{\varphi_j} Z_2(\varphi)]\cdot\partial_{\varphi_k} \vartheta_0(\varphi)
			-[{\mathtt J} \partial_{\varphi_j} Z_1(\varphi)] \cdot\partial_{\varphi_k} I_0(\varphi)
			-\big\langle\mathcal{J}^{-1} \partial_{\varphi_j} Z_3(\varphi),\partial_{\varphi_k}  z_0(\varphi)\big\rangle_{L^2(\mathbb{T})\times L^2(\mathbb{T})}\nonumber
			\\
			&+ [{\mathtt J} \partial_{\varphi_j} I_0(\varphi)]\cdot\partial_{\varphi_k} Z_2(\varphi)
			-[{\mathtt J} \partial_{\varphi_j} \vartheta_0(\varphi)]\cdot\partial_{\varphi_k} Z_1(\varphi) 
			-\big\langle\mathcal{J}^{-1} \partial_{\varphi_j}  z_0(\varphi) ,\partial_{\varphi_k} Z_3(\varphi)\big\rangle_{L^2(\mathbb{T})\times L^2(\mathbb{T})}.
		\end{align*}
		where $(\underline{e}_1,\ldots,\underline{e}_d)$ denotes the canonical basis of $\mathbb{R}^d$.

	\end{lem}
	
	\begin{proof}
		
		From  the expression of the coefficients $\mathcal{A}_{kj}$ one has
		\begin{align*}
			\omega\cdot\partial_\varphi \mathcal{A}_{kj}(\varphi)
			&=\big\langle  \mathtt{J}\partial_{\varphi_j} \omega\cdot\partial_\varphi I_0(\varphi),\partial_{\varphi_k} \vartheta_0(\varphi)\big\rangle_{\mathbb{R}^d}+\big\langle \mathtt{J} \partial_{\varphi_j} I_0(\varphi),\partial_{\varphi_k} \omega\cdot\partial_\varphi\vartheta_0(\varphi)\big\rangle_{\mathbb{R}^d}
			\\
			&
			-\big\langle \mathtt{J}\partial_{\varphi_j} \omega\cdot\partial_\varphi\vartheta_0(\varphi),\partial_{\varphi_k} I_0(\varphi)\big\rangle_{\mathbb{R}^d} 
			-\big\langle \mathtt{J}\partial_{\varphi_j} \vartheta_0(\varphi),\partial_{\varphi_k} \omega\cdot\partial_\varphi I_0(\varphi)\big\rangle_{\mathbb{R}^d}   
			\\
			&
			-\big\langle\mathcal{J}^{-1} \partial_{\varphi_j} \omega\cdot\partial_\varphi z_0(\varphi) ,\partial_{\varphi_k}  z_0(\varphi)\big\rangle_{L^2(\mathbb{T})\times L^2(\mathbb{T})}-\big\langle\mathcal{J}^{-1} \partial_{\varphi_j}  z_0(\varphi) ,\partial_{\varphi_k} \omega\cdot\partial_\varphi z_0(\varphi)\big\rangle_{L^2(\mathbb{T})\times L^2(\mathbb{T})} 
		\end{align*}
		In view of \eqref{def Z} we get
		\begin{align}\label{omdphiakj}
			\omega\cdot\partial_\varphi \mathcal{A}_{kj}(\varphi)
			&=\big\langle \mathtt{J}\partial_{\varphi_j} Z_2(\varphi),\partial_{\varphi_k} \vartheta_0(\varphi)\big\rangle_{\mathbb{R}^d}+\big\langle \mathtt{J}\partial_{\varphi_j} I_0(\varphi),\partial_{\varphi_k} Z_1(\varphi)\big\rangle_{\mathbb{R}^d}  
			\\
			&
			\quad-\big\langle \mathtt{J}\partial_{\varphi_j} Z_1(\varphi),\partial_{\varphi_k} I_0(\varphi)\big\rangle_{\mathbb{R}^d}  
			-\big\langle \mathtt{J}\partial_{\varphi_j} \vartheta_0(\varphi),\partial_{\varphi_k} Z_2(\varphi)\big\rangle_{\mathbb{R}^d}\nonumber
			\\
			&
			\quad-\big\langle  \mathcal{J}^{-1} \partial_{\varphi_j}  z_0(\varphi) ,\partial_{\varphi_k} Z_3(\varphi)\big\rangle_{L^2(\mathbb{T})\times L^2(\mathbb{T})}
			-\big\langle \mathcal{J}^{-1} \partial_{\varphi_j} Z_3(\varphi) ,\partial_{\varphi_k}  z_0(\varphi)\big\rangle_{L^2(\mathbb{T})\times L^2(\mathbb{T})}\nonumber
			\\ &\quad+\mathcal{B}_{kj}^1(\varphi)+\mathcal{B}_{kj}^2(\varphi)+\mathcal{B}_{kj}^3(\varphi),\nonumber
		\end{align}
		where
		\begin{align*}
			\mathcal{B}_{kj}^1(\varphi)&\triangleq-\big\langle \partial_{\varphi_j} ( \partial_{I} K_\varepsilon^{\alpha_0}) (i_0(\varphi)),\partial_{\varphi_k} I_0(\varphi)\big\rangle_{\mathbb{R}^d}+\big\langle\partial_{\varphi_j} I_0(\varphi),\partial_{\varphi_k} ( \partial_{I} K_\varepsilon^{\alpha_0}) (i_0(\varphi))\big\rangle_{\mathbb{R}^d}, 
			\\ 
			\mathcal{B}_{kj}^2(\varphi)&\triangleq-\big\langle \partial_{\varphi_j} ( \partial_{\vartheta} K_\varepsilon^{\alpha_0}) (i_0(\varphi)),\partial_{\varphi_k} \vartheta_0(\varphi)\big\rangle_{\mathbb{R}^d}+\big\langle\partial_{\varphi_j} \vartheta_0(\varphi),\partial_{\varphi_k} ( \partial_{\vartheta} K_\varepsilon^{\alpha_0}) (i_0(\varphi))\big\rangle_{\mathbb{R}^d}, 
			\\ 
			\mathcal{B}_{kj}^3(\varphi)&\triangleq \big\langle \partial_{\varphi_j}  z_0(\varphi) ,\partial_{\varphi_k}  ( \nabla_{z} K_\varepsilon^{\alpha_0}) (i_0(\varphi))\big\rangle_{L^2(\mathbb{T})\times L^2(\mathbb{T})}
			-\big\langle \partial_{\varphi_j} ( \nabla_{z} K_\varepsilon^{\alpha_0}) (i_0(\varphi)) ,\partial_{\varphi_k}  z_0(\varphi)\big\rangle_{L^2 \times L^2}.
		\end{align*}
		Straightforward computations leads to
		\begin{align*}
			\mathcal{B}_{kj}^1(\varphi)&
			=-\big\langle  ( \partial_{I\vartheta} K_\varepsilon^{\alpha_0}) (i_0(\varphi)) \partial_{\varphi_j}\vartheta_0(\varphi),\partial_{\varphi_k} I_0(\varphi)\big\rangle_{\mathbb{R}^d}
			-\big\langle  ( \partial_{zI} K_\varepsilon^{\alpha_0}) (i_0(\varphi))\partial_{\varphi_j} z_0 (\varphi),\partial_{\varphi_k} I_0(\varphi)\big\rangle_{\mathbb{R}^d}
			\\ & \quad 
			+\big\langle\partial_{\varphi_j} I_0(\varphi), ( \partial_{I\vartheta} K_\varepsilon^{\alpha_0}) (i_0(\varphi))\partial_{\varphi_k}\vartheta_0(\varphi)\big\rangle_{\mathbb{R}^d}
			+\big\langle\partial_{\varphi_j} I_0(\varphi),( \partial_{I z} K_\varepsilon^{\alpha_0}) (i_0(\varphi)) \partial_{\varphi_k} z_0(\varphi)\big\rangle_{\mathbb{R}^d},\nonumber
			\\ 
			\mathcal{B}_{kj}^2(\varphi)&=-\big\langle  ( \partial_{I\vartheta} K_\varepsilon^{\alpha_0}) (i_0(\varphi)) \partial_{\varphi_j}I_0(\varphi),\partial_{\varphi_k} \vartheta_0(\varphi)\big\rangle_{\mathbb{R}^d}
			-\big\langle  ( \partial_{z\vartheta} K_\varepsilon^{\alpha_0}) (i_0(\varphi))\partial_{\varphi_j} z_0 (\varphi),\partial_{\varphi_k} \vartheta_0(\varphi)\big\rangle_{\mathbb{R}^d}
			\\ &
			\quad+\big\langle\partial_{\varphi_j} \vartheta_0(\varphi), ( \partial_{I\vartheta} K_\varepsilon^{\alpha_0}) (i_0(\varphi))\partial_{\varphi_k}I_0(\varphi)\big\rangle_{\mathbb{R}^d}
			+\big\langle\partial_{\varphi_j} \vartheta_0(\varphi),( \partial_{\vartheta z} K_\varepsilon^{\alpha_0}) (i_0(\varphi)) \partial_{\varphi_k} z_0(\varphi)\big\rangle_{\mathbb{R}^d},\nonumber
			\\
			\mathcal{B}_{kj}^3(\varphi)&=
			\big\langle  (\partial_I \nabla_{z} K_\varepsilon^{\alpha_0}) (i_0(\varphi)) \partial_{\varphi_k}I_0(\varphi) ,\partial_{\varphi_j}  z_0(\varphi)\big\rangle_{L^2(\mathbb{T})\times L^2(\mathbb{T})}
			\\
			&\quad-\big\langle \partial_{\varphi_k}  z_0(\varphi) , ( \partial_I \nabla_{z} K_\varepsilon^{\alpha_0}) (i_0(\varphi)) \partial_{\varphi_j}I_0(\varphi)\big\rangle_{L^2(\mathbb{T})\times L^2(\mathbb{T})}
			\\ &\quad+\big\langle( \partial_\vartheta \nabla_{ z} K_\varepsilon^{\alpha_0}) (i_0(\varphi)) \partial_{\varphi_k}\vartheta_0(\varphi),\partial_{\varphi_j}  z_0(\varphi)\big\rangle_{L^2(\mathbb{T})\times L^2(\mathbb{T})}
			\\
			&\quad-\big\langle \partial_{\varphi_k}  z_0(\varphi) ,( \partial_\vartheta \nabla_{ z} K_\varepsilon^{\alpha_0}) (i_0(\varphi)) \partial_{\varphi_j}\vartheta_0(\varphi)\big\rangle_{L^2(\mathbb{T})\times L^2(\mathbb{T})}.\nonumber
		\end{align*}
		Combining the last three  identities we obtain
		$$
		\mathcal{B}_{kj}^1(\varphi)+\mathcal{B}_{kj}^2(\varphi)+\mathcal{B}_{kj}^3(\varphi)=0.
		$$
		This with  \eqref{omdphiakj} concludes the proof of the lemma.
	\end{proof}
	We define the sequence $(N_n)_{n\in\N\cup\{-1\}}$ as
	\begin{equation}\label{def geo Nn}
		N_{-1}\triangleq 1,\qquad \forall n\in\mathbb{N},\quad N_{n}\triangleq N_{0}^{\left(\frac{3}{2}\right)^{n}},\qquad \hbox{with} \qquad N_0\geqslant2.
	\end{equation}
	The following lemma  is proved in \cite[Lemma 5.3]{BM18} and \cite[Lemma 6.2]{HHM21}.
	\begin{lem}\label{lem:est-akj}
		The coefficients $\mathcal{A}_{jk}$, defined in Lemma  $\ref{lemma1est}$, decomposes as
		\begin{equation}\label{Akj decomposition}
			\mathcal{A}_{kj}=\mathcal{A}_{kj}^{(n)}+\mathcal{A}_{kj}^{(n),\perp},\qquad\hbox{with} \qquad \mathcal{A}_{kj}^{(n)}\triangleq \Pi_{N_n}\mathcal{A}_{kj}\qquad \hbox{and}\qquad \mathcal{A}_{kj}^{(n),\perp}\triangleq \Pi_{N_n}^\perp\mathcal{A}_{kj}.
		\end{equation}
		In addition, the following properties hold true.
		\begin{enumerate}
			\item The function  $\mathcal{A}_{kj}^{(n),\perp}$ satisfies  for any $s\in \mathbb{R}$,
			\begin{equation*}
				\forall \,\mathtt{b}\geqslant 0,\quad\| \mathcal{A}_{k j}^{(n),\perp} \|_{s}^{q,\gamma,\mathtt{m}} \lesssim N_n^{-\mathtt{b}}  \|  {\mathfrak I}_0 \|_{s+1+\mathtt{b}}^{q,\gamma,\mathtt{m}}.
			\end{equation*}
			\item There exist functions $\mathcal{A}_{kj}^{(n),\textnormal{ext}}$ defined for any $(b,\omega)\in \mathcal{O}$ and satisfying, 
			for any $s\geqslant s_0$, the estimate
			\begin{equation*}
				\|  \mathcal{A}_{k j}^{(n),\textnormal{ext}} \|_{s}^{q,\gamma,\mathtt{m}} \lesssim \gamma^{-1}
				\big(\| Z \|_{s+\tau_1(q + 1)+1 }^{q,\gamma,\mathtt{m}} + \| Z \|_{s_0+1}^{q,\gamma,\mathtt{m}} \|  {\mathfrak I}_0 \|_{s+\tau_1(q + 1) +1}^{q,\gamma,\mathtt{m}}\big)\,.
			\end{equation*}
			Moreover,  $\mathcal{A}_{k j}^{(n),\textnormal{ext}}$ coincides with $\mathcal{A}_{k j}^{(n)}$ on the Cantor set
			\begin{equation}\label{DC tau gamma N}
				\mathtt {DC}_{N_n} (\gamma, \tau_1) \triangleq \bigcap_{l\in\mathbb{Z}^{d}\setminus\{0\}\atop|l|\leqslant N_n}\Big\{ \omega \in \mathbb{R}^d\quad\textnormal{s.t.}\quad|\omega \cdot l | \geqslant \tfrac{\gamma}{\langle l \rangle^{\tau_1}}\Big\}. \, 
			\end{equation}
		\end{enumerate}
	\end{lem}

	\subsection{Construction of an approximate inverse} 
	According to Proposition \ref{Proposition-Conjugation}-{\rm (ii)} and Lemma \ref{lem:est-akj}, the error term $\mathbb{E}$  is zero at an exact solution, up to an exponentially small remainder on the Cantor set $\mathtt {DC}_{N_n} (\gamma, \tau_1) $. Therefore,  in order to find an approximate inverse of the linear operator in \eqref{Id-conj} it is sufficient to almost invert the operator $\mathbb{D}$, which is triangular. More precisely, we first invert the action-component equation, in the  linear system $\mathbb{D}[\widehat u]=(g_1,g_2,g_3)$, which is decoupled from the other equations, 
	$$
	{\omega\cdot\partial_\varphi  \widehat y=g_2-\mathcal{B}(\varphi)\widehat \alpha.}
	$$
	Then, we shall solve the last normal-component equation
	$$
	\omega\cdot\partial_\varphi \widehat w- 
	\mathcal{J}  K_{02}(\varphi)\widehat w =g_3+\mathcal{J}\big[K_{11}(\varphi) \widehat y +L_2^{\top}(\varphi)\widehat\alpha   \big].
	$$
	For this aim we need to find an approximate  right inverse of the linearized operator in the normal direction 
	\begin{equation}\label{def hat L}
		\widehat{\mathcal{L}}\triangleq \Pi_{\overline{\mathbb{S}}_0}^\bot \big(\omega\cdot \partial_\varphi   - 
		\mathcal{J}  K_{02}(\varphi) \big)\Pi_{\overline{\mathbb{S}}_0}^\bot
	\end{equation}
	when the set of parameters is restricted to a Cantor-like set. Here the projector $\Pi_{\overline{\mathbb{S}}_0}^\bot$ is the one defined in \eqref{proj-nor1}. Finally, we shall solve the first equation in $\mathbb{D}[\widehat u]=(g_1,g_2,g_3)$ after  choosing $\widehat \alpha $ in such way we get zero average in the equation.

	The following proposition gives a brief statement about the invertibility in the normal direction; the construction of an approximate right inverse of the operator $\widehat{\mathcal{L}}$ is the subject of Section \ref{reduction}  and a precise statement with a detailed description of Cantor like sets,  see  Proposition~\ref{prop inv linfty}.
	
	\begin{prop}\label{thm:inversion of the linearized operator in the normal directions}
		Given the conditions  \eqref{setting tau1 and tau2}, \eqref{init Sob cond}, \eqref{p-RR} and \eqref{sml-RR}. There exists $\sigma_5\triangleq \sigma_5(\tau_1,\tau_2,q,d)>0$   such that if 
		\begin{equation*}
			\|\mathfrak{I}_0\|_{s_h+\sigma_5}^{q,\gamma,\mathtt{m}}\leqslant 1,
		\end{equation*}
		then there exists a family of  linear operators $\big(\widehat{\mathtt{T}}_{n}\big)_{n\in\mathbb{N}}$  defined in $\mathcal{O}$ and  satisfying the estimate
		\begin{equation*}
			\forall \, s\in\,[ s_0, S],\quad\sup_{n\in\mathbb{N}}\|\widehat{\mathtt{T}}_{n}\rho\|_{s}^{q,\gamma ,\mathtt{m}}\lesssim\gamma^{-1}\left(\|\rho\|_{s+\sigma_5}^{q,\gamma ,\mathtt{m}}+\| \mathfrak{I}_{0}\|_{s+\sigma_5}^{q,\gamma ,\mathtt{m}}\|\rho\|_{s_{0}+\sigma_5}^{q,\gamma,\mathtt{m}}\right)
		\end{equation*}
		and,  for any $n\in\mathbb{N}$, we have the following splitting
		$$\widehat{\mathcal{L}}=\widehat{\mathtt{L}}_{n}+\widehat{\mathtt{R}}_{n},\qquad\textnormal{with}\qquad\widehat{\mathtt{L}}_{n}\widehat{\mathtt{T}}_{n}=\textnormal{Id},$$
		in a Cantor set
		$\mathtt{G}_n\triangleq \mathtt{G}_n(\gamma,\tau_{1},\tau_{2},i_{0})\subset \mathtt {DC}_{N_n} (\gamma, \tau_1) \times (b_*, b^*),$
		where the operators $\widehat{\mathtt{L}}_{n}$ and $\widehat{\mathtt{R}}_{n}$ are defined in  $\mathcal{O}$ and satisfy 
		\begin{align*}
			\forall s\in[s_{0},S],\quad& \sup_{n\in\mathbb{N}}\|\widehat{\mathtt{L}}_{n}\rho\|_{s}^{q,\gamma,\mathtt{m}}\lesssim\|\rho\|_{s+1}^{q,\gamma,\mathtt{m}}+\varepsilon\gamma^{-2}\|\mathfrak{I}_{0}\|_{s+\sigma_5}^{q,\gamma,\mathtt{m}}\|\rho\|_{s_{0}+1}^{q,\gamma,\mathtt{m}},\\
			\forall s\in[s_{0},S],\quad &\|\widehat{\mathtt{R}}_{n}\rho\|_{s_{0}}^{q,\gamma,\mathtt{m}}\lesssim N_{n}^{s_{0}-s}\gamma^{-1}\left(\|\rho\|_{s+\sigma_5}^{q,\gamma,\mathtt{m}}+\varepsilon\gamma^{-2}\|\mathfrak{I}_{0}\|_{s+\sigma_5}^{q,\gamma,\mathtt{m}}\|\rho\|_{s_{0}+\sigma_5}^{q,\gamma,\mathtt{m}}\right)\\
			&\qquad\qquad\qquad\quad+\varepsilon\gamma^{-3}N_{0}^{\mu_{2}}N_{n+1}^{-\mu_{2}}\|\rho\|_{s_{0}+\sigma_5}^{q,\gamma,\mathtt{m}}.
		\end{align*}
		
	\end{prop}
	
	The main goal is to find an approximate inverse to the operator $[D G_0({\mathtt u}_0)]^{-1}  d_{(i,\alpha)} {\mathcal F} (i_0,\alpha_0) D\widetilde G_0({\mathtt u}_0)$  in \eqref{Id-conj}. For this aim, since we require only  finitely many non-resonance conditions \eqref{DC tau gamma N}, for any $\omega\in\mathbb{R}^d$, we   decompose $\omega\cdot \partial_\varphi$ as
	\begin{equation}\label{omeg-phi-decomposition}
		\omega\cdot \partial_\varphi=\mathcal{D}_{(n)} +\mathcal{D}_{(n)}^{\perp},\qquad\mathcal{D}_{(n)}\triangleq \omega\cdot \partial_\varphi\,\Pi_{N_n}+ \Pi_{N_n,\mathtt{g}}^\perp, \qquad\mathcal{D}_{(n)}^{\perp}\triangleq  \omega\cdot \partial_\varphi\, \Pi_{N_n}^\perp- \Pi_{N_n, \mathtt{g}}^\perp,
	\end{equation}
	where
	$$
	\Pi_{N_n, \mathtt{g}}^\perp\sum_{l\in\Z^{d}\setminus\{0\}}h_{l} {\bf{e}}_{l}\triangleq\sum_{l\in\Z^{d}\setminus\{0\}\atop |l|> N_n} \mathtt{g}(l) h_{l} {\bf{e}}_{l}.
	$$
	and the function $\mathtt{g}:\mathbb{Z}^d\setminus \{0\}\to \{-1,1\}$ is defined, for all $l=(l_1,\cdots,l_d)\in \mathbb{Z}^d\setminus \{0\}$,   as the sign of the first non-zero component in the vector $l$. Thus, it satisfies 
	$$\forall l\in \mathbb{Z}^d\setminus \{0\} ,\quad\mathtt{g}(-l)=-\mathtt{g}(l).$$
	The projector  $\Pi_{N_n, \mathtt{g}}^\perp$ is used here instead of $\Pi_{N_n}^\perp$ in order to preserve the reversibility property. Then,
	according to Proposition~\ref{Proposition-Conjugation}, the identities \eqref{Akj decomposition}-\eqref{omeg-phi-decomposition} and  Proposition~\ref{thm:inversion of the linearized operator in the normal directions} we have the following decomposition 
	\begin{equation}\label{decomposition conj op}
		[D G_0({\mathtt u}_0)]^{-1}  d_{(i,\alpha)} {\mathcal F} (i_0,\alpha_0) D\widetilde G_0({\mathtt u}_0)={\mathbb D}_n+{\mathbb E}_n+ {\mathscr P}_n+{\mathscr Q}_n,
	\end{equation}
	with
	\begin{align*}
		&{\mathbb D}_n [\widehat \phi, \widehat y, \widehat w, \widehat \alpha ]  \triangleq
		\left(
		\begin{array}{c}
			\mathcal{D}_{(n)}  \widehat\phi-K_{20}\widehat y-K_{11}^\top \widehat w -L_1^{\top}\widehat \alpha
			\\
			\mathcal{D}_{(n)} \widehat y+ \mathcal{B} \widehat \alpha \\
			\widehat{\mathtt{L}}_{\omega,n} \widehat w -\mathcal{J}\big[K_{11} \widehat y+ L_2^{\top} \widehat\alpha   \big] 
		\end{array}
		\right),
		\\ &{\mathbb E}_{n} [\widehat \phi, \widehat y, \widehat w]  \triangleq
		[D G_0({\mathtt u}_0)]^{-1}    [\partial_\varphi Z] \widehat\phi+\left(
		\begin{array}{c}
			0 
			\\
			\mathcal{A}^{(n)}\big[K_{20} \widehat y+K_{11}^\top \widehat w\big]-R_{10}\widehat y-R_{01}\widehat w
			\\
			0   
		\end{array}
		\right) 
		\\
		&{\mathscr P}_n [\widehat \phi, \widehat y, \widehat w ]  \triangleq
		\left(
		\begin{array}{c}
			\mathcal{D}_{(n)}^{\perp}  \widehat\phi
			\\
			\mathcal{D}_{(n)}^{\perp} \widehat y + \mathcal{A}^{(n),\perp}\big[K_{20} \widehat y+K_{11}^\top \widehat w\big]
			\\
			0
		\end{array}
		\right),\quad {\mathscr Q}_n  [\widehat \phi, \widehat y, \widehat w ]  \triangleq
		\left(
		\begin{array}{c}
			0
			\\
			0
			\\
			\widehat{\mathtt{R}}_n[ \widehat  w]
		\end{array}
		\right),
	\end{align*}
	where $\mathcal{A}^{(n)}$ and $\mathcal{A}^{(n),\perp}$ are the matrices with coefficients 
	$\mathcal{A}_{kj}^{(n)}$ and $\mathcal{A}_{kj}^{(n),\perp}$ respectively, see \eqref{Akj decomposition}.
	We define the linear operator $\mathbb{L}_{\textnormal{ext}}$ as 
	\begin{equation}\label{def Lext}
		\mathbb{L}_{\textnormal{ext}}\triangleq \mathbb{D}_n+{\mathbb E}_{n}^{{\rm ext}}+{\mathscr P}_n+{\mathscr Q}_n,
	\end{equation}
	where the operator  ${\mathbb E}_{n}^{{\rm ext}}$ vanishes at  exact solutions on the whole set of parameters $\mathcal{O}$ and it is given by
	\begin{align*}
		{\mathbb E}_{n}^{{\rm ext}} [\widehat \phi, \widehat y, \widehat w ] & \triangleq
		[D G_0({\mathtt u}_0)]^{-1}    [\partial_\varphi Z] \widehat\phi+\left(
		\begin{array}{c}
			0 
			\\
			\mathcal{A}^{(n),\textnormal{ext}}\big[K_{20}(\varphi) \widehat y+K_{11}^\top \widehat w\big]-R_{10}\widehat y-R_{01}\widehat w
			\\
			0   
		\end{array}
		\right), 
	\end{align*}
	with $\mathcal{A}^{(n),\textnormal{ext}}$ is the matrix with coefficients 
	$\mathcal{A}_{kj}^{(n),\textnormal{ext}}$, see \eqref{Akj decomposition}.
	The operator $\mathbb{L}_{\textnormal{ext}}$ is  defined on the whole set $\mathcal{O}$ and, by construction,  coincides with the linear operator in \eqref{decomposition conj op} on the Cantor set $\mathtt{G}_n$,
	\begin{equation}\label{lext-f}
		\forall (b,\omega) \in \mathtt{G}_n,\quad\mathbb{L}_{\textnormal{ext}}=[D G_0({\mathtt u}_0)]^{-1}  d_{(i,\alpha)} {\mathcal F} (i_0,\alpha_0) D \tilde G_0({\mathtt u}_0).
	\end{equation}

	The following proposition  shows that the principal term  $\mathbb{D}_n$ has an exact inverse. Its proof can be found in  \cite[Prop. 6.3]{HHM21}
	\begin{prop}\label{prop:decomp-lin}
		Given the conditions  \eqref{setting tau1 and tau2}, \eqref{init Sob cond}, \eqref{p-RR} and \eqref{sml-RR}. There exists $\sigma_6\triangleq \sigma_6(\tau_1,\tau_2,q,d) >0$   such that if 
		\begin{equation*}
			\|\mathfrak{I}_0\|_{s_h+\sigma_6}^{q,\gamma,\mathtt{m}}\leqslant 1,
		\end{equation*}
		then there exist a family of operators $\big([{\mathbb D}_n]_{\textnormal{ext}}^{-1}\big)_n$ such that for all $ g \triangleq (g_1, g_2, g_3) $ 
		satisfying the reversibility and $\mathtt{m}$-fold symmetry properties 
		\begin{equation}\label{symmetry g1 g2 g3}
			g_1(\varphi) = g_1(- \varphi),\qquad g_2(\varphi) = - g_2(- \varphi),\qquad g_3(\varphi) = - ({\mathcal S} g_3)(\varphi),\qquad(\mathscr{T}_{\mathtt{m}}g_3)(\varphi)=g_3(\varphi), 
		\end{equation}
		the function 
		$ [{\mathbb D}_n]_{\textnormal{ext}}^{-1} g $ 
		satisfies the estimate, for all $s_0 \leqslant s \leqslant S$,
		\begin{equation*} 
			\| [{\mathbb D}_n]_{\textnormal{ext}}^{-1}g \|_{s}^{q,\gamma,\mathtt{m}}
			\lesssim \gamma^{-1} \big( \| g \|_{s + \sigma_6}^{q,\gamma,\mathtt{m}}
			+  \| {\mathfrak I}_0  \|_{s + \sigma_6}^{q,\gamma,\mathtt{m}}
			\| g \|_{s_0 + \sigma_6}^{q,\gamma,\mathtt{m}}  \big)
		\end{equation*}
		and for all $(b,\omega) \in \mathtt{G}_n$ one has
		$$
		{\mathbb D}_n [{\mathbb D}_n]_{\textnormal{ext}}^{-1} =\textnormal{Id}.
		$$	
	\end{prop}
	Coming back to the linear operator  $d_{i,\alpha}\mathcal{F}(i_{0},\alpha_{0})$,	according to  \eqref{def Lext} and 
	\eqref{lext-f}, on the Cantor set $ \mathtt{G}_n, $ we have the decomposition
	\begin{equation*}
		\begin{aligned}
			d_{i,\alpha}\mathcal{F}(i_{0},\alpha_{0})
			&=DG_{0}({\mathtt u}_0) \, {\mathbb{D}}_n\, [D\widetilde{G}_{0}({\mathtt u}_0)]^{-1}+ DG_{0}({\mathtt u}_0) \, {\mathbb E}_{n}^{{\rm ext}} \, [D\widetilde{G}_{0}({\mathtt u}_0)]^{-1}\\ &\quad+DG_{0}({\mathtt u}_0)\,   {\mathscr P}_n\, [D\widetilde{G}_{0}({\mathtt u}_0)]^{-1}+DG_{0}({\mathtt u}_0) \, {\mathscr Q}_n\, [D\widetilde{G}_{0}({\mathtt u}_0)]^{-1}.
		\end{aligned}
	\end{equation*}
	Applying the operator 
	\begin{equation}\label{def inverse T} 
		{\rm T}_0 \triangleq {\rm T}_0(i_0) \triangleq D { \widetilde G}_0({\mathtt u}_0)\, [{\mathbb D}_n]_{\textnormal{ext}}^{-1}\,[D G_0({\mathtt u}_0)]^{-1} 
	\end{equation} 
	to the right of the last identity we get for all $(b,\omega)\in  \mathtt{G}_n,$
	\begin{align*}
	d_{i,\alpha}\mathcal{F}(i_{0},\alpha_{0}) {\rm T}_{0}-\textnormal{Id}=\mathcal{E}_1^{(n)}+\mathcal{E}_2^{(n)}+\mathcal{E}_3^{(n)}\quad \textnormal{with}\quad\begin{array}[t]{rcl}
			&\mathcal{E}_1^{(n)}\triangleq DG_{0}({\mathtt u}_0) \,  {\mathbb E}_{n}^{{\rm ext}} \, [D\widetilde{G}_{0}({\mathtt u}_0)]^{-1}{\rm T}_{0},
		\\
		&\mathcal{E}_2^{(n)}\triangleq DG_{0}({\mathtt u}_0)\,  {\mathscr P}_n\, [D\widetilde{G}_{0}({\mathtt u}_0)]^{-1}{\rm T}_{0},
		\\
		&\mathcal{E}_3^{(n)}\triangleq DG_{0}({\mathtt u}_0)\,  {\mathscr Q}_n\, [D\widetilde{G}_{0}({\mathtt u}_0)]^{-1}{\rm T}_{0}.
		\end{array}
			\end{align*}
	Consequently, the operator ${\rm T}_0$ is an approximate right inverse for $d_{i,\alpha} {\mathcal F}(i_0,\alpha_0)$. In particular, we have the following result, whose proof is similar to \cite[Theorem 5.1]{HR21}.
	
	\begin{theo}  \label{theo appr inv}
		{\bf (Approximate inverse)}
		Let $(\gamma,q,d,\tau_{1},s_{0},\mu_2,s_h,S)$ satisfy \eqref{setting tau1 and tau2}--\eqref{init Sob cond} and \eqref{p-RR}--\eqref{sml-RR}. There exists $ { \overline\sigma}= { \overline\sigma}(\tau_1,\tau_2,d,q)>0$   such that if 
		\begin{equation}\label{bnd frkIn-final}
			\|\mathfrak{I}_0\|_{s_h+\overline\sigma}^{q,\gamma,\mathtt{m}}\leqslant 1,
		\end{equation}
		then for smooth $ g = (g_1, g_2, g_3) $, satisfying \eqref{symmetry g1 g2 g3},  
		the operator $ {\rm T}_0  $ defined in \eqref{def inverse T} is reversible, $\mathtt{m}$-fold preserving and satisfies 
		\begin{equation}\label{tame T0}
			\forall s\in [s_0,S],\quad \| {\rm T}_0 g\|_{s}^{q,\gamma,\mathtt{m}}\lesssim\gamma^{-1}\left(\|g\|_{s+{\overline\sigma}}^{q,\gamma,\mathtt{m}}+\|\mathfrak{I}_{0}\|_{s+{\overline\sigma}}^{q,\gamma,\mathtt{m}}\|g\|_{s_{0}+\overline{\sigma}}^{q,\gamma,\mathtt{m}}\right).
		\end{equation}
		Moreover  ${\rm T}_0$ is an almost-approximate  
		right 
		inverse of $d_{i, \alpha} 
		\mathcal{ F}(i_0, \alpha_0)$ on the Cantor set $ \mathtt{G}_n$. More precisely,   for all $(b,\omega)\in \mathtt{G}_n $ one has
		\begin{equation}\label{splitting of approximate inverse}
			d_{i,\alpha} \mathcal{ F} (i_0,\alpha_0)  {\rm T}_0
			- {\rm Id} = \mathcal{E}^{(n)}_1+\mathcal{E}^{(n)}_2+\mathcal{E}^{(n)}_3,
		\end{equation}
		where the operators $\mathcal{E}^{(n)}_1$, $\mathcal{E}^{(n)}_2$ and $\mathcal{E}^{(n)}_3$ are defined in the whole set $\mathcal{O}$ with the estimates
		\begin{align}
			\|{\mathcal{E}_1^{(n)}} \rho \|_{s_0}^{q,\gamma,\mathtt{m}} & \lesssim  \gamma^{-1 } \| \mathcal{ F}(i_0, \alpha_0) \|_{s_0 +\overline\sigma}^{q,\gamma,\mathtt{m}} \|\rho\|_{s_0 + \overline\sigma}^{q,\gamma,\mathtt{m}},\label{calE1}
			\\
			\forall\,  \mathtt{b}\geqslant 0,
			\quad \| \mathcal{E}_2^{(n)} \rho \|_{s_0}^{q,\gamma,\mathtt{m}}& \lesssim 
			\gamma^{-1} N_n^{-\mathtt{b}} \big(\|\rho\|_{s_0 +\overline\sigma +\mathtt{b}}^{q,\gamma,\mathtt{m}}+
			\|\mathfrak{I}_{0}\|_{s_0+\overline\sigma+\mathtt{b}}^{q,\gamma,\mathtt{m}}\big\|\rho\|_{s_0+\overline\sigma}^{q,\gamma,\mathtt{m}}\big),\label{calE2}
			\\
			\forall\, \mathtt{b}\in [0,S],
			\quad \| \mathcal{E}_3^{(n)} \rho \|_{s_0}^{q,\gamma,\mathtt{m}}& \lesssim N_n^{-\mathtt{b}}\gamma^{-2}\Big( \|\rho\|_{s_0+\mathtt{b}+\overline\sigma}^{q,\gamma,\mathtt{m}}+{\varepsilon\gamma^{-2}}\| \mathfrak{I}_{0}\|_{s_0+\mathtt{b}+\overline\sigma}^{q,\gamma,\mathtt{m}}\|\rho\|_{s_0+{\overline\sigma}}^{q,\gamma,\mathtt{m}} \Big)\label{calE3}\\ &\quad\quad+ \varepsilon\gamma^{-4}N_{0}^{{\mu}_{2}}{N_{n}^{-\mu_{2}}} \|\rho\|_{s_0+\overline\sigma}^{q,\gamma,\mathtt{m}}.\nonumber
		\end{align}
	\end{theo}
	
	\section{Reduction}\label{reduction}
	This section is devoted to the reducibility of the linearized operator associated to the nonlinear equation \eqref{nonlinear-func}, whose structure is detailed in Proposition \ref{prop:conjP}. The first main step is to conjugate it into a diagonal matrix Fourier multiplier using a suitable quasi-periodic symplectic change of coordinates as in \cite{HHM21,HR21}. The second part deals with  the asymptotic structure of the operator localized on the normal directions. In the last part, we focus on the remainder reduction. To formulate our statements we need to  introduce the following parameters.
	\begin{equation}\label{param}
		\begin{array}{ll} 
			s_{l}\triangleq s_0+\tau_1 q+\tau_1 +2,\qquad\qquad& \overline{\mu}_2\triangleq 4\tau_1 q+6\tau_1 +3, \\
			\overline{s}_{l}\triangleq s_l+\tau_2 q+\tau_2 , & \overline{s}_{h}\triangleq \frac{3}{2}\overline{\mu}_{2}+s_{l}+1
		\end{array}
	\end{equation}
	and 
	\begin{equation}\label{sigma-F}
		\sigma_{1}\triangleq s_0+\tau_1 q+2\tau_1+4,\qquad \sigma_2\triangleq\sigma_{1}+3.		
	\end{equation}
	Throughout this section and we shall work under the following assumption
	\begin{align}\label{ouvert-sym}
		\mathcal{O}\triangleq(b_*,b^*)\times \mathscr{U},\qquad\hbox{with}\qquad 0<b_*<b^*<1\qquad \hbox{and} \qquad \mathtt{m}\geqslant \mathtt{m}^*,
	\end{align}
	where $\mathtt{m}^*$ is defined in Corollary \ref{coro-equilib-freq}. The set $\mathscr{U}$ is an open subset of $\mathbb{R}^{d}$ containing the equilibrum frequency vector curve, namely, we choose
	\begin{equation}\label{def scrU}
		\mathscr{U}\triangleq B(0,R_0)\qquad\textnormal{s.t.}\qquad\omega_{\textnormal{Eq}}\big([b_*,b^*]\big)\subset B\big(0,\tfrac{R_0}{2}\big),\qquad R_0>0.
	\end{equation}
We denote 
	$$\mathbf{H}^s_{\perp,\m}\triangleq\mathbf{H}_{\mathtt{m}}^{s}\cap \mathbf{H}_{\overline{\mathbb{S}}_0}^{\perp}$$
	and equip this space with the same norm as Sobolev spaces.
	\subsection{Structure of the linearized operator restricted to the normal directions}
	Here, we present the structure of the linearized operator in the normal directions
	\begin{equation*}
		\widehat{\mathcal{L}}=\widehat{\mathcal{L}}(i_0)= \Pi_{\overline{\mathbb{S}}_0}^\bot \big(\omega\cdot \partial_\varphi   - 
		\mathcal{J}  K_{02}(\varphi) \big)\Pi_{\overline{\mathbb{S}}_0}^\bot
	\end{equation*} defined through \eqref{def hat L} and \eqref{def K02}, where $i_0=(\vartheta_0,I_0,z_0)$ is a $\mathtt{m}$-fold reversible torus (satisfying \eqref{parity solution}) and whose periodic component $\mathfrak{I}_0$  satisfy the smallness condition
	$$\|\mathfrak{I}_0\|_{s_{0}+2}^{q,\gamma,\mathtt{m}}\leqslant 1,$$  
	given in Lemma  \eqref{tame X per}. The linear operator $\widehat{\mathcal{L}}$ decomposes as a finite rank perturbation of the linearized operator associated with the original problem,  as the following shows. 
	We refer the reader to \cite[Prop. 6.1]{HR21} for a detailed proof that one can adapt to our matrix case.	We mention that the $\mathtt{m}$-fold symmetry property can also be easily tracked.	
			\begin{prop}\label{lemma-normal-s}
		Let $(\gamma,q,d,s_{0})$ satisfy \eqref{init Sob cond}. 
		Then the operator $\widehat{\mathcal{L}}$ defined in \eqref{def hat L} takes the form 
		$$\widehat{\mathcal{L}}=\Pi_{\overline{\mathbb{S}}_0}^{\perp}\left(\mathcal{L}-\varepsilon\partial_\theta\mathcal{R}\right)\Pi_{\overline{\mathbb{S}}_0}^{\perp},\qquad \mathcal{L}\triangleq\omega\cdot\partial_{\varphi}\mathbf{I}_{\mathtt{m}}+\mathfrak{L}_{\varepsilon r},\qquad \mathcal{R}\triangleq \begin{pmatrix} \mathcal{T}_{J_{1,1}}({r}) & \mathcal{T}_{J_{1,2}}({r})\\
			\mathcal{T}_{J_{2,1}}({r}) & \mathcal{T}_{J_{2,2}}({r})
		\end{pmatrix},
		$$
		where 
		$$
		\mathbf{I}_{\mathtt{m}}\triangleq \begin{pmatrix}
			\mathbb{I}_{\mathtt{m}} &0\\
			0& \mathbb{I}_{\mathtt{m}}
		\end{pmatrix}
		$$
		denotes the identity map of $L^2_{\mathtt{m}}(\mathbb{T}^{d+1})\times L^2_{\mathtt{m}}(\mathbb{T}^{d+1})$. The operator $\mathfrak{L}_{\varepsilon r}$ is defined in Proposition  \ref{prop:conjP} and from \eqref{aa-coord-00} we have 
		\begin{align*}r(\varphi)
			&=\mathbf{A}\big(\vartheta_{0}(\varphi),\, I_{0}(\varphi),z_0(\varphi)\big),
		\end{align*}
		with $\mathbf{A}$ as in \eqref{aa-coord-00}, supplemented with the reversibility and $\mathtt{m}$-fold properties
		\begin{equation*}
			r(-\varphi,-\theta)=r(\varphi,\theta)=r\big(\varphi,\theta+\tfrac{2\pi}{\mathtt{m}}\big).
		\end{equation*}
		Moreover, for  any $k,n\in\{1,2\},$ the operator $\mathcal{T}_{J_{k,n}}({r})$ is an integral  operator in the form \eqref{Top-op1},  whose  kernel
		$J_{k,n}(r)$ is $\mathtt{m}$-fold reversibility preserving. In addition, under the assumption
		\begin{equation}\label{frakI0 bnd}
			\|\mathfrak{I}_0\|_{s_0}^{q,\gamma,\m}\leqslant  1,	
		\end{equation} we have for all $s\geqslant s_{0}$, 
		\begin{enumerate}[label=(\roman*)]
			\item  the function $r$ satisfies the estimates,	
			\begin{equation}\label{esti r I0}
				\|r\|_{s}^{q,\gamma,\m}\lesssim 1+\|\mathfrak{I}_{0}\|_{s}^{q,\gamma,\m}
			\end{equation}
			and 
			\begin{equation}\label{esti r I0d}
				\|\Delta_{12}r\|_{s}^{q,\gamma,\m}\lesssim\|\Delta_{12}i\|_{s}^{q,\gamma,\m}+\| \Delta_{12}i\|_{s_0}^{q,\gamma,\m}\max_{\ell\in\{1,2\}}\|\mathfrak{I}_{\ell}\|_{s}^{q,\gamma,\m}.
			\end{equation}
			\item for any $k,n\in\{1,2\},$ the kernel $J_{k,n}$ satisfies the following estimates
			\begin{equation}\label{e-Jkn}
				\|J_{k,n}\|_{s}^{q,\gamma,\m}\lesssim 1+\|\mathfrak{I}_{0}\|_{s+3}^{q,\gamma,\m}
			\end{equation}
			and
			\begin{equation}\label{e-Jknd} 
				\|\Delta_{12}J_{k,n}\|_{s}^{q,\gamma,\m}\lesssim\|\Delta_{12}i\|_{s+3}^{q,\gamma,\m}+\|\Delta_{12}i\|_{s_0+3}^{q,\gamma,\m}\max_{\ell\in\{1,2\}}\|\mathfrak{I}_{\ell}\|_{s+3}^{q,\gamma,\m}.
			\end{equation}
			Here  $\displaystyle\mathfrak{I}_{\ell}(\varphi)=i_{\ell}(\varphi)-(\varphi,0,0), $ and  for any function $f$,   $\Delta_{12} f\triangleq f(i_1)-f(i_2)$ refers to the difference of $f$ taken  at two different states $i_1$ and $i_2$ satisfying \eqref{frakI0 bnd}. 
		\end{enumerate}
	\end{prop}
	
	\subsection{Reduction of the transport part}
	The main purpose is to reduce to  constant coefficients  the transport parts in the linearized operator, described in Proposition \ref{prop:conjP}. Notice that the transport operator is diagonal, therefore we shall reduce each scalar component  apart.  This  was done by a KAM iterative scheme in \cite{HHM21,HR21}, in the same spirit of the papers \cite{BBMH18,BM21,BFM21,FGMP19}.   We skip the proof of the following proposition  since it is the same as in \cite[Prop. 6.2]{HR21}, where the scheme is initialized by \eqref{f0+-}, \eqref{es-f0} and  \eqref{sml-r0}. Moreover, the persistence of  the $\mathtt{m}$-fold symmetry property can be  easily checked along the scheme.
	\begin{prop}\label{prop strighten}
		Given the conditions  \eqref{init Sob cond}--\eqref{setting tau1 and tau2} and \eqref{param}--\eqref{ouvert-sym}. Let $\upsilon \in(0,\tfrac{1}{q+1}]$.  For any $(\mu_2,\mathtt{p},s_h)$ satisfying 
		\begin{equation}\label{p-trs}
			\mu_{2}\geqslant \overline{\mu}_{2}, \qquad\mathtt{p}\geqslant 0,\qquad s_{h}\geqslant\max\left(\tfrac{3}{2}\mu_{2}+s_{l}+1,\overline{s}_{h}+\mathtt{p}\right),
		\end{equation}
		there exists $\varepsilon_{0}>0$ such that if
		\begin{equation}\label{sml-trs}
			\varepsilon\gamma^{-1}N_{0}^{\mu_{2}}\leqslant\varepsilon_{0}\qquad\textnormal{and}\qquad \|\mathfrak{I}_{0}\|_{s_{h}+\sigma_{1}}^{q,\gamma,\mathtt{m}}\leqslant 1,
		\end{equation}
	 then for all $k\in\{1,2\}$ there exist
		$
		\mathtt{c}_k\triangleq \mathtt{c}_k(b,\omega,i_0)\in W^{q,\infty,\gamma }(\mathcal{O},\mathbb{C})$ and $\beta_{k}\triangleq \beta_k(b,\omega,i_0)\in W^{q,\infty,\gamma }(\mathcal{O},H_{\m}^{S})$
		such that the following results hold true.
		\begin{enumerate}[label=(\roman*)]
			\item The constants $\mathtt{c}_k$  satisfy the following estimate,
			\begin{equation}\label{sml-r0}
				\| \mathtt{v}_k-\mathtt{c}_k\|^{q,\gamma}\lesssim\varepsilon,
			\end{equation}
			where $\mathtt{v}_k$ is defined in \eqref{def V10 V20}.
			\item The transformation $\mathscr{B}_{k}^{\pm 1}$, related to the functions $\beta_{k}$ and $\widehat{\beta}_{k}$ through \eqref{def symplctik CVAR}-\eqref{mathscrB1}, are  $\mathtt{m}$-fold  reversibility preserving and satisfying  the following estimates: for all $s\in[s_{0},S]$ 
			\begin{align}\label{cont Bk}
				\|\mathscr{B}_{k}^{\pm 1}\rho\|_{s}^{q,\gamma ,\mathtt{m}}
				&\lesssim\|\rho\|_{s}^{q,\gamma ,\mathtt{m}}+\varepsilon\gamma ^{-1}\| \mathfrak{I}_{0}\|_{s+\sigma_{1}}^{q,\gamma ,\m}\|\rho\|_{s_{0}}^{q,\gamma ,\mathtt{m}},
				\\
				\label{sml betak I0}
				\|\widehat{\beta}_{k}\|_{s}^{q,\gamma,\mathtt{m}}\lesssim\|\beta_{k}\|_{s}^{q,\gamma ,\mathtt{m}}&\lesssim \varepsilon\gamma ^{-1}\left(1+\| \mathfrak{I}_{0}\|_{s+\sigma_{1}}^{q,\gamma ,\mathtt{m}}\right).
			\end{align}
			\item In the Cantor set
			\begin{align}\label{Cant-trs}
				\mathcal{O}_{\infty,n}^{\gamma,\tau_1}(i_{0})\triangleq& \bigcap_{ \underset{(l,j)\in\mathbb{Z}^d\times\mathbb{Z}_{\mathtt{m}}\backslash\{(0,0)\}\atop|l|\leqslant N_{n}}{k\in\{1,2\}}}\left\lbrace(b,\omega)\in \mathcal{O}\quad\textnormal{s.t.}\quad\big|\omega\cdot l+j\mathtt{c}_k(b,\omega)\big|>\tfrac{4\gamma^{\upsilon}\langle j\rangle}{\langle l\rangle^{\tau_1}}\right\rbrace
			\end{align}
			we have 
			\begin{equation}\label{Strighten Vk}
				\mathscr{B}_k^{-1}\Big(
				\omega\cdot\partial_{\varphi} \mathbb{I}_{\mathtt{m}}+\partial_\theta \big(\mathcal{V}_k(\varepsilon r)\, \cdot\big)\Big)\mathscr{B}_k=
				\omega\cdot\partial_{\varphi} \mathbb{I}_{\mathtt{m}}+\mathtt{c}_{k}\partial_{\theta} +\mathscr{E}_{n,k},
			\end{equation}
			where $\mathcal{V}_k$ are defined in Proposition \ref{prop:conjP}-1 and  $\mathscr{E}_{n,k}\triangleq \mathscr{E}_{n,k}(b,\omega,i_{0})$ are linear operators satisfying
			\begin{equation}\label{scr-Enk}
				\|\mathscr{E}_{n,k}\rho\|_{s_{0}}^{q,\gamma,\m}\lesssim\varepsilon N_{0}^{\mu_{2}}N_{n+1}^{-\mu_{2}}\|\rho\|_{s_{0}+2}^{q,\gamma,\m}.
			\end{equation}
			\item Given two tori $i_{1}$ and $i_{2}$ both satisfying \eqref{sml-trs} (replacing $\mathfrak{I}_{0}$ by $\mathfrak{I}_{1}$ or $\mathfrak{I}_{2}$), we have 
			\begin{align}\label{diff Vpm}
				\|\Delta_{12}\mathtt{c}_k\|^{q,\gamma}&\lesssim\varepsilon\| \Delta_{12}i\|_{\overline{s}_{h}+\sigma_{1}}^{q,\gamma,\mathtt{m}},
				\\ 
				\label{diff betak}
				\|\Delta_{12}\beta_{k}\|_{\overline{s}_{h}+\mathtt{p}}^{q,\gamma,\mathtt{m}}+\|\Delta_{12}\widehat{\beta}_{k}\|_{\overline{s}_{h}+\mathtt{p}}^{q,\gamma,\mathtt{m}}&\lesssim\varepsilon\gamma^{-1}\|\Delta_{12}i\|_{\overline{s}_{h}+\mathtt{p}+\sigma_{1}}^{q,\gamma,\mathtt{m}}.
			\end{align}
		\end{enumerate}
	\end{prop}
	Define the matrix operators 
	\begin{equation}\label{def BE}
		{\mathscr{B}}\triangleq \begin{pmatrix}
			\mathscr{B}_{1} & 0\\
			0 & \mathscr{B}_{2}
		\end{pmatrix}\qquad\hbox{and}  \qquad 			
		{\mathscr{E}}_{n}\triangleq \begin{pmatrix}
			\mathscr{E}_{n,1} & 0\\
			0 & \mathscr{E}_{n,2}
		\end{pmatrix},
	\end{equation}
	where $\mathscr{B}_{k}$ and $\mathscr{E}_{n,k}$ have  been defined in Proposition \ref{prop strighten}. Next, we plan to describe the action of the transformation $\mathscr{B}$ on the linearized operator introduced in Proposition \ref{prop:conjP} and derive some useful estimates.
	\begin{prop}\label{prop RTNL}
		Given the conditions  \eqref{setting tau1 and tau2}--\eqref{init Sob cond}, \eqref{param}, \eqref{ouvert-sym} and \eqref{p-trs}. Then, there exists $\varepsilon_{0}>0$ such that if
		\begin{equation}\label{sml-nl}
			\varepsilon\gamma^{-1}N_{0}^{\mu_{2}}\leqslant\varepsilon_{0}\qquad\textnormal{and}\qquad \|\mathfrak{I}_{0}\|_{q,s_{h}+\sigma_{2}}^{\gamma,\mathtt{m}}\leqslant 1,
		\end{equation}
		where $\sigma_2$ is given by \eqref{sigma-F},  
		then by restricting the parameters to  the Cantor set defined in \eqref{Cant-trs}  we get
		\begin{align}\label{b-1lb}
			\mathscr{L}\triangleq {\mathscr{B}}^{-1}\big(\omega\cdot\partial_{\varphi} {\mathbf{I}}_{\mathtt{m}}+\mathfrak{L}_{\varepsilon r}\big) {\mathscr{B}}&=\omega\cdot\partial_{\varphi} \mathbf{I}_{\mathtt{m}}+{\mathscr{D}}+{\mathscr{R}}+{\mathscr{E}}_{n},
		\end{align}
		where
		$$\mathscr{D}\triangleq \begin{pmatrix}
			\mathtt{c}_1\partial_{\theta}\, \cdot+\tfrac12\mathcal{H}+\partial_\theta  \mathcal{Q}\ast\cdot& 0\\
			0 &  \mathtt{c}_2\partial_{\theta}\, \cdot-\big(\tfrac12\mathcal{H}+\partial_\theta  \mathcal{Q}\ast\cdot)
		\end{pmatrix},$$
		and the operator ${\mathscr{R}}\triangleq {\mathscr{R}}({\varepsilon r})$ is a real, $\mathtt{m}$-fold and reversibility preserving matricial integral operator satisfying
		\begin{equation}\label{e-Rnl}
			\forall s\in[s_0,S],\quad\interleave{\mathscr{R}}\interleave_{s}^{q,\gamma,\mathtt{m}}\lesssim\varepsilon\gamma^{-1}\Big(1+\|\mathfrak{I}_0\|_{s+\sigma_{2}}^{q,\gamma,\mathtt{m}}\Big).
		\end{equation}
		Moreover, given two tori $i_{1}$ and $i_{2}$ both satisfying \eqref{sml-nl} (replacing $\mathfrak{I}_{0}$ by $\mathfrak{I}_{1}$ or $\mathfrak{I}_{2}$), we have 
		\begin{equation}\label{e-Rnld}
			\interleave\Delta_{12}{\mathscr{R}}\interleave_{\overline{s}_h+\mathtt{p}}^{q,\gamma,\mathtt{m}}\lesssim\varepsilon\gamma^{-1}\|\Delta_{12}i\|_{\overline{s}_h+\mathtt{p}+\sigma_{2}}^{q,\gamma,\mathtt{m}}.
		\end{equation}
	\end{prop}
	
	\begin{proof}
		From \eqref{def BE} and \eqref{defLr2} we may write
		\begin{align*}
			{\mathscr{B}}^{-1}\big(\omega\cdot\partial_{\varphi} \mathbf{I}_{\mathtt{m}}+\mathfrak{L}_{\varepsilon r}\big){\mathscr{B}}
			&=
			\begin{pmatrix}\mathscr{B}_1^{-1}\big(\omega\cdot\partial_{\varphi} \mathbb{I}_{\mathtt{m}}+\partial_\theta \big(\mathcal{V}_1(\varepsilon r)\, \cdot\big)\big)\mathscr{B}_1 &0\\ 0& \mathscr{B}_2^{-1}\big(	\omega\cdot\partial_{\varphi} \mathbb{I}_{\mathtt{m}}+\partial_\theta \big(\mathcal{V}_2(\varepsilon r)\, \cdot\big)\big)\mathscr{B}_2\end{pmatrix}
			\\
			&\quad +\begin{pmatrix}
				\tfrac{1}{2}\mathcal{H}+\partial_\theta \mathcal{Q}\ast\cdot & 0\\
				0 & -\tfrac{1}{2}\mathcal{H}-\partial_\theta \mathcal{Q}\ast\cdot
			\end{pmatrix}+\begin{pmatrix}
				\mathscr{R}_{1,1} & \mathscr{R}_{1,2}\\
				\mathscr{R}_{2,1} & \mathscr{R}_{2,2}
			\end{pmatrix},
		\end{align*}
		where
		\begin{align*}
			\quad \mathscr{R}_{k,k'}\triangleq \mathscr{B}_k^{-1}\partial_{\theta}\mathcal{T}_{\mathscr{K}_{k,k'}(\varepsilon r)}\mathscr{B}_{k'}+(-1)^{k+1}\delta_{k,k'}\Big(\tfrac{1}{2}\mathscr{B}_{k}^{-1}\mathcal{H}\mathscr{B}_{k}-\tfrac{1}{2}\mathcal{H}+\mathscr{B}_{k}^{-1}\big(\partial_\theta \mathcal{Q}\ast\cdot\big)\mathscr{B}_{k}-\partial_\theta \mathcal{Q}\ast\cdot\Big)
		\end{align*}
		and  $\delta_{k,k'}$ denotes the usual Kronecker symbol. Putting together  \eqref{Strighten Vk} and  \eqref{defLr2} allows to get  in the Cantor set $\mathcal{O}_{\infty,n}^{\gamma,\tau_{1}}(i_{0})$ the decomposition \eqref{b-1lb} 
		with
		\begin{align*}
			{\mathscr{R}}&=\begin{pmatrix}
				\mathscr{R}_{1,1} & \mathscr{R}_{1,2}\\
				\mathscr{R}_{2,1} & \mathscr{R}_{2,2}
			\end{pmatrix}.
		\end{align*} 
		The symmetry properties of $\mathcal{Q}$, $\beta$ and $\widehat{\beta}$, given by Proposition \ref{prop:conjP}-2 and Proposition \ref{prop strighten}-2, together with Lemma \ref{lem sym--rev} and Lemma \ref{lemma:conjug-Hilbert} imply that $\mathscr{B}_{k}^{-1}\mathcal{H}\mathscr{B}_{k}-\mathcal{H}$ and $\mathscr{B}_{k}^{-1}\mathcal{Q}\mathscr{B}_{k}-\mathcal{Q}$ are real, reversible and $\mathtt{m}$-fold preserving integral operators. In view of  Lemma \ref{lemma:conjug-Hilbert}, \eqref{sml betak I0}, \eqref{diff betak}, \eqref{sml-nl} and \eqref{sigma-F} the operator $\mathscr{B}_{k}^{-1}\mathcal{H}\mathscr{B}_{k}-\mathcal{H}$  is an integral operator and satisfies
		\begin{align}\label{e-Box H}
			\|\mathscr{B}_{k}^{-1}\mathcal{H}\mathscr{B}_{k}-\mathcal{H}\|_{\textnormal{\tiny{I-D}},s}^{q,\gamma,\mathtt{m}}&\lesssim\varepsilon\gamma^{-1}\left(1+\|\mathfrak{I}_0\|_{s+2+\sigma_{1}}^{q,\gamma,\mathtt{m}}\right),
			\\
			\label{e-Box Hd}
			\|\Delta_{12}(\mathscr{B}_{k}^{-1}\mathcal{H}\mathscr{B}_{k}-\mathcal{H})\|_{\textnormal{\tiny{I-D}},\overline{s}_h+\mathtt{p}}^{q,\gamma,\mathtt{m}}&\lesssim\varepsilon\gamma^{-1}\|\Delta_{12}i\|_{\overline{s}_h+\mathtt{p}+2+\sigma_{1}}^{q,\gamma,\mathtt{m}}.
		\end{align}		
		As for the term in $\mathcal{Q}$, we observe according to the notation \eqref{Top-op1},  \eqref{dcp calDeb0} and \eqref{ASYFR1-} that	$$\mathcal{Q}\ast\rho=\mathcal{T}_{\widehat{Q}}\,\rho,\qquad\hbox{with}\qquad \widehat{Q}(b,\varphi,\theta,\eta)\triangleq \mathcal{Q}(b,\theta-\eta)$$ 
		and
		$$\|\widehat{Q}\|_{s}^{q,\gamma,\mathtt{m}}\leqslant C(q,s),$$
		for some constant $C(q,s)>0$. Therefore, applying  Lemma \ref{lem CVAR kernel} together with \eqref {sml betak I0}, \eqref{diff betak} and the smallness condition \eqref{sml-nl} yield to  
		\begin{align}\label{e-Box Q}
			\|\mathscr{B}_{k}^{-1}(\partial_\theta \mathcal{Q}\ast\cdot)\mathscr{B}_{k}-\partial_\theta \mathcal{Q}\ast\cdot\|_{\textnormal{\tiny{I-D}},s}^{q,\gamma,\mathtt{m}}&\lesssim\varepsilon\gamma^{-1}\left(1+\|\mathfrak{I}_0\|_{s+\sigma_{1}+1}^{q,\gamma,\mathtt{m}}\right),
			\\
			\label{e-Box Qd}
			\|\Delta_{12}\big(\mathscr{B}_{k}^{-1}(\partial_\theta \mathcal{Q}\ast\cdot)\mathscr{B}_{k}-\partial_\theta \mathcal{Q}\ast\cdot\big)\|_{\textnormal{\tiny{I-D}},\overline{s}_h+\mathtt{p}}^{q,\gamma,\mathtt{m}}&\lesssim\varepsilon\gamma^{-1}\|\Delta_{12}i\|_{\overline{s}_h+\mathtt{p}+\sigma_{1}+1}^{q,\gamma,\mathtt{m}}.
		\end{align}
		As for the operator $\mathscr{B}_k^{-1}\partial_{\theta}\mathcal{T}_{\mathscr{K}_{k,k'}(\varepsilon r)}\mathscr{B}_{k'}$, first it is real, $\mathtt{m}$-fold preserving and reversible Toeplitz in time operator according to Lemma \ref{lem sym--rev} together with \eqref{sym scrKkn}, \eqref{sym-m scrKkn} and the symmetry properties of $\beta,\widehat{\beta}$, given by Proposition \ref{prop strighten}-2. In addition, using  \eqref{e-odsBtB}, Proposition \ref{prop:conjP}-3, \eqref{sml betak I0}, \eqref{esti r I0}, \eqref{sml-nl} and \eqref{sigma-F}, we deduce that
		\begin{align}\label{e-Box calT}
			\|\mathscr{B}_k^{-1}\partial_{\theta}\mathcal{T}_{\mathscr{K}_{k,k'}(\varepsilon r)}\mathscr{B}_{k'}\|_{\textnormal{\tiny{I-D}},s}^{q,\gamma,\mathtt{m}}&\lesssim \|\mathscr{K}_{k,k'}(\varepsilon r)\|_{s+1}^{q,\gamma,\mathtt{m}}+\|\mathscr{K}_{k,k'}({\varepsilon r})\|_{s_0}^{q,\gamma,\mathtt{m}}\,\max_{\ell\in\{1,2\}}\|\beta_\ell\|_{s+2}^{q,\gamma,\mathtt{m}}\nonumber\\
			&\lesssim\varepsilon\gamma^{-1}\left(1+\|\mathfrak{I}_0\|_{s+\sigma_1+2}^{q,\gamma,\mathtt{m}}\right).
		\end{align}
		Using Proposition \ref{prop:conjP}-3 supplemented by \eqref{esti r I0} and \eqref{esti r I0d}, we obtain
		\begin{align}\label{ed-scrKkner}
			\nonumber \|\Delta_{12}\mathscr{K}_{k,k'}({\varepsilon r})\|_{s}^{q,\gamma,\mathtt{m}}&\lesssim\varepsilon\|\Delta_{12}r\|_{s+1}^{q,\gamma,\m}+\varepsilon\|\Delta_{12}r\|_{s_0+1}^{q,\gamma,\m}\max_{\ell\in\{1,2\}}\|r_\ell\|_{s+1}^{q,\gamma,\m}\\
			&\lesssim \varepsilon\Big(\|\Delta_{12}i\|_{s+1}^{q,\gamma,\mathtt{m}}+\|\Delta_{12}i\|_{s_0+1}^{q,\gamma,\mathtt{m}}\max_{\ell\in\{1,2\}}\|\mathfrak{I}_{\ell}\|_{s+1}^{q,\gamma,\mathtt{m}}\Big).
		\end{align}
		Therefore, applying \eqref{diff12BoxBtB} together with \eqref{ed-scrKkner}, \eqref{sml betak I0}, \eqref{diff betak} and \eqref{sml-nl}, we get
		\begin{align}\label{e-Box calTd}
			\|\Delta_{12}\mathscr{B}_k^{-1}\partial_{\theta}\mathcal{T}_{\mathscr{K}_{k,k'}(\varepsilon r)}\mathscr{B}_{k'}\|_{\textnormal{\tiny{I-D}},\overline{s}_h+\mathtt{p}}^{q,\gamma,\mathtt{m}}\lesssim\varepsilon\gamma^{-1}\|\Delta_{12}i\|_{\overline{s}_h+\mathtt{p}+2+\sigma_{1}}^{q,\gamma,\mathtt{m}}.
		\end{align}
		Combining  \eqref{e-Box H}, \eqref{e-Box Q} and \eqref{e-Box calT}, we find 
		\begin{align*}
			\|{\mathscr{R}}_{k,k'}\|_{\textnormal{\tiny{I-D}},s}^{q,\gamma,\mathtt{m}}&\lesssim\varepsilon\gamma^{-1}\left(1+\|\mathfrak{I}_0\|_{s+2+\sigma_{1}}^{q,\gamma,\mathtt{m}}\right).
		\end{align*}
		Moreover, putting together \eqref{e-Box Hd}, \eqref{e-Box Qd} and \eqref{e-Box calTd} implies
		\begin{align*}
			\|\Delta_{12}{\mathscr{R}}_{k,k'}\|_{\textnormal{\tiny{I-D}},\overline{s}_{h}+\mathtt{p}}^{q,\gamma,\mathtt{m}}&\lesssim\varepsilon\gamma^{-1}\|\Delta_{12}i\|_{\overline{s}_{h}+\mathtt{p}+2+\sigma_{1}}^{q,\gamma,\mathtt{m}}.
		\end{align*}
		This proves the Proposition \ref{prop RTNL}.
	\end{proof}	
	\subsection{Localization into  the normal directions}
	We shall focus in this section on the   localization effects in the normal directions for the reduction of the transport part. For this aim, we consider the localized quasi-periodic symplectic change of coordinates defined by
	$${\mathscr{B}}_{\perp}\triangleq \Pi_{\overline{\mathbb{S}}_{0}}^{\perp}{\mathscr{B}}\Pi_{\overline{\mathbb{S}}_{0}}^{\perp}=\begin{pmatrix}
		\Pi_{1}^{\perp}\mathscr{B}_1\Pi_{1}^{\perp} & 0\\
		0 & \Pi_{2}^{\perp}\mathscr{B}_2\Pi_{2}^{\perp}
	\end{pmatrix},$$
	where the projectors are defined in \eqref{proj-nor1}-\eqref{PI1-PI2}.
	Then, the main result of this section reads as follows.
	\begin{prop}\label{prop proj nor dir}
		Let $(\gamma,q,d,\tau_1,s_{0},S,\mathtt{m})$ satisfy \eqref{setting tau1 and tau2}--\eqref{init Sob cond} and \eqref{ouvert-sym}.
		Let  $(\overline{\mu}_2,s_l,\overline{s}_h,\mu_2,\mathtt{p},s_h)$  satisfy \eqref{param} and  \eqref{p-trs}. 
		There exist $\varepsilon_0>0$ and $\sigma_{3}\triangleq\sigma_{3}(\tau_{1},q,d,s_{0})\geqslant\sigma_{2}$, where $\sigma_2$ is given by \eqref{sigma-F}, such that if 
		\begin{equation}\label{sml-pnor}
			\varepsilon\gamma^{-1}N_0^{\mu_2}\leqslant\varepsilon\qquad\textnormal{and}\qquad\|\mathfrak{I}_0\|_{s_h+\sigma_3}^{q,\gamma,\mathtt{m}}\leqslant 1,
		\end{equation}
		then the following assertions hold true.
		\begin{enumerate}[label=(\roman*)]
			\item The operators $\mathscr{B}_{\perp}^{\pm 1}$ satisfy the following estimate
			\begin{equation}\label{e-vectBnor}
				\|{\mathscr{B}}_{\perp}^{\pm 1}\rho\|_{s}^{q,\gamma ,\mathtt{m}}\lesssim\|\rho\|_{s}^{q,\gamma ,\mathtt{m}}+\varepsilon\gamma ^{-1}\| \mathfrak{I}_{0}\|_{s+\sigma_{3}}^{q,\gamma ,\mathtt{m}}\|\rho\|_{s_{0}}^{q,\gamma ,\mathtt{m}}.
			\end{equation}
			\item For any $n\in\mathbb{N}^{*},$ in the Cantor set $\mathcal{O}_{\infty,n}^{\gamma,\tau_1}(i_{0})$ introduced in \eqref{Cant-trs}, we have 
			\begin{equation}\label{reduction on nor}
				{\mathscr{B}}_{\perp}^{-1}\widehat{{\mathcal{L}}}{\mathscr{B}}_{\perp}={\mathscr{L}}_{0}+{\mathscr{E}}_{n}^0,\qquad
				{\mathscr{L}}_{0}\triangleq \omega\cdot\partial_{\varphi}\mathbf{I}_{\mathtt{m},\perp}+{\mathscr{D}}_{0}+{\mathscr{R}}_{0},
			\end{equation}
			where $\mathbf{I}_{\mathtt{m},\perp}\triangleq \Pi_{\overline{\mathbb{S}}_{0}}^{\perp}\mathbf{I}_{\mathtt{m}}$ 
			and ${\mathscr{D}}_{0}=\Pi_{\overline{\mathbb{S}}_{0}}^{\perp} {\mathscr{D}}_{0}\Pi_{\overline{\mathbb{S}}_{0}}^{\perp}$ is a reversible Fourier multiplier operator given by 
			$${\mathscr{D}}_{0}\triangleq \begin{pmatrix}
				\mathscr{D}_{0,1} & 0\\
				0 & \mathscr{D}_{0,2}
			\end{pmatrix},\qquad{\mathscr{D}}_{0,k}\triangleq \left(\ii\mu_{j,k}^{(0)}\right)_{j\in\Z_{\m}\backslash\overline{\mathbb{S}}_{0,k}},\qquad\mu_{-j,k}^{(0)}(b,\omega)=-\mu_{j,k}^{(0)}(b,\omega),$$
			with 
			\begin{equation}\label{mu0 r0}
				\mu_{j,k}^{(0)}(b,\omega,i_{0})\triangleq \Omega_{j,k}(b)+jr_{k}^{(0)}(b,\omega,i_{0}),\qquad r_{k}^{(0)}(b,\omega,i_{0})\triangleq \mathtt{c}_k(b,\omega)-\mathtt{v}_k(b)
			\end{equation}
			and such that
			\begin{equation}\label{e-ed-r0}
				\|r_{k}^{(0)}\|^{q,\gamma}\lesssim \varepsilon \qquad\textnormal{and}\qquad \|\Delta_{12}r_{k}^{(0)}\|^{q,\gamma}\lesssim \varepsilon \| \Delta_{12}i\|_{\overline{s}_{h}+\sigma_{1}}^{q,\gamma,\mathtt{m}}.
			\end{equation}
			Notice that the frequencies $\Omega_{j,k}(b)$ are defined in \eqref{omega jk b}.
			\item The operator ${\mathscr{E}}_{n}^0$ satisfies the following estimate
			\begin{equation}\label{e-scrEn0}
				\|{\mathscr{E}}_{n}^0\rho\|_{s_{0}}^{q,\gamma,\mathtt{m}}\lesssim \varepsilon N_{0}^{\mu_{2}}N_{n+1}^{-\mu_{2}}\|\rho\|_{s_{0}+2}^{q,\gamma,\mathtt{m}}.
			\end{equation}
			\item The operator ${\mathscr{R}}_{0}$ is an  $\mathtt{m}$-fold preserving and reversible Toeplitz in time matricial operator satisfying 
			\begin{equation}\label{e-hyb-scrR0}
				\forall s\in [s_{0},S],\quad \interleave{\mathscr{R}}_{0}\interleave_{s}^{q,\gamma,\mathtt{m}}\lesssim\varepsilon\gamma^{-1}\left(1+\| \mathfrak{I}_{0}\|_{s+\sigma_{3}}^{q,\gamma,\mathtt{m}}\right)
			\end{equation}
			and
			\begin{equation}\label{ed-hyb-scrR0}
				\interleave\Delta_{12}{\mathscr{R}}_{0}\interleave_{\overline{s}_{h}+\mathtt{p}}^{q,\gamma,\mathtt{m}}\lesssim\varepsilon\gamma^{-1}\| \Delta_{12}i\|_{\overline{s}_{h}+\mathtt{p}+\sigma_{3}}^{q,\gamma,\mathtt{m}}.
			\end{equation}
			\item The operator $\mathscr{L}_{0}$ satisfies
			\begin{equation}\label{e-scrL0}
				\forall s\in[s_{0},S],\quad\|{\mathscr{L}}_{0}\rho\|_{s}^{q,\gamma,\mathtt{m}}\lesssim\|\rho\|_{s+1}^{q,\gamma,\mathtt{m}}+\varepsilon\gamma ^{-1}\| \mathfrak{I}_{0}\|_{s+\sigma_{3}}^{q,\gamma,\mathtt{m}}\|\rho\|_{s_{0}}^{q,\gamma ,\mathtt{m}}.
			\end{equation}
		\end{enumerate}
	\end{prop}
	\begin{proof}
		\textbf{(i)} It is obtained using \eqref{cont Bk} and Lemma \ref{lem funct prop}-(i).\\
		\textbf{(ii)} The first estimate of \eqref{e-ed-r0} follows from \eqref{sml-r0} and the second one from \eqref{diff Vpm}. On the other hand,  using the expression of $\widehat{\mathcal{L}}$ detailed  in Proposition \ref{lemma-normal-s}, combined with  the decomposition $\textnormal{Id}=\Pi_{\mathbb{S}_{0}}+\Pi_{\overline{\mathbb{S}}_{0}}^{\perp}$  we write
		\begin{align*}
			\mathscr{B}_{\perp}^{-1}\widehat{\mathcal{L}}\mathscr{B}_{\perp}&=\mathscr{B}_{\perp}^{-1}\Pi_{\overline{\mathbb{S}}_{0}}^{\perp}({\mathcal{L}}-\varepsilon\partial_{\theta}\mathcal{R})\mathscr{B}_{\perp}
			\\
			&=\mathscr{B}_{\perp}^{-1}\Pi_{\overline{\mathbb{S}}_{0}}^{\perp}{\mathcal{L}}\mathscr{B}\Pi_{\overline{\mathbb{S}}_{0}}^{\perp}-\mathscr{B}_{\perp}^{-1}\Pi_{\overline{\mathbb{S}}_{0}}^{\perp}{\mathcal{L}}\Pi_{\mathbb{S}_{0}}\mathscr{B}\Pi_{\overline{\mathbb{S}}_{0}}^{\perp}-\varepsilon\mathscr{B}_{\perp}^{-1}\Pi_{\overline{\mathbb{S}}_{0}}^{\perp}\partial_{\theta}\mathcal{R}\mathscr{B}_{\perp}.
		\end{align*}
		By virtue of
		Proposition \ref{prop RTNL} 
		one has in the Cantor set $\mathcal{O}_{\infty,n}^{\gamma,\tau_{1}}(i_{0})$,
		$$
		{\mathcal{L}}\mathscr{B}=\mathscr{B}\mathscr{L}
		$$
		and therefore, using also that $\mathscr{B}_{\perp}^{-1}\Pi_{\overline{\mathbb{S}}_0}^{\perp}=\mathscr{B}_{\perp}^{-1},$ we get
		\begin{align*}
			\mathscr{B}_{\perp}^{-1}\widehat{\mathcal{L}}\mathscr{B}_{\perp}&=\mathscr{B}_{\perp}^{-1}\mathscr{B}\mathscr{L}\Pi_{\overline{\mathbb{S}}_{0}}^{\perp}-\mathscr{B}_{\perp}^{-1}\mathcal{L}\Pi_{\mathbb{S}_{0}}\mathscr{B}\Pi_{\overline{\mathbb{S}}_{0}}^{\perp}
			-\varepsilon\mathscr{B}_{\perp}^{-1}\partial_{\theta}\mathcal{R}\mathscr{B}_{\perp}.
		\end{align*}
		Thus, using  \eqref{b-1lb} we deduce that
		\begin{align*}
			\mathscr{B}_{\perp}^{-1}\mathscr{B}\mathscr{L}\Pi_{\overline{\mathbb{S}}_{0}}^{\perp}&=\left(\omega\cdot\partial_{\varphi}\mathbf{I}_{\mathtt{m}}+\mathscr{D}\right)\Pi_{\overline{\mathbb{S}}_{0}}^{\perp}+{\mathscr{E}}_{n}^0+\mathscr{B}_{\perp}^{-1}\mathscr{B}\mathscr{R}\Pi_{\overline{\mathbb{S}}_{0}}^{\perp},
		\end{align*}
		with
		\begin{equation}\label{def scr En0}
			{\mathscr{E}}_{n}^0\triangleq \mathscr{B}_{\perp}^{-1}\mathscr{B}\mathscr{E}_{n}\Pi_{\overline{\mathbb{S}}_{0}}^{\perp}.
		\end{equation}
		Consequently, in the Cantor set $\mathcal{O}_{\infty,n}^{\gamma,\tau_1}(i_{0})$, one has the following reduction
		\begin{align*}
			\mathscr{B}_{\perp}^{-1}\widehat{\mathcal{L}}\mathscr{B}_{\perp}&=\omega\cdot\partial_{\varphi}\mathbf{I}_{\mathtt{m},\perp}+\mathscr{D}_0+\mathscr{R}_0+{\mathscr{E}}_{n}^0,
		\end{align*}
		where we set
		$$\mathscr{D}_0\triangleq \begin{pmatrix}
			\mathtt{c}_1\partial_{\theta}\, \cdot+\tfrac{1}{2}\mathcal{H}+\partial_\theta  \mathcal{Q}\ast\cdot& 0\\
			0 &  \mathtt{c}_2\partial_{\theta}\, \cdot-\tfrac{1}{2}\mathcal{H}-\partial_\theta  \mathcal{Q}\ast\cdot
		\end{pmatrix}\Pi_{\overline{\mathbb{S}}_{0}}^{\perp},$$
		\begin{align*}
			\mathscr{R}_0&\triangleq -\mathscr{B}_{\perp}^{-1}{\mathcal{L}}\Pi_{\mathbb{S}_{0}}\mathscr{B}\Pi_{\overline{\mathbb{S}}_{0}}^{\perp}-\varepsilon\mathscr{B}_{\perp}^{-1}\partial_{\theta}\mathcal{R}\mathscr{B}_{\perp}+\mathscr{B}_{\perp}^{-1}\mathscr{B}\mathscr{R}\Pi_{\overline{\mathbb{S}}_{0}}^{\perp}.
		\end{align*}
		\textbf{(iii)} It can be obtained from \eqref{def scr En0}, \eqref{e-vectBnor}, \eqref{cont Bk} and \eqref{scr-Enk}.\\
		\textbf{(iv)}
		To get the estimates \eqref{e-hyb-scrR0} and \eqref{ed-hyb-scrR0}, we may refer to Lemma \cite[Prop 6.3 and Lem. 6.3]{HR21} up to very slight modifications corresponding to the hybrid topology introduced in Section \ref{sec funct set}. The computations are long and based on a duality representations of $\mathscr{B}_{\perp}^{\pm 1}$ and $\mathscr{B}^{\pm 1}$. In particular, one may use Lemma \ref{lem sym--rev}, \eqref{e-Rnl}, \eqref{e-Rnld}, \eqref{e-Jkn}, \eqref{e-Jknd} and \eqref{diff betak}.\\
		\textbf{(v)} This estimate follows from \eqref{reduction on nor}, \eqref{sml-r0}, \eqref{e-hyb-scrR0}, \eqref{sml-pnor} and Corollary \ref{cor-hyb-nor}-(iii).
	\end{proof}

	\subsection{Reduction of the remainder}\label{Reduction-Remaind}
	This section is devoted to the conjugation of the operator $\mathscr{L}_{0}$ defined in Proposition \ref{prop proj nor dir} to a diagonal one, up to a fast decaying small remainders. This will be achieved through a standard KAM reducibility techniques well-adapted to the operators setting. This will be implemented by taking advantage of the exterior parameters which are restricted to a suitable Cantor set that prevents the resonances in the second order Melnikov assumption. Notice that one gets from this study some estimates on the distribution of the eigenvalues and their stability with respect to the torus parametrization. This is considered as the key step not only to get an approximate inverse but also to achieve the Nash-Moser scheme with a final massive Cantor set.	 We may refer for instance to \cite{BBM14,B19,FP14,HHM21,HR21} for some implementations of this KAM strategy to PDEs.
	\begin{prop}\label{prop RR}
		Let $(\gamma,q,d,\tau_1,\tau_2,s_{0},\overline{s}_l,\overline{\mu}_{2},S,\mathtt{m})$ satisfy \eqref{init Sob cond}, \eqref{setting tau1 and tau2}, \eqref{param} and   \eqref{ouvert-sym}. For any $(\mu_2,s_h)$ satisfying  
		\begin{align}\label{p-RR}
			\mu_2\geqslant \overline{\mu}_2+2\tau_2q+2\tau_2\qquad\textnormal{and}\qquad  s_h\geqslant \frac{3}{2}\mu_{2}+\overline{s}_{l}+1,
		\end{align}
		there exist $\varepsilon_{0}\in(0,1)$ and $\sigma_{4}\triangleq\sigma_{4}(\tau_1,\tau_2,q,d)\geqslant\sigma_{3}$, with $\sigma_{3}$ defined in Proposition $\ref{prop proj nor dir},$ such that if 
		\begin{equation}\label{sml-RR}
			\varepsilon\gamma^{-2-q}N_{0}^{\mu_{2}}\leqslant \varepsilon_{0}
		\end{equation}
		and
		\begin{equation}\label{bnd frkI0-4}
			\|\mathfrak{I}_{0}\|_{s_{h}+\sigma_{4}}^{q,\gamma,\mathtt{m}}\leqslant 1,
		\end{equation}
		then the following assertions hold true.
		\begin{enumerate}[label=(\roman*)]
			\item There exists a  family of invertible linear operator $\Phi_{\infty}:\mathcal{O}\to \mathcal{L}\big(\mathbf{H}^s_{\perp,\m}\big)$ satisfying the estimates
			\begin{equation}\label{cont-Phifty}
				\forall s\in[s_{0},S],\quad \mbox{ }\|\Phi_{\infty}^{\pm 1}\rho\|_{s}^{q,\gamma ,\m}\lesssim \|\rho\|_{s}^{q,\gamma,\m}+\varepsilon\gamma^{-2}\| \mathfrak{I}_{0}\|_{s+\sigma_{4}}^{q,\gamma,\m}\|\rho\|_{s_{0}}^{q,\gamma,\m}.
			\end{equation}
			There exists a diagonal operator $\mathscr{L}_\infty\triangleq\mathscr{L}_{\infty}(b,\omega,i_{0})$ taking the form
			\begin{align*}\mathscr{L}_{\infty}&=\omega\cdot\partial_{\varphi}\mathbf{I}_{\mathtt{m},\perp}+\mathscr{D}_{\infty}
			\end{align*}
			where $\mathscr{D}_{\infty}=\Pi_{\overline{\mathbb{S}}_{0}}^{\perp}\mathscr{D}_{\infty}\Pi_{\overline{\mathbb{S}}_{0}}^{\perp}=\mathscr{D}_{\infty}(b,\omega,i_{0})$ is a diagonal operator  with reversible Fourier multiplier entries, namely
			$$\mathscr{D}_{\infty}\triangleq\begin{pmatrix}
				\mathscr{D}_{\infty,1} & 0\\
				0 & \mathscr{D}_{\infty,2}
			\end{pmatrix},\qquad\mathscr{D}_{\infty,k}\triangleq\left(\ii\mu_{j,k}^{(\infty)}\right)_{j\in\mathbb{Z}_{\mathtt{m}}\setminus\overline{\mathbb{S}}_{0,k}},\qquad\mu_{-j,k}^{(\infty)}(b,\omega)=-\mu_{j,k}^{(\infty)}(b,\omega),$$
			with
			\begin{align}\label{def mu lim}
				\forall j\in\Z_{\m}\backslash \overline{\mathbb{S}}_{0,k},&\quad\mu_{j,k}^{(\infty)}(b,\omega,i_{0})\triangleq\mu_{j,k}^{(0)}(b,\omega,i_{0})+r_{j,k}^{(\infty)}(b,\omega,i_{0})
			\end{align}
			and 
			\begin{align}\label{e-rjfty}
				\sup_{j\in\mathbb{Z}_{\mathtt{m}}\setminus\overline{\mathbb{S}}_{0,k}}|j|\left\| r_{j,k}^{(\infty)}\right\|^{q,\gamma}\lesssim\varepsilon\gamma^{-1},
			\end{align}
			such that in the Cantor set
			\begin{align*}
				&\mathscr{O}_{\infty,n}^{\gamma,\tau_{1},\tau_{2}}(i_{0})\triangleq \mathcal{O}_{\infty,n}^{\gamma,\tau_1}(i_{0})\\
				&\bigcap_{\underset{\,j, j_{0}\in\Z_{\m}\backslash\overline{\mathbb{S}}_{0,k}}{ {k\in\{1,2\}}}}\bigcap_{\underset{|l|\leqslant N_{n}}{l\in\mathbb{Z}^{d}\atop(l,j)\neq(0,j_{0})}}\Big\{(b,\omega)\in\mathcal{O}\quad\textnormal{s.t.}\quad\big|\omega\cdot l+\mu_{j,k}^{(\infty)}(b,\omega,i_{0})-\mu_{j_{0},k}^{(\infty)}(b,\omega,i_{0})\big|>\tfrac{2\gamma\langle j-j_0\rangle}{\langle l\rangle^{\tau_2}}\Big\}\\
				&\bigcap_{\underset{\,j_0\in\Z_{\m}\backslash\overline{\mathbb{S}}_{0,2}}{ j\in\Z_{\m}\backslash\overline{\mathbb{S}}_{0,1}}}\bigcap_{\underset{\langle l,j,j_0\rangle\leqslant N_{n}}{l\in\mathbb{Z}^{d}}}\Big\{(b,\omega)\in\mathcal{O}\quad\textnormal{s.t.}\quad\big|\omega\cdot l+\mu_{j,1}^{(\infty)}(b,\omega,i_{0})-\mu_{j_{0},2}^{(\infty)}(b,\omega,i_{0})\big|>\tfrac{2\gamma}{\langle l,j,j_0\rangle^{\tau_2}}\Big\}
			\end{align*}
			we have 
			\begin{align*}
				\Phi_{\infty}^{-1}\mathscr{L}_{0}\Phi_{\infty}&=\mathscr{L}_{\infty}+{\mathscr{E}}_{n}^1,
			\end{align*}
			and the linear operator  ${\mathscr{E}}_{n}^1$ satisfies the estimate
			\begin{equation}\label{e-scrEn1}
				\|{\mathscr{E}}_{n}^1\rho\|_{s_0}^{q,\gamma ,\m}\lesssim \varepsilon\gamma^{-2}N_{0}^{{\mu}_{2}}N_{n+1}^{-\mu_{2}} \|\rho\|_{s_0+1}^{q,\gamma,\m}.
			\end{equation}
			We refer to \eqref{Cant-trs}, \eqref{reduction on nor} and \eqref{mu0 r0} for the definition of  $\mathcal{O}_{\infty,n}^{\gamma,\tau_1}(i_{0}),$ $\mathscr{L}_{0}$ and $\big(\mu_{j,k}^{(0)}(b,\omega,i_{0})\big)_{j\in\Z_{\m}\backslash\overline{\mathbb{S}}_{0,k}}$, respectively.
			\item Given two tori $i_{1}$ and $i_{2}$ both satisfying \eqref{sml-RR}-\eqref{bnd frkI0-4}, then for $k\in\{1,2\}$ we have 
			\begin{align}\label{ed-rjkfty}
				\forall j\in\Z_{\m}\backslash \overline{\mathbb{S}}_{0,k},&\quad\left\|\Delta_{12}r_{j,k}^{(\infty)}\right\|^{q,\gamma}\lesssim\varepsilon\gamma^{-1}\|\Delta_{12}i\|_{\overline{s}_{h}+\sigma_{4}}^{q,\gamma,\m}
			\end{align}
			and
			\begin{align}\label{ed-mujkfty}
				\forall j\in\Z_{\m}\backslash \overline{\mathbb{S}}_{0,k},&\quad\left\|\Delta_{12}\mu_{j,k}^{(\infty)}\right\|^{q,\gamma}\lesssim\varepsilon\gamma^{-1}|j|\| \Delta_{12}i\|_{\overline{s}_{h}+\sigma_{4}}^{q,\gamma,\m}.
			\end{align}
		\end{enumerate}
	\end{prop}
	
	\begin{proof}
		{\bf{(i)}} First recall that Proposition \ref{prop proj nor dir} states that in restriction to the Cantor set $\mathcal{O}_{\infty,n}^{\gamma,\tau_1}(i_0)$ the following identity holds
		$$\mathscr{B}_{\perp}^{-1}\widehat{\mathcal{L}}\mathscr{B}_{\perp}=\mathscr{L}_{0}+{\mathscr{E}}_{n}^0,$$
		where the operator $\mathscr{L}_{0}$ decomposes as follows 
		\begin{equation*}
			\mathscr{L}_0=\omega\cdot\partial_{\varphi}\mathbf{I}_{\mathtt{m},\perp}+\mathscr{D}_0+\mathscr{R}_0,
		\end{equation*}
		with
		$$\mathscr{D}_0=\Pi_{\overline{\mathbb{S}}_{0}}^{\perp}\mathscr{D}_0\Pi_{\overline{\mathbb{S}}_{0}}^{\perp}=\begin{pmatrix}
			\mathscr{D}_{0,1} & 0\\
			0 & \mathscr{D}_{0,2}
		\end{pmatrix},\qquad\mathscr{D}_{0,k}=\Big(\ii\mu_{j,k}^{(0)}\Big)_{j\in\Z_{\m}\backslash\overline{\mathbb{S}}_{0,k}},\qquad\mu_{-j,k}^{(0)}(b,\omega)=-\mu_{j,k}^{(0)}(b,\omega)$$
		and $\mathscr{R}_0$ a real and reversible Toeplitz in time operator of zero order satisfying $\Pi_{\overline{\mathbb{S}}_{0}}^{\perp} \mathscr{R}_0\Pi_{\overline{\mathbb{S}}_{0}}^{\perp} =\mathscr{R}_0.$
		Let us define the quantity 
		$$
		\delta_{0}(s)=\gamma ^{-1}\interleave\mathscr{R}_{0}\interleave_{s}^{q,\gamma,\m},
		$$
		By virtue of \eqref{e-hyb-scrR0}, we find
		\begin{align}\label{edlt0s}
			\delta_{0}(s)\leqslant C\varepsilon\gamma^{-2}\left(1+\|\mathfrak{I}_{0}\|_{s+\sigma_{3}}^{q,\gamma,\m}\right).
		\end{align}
		Thus, combining \eqref{p-RR},  \eqref{sml-RR} and the fact that $\sigma_4\geqslant \sigma_3$ yields
		\begin{align}\label{e-dlt0sh init}
			\nonumber N_{0}^{\mu_{2}}\delta_{0}(s_{h}) &\leqslant  C N_{0}^{\mu_{2}}\varepsilon\gamma^{-2}\\
			&\leqslant C\varepsilon_{0}.
		\end{align}
		The smallness conditions \eqref{edlt0s} and \eqref{e-dlt0sh init} allow to start a KAM reduction procedure similarly to the scalar case \cite[Prop. 6.5]{HR21}. Nevertheless the following KAM iteration is done at the matricial level. For this aim, we need to consider the hybrid norm \eqref{hyb nor} to overcome spatial resonances coming from the anti-diagonal entries when solving the homological equations. To clarify this point, let us first discuss a general KAM step of the procedure.\\
		$\blacktriangleright$ \textbf{KAM step.}  Now, we explain the typical KAM step used in the reduction of the remainder. Assume that we have a linear operator $\mathscr{L}$ taking the following form when the parameters are restricted to some Cantor set $\mathscr{O}$  
		$$\mathscr{L}=\omega\cdot\partial_{\varphi}\mathbf{I}_{\mathtt{m},\perp}+\mathscr{D}+\mathscr{R},$$
		with
		\begin{equation}\label{struc scr D}
			\mathscr{D}=\Pi_{\overline{\mathbb{S}}_{0}}^{\perp}\mathscr{D}\Pi_{\overline{\mathbb{S}}_{0}}^{\perp}=\begin{pmatrix}
				\mathscr{D}_1 & 0\\
				0 & \mathscr{D}_2
			\end{pmatrix},\qquad\mathscr{D}_{k}=\big(\ii\mu_{j,k}\big)_{j\in\Z_{\m}\backslash \overline{\mathbb{S}}_{0,k}},\qquad\mu_{-j,k}(b,\omega)=-\mu_{j,k}(b,\omega).
		\end{equation}
		In addition we assume that the matrix operator  
		$$\mathscr{R}=\begin{pmatrix}
			\mathscr{R}_1 & \mathscr{R}_3\\
			\mathscr{R}_4 & \mathscr{R}_2
		\end{pmatrix}$$ 
		is real, reversible Toeplitz in time of zero order and satisfies
		$$\Pi_{\overline{\mathbb{S}}_{0}}^{\perp} \mathscr{R}\Pi_{\overline{\mathbb{S}}_{0}}^{\perp} =\mathscr{R}.$$
		One may check from \eqref{proj-nor1} that this latter assumption is equivalent to 
		\begin{equation}\label{cond proj R}
			\begin{pmatrix}
				\Pi_{1}^{\perp}\mathscr{R}_1\Pi_{1}^{\perp} & \Pi_{1}^{\perp}\mathscr{R}_3\Pi_{2}^{\perp}\vspace{0.3cm}\\
				\Pi_{2}^{\perp}\mathscr{R}_4\Pi_{1}^{\perp} & \Pi_{2}^{\perp}\mathscr{R}_2\Pi_{2}^{\perp}
			\end{pmatrix}=\begin{pmatrix}
				\mathscr{R}_1 & \mathscr{R}_3\\
				\mathscr{R}_4 & \mathscr{R}_2
			\end{pmatrix}.
		\end{equation}
		According to Definition  \ref{Def-Rev}, the real and reversibility properties of $\mathscr{R}_k$ are equivalent to say
		\begin{equation}\label{coef scrRk}
			(\mathscr{R}_k)_{l_{0},j_{0}}^{l,j}\triangleq \ii\,r_{j_{0},k}^{j}(b,\omega,l-l_{0})\in \ii\,\mathbb{R}\qquad\mbox{and}\qquad (\mathscr{R}_k)_{-l_{0},-j_{0}}^{-l,-j}=-(\mathscr{R}_k)_{l_{0},j_{0}}^{l,j}.
		\end{equation}
		Moreover, the condition \eqref{cond proj R} is equivalent to
		\begin{align}
			&\forall k\in\{1,2\},\quad\forall \,l\,\in\mathbb{Z}^d,\quad\forall\, j\,\,\hbox{or}\,\,\,j_0\in\overline{\mathbb{S}}_{0,k},\quad r_{j_{0},k}^{j}(b,\omega,l)=0,\label{restri rk}\\
			&\forall \ell\in\{3,4\},\quad\forall \,l\,\in\mathbb{Z}^d,\quad\forall\, j\in\overline{\mathbb{S}}_{0,\ell-2}\,\,\hbox{or}\,\,\,j_0\in\overline{\mathbb{S}}_{0,5-\ell},\quad r_{j_{0},\ell}^{j}(b,\omega,l)=0.\label{restri rl}
		\end{align}
		Now, consider a linear invertible transformation close to the identity \begin{equation}\label{ansatz Phi reduc rem}
			\Phi=\mathbf{I}_{\mathtt{m},\perp}+\Psi:\mathcal{O}\rightarrow\mathcal{L}({\mathbf{H}}_{\perp,\m}^{s}),\qquad\Psi=\Pi_{\overline{\mathbb{S}}_{0}}^{\perp}\Psi\Pi_{\overline{\mathbb{S}}_{0}}^{\perp}=\begin{pmatrix}
				\Psi_{1} & \Psi_3\\
				\Psi_4 & \Psi_2
			\end{pmatrix},
		\end{equation}
		with $\Psi$ depending on $\mathscr{R}$ and {\it small} in a suitable sense related to the hybrid norm \eqref{hyb nor}. Then, one readily obtains, in restriction to $\mathscr{O},$ the following decomposition
		\begin{align*}
			\Phi^{-1}\mathscr{L}\Phi & =  \Phi^{-1}\Big(\Phi\left(\omega\cdot\partial_{\varphi}\Pi_{\overline{\mathbb{S}}_{0}}^{\perp}+\mathscr{D}\right)+\left[\omega\cdot\partial_{\varphi}\Pi_{\overline{\mathbb{S}}_{0}}^{\perp}+\mathscr{D},\Psi\right]+\mathscr{R}+\mathscr{R}\Psi\Big)\\
			& =  \omega\cdot\partial_{\varphi}\Pi_{\overline{\mathbb{S}}_{0}}^{\perp}+\mathscr{D}+\Phi^{-1}\Big(\big[\omega\cdot\partial_{\varphi}\Pi_{\overline{\mathbb{S}}_{0}}^{\perp}+\mathscr{D}\, ,\, \Psi\big]+\mathbf{P_{N}}\mathscr{R}+\mathbf{P}_{\mathbf{N}}^{\perp}\mathscr{R}+\mathscr{R}\Psi\Big),
		\end{align*}
		where $\mathbf{P_{N}}\mathscr{R}$ and $\mathbf{P}_{\mathbf{N}}^{\perp}\mathscr{R}$ are defined as in \eqref{proj mat}. We shall select $\Psi$ such that the above expression contains a new remainder $\mathscr{R}_{\textnormal{\tiny{next}}}$ quadratically smaller than the previous one $\mathscr{R}$ up to modify the diagonal part $\mathscr{D}$ into a new one $\mathscr{D}_{\textnormal{\tiny{next}}}$ with the same structure \eqref{struc scr D}. Therefore, we choose $\Psi$ such that it solves the following \textit{matricial homological equation}
		\begin{equation}\label{hom eq Psi}
			\big[\omega\cdot\partial_{\varphi}\mathbf{I}_{\mathtt{m},\perp}+\mathscr{D}\, ,\,\Psi\big]+\mathbf{P_{N}}\mathscr{R}=\lfloor \mathbf{P_{N}}\mathscr{R}\rfloor,
		\end{equation}
		where $\lfloor \mathbf{P_{N}}\mathscr{R}\rfloor$ is the diagonal part of the matrix operator $\mathbf{P_{N}}\mathscr{R}$ as defined by \eqref{def diag-diag}-\eqref{def diag opi}, namely
		$$\lfloor \mathbf{P_{N}}\mathscr{R}\rfloor=\begin{pmatrix}
			\lfloor P_{N}^1\mathscr{R}_1\rfloor & 0\\
			0 & \lfloor P_{N}^1\mathscr{R}_2\rfloor
		\end{pmatrix}.$$
		The matricial equation \eqref{hom eq Psi} is equivalent to the following set of four scalar homological equations
		\begin{equation}\label{4Homeq}
			\left\lbrace\begin{array}{l}
				\big[\omega\cdot\partial_{\varphi}\Pi_{1}^{\perp}+\mathscr{D}_1\, ,\,\Psi_1\big]=\lfloor P_{N}^{1}\mathscr{R}_1\rfloor-P_{N}^{1}\mathscr{R}_1,\vspace{0.1cm}\\
				\big[\omega\cdot\partial_{\varphi}\Pi_{2}^{\perp}+\mathscr{D}_2\, ,\,\Psi_2\big]=\lfloor P_{N}^{1}\mathscr{R}_2\rfloor-P_{N}^{1}\mathscr{R}_2,\vspace{0.1cm}\\
				\big(\omega\cdot\partial_{\varphi}\Pi_{1}^{\perp}+\mathscr{D}_1\big)\Psi_{3}-\Psi_3\big(\omega\cdot\partial_{\varphi}\Pi_{2}^{\perp}+\mathscr{D}_2\big)=-P_{N}^{2}\mathscr{R}_3,\vspace{0.1cm}\\
				\big(\omega\cdot\partial_{\varphi}\Pi_{2}^{\perp}+\mathscr{D}_2\big)\Psi_{4}-\Psi_4\big(\omega\cdot\partial_{\varphi}\Pi_{1}^{\perp}+\mathscr{D}_1\big)=-P_{N}^{2}\mathscr{R}_4.
			\end{array}\right.
		\end{equation}
		As we shall see these equations can be solved modulo the selection of suitable parameters $(b,\omega)$ among a Cantor-type set connected to non-resonance conditions. 
		Let us begin with the diagonal equations on $\Psi_1$ and $\Psi_2$ which can  be treated in a similar way. Fix $k\in\{1,2\},$ then we are interested in solving the equation
		$$\big[\omega\cdot\partial_{\varphi}\Pi_{k}^{\perp}+\mathscr{D}_k\, ,\,\Psi_k\big]=\lfloor P_{N}^{1}\mathscr{R}_k\rfloor-P_{N}^{1}\mathscr{R}_k.$$
		This will be done by using the Fourier expansion of our operators. First notice that similarly to \eqref{cond proj R}-\eqref{restri rk}, the condition $\Psi_k=\Pi_{k}^{\perp}\Psi_k\Pi_{k}^{\perp}$ is equivalent to say that the Fourier coefficients of $\Psi_{k}$ satisfy 
		\begin{equation}\label{restri Psik}
			\forall(l,l_0)\in(\mathbb{Z}^{d})^2,\quad\forall j\,\,\textnormal{or}\,\,j_0\in\overline{\mathbb{S}}_{0,k},\quad(\Psi_k)_{l_{0},j_{0}}^{l,j}=0.
		\end{equation}
		Straightforward computations lead to
		$$
		\forall (l_0,j_0)\in\mathbb{Z}^{d}\times(\mathbb{Z}_{\mathtt{m}}\setminus\overline{\mathbb{S}}_{0,k}),\quad\big[\omega\cdot\partial_{\varphi}\Pi_{k}^{\perp},\Psi_k\big]\mathbf{e}_{l_{0},j_{0}}=\ii\sum_{(l,j)\in\mathbb{Z}^{d }\times(\mathbb{Z}_{\mathtt{m}}\setminus\overline{\mathbb{S}}_{0,k})}(\Psi_k)_{l_{0},j_{0}}^{l,j}\,\,\omega\cdot(l-l_{0})\,\,\mathbf{e}_{l,j}$$
		and using \eqref{struc scr D}
		$$\forall (l_0,j_0)\in\mathbb{Z}^{d}\times(\mathbb{Z}_{\mathtt{m}}\setminus\overline{\mathbb{S}}_{0,k}),\quad[\mathscr{D}_{k},\Psi_k]\mathbf{e}_{l_{0},j_{0}}=\ii\sum_{(l,j)\in\mathbb{Z}^{d }\times(\mathbb{Z}_{\mathtt{m}}\setminus\overline{\mathbb{S}}_{0,k})}(\Psi_k)_{l_{0},j_{0}}^{l,j}\big(\mu_{j,k}(b,\omega)-\mu_{j_0,k}(b,\omega)\big)\mathbf{e}_{l,j}.$$
		Consequently  $\Psi_k$ is a solution of \eqref{hom eq Psi} if and only if for any $(l_0,j_0)\in\mathbb{Z}^{d}\times(\mathbb{Z}_{\mathtt{m}}\setminus\overline{\mathbb{S}}_{0,k}),$
		$$\sum_{(l,j)\in\mathbb{Z}^{d}\times(\mathbb{Z}_{\mathtt{m}}\setminus\overline{\mathbb{S}}_{0,k})}(\Psi_k)_{l_{0},j_{0}}^{l,j}\Big(\omega\cdot (l-l_{0})+\mu_{j,k}(b,\omega)-\mu_{j_0,k}(b,\omega)\Big)\mathbf{e}_{l,j}=-\sum_{(l,j)\in\mathbb{Z}^{d}\times(\mathbb{Z}_{\mathtt{m}}\setminus\overline{\mathbb{S}}_{0,k})\atop\underset{(l,j)\neq(l,j_0)}{\langle l-l_0,j-j_0\rangle\leqslant N}}r_{j_0,k}^{j}(b,\omega,l-l_0)\mathbf{e}_{l,j}.$$
		By identification, we deduce that for any $(l,l_0,j,j_0)\in(\mathbb{Z}^d)^2\times(\mathbb{Z}_{\mathtt{m}}\setminus\overline{\mathbb{S}}_{0,k})^2,$ if $\langle l-l_0,j-j_0\rangle>N$, then $(\Psi_{k})_{l_0,j_0}^{l,j}=0,$
		otherwise we have
		$$(\Psi_k)_{l_{0},j_{0}}^{l,j}\Big(\omega\cdot (l-l_{0})+\mu_{j,k}(b,\omega)-\mu_{j_0,k}(b,\omega)\Big)=\left\lbrace\begin{array}{ll}
			-r_{j_{0},k}^{j}(b,\omega,l-l_{0}), & \mbox{if }(l,j)\neq(l_{0},j_{0}),\\
			0, & \mbox{if }(l,j)=(l_{0},j_{0}).
		\end{array}\right.$$
		As a consequence, we have that $\Psi_k$ is a Toeplitz in time operator with $(\Psi_k)_{j_{0}}^{j}(l-l_0)\triangleq (\Psi_k)_{l_{0},j_{0}}^{l,j}$. In addition, for  $(l,j,j_{0})\in\mathbb{Z}^{d }\times(\mathbb{Z}_{\mathtt{m}}\setminus\overline{\mathbb{S}}_{0,k})^{2}$ with $\langle l,j-j_0\rangle\leqslant N,$ one gets
		\begin{equation}\label{choice psik}
			(\Psi_k)_{j_{0}}^{j}(b,\omega,l)=\left\lbrace\begin{array}{ll}
				\frac{-r_{j_{0},k}^{j}(b,\omega,l)}{\omega\cdot l+\mu_{j,k}(b,\omega)-\mu_{j_{0},k}(b,\omega)}, & \mbox{if }(l,j)\neq(0,j_{0}),\\
				0, & \mbox{if }(l,j)=(0,j_{0}),
			\end{array}\right.
		\end{equation}
		provided that the denominator is non zero.  This latter fact is imposed by selecting suitable values of the parameters $(b,\omega)$ among the following set
		$$\mathscr{O}_{k}=\bigcap_{\underset{|l|\leqslant N}{(l,j,j_{0})\in\mathbb{Z}^{d }\times\left(\mathbb{Z}_{\mathtt{m}}\setminus\overline{\mathbb{S}}_{0,k}\right)^{2}}\atop (l,j)\neq(0,j_0)}\left\lbrace(b,\omega)\in\mathscr{O}
		\quad\textnormal{s.t.}\quad|\omega\cdot l+\mu_{j,k}(b,\omega)-\mu_{j_{0},k}(b,\omega)|>\tfrac{\gamma\langle j-j_0\rangle}{\langle l\rangle^{\tau_2}}\right\rbrace.$$
		This restriction avoids the resonances and implies that the identity  \eqref{choice psik} is well defined. Now, we shall extend $\Psi_k$ to the whole set $\mathcal{O}$ by using the cut-off function $\chi\in C^{\infty}(\mathbb{R},[0,1])$ defined by
		\begin{equation}\label{def cut-off chi}
			\chi(x)=\left\lbrace\begin{array}{ll}
				1 & \textnormal{if }|x|\geqslant\frac{1}{2}\vspace{0.1cm}\\
				0 & \textnormal{if }|x|\leqslant\frac{1}{3}.
			\end{array}\right.
		\end{equation}
		Then, the extension of $\Psi_k$, still denoted $\Psi_k$, is obtained by defining the Fourier coefficients by \eqref{restri Psik} and for $(l,j,j_{0})\in\mathbb{Z}^{d }\times(\mathbb{Z}_{\mathtt{m}}\setminus\overline{\mathbb{S}}_{0,k})^{2}$ with $\langle l,j-j_0\rangle\leqslant N,$ 
		\begin{equation}\label{coef ext psik}
			(\Psi_k)_{j_{0}}^{j}(b,\omega,l)=\left\lbrace\begin{array}{ll}
				-\varrho_{j_{0},k}^{j}(b,\omega,l)\,\, r_{j_{0},k}^{j}(b,\omega,l),& \mbox{if }\quad (l,j)\neq(0,j_{0}),\\
				0, & \mbox{if }\quad (l,j)=(0,j_{0}),
			\end{array}\right.
		\end{equation}
		with
		\begin{align}\label{varro psik}
			\varrho_{j_{0},k}^{j}(b,\omega,l)\triangleq \frac{\chi\left((\omega\cdot l+\mu_{j,k}(b,\omega)-\mu_{j_{0},k}(b,\omega))(\gamma\langle j-j_{0}\rangle)^{-1}\langle l\rangle^{\tau_2}\right)}{\omega\cdot l+\mu_{j,k}(b,\omega)-\mu_{j_{0},k}(b,\omega)}\cdot
		\end{align}
		The extension \eqref{coef ext psik} is smooth and coincides with \eqref{choice psik} for the parameters taken in $\mathscr{O}_{k}.$ In addition, putting together \eqref{coef ext psik}, \eqref{varro psik}, \eqref{coef scrRk} and \eqref{struc scr D} gives	$$(\Psi_k)_{j_{0}}^{j}(l)\in\mathbb{R}\qquad\textnormal{and}\qquad (\Psi_k)_{-j_{0}}^{-j}(-l)=(\Psi_k)_{j_{0}}^{j}(l).$$
		Therefore Definition \ref{Def-Rev} implies that $\Psi_k$ is a real and reversibility preserving Toeplitz in time operator. We now turn to the anti-diagonal equations satisfied by $\Psi_3$ and $\Psi_4$ in \eqref{4Homeq} which can be unified in the following form. Fix $\ell\in\{3,4\}$, then both equations of interest write
		\begin{equation}\label{hom eq Psil}
			\big(\omega\cdot\partial_{\varphi}\Pi_{\ell-2}^{\perp}+\mathscr{D}_{\ell-2}\big)\Psi_{\ell}-\Psi_{\ell}\big(\omega\cdot\partial_{\varphi}\Pi_{5-\ell}^{\perp}+\mathscr{D}_{5-\ell}\big)=-P_{N}^2\mathscr{R}_{\ell}.
		\end{equation}
		First notice that similarly to \eqref{cond proj R}-\eqref{restri rl}, the condition $\Psi_\ell=\Pi_{\ell-2}^{\perp}\Psi_\ell\Pi_{5-\ell}^{\perp}$ implies that the Fourier coefficients of $\Psi_{\ell}$ satisfy 
		\begin{equation}\label{restri Psil}
			\forall(l,l_0)\in(\mathbb{Z}^{d})^2,\quad\forall j\in\overline{\mathbb{S}}_{0,\ell-2}\,\,\textnormal{or}\,\,j_0\in\overline{\mathbb{S}}_{0,5-\ell},\quad(\Psi_\ell)_{l_{0},j_{0}}^{l,j}=0.
		\end{equation}
		One readily has that $\Psi_\ell$ is a solution of \eqref{hom eq Psil} if and only if for any $(l_0,j_0)\in\mathbb{Z}^{d}\times(\mathbb{Z}_{\mathtt{m}}\setminus\overline{\mathbb{S}}_{0,5-\ell}),$
		\begin{align*}
			\sum_{(l,j)\in\mathbb{Z}^{d}\times(\mathbb{Z}_{\mathtt{m}}\setminus\overline{\mathbb{S}}_{0,\ell-2})}(\Psi_\ell)_{l_{0},j_{0}}^{l,j}\Big(\omega\cdot (l-l_{0})+\mu_{j,\ell-2}(b,\omega)-&\mu_{j_0,5-\ell}(b,\omega)\Big)\mathbf{e}_{l,j}\\
			&=-\sum_{(l,j)\in\mathbb{Z}^{d}\times(\mathbb{Z}_{\mathtt{m}}\setminus\overline{\mathbb{S}}_{0,\ell-2})\atop\langle l-l_0,j,j_0\rangle\leqslant N}r_{j_0,\ell}^{j}(b,\omega,l-l_0)\mathbf{e}_{l,j}.
		\end{align*}
		By identification, we deduce that for any $(l,l_0,j,j_0)\in(\mathbb{Z}^d)^2\times(\mathbb{Z}_{\mathtt{m}}\setminus\overline{\mathbb{S}}_{0,\ell-2})\times(\mathbb{Z}_{\mathtt{m}}\setminus\overline{\mathbb{S}}_{0,5-\ell}),$ if $\langle l-l_0,j,j_0\rangle>N$, then $(\Psi_{\ell})_{l_0,j_0}^{l,j}=0,$
		otherwise we have
		$$(\Psi_\ell)_{l_{0},j_{0}}^{l,j}\Big(\omega\cdot (l-l_{0})+\mu_{j,\ell-2}(b,\omega)-\mu_{j_0,5-\ell}(b,\omega)\Big)=-r_{j_{0},\ell}^{j}(b,\omega,l-l_{0}).$$
		As a consequence, we have that $\Psi_\ell$ is a Toeplitz in time operator with $(\Psi_\ell)_{j_{0}}^{j}(l-l_0)\triangleq (\Psi_\ell)_{l_{0},j_{0}}^{l,j}$. In addition, for  $(l,j,j_{0})\in\mathbb{Z}^{d }\times(\mathbb{Z}_{\mathtt{m}}\setminus\overline{\mathbb{S}}_{0,\ell-2})\times(\mathbb{Z}_{\mathtt{m}}\setminus\overline{\mathbb{S}}_{0,5-\ell})$ with $\langle l,j,j_0\rangle\leqslant N,$ one gets
		\begin{equation}\label{choice psil}
			(\Psi_\ell)_{j_{0}}^{j}(b,\omega,l)=\frac{-r_{j_{0},\ell}^{j}(b,\omega,l)}{\omega\cdot l+\mu_{j,\ell-2}(b,\omega)-\mu_{j_{0},5-\ell}(b,\omega)}
		\end{equation}
		provided that the denominator is non zero.  This latter fact is imposed by selecting suitable values of the parameters $(b,\omega)$ among the following set
		$$\mathscr{O}_{1,2}=\bigcap_{(l,j,j_{0})\in\mathbb{Z}^{d}\times(\mathbb{Z}_{\mathtt{m}}\setminus\overline{\mathbb{S}}_{0,1})\times(\mathbb{Z}_{\mathtt{m}}\setminus\overline{\mathbb{S}}_{0,2})\atop\langle l,j,j_0\rangle\leqslant N_{n}}\Big\{(b,\omega)\in\mathcal{O}_{\infty,n}^{\gamma,\tau_1}(i_0)\quad\textnormal{s.t.}\quad\big|\omega\cdot l+\mu_{j,1}(b,\omega)-\mu_{j_0,2}(b,\omega)\big|>\tfrac{\gamma}{\langle l,j,j_0\rangle^{\tau_2}}\Big\}.$$
		This  implies that the identity  \eqref{choice psil} is well defined. Now, the extension of $\Psi_\ell$, still denoted $\Psi_\ell$, is obtained by defining the Fourier coefficients by \eqref{restri Psil} and for $(l,j,j_{0})\in\mathbb{Z}^{d }\times(\mathbb{Z}_{\mathtt{m}}\setminus\overline{\mathbb{S}}_{0,\ell-2})\times(\mathbb{Z}_{\mathtt{m}}\setminus\overline{\mathbb{S}}_{0,5-\ell})$ with $\langle l,j,j_0\rangle\leqslant N,$ 
		\begin{equation}\label{coef ext psil}
			(\Psi_\ell)_{j_{0}}^{j}(b,\omega,l)=-\varrho_{j_{0},\ell}^{j}(b,\omega,l)\,\, r_{j_{0},\ell}^{j}(b,\omega,l)
		\end{equation}
		with
		\begin{align}\label{varro psil}
			\varrho_{j_{0},\ell}^{j}(b,\omega,l)\triangleq \frac{\chi\left((\omega\cdot l+\mu_{j,\ell-2}(b,\omega)-\mu_{j_{0},5-\ell}(b,\omega))\gamma^{-1}\langle l,j,j_0\rangle^{\tau_2}\right)}{\omega\cdot l+\mu_{j,\ell-2}(b,\omega)-\mu_{j_{0},5-\ell}(b,\omega)},
		\end{align}
		where $\chi$ is the cut-off function introduced in \eqref{def cut-off chi}. The extension \eqref{coef ext psil} is smooth and coincides with \eqref{choice psil} for the parameters taken in $\mathscr{O}_{1,2}.$ In addition, putting together \eqref{coef ext psil}, \eqref{varro psil}, \eqref{coef scrRk} and \eqref{struc scr D} gives	$$(\Psi_\ell)_{j_{0}}^{j}(l)\in\mathbb{R}\qquad\textnormal{and}\qquad (\Psi_\ell)_{-j_{0}}^{-j}(-l)=(\Psi_\ell)_{j_{0}}^{j}(l).$$
		Therefore Definition \ref{Def-Rev} implies that $\Psi_{\ell}$ is a real and reversibility preserving Toeplitz in time operator. Now consider,
		\begin{equation}\label{D-R next}
			\mathscr{D}_{\textnormal{\tiny{next}}}\triangleq \mathscr{D}+\lfloor \mathbf{P_{N}}\mathscr{R}\rfloor,\quad \quad\mathscr{R}_{\textnormal{\tiny{next}}}\triangleq \Phi^{-1}\big(-\Psi\,\lfloor \mathbf{P_{N}}\mathscr{R}\rfloor +\mathbf{P}_{\mathbf{N}}^{\perp}\mathscr{R}+\mathscr{R}\Psi\big)
		\end{equation}
		and
		$$\mathscr{L}_{\textnormal{\tiny{next}}}\triangleq \omega\cdot\partial_{\varphi}\mathbf{I}_{\mathtt{m},\perp}+\mathscr{D}_{\textnormal{\tiny{next}}}+\mathscr{R}_{\textnormal{\tiny{next}}}.$$
		Recall that $\mathscr{D},$ $\mathscr{R}$ and $\Psi$ satisfy the localizations properties \eqref{struc scr D}, \eqref{cond proj R} and \eqref{ansatz Phi reduc rem}, respectively. One can easily check that this property is stable under composition/addition and therefore obtains
		$$\Pi_{\mathbb{S}_{0}}^{\perp}\mathscr{D}_{\textnormal{\tiny{next}}}\Pi_{\mathbb{S}_0}^{\perp}=\mathscr{D}_{\textnormal{\tiny{next}}}\qquad\textnormal{and}\qquad\Pi_{\mathbb{S}_0}^{\perp}\mathscr{R}_{\textnormal{\tiny{next}}}\Pi_{\mathbb{S}_0}^{\perp}=\mathscr{R}_{\textnormal{\tiny{next}}}.$$
		Therefore, in restriction to the Cantor set 
		$$\mathscr{O}_{\textnormal{\tiny{next}}}^{\gamma}\triangleq \mathscr{O}_{1}\cap\mathscr{O}_{2}\cap\mathscr{O}_{1,2},$$
		the above construction implies that
		$$\mathscr{L}_{\textnormal{\tiny{next}}}=\Phi^{-1}\mathscr{L}\Phi.$$
		To end this KAM step, we shall now give some quantitative estimates in order to prove the convergence of the scheme. For this aim, we assume that the following estimates hold true.
		\begin{equation}\label{ass-mukd}
			\forall k\in\{1,2\},\quad\forall (j,j_0)\in(\mathbb{Z}_{\mathtt{m}}\setminus\overline{\mathbb{S}}_{0,k})^2,\quad\max_{\alpha\in\mathbb{N}^{d+1}\atop|\alpha|\in\llbracket 0,q\rrbracket}\sup_{(b,\omega)\in\mathcal{O}}\Big|\partial_{b,\omega}^{\alpha}\Big(\mu_{j,k}(b,\omega)-\mu_{j_0,k}(b,\omega)\Big)\Big|\leqslant C|j-j_0|
		\end{equation}
		and
		\begin{equation}\label{ass-mu12d}
			\forall (j,j_0)\in(\mathbb{Z}_{\mathtt{m}}\setminus\overline{\mathbb{S}}_{0,1})\times(\mathbb{Z}_{\mathtt{m}}\setminus\overline{\mathbb{S}}_{0,2}),\quad\max_{\alpha\in\mathbb{N}^{d+1}\atop|\alpha|\in\llbracket 0,q\rrbracket}\sup_{(b,\omega)\in\mathcal{O}}\Big|\partial_{b,\omega}^{\alpha}\Big(\mu_{j,1}(b,\omega)-\mu_{j_0,2}(b,\omega)\Big)\Big|\leqslant C\langle j,j_0\rangle.
		\end{equation}
		We denote
		$$A_{l,j,j_{0}}^{k,k'}(b,\omega)\triangleq \omega\cdot l+\mu_{j,k}(b,\omega)-\mu_{j_{0},k'}(b,\omega), \quad a_{l,j,j_{0}}\triangleq (\gamma\langle j-j_0\rangle)^{-1}\langle l\rangle^{\tau_2},\quad \widetilde{a}_{l,j,j_0}\triangleq \gamma^{-1}\langle l,j,j_0\rangle^{\tau_2}.$$
		Then, we can write
		\begin{align*}
			\forall k\in\{1,2\},\quad\varrho_{j_{0},k}^{j}(b,\omega,l)&=a_{l,j,j_{0}}\widehat \chi\left(a_{l,j,j_{0}} A_{l,j,j_{0}}^{k,k}(b,\omega)\right),\\
			\forall\ell\in\{3,4\},\quad\varrho_{j_{0},\ell}^{j}(b,\omega,l)&=\widetilde{a}_{l,j,j_{0}}\widehat \chi\left(\widetilde{a}_{l,j,j_{0}} A_{l,j,j_{0}}^{\ell-2,5-\ell}(b,\omega)\right),
		\end{align*}
		where   $\widehat{\chi}(x)\triangleq \frac{\chi(x)}{x}$  is $\mathcal{C}^\infty$ with bounded derivatives. The assumptions \eqref{ass-mukd}-\eqref{ass-mu12d} imply
		\begin{align*}
			\forall\, (l,j,j_{0})\in\mathbb{Z}^{d}\times(\mathbb{Z}_{\mathtt{m}}\setminus\overline{\mathbb{S}}_{0,k})^2,\quad\max_{\alpha\in\mathbb{N}^{d+1}\atop|\alpha|\in\llbracket 0, q\rrbracket}\sup_{(b,\omega)\in\mathcal{O}}\Big|\partial_{b,\omega}^{\alpha}A_{l,j,j_{0}}^{k,k}(b,\omega)\Big|\leqslant C\langle l,j-j_0\rangle
		\end{align*}
		and
		\begin{align*}
			\forall\, (l,j,j_{0})\in\mathbb{Z}^{d}\times(\mathbb{Z}_{\mathtt{m}}\setminus\overline{\mathbb{S}}_{0,\ell-2})\times(\mathbb{Z}_{\mathtt{m}}\setminus\overline{\mathbb{S}}_{0,5-\ell}),\quad\max_{\alpha\in\mathbb{N}^{d+1}\atop|\alpha|\in\llbracket 0, q\rrbracket}\sup_{(b,\omega)\in\mathcal{O}}\Big|\partial_{b,\omega}^{\alpha}A_{l,j,j_{0}}^{\ell-2,5-\ell}(b,\omega)\Big|\leqslant C\langle l,j,j_0\rangle.
		\end{align*}
		Applying Lemma \ref{lem funct prop}-(iv), we obtain for any $\alpha\in\mathbb{N}^{d+1}$ with $|\alpha|\in\llbracket 0,q\rrbracket,$
		\begin{align*}
			\forall k\in\{1,2\},\quad\sup_{(b,\omega)\in\mathcal{O}}\Big|\partial_{b,\omega}\varrho_{j_0,k}^{j}(b,\omega,l)\Big|&\leqslant C\gamma^{-(|\alpha|+1)}\langle l,j-j_0\rangle^{\tau_2 |\alpha|+\tau_2+|\alpha|},\\
			\forall\ell\in\{3,4\},\quad\sup_{(b,\omega)\in\mathcal{O}}\Big|\partial_{b,\omega}\varrho_{j_0,\ell}^{j}(b,\omega,l)\Big|&\leqslant C\gamma^{-(|\alpha|+1)}\langle l,j,j_0\rangle^{\tau_2 |\alpha|+\tau_2+|\alpha|}.
		\end{align*}
		In similar way to  \cite[Prop. 6.5]{HR21}, making use Leibniz rule implies
		\begin{align}
			\forall k\in\{1,2\},\quad\|\Psi_{k}\|_{\textnormal{\tiny{O-d}},s}^{q,\gamma,\mathtt{m}}&\leqslant C\gamma^{-1}\|P_{N}^{1}\mathscr{R}_{k}\|_{\textnormal{\tiny{O-d}},s+\tau_2 q+\tau_2}^{q,\gamma,\mathtt{m}},\label{e-psik-scrRk}\\
			\forall\ell\in\{3,4\},\quad\|\Psi_{\ell}\|_{\textnormal{\tiny{I-D}},s}^{q,\gamma,\mathtt{m}}&\leqslant C\gamma^{-1}\|P_{N}^{2}\mathscr{R}_{\ell}\|_{\textnormal{\tiny{I-D}},s+\tau_2 q+\tau_2}^{q,\gamma,\mathtt{m}}.\label{e-psil-scrRl}
		\end{align}
		Combining \eqref{e-psik-scrRk},  \eqref{e-psil-scrRl}, \eqref{hyb nor} and Corollary \ref{cor-hyb-nor}, we get
		\begin{align}\label{e-Psi-scrR-hyb}
			\interleave\Psi\interleave_{s}^{q,\gamma,\mathtt{m}}&\leqslant C\gamma^{-1}\interleave\mathbf{P_N}\mathscr{R}\interleave_{s+\tau_2 q+\tau_2}^{q,\gamma,\mathtt{m}}\nonumber\\
			&\leqslant C\gamma^{-1}N^{\tau_2q+\tau_2}\interleave\mathscr{R}\interleave_{s}^{q,\gamma,\mathtt{m}}.
		\end{align}
		Now assume the following smallness condition
		\begin{align}\label{ass-sml-scrR}
			\gamma ^{-1}N^{\tau_2 q+\tau_2}\interleave \mathscr{R}\interleave_{s_0}^{q,\gamma,\mathtt{m}}\leqslant  C\varepsilon_{0}.
		\end{align}
		Putting together \eqref{e-Psi-scrR-hyb} and \eqref{ass-sml-scrR}, we obtain
		\begin{equation}\label{sml-Psi-hyb}
			\interleave\Psi\interleave_{s_0}^{q,\gamma,\mathtt{m}}\leqslant C\varepsilon_{0}.
		\end{equation}
		We deduce that, for $\varepsilon_0$ small enough, the operator $\Phi$ is invertible and its inverse is given by
		$$\Phi^{-1}=\displaystyle\sum_{n=0}^{\infty}(-1)^{n}\Psi^{n}\triangleq \textnormal{Id}+\Sigma.$$
		According to Corollary \ref{cor-hyb-nor}-(ii), \eqref{e-Psi-scrR-hyb} and \eqref{sml-Psi-hyb}, one obtains
		\begin{align}
			\displaystyle\interleave\Sigma\interleave_{s}^{q,\gamma,\mathtt{m}} & \leqslant  \displaystyle\interleave\Psi\interleave_{s}^{q,\gamma,\mathtt{m}}\left(1+\sum_{n=1}^{\infty}\left(C\interleave\Psi\interleave_{s_{0}}^{q,\gamma,\mathtt{m}}\right)^{n}\right)\label{e2-Sigma}\\
			&\leqslant \displaystyle C\,\gamma ^{-1}N^{\tau_2 q+\tau_2}\interleave\mathscr{R}\interleave_{s}^{q,\gamma,\mathtt{m}}.\label{el-Sigma}
		\end{align}
		In particular, \eqref{ass-sml-scrR} implies
		\begin{equation}\label{sml-Sigma}
			\interleave\Sigma\interleave_{s_0}^{q,\gamma,\mathtt{m}}\leqslant C\gamma ^{-1}N^{\tau_2 q+\tau_2}\interleave\mathscr{R}\interleave_{s_0}^{q,\gamma,\mathtt{m}}\leqslant C\varepsilon_0.
		\end{equation}
		The second identity in \eqref{D-R next} also writes
		$$\mathscr{R}_{\textnormal{\tiny{next}}}=\mathbf{P}_{\mathbf{N}}^{\perp}\mathscr{R}+\Phi^{-1}\mathscr{R}\Psi-\Psi\,\lfloor \mathbf{P_{N}}\mathscr{R}\rfloor +\Sigma\big(\mathbf{P}_{\mathbf{N}}^{\perp}\mathscr{R}-\Psi\,\lfloor \mathbf{P_{N}}\mathscr{R}\rfloor\big).$$
		Hence, one gets from Corollary \ref{cor-hyb-nor}-(ii), \eqref{e2-Sigma}, \eqref{sml-Psi-hyb} and \eqref{sml-Sigma},
		\begin{align}\label{e-scrRnext1}
			\nonumber \interleave\mathscr{R}_{\textnormal{\tiny{next}}}\interleave_{s}^{q,\gamma,\mathtt{m}} & \leqslant\interleave\mathbf{P}_{\mathbf{N}}^{\perp}\mathscr{R}\interleave_{s}^{q,\gamma,\mathtt{m}}+C\interleave\Sigma\interleave_{s}^{q,\gamma,\mathtt{m}}\left(\interleave\mathbf{P}_{\mathbf{N}}^{\perp}\mathscr{R}\interleave_{s_{0}}^{q,\gamma,\mathtt{m}}+\interleave\Psi\interleave_{s_{0}}^{q,\gamma,\mathtt{m}}\interleave\mathscr{R}\interleave_{s_{0}}^{q,\gamma,\mathtt{m}}\right)\\
			&\quad+C\left(1+\interleave\Sigma\interleave_{s_{0}}^{q,\gamma,\mathtt{m}}\right)
			\left(\interleave\Psi\interleave_{s}^{q,\gamma,\mathtt{m}}\interleave\mathscr{R}\interleave_{s_{0}}^{q,\gamma,\mathtt{m}}+\interleave\Psi\interleave_{s_{0}}^{q,\gamma,\mathtt{m}}\interleave\mathscr{R}\interleave_{s}^{q,\gamma,\mathtt{m}}\right).
		\end{align}
		Using Corollary \ref{cor-hyb-nor}-(i), \eqref{e-Psi-scrR-hyb}, \eqref{ass-sml-scrR}, \eqref{sml-Psi-hyb}, \eqref{el-Sigma}, \eqref{sml-Sigma} and \eqref{e-scrRnext1}, we have for all $S\geqslant\overline{s}\geqslant s\geqslant s_{0}$,
		\begin{equation}\label{e-KAM Rnext}
			\interleave\mathscr{R}_{\textnormal{\tiny{next}}}\interleave_{s}^{q,\gamma,\mathtt{m}}\leqslant N^{s-\overline{s}}\interleave\mathscr{R}\interleave_{\overline{s}}^{q,\gamma,\mathtt{m}}+C\gamma^{-1}N^{\tau_2 q+\tau_2}\interleave\mathscr{R}\interleave_{s_{0}}^{q,\gamma,\mathtt{m}}\interleave\mathscr{R}\interleave_{s}^{q,\gamma,\mathtt{m}}.
		\end{equation}
		One also has
		\begin{align*}
			\partial_{\theta}\mathscr{R}_{\textnormal{\tiny{next}}}&=\Phi^{-1}\Big(\mathbf{P}_{\mathbf{N}}^{\perp}\partial_{\theta}\mathscr{R}+\partial_{\theta}\mathscr{R}\Psi-\Psi\partial_{\theta}\lfloor\mathbf{P}_{\mathbf{N}}\mathscr{R}\rfloor-[\partial_{\theta},\Psi]\lfloor\mathbf{P}_{\mathbf{N}}\mathscr{R}\rfloor\Big)\\
			&\quad+[\partial_{\theta},\Sigma]\Big(\mathbf{P}_{\mathbf{N}}^{\perp}\mathscr{R}+\mathscr{R}\Psi-\Psi\lfloor\mathbf{P}_{\mathbf{N}}\mathscr{R}\rfloor\Big).
		\end{align*} 
		Using the fact that for any scalar operator $T$,
		$$\|[\partial_{\theta},T]\|_{\textnormal{\tiny{O-d}},s}^{q,\gamma,\mathtt{m}}\lesssim\|T\|_{\textnormal{\tiny{O-d}},s+1}^{q,\gamma,\mathtt{m}},\qquad\|[\partial_{\theta},T]\|_{\textnormal{\tiny{I-D}},s}^{q,\gamma,\mathtt{m}}\lesssim\|T\|_{\textnormal{\tiny{I-D}},s+1}^{q,\gamma,\mathtt{m}},$$
		one has for any matricial operator $\mathbf{T},$
		$$\interleave[\partial_{\theta},\mathbf{T}]\interleave_{s}^{q,\gamma,\mathtt{m}}\lesssim\interleave\mathbf{T}\interleave_{s+1}^{q,\gamma,\mathtt{m}}.$$
		Thus, in a similar way to \eqref{e-KAM Rnext}, one obtains for any $s_0\leqslant s\leqslant\overline{s}\leqslant S,$
		\begin{equation}\label{e-KAM dRnext}
			\widehat{\delta}_{\textnormal{\tiny{next}}}(s)\leqslant N^{s-\overline{s}}\,\widehat{\delta}(\overline{s})+C\gamma^{-1}N^{\tau_2 q+\tau_2+1}\widehat{\delta}(s_0)\widehat{\delta}(s),
		\end{equation}
		where
		$$\widehat{\delta}(s)\triangleq \max\Big(\gamma^{-1}\interleave\partial_{\theta}\mathscr{R}\interleave_{s}^{q,\gamma,\mathtt{m}}\,,\,\delta(s)\Big).$$
		$\blacktriangleright$ \textbf{Initialization} Now, we shall check the validity of the assumptions \eqref{ass-mukd}, \eqref{ass-mu12d} and \eqref{ass-sml-scrR} for the initial operator $\mathscr{L}=\mathscr{L}_0$ in \eqref{reduction on nor}. It is clear from \eqref{ASYFR1+}-\eqref{ASYFR1-} that
		$$\forall k\in\{1,2\},\quad\forall(j,j_{0})\in\mathbb{Z}^{2},\quad\max_{q'\in\llbracket 0,q\rrbracket}\sup_{b\in(b_{*},b^{*})}\big|\partial_{b}^{q'}\big(\Omega_{j,k}(b)-\Omega_{j_{0},k}(b)\big)\big|\leqslant C\,|j-j_{0}|$$
		and
		$$\forall(j,j_{0})\in\mathbb{Z}^{2},\quad\max_{q'\in\llbracket 0, q\rrbracket}\sup_{b\in(b_{*},b^{*})}\big|\partial_{b}^{q'}\big(\Omega_{j,1}(b)-\Omega_{j_{0},2}(b)\big)\big|\leqslant C\,\langle j,j_{0}\rangle.$$
		Consequently, we infer from \eqref{mu0 r0}-\eqref{e-ed-r0},
		\begin{align*}
			\forall k\in\{1,2\},\quad\forall(j,j_{0})\in\mathbb{Z}^{2},\quad\max_{\alpha\in\mathbb{N}^{d+1}\atop|\alpha|\in\llbracket 0, q\rrbracket}\sup_{(b,\omega)\in\mathcal{O}}\left|\partial_{b,\omega}^{\alpha}\left(\mu_{j,k}^{(0)}(b,\omega)-\mu_{j_{0},k}^{(0)}(b,\omega)\right)\right|\leqslant C\,|j-j_{0}|
		\end{align*}
		and
		\begin{align*}
			\forall(j,j_{0})\in\mathbb{Z}^{2},\quad\max_{\alpha\in\mathbb{N}^{d+1}\atop|\alpha|\in\llbracket 0, q\rrbracket}\sup_{(b,\omega)\in\mathcal{O}}\left|\partial_{b,\omega}^{\alpha}\left(\mu_{j,1}^{(0)}(b,\omega)-\mu_{j_{0},2}^{(0)}(b,\omega)\right)\right|\leqslant C\,\langle j,j_{0}\rangle.
		\end{align*}
		This proves the initial assumptions \eqref{ass-mukd}-\eqref{ass-mu12d}. Now let us focus on the assumption \eqref{ass-sml-scrR}. This latter is obtained by gathering \eqref{e-hyb-scrR0}, \eqref{sml-RR} and \eqref{p-RR}. Indeed, 
		\begin{align*}
			\gamma ^{-1}N_0^{\tau_2q+\tau_2}\interleave\mathscr{R}_{0}\interleave_{s_0}^{q,\gamma,\mathtt{m}}&\leqslant C\varepsilon\gamma ^{-2}N_0^{\mu_2}\left(1+\|\mathfrak{I}_{0}\|_{s_h+\sigma_{4}}^{q,\gamma,\mathtt{m}}\right)\\
			&\leqslant C\varepsilon_{0}.
		\end{align*}
		$\blacktriangleright$ \textbf{KAM iteration}. Now, we shall implement the complete KAM reduction scheme. Given $m\in\mathbb{N}$ we assume that we have constructed a linear operator 
		\begin{align}\label{Op-Lm}
			\mathscr{L}_{m}\triangleq \omega\cdot\partial_{\varphi}\mathbf{I}_{\mathtt{m},\perp}+\mathscr{D}_{m}+\mathscr{R}_{m},
		\end{align}
		with
		$$\mathscr{D}_m=\begin{pmatrix}
			\mathscr{D}_{m,1} & 0\\
			0 & \mathscr{D}_{m,2}
		\end{pmatrix},\qquad\mathscr{D}_{m,k}=\left(\ii \mu_{j,k}^{(m)}\right)_{j\in\mathbb{Z}_{\mathtt{m}}\setminus\overline{\mathbb{S}}_{0,k}},\qquad\mu_{-j,k}^{(m)}(b,\omega)=-\mu_{j,k}^{(m)}(b,\omega)$$ 
		and $\mathscr{R}_m$ a real and reversible Toeplitz in time matrix operator of zero order satisfying $\Pi_{\overline{\mathbb{S}}_{0}}^{\perp} \mathscr{R}_m\Pi_{\overline{\mathbb{S}}_{0}}^{\perp} =\mathscr{R}_m.$ In addition, we assume that the assumptions \eqref{ass-mukd}, \eqref{ass-mu12d} and \eqref{ass-sml-scrR} hold for $\mathscr{D}_m$ and $\mathscr{R}_m$. Notice that for $m=0$ we take the operator $\mathscr{L}_0$ defined in \eqref{reduction on nor}. Applying the KAM step we can construct a linear invertible operator $\Phi_{m}=\mathbf{I}_{\mathtt{m},\perp}+\Psi_{m}$ with $\Psi_m$ living in $\mathcal{O}$ such that in restriction to the Cantor set	
		\begin{align}\label{Cantor-SX}
			&\mathscr{O}_{m+1}^{\gamma }=
			\bigcap_{\underset{\,j, j_{0}\in\Z_{\m}\backslash\overline{\mathbb{S}}_{0,k}}{ {k\in\{1,2\}}}}\bigcap_{\underset{|l|\leqslant N_{m}}{l\in\mathbb{Z}^{d}\atop(l,j)\neq(0,j_{0})}}\left\lbrace(b,\omega)\in\mathscr{O}_m^\gamma\quad\textnormal{s.t.}\quad\left|\omega\cdot l+\mu_{j,k}^{(m)}(b,\omega,i_{0})-\mu_{j_{0},k}^{(m)}(b,\omega,i_{0})\right|>\tfrac{\gamma\langle j-j_0\rangle}{\langle l\rangle^{\tau_2}}\right\rbrace\nonumber\\
			&\bigcap_{(l,j,j_{0})\in\mathbb{Z}^{d }\times(\mathbb{Z}_{\mathtt{m}}\setminus\overline{\mathbb{S}}_{0,1})\times(\mathbb{Z}_{\mathtt{m}}\setminus\overline{\mathbb{S}}_{0,2})\atop\langle l,j,j_0\rangle\leqslant N_{m}}\left\lbrace(b,\omega)\in\mathscr{O}_{m}^{\gamma }\quad\textnormal{s.t.}\quad \left|\omega\cdot l+\mu_{j,1}^{(m)}(b,\omega)-\mu_{j_{0},2}^{(m)}(b,\omega)\right|>\tfrac{\gamma}{\langle l,j,j_0\rangle^{\tau_2}}\right\rbrace,
		\end{align}
		the operator $\Psi_{m}$ satisfies the following homological equation
		$$\big[\omega\cdot\partial_{\varphi}\mathbf{I}_{\mathtt{m},\perp}+\mathscr{D}_{m},\Psi_{m}\big]+\mathbf{P_{N_{m}}}\mathscr{R}_{m}=\lfloor \mathbf{P_{N_{m}}}\mathscr{R}_{m}\rfloor$$
		and consequently, the following identity holds in $\mathscr{O}_{m+1}^{\gamma}$
		\begin{align}\label{Op-Lm1}
			\Phi_{m}^{-1}\mathscr{L}_{m}\Phi_{m}=\omega\cdot\partial_{\varphi}\mathbf{I}_{\mathtt{m},\perp}+\mathscr{D}_{m+1}+\mathscr{R}_{m+1},
		\end{align}
		with
		\begin{align}\label{scr Dm+1-Rm+1}
			\mathscr{D}_{m+1}\triangleq \mathscr{D}_{m}+\lfloor \mathbf{P_{N_{m}}}\mathscr{R}_{m}\rfloor\quad\mbox{and}\quad \mathscr{R}_{m+1}\triangleq \Phi_{m}^{-1}\left(-\Psi_{m}\,\lfloor \mathbf{P_{N_{m}}}\mathscr{R}_{m}\rfloor +\mathbf{P_{N_{m}}}^{\perp}\mathscr{R}_{m}+\mathscr{R}_{m}\Psi_{m}\right).
		\end{align}
		Recall that the operator $\lfloor \mathbf{P_{N_{m}}}\mathscr{R}_{m}\rfloor$ is defined by
		$$\lfloor \mathbf{P_{N_{m}}}\mathscr{R}_{m}\rfloor=\begin{pmatrix}
			\lfloor P_{N_{m}}^1\mathscr{R}_{m,1}\rfloor & 0\\
			0 & \lfloor P_{N_{m}}^1\mathscr{R}_{m,2}\rfloor
		\end{pmatrix},$$
		with
		$$\lfloor P_{N_{m}}^1\mathscr{R}_{m,k}\rfloor=\left(\ii r_{j,k}^{(m)}\right)_{j\in\mathbb{Z}_{\mathtt{m}}\setminus\overline{\mathbb{S}}_{0,k}},\qquad r_{-j,k}^{(m)}(b,\omega)=-r_{j,k}^{(m)}(b,\omega).$$
		Observe that the symmetry condition for $r_{j,k}^{(m)}$ is a consequence of the reversibility of $\mathscr{R}_m.$ By construction, we find
		\begin{align}\label{spec rec def}
			\mu_{j,k}^{(m+1)}=\mu_{j,k}^{(m)}+r_{j,k}^{(m)}.
		\end{align}
		We point out that working with this extension for $\Psi_m$ allows to extend  both $\mathscr{D}_{m+1}$ and the remainder $\mathscr{R}_{m+1}$ provided that the operators $\mathscr{D}_{m}$ and $\mathscr{R}_{m}$ are defined in the whole range of parameters. Thus the operator defined by the right-hand side in \eqref{Op-Lm1} can be extended to the whole set $\mathcal{O}$ and we denote this extension by $\mathscr{L}_{m+1}$. that is,
		\begin{align*}
			\omega\cdot\partial_{\varphi}\mathbf{I}_{\mathtt{m},\perp}+\mathscr{D}_{m+1}+\mathscr{R}_{m+1}\triangleq \mathscr{L}_{m+1}.
		\end{align*}
		This enables to construct by induction the sequence of operators $\left(\mathscr{L}_{m+1}\right)$ in the full set  $\mathcal{O}$. Similarly the operator $\Phi_{m}^{-1}\mathscr{L}_{m}\Phi_{m}$ admits an extension in $\mathcal{O}$ induced by the extension of $\Phi_m^{\pm1}$ . However, by   construction the identity $\mathscr{L}_{m+1}=\Phi_{m}^{-1}\mathscr{L}_{m}\Phi_{m}$ in \eqref{Op-Lm1} occurs in the  Cantor set $\mathscr{O}_{m+1}^{\gamma }$ and may fail outside this set. Define
		\begin{equation}\label{def dltm dltmh}
			\delta_{m}(s)\triangleq \gamma ^{-1}\interleave\mathscr{R}_{m}\interleave_{s}^{q,\gamma,\mathtt{m}}\qquad\textnormal{and}\qquad\widehat{\delta}_{m}(s)\triangleq \max\Big(\delta_m(s)\,,\,\gamma^{-1}\interleave\partial_{\theta}\mathscr{R}_m\interleave_{s}^{q,\gamma,\mathtt{m}}\Big).
		\end{equation}
		Assume that the following estimates hold
		\begin{align}\label{e-mujmk} \forall\,(j,j_{0})\in(\mathbb{Z}_{\mathtt{m}}\setminus\overline{\mathbb{S}}_{0,k})^{2},\quad\max_{|\alpha| \in\llbracket 0,q\rrbracket}\sup_{(b,\omega)\in\mathcal{O}}\left|\partial_{b,\omega}^{\alpha}\left(\mu_{j,k}^{(m)}(b,\omega)-\mu_{j_{0},k}^{(m)}(b,\omega)\right)\right|\leqslant C\,|j-j_{0}|
		\end{align}
		and
		\begin{align}\label{e-mujm12} \forall\,(j,j_{0})\in(\mathbb{Z}_{\mathtt{m}}\setminus\overline{\mathbb{S}}_{0,1})\times(\mathbb{Z}_{\mathtt{m}}\setminus\overline{\mathbb{S}}_{0,2}),\quad\max_{|\alpha| \in\llbracket 0,q\rrbracket}\sup_{(b,\omega)\in\mathcal{O}}\left|\partial_{b,\omega}^{\alpha}\left(\mu_{j,1}^{(m)}(b,\omega)-\mu_{j_{0},2}^{(m)}(b,\omega)\right)\right|\leqslant C\,\langle j,j_{0}\rangle.
		\end{align}
		Applying the KAM step, we deduce from \eqref{e-KAM Rnext} and \eqref{e-KAM dRnext} the following induction formulae true for any $s_0\leqslant s\leqslant\overline{s}\leqslant S,$
		\begin{align*}
			\delta_{m+1}(s)&\leqslant N_{m}^{s-\overline{s}}\delta_{m}(\overline{s})+CN_{m}^{\tau_2q+\tau_2}\delta_{m}(s_0)\delta_{m}(s),\\
			\widehat{\delta}_{m+1}(s)&\leqslant N_{m}^{s-\overline{s}}\,\widehat{\delta}_{m}(\overline{s})+CN_{m}^{\tau_2q+\tau_2+1}\widehat{\delta}_{m}(s_0)\widehat{\delta}_{m}(s).
		\end{align*}
		Hence, in a similar way to \cite[Prop. 6.5]{HR21}, our choice of parameters \eqref{param} allow to prove by induction on $m\in\mathbb{N}$ that 
		\begin{equation}\label{hyp-ind dltp}
			\forall\, m\in\mathbb{N},\quad \delta_{m}(\overline{s}_{l})\leqslant \delta_{0}(s_{h})N_{0}^{\mu_{2}}N_{m}^{-\mu_{2}}\quad \mbox{ and }\quad \delta_{m}(s_{h})\leqslant\left(2-\tfrac{1}{m+1}\right)\delta_{0}(s_{h}),
		\end{equation}
		and
		\begin{equation}\label{hyp-ind dltph}
			\forall m\in\mathbb{N},\quad\widehat{\delta}_{m}(s_{0})\leqslant \widehat{\delta}_{0}(s_{h})N_{0}^{\mu_{2}}N_{m}^{-\mu_{2}}\quad \mbox{ and }\quad \widehat{\delta}_{m}(s_{h})\leqslant\left(2-\tfrac{1}{m+1}\right)\widehat{\delta}_{0}(s_{h}).
		\end{equation}
		Observe that the first condition in \eqref{hyp-ind dltp} together with \eqref{e-dlt0sh init} and \eqref{p-RR} implies that the smallness condition \eqref{ass-sml-scrR} is satisfied for any $m$ (replacing $\mathscr{R}$ by $\mathscr{R}_m$ and $N$ by $N_m$). Using the Topeplitz structure of $\mathscr{R}_{m,k}$ and an integration by parts, we get from \eqref{spec rec def}
		\begin{align*}
			\left\|\mu_{j,k}^{(m+1)}-\mu_{j,k}^{(m)}\right\|^{q,\gamma}&=\big\|\big\langle P_{N_{m}}^1\mathscr{R}_{m,k}\mathbf{e}_{l,j}\,,\,\mathbf{e}_{l,j}\big\rangle_{L^{2}(\mathbb{T}^{d +1})}\big\|^{q,\gamma}\\
			&=\big\|\big\langle P_{N_{m}}^1\mathscr{R}_{m,k}\mathbf{e}_{0,j}\,,\,\mathbf{e}_{0,j}\big\rangle_{L^{2}(\mathbb{T}^{d+1})}\big\|^{q,\gamma}\\
			&=\tfrac{1}{|j|}\big\|\big\langle P_{N_{m}}^1\mathscr{R}_{m,k}\mathbf{e}_{0,j}\,,\,\partial_{\theta}\mathbf{e}_{0,j}\big\rangle_{L^{2}(\mathbb{T}^{d+1})}\big\|^{q,\gamma}\\
			&=\tfrac{1}{|j|}\big\|\big\langle P_{N_{m}}^1\partial_{\theta}\mathscr{R}_{m,k}\mathbf{e}_{0,j}\,,\,\mathbf{e}_{0,j}\big\rangle_{L^{2}(\mathbb{T}^{d+1})}\big\|^{q,\gamma}.
		\end{align*}
		Therefore, a duality argument together with \eqref{def dltm dltmh}, \eqref{hyp-ind dltph}, \eqref{edlt0s}, \eqref{sml-RR} and Corollary \ref{cor-hyb-nor}-(iii) imply
		\begin{align}\label{bound Cv mujkm}
			|j|\left\|\mu_{j,k}^{(m+1)}-\mu_{j,k}^{(m)}\right\|^{q,\gamma}&\leqslant \big\|\partial_{\theta}\mathscr{R}_{m}\mathbf{e}_{0,j}\|_{s_0}^{q,\gamma,\mathtt{m}}\,\,\langle j\rangle^{-s_0}\nonumber\\
			&\leqslant C\interleave\partial_{\theta}\mathscr{R}_{m}\interleave_{s_0}^{q,\gamma,\mathtt{m}}\|\mathbf{e}_{0,j}\|_{H^{s_0}}\,\,\langle j\rangle^{-s_0}\nonumber\\
			&\leqslant C \gamma\,\widehat{\delta}_{0}(s_{h})N_{0}^{\mu_{2}}N_{m}^{-\mu_{2}}\nonumber\\
			&\leqslant C \varepsilon\gamma^{-1}N_0^{\mu_{2}}N_{m}^{-\mu_{2}}.
		\end{align}
		Now we shall check that the assumptions \eqref{e-mujmk} and \eqref{e-mujm12} are satisfied for the next step. Combining \eqref{bound Cv mujkm} with \eqref{e-mujmk} we infer that for $k\in\{1,2\}$ and  $(j,j_{0})\in(\mathbb{Z}_{\mathtt{m}}\setminus\overline{\mathbb{S}}_{0,k})^{2},$
		\begin{align*}
			\max_{|\alpha| \in\llbracket 0,q\rrbracket}\sup_{(b,\omega)\in\mathcal{O}}\left|\partial_{b,\omega}^{\alpha}\left(\mu_{j,k}^{(m+1)}(b,\omega)-\mu_{j_{0},k}^{(m+1)}(b,\omega)\right)\right|\leqslant C\big(1+\varepsilon\gamma^{-1-q}N_0^{\mu_2}N_{m}^{-\mu_{2}}\big)\,|j-j_{0}|.
		\end{align*}
		Now putting together \eqref{bound Cv mujkm} and \eqref{e-mujm12}, we get for $(j,j_{0})\in(\mathbb{Z}_{\mathtt{m}}\setminus\overline{\mathbb{S}}_{0,1})\times(\mathbb{Z}_{\mathtt{m}}\setminus\overline{\mathbb{S}}_{0,2}),$
		\begin{align*}
			\max_{|\alpha| \in\llbracket 0,q\rrbracket}\sup_{(b,\omega)\in\mathcal{O}}\left|\partial_{b,\omega}^{\alpha}\left(\mu_{j,1}^{(m+1)}(b,\omega)-\mu_{j_{0},2}^{(m+1)}(b,\omega)\right)\right|\leqslant C\big(1+\varepsilon\gamma^{-1-q}N_0^{\mu_2}N_{m}^{-\mu_{2}}\big)\,\langle j,j_{0}\rangle.
		\end{align*}
		The convergence of the series $\sum N_m^{-\mu_2}$  implies  the desired result with a constant $C$ uniform in $m.$ This achieves the induction argument. Observe that the bound \eqref{bound Cv mujkm} implies the convergence of the sequence  $\left(\mu_{j,k}^{(m)}\right)_{m\in\mathbb{N}}$ toward some $\mu_{j,k}^{(\infty)}\in W^{q,\infty,\gamma }(\mathcal{O},\mathbb{C})$ given by
		\begin{equation}\label{constru mujkfty}
			\mu_{j,k}^{(\infty)}=\mu_{j,k}^{(0)}+\sum_{m=0}^\infty\left(\mu_{j,k}^{(m+1)}-\mu_{j,k}^{(m)}\right)\triangleq \mu_{j,k}^{(0)}+r_{j,k}^{(\infty)},
		\end{equation}
		where $(\mu_{j,k}^{(0)})$ was introduced in Proposition \ref{prop proj nor dir}, writes  
		$$\mu_{j,k}^{(0)}(b,\omega,i_{0})=\Omega_{j,k}(b)+j\big(c_{k}(b,\omega,i_{0})-\mathtt{v}_k(b)\big).$$
		The estimate \eqref{e-rjfty} follows immediately from \eqref{constru mujkfty} and \eqref{bound Cv mujkm}. Define the diagonal operator $\mathscr{D}_{\infty,k}$ defined on the normal modes by
		\begin{align}\label{Dinfty-op}
			\forall (l,j)\in\mathbb{Z}^d\times(\mathbb{Z}_{\mathtt{m}}\setminus\overline{\mathbb{S}}_{0,k}),\quad \mathscr{D}_{\infty,k} {\bf e}_{l,j}=\ii\mu_{j,k}^{(\infty)}{\bf e}_{l,j}.
		\end{align}
		By definition of the off-diagonal norm and \eqref{bound Cv mujkm}, we obtain
		\begin{align}\label{Cv-od-scrDn}
			\|\mathscr{D}_{m,k}-\mathscr{D}_{\infty,k}\|_{\textnormal{\tiny{O-d}},s_{0}}^{q,\gamma,\mathtt{m}}=\sup_{j\in\mathbb{Z}_{\mathtt{m}}\setminus\overline{\mathbb{S}}_{0,k}}\left\|\mu_{j,k}^{(m)}-\mu_{j,k}^{(\infty)}\right\|^{q,\gamma}\leqslant C\, \gamma\,\delta_{0}(s_{h})N_{0}^{\mu_{2}} N_{m}^{-\mu_{2}}.
		\end{align}
		Consider the diagonal operator  $\mathscr{L}_{\infty}\triangleq \omega\cdot\partial_{\varphi}\mathbf{I}_{\mathtt{m},\perp}+\mathscr{D}_{\infty},$ where $\mathscr{D}_\infty$ is introduced in \eqref{Dinfty-op}. For any $m\in\mathbb{N},$ applying  \eqref{Cv-od-scrDn} and \eqref{hyp-ind dltp} yields
		\begin{align*}
			\interleave\mathscr{L}_{m}-\mathscr{L}_{\infty}\interleave_{s_{0}}^{q,\gamma,\mathtt{m}}&\leqslant2\max_{k\in\{1,2\}}\|\mathscr{D}_{m,k}-\mathscr{D}_{\infty,k}\|_{\textnormal{\tiny{O-d}},s_{0}}^{q,\gamma,\mathtt{m}}+\interleave\mathscr{R}_{m}\interleave_{s_{0}}^{q,\gamma,\mathtt{m}}\\
			&\leqslant C\, \gamma\,\delta_{0}(s_{h})N_{0}^{\mu_{2}} N_{m}^{-\mu_{2}},
		\end{align*}
		where $\mathscr{L}_{m}$ is given in \eqref{Op-Lm}. As a consequence,
		\begin{align*}
			\lim_{m\rightarrow\infty} \interleave\mathscr{L}_{m}-\mathscr{L}_{\infty}\interleave_{s_{0}}^{q,\gamma,\mathtt{m}}=0.
		\end{align*}
		Now we define the sequence $\left(\widehat{\Phi}_m\right)_{m\in\mathbb{N}}$ of the successive transformations as follows 
		\begin{equation}\label{Def-Phi}
			\widehat\Phi_0\triangleq \Phi_0\qquad\textnormal{and}\qquad \forall m\geqslant1,\quad \widehat\Phi_m\triangleq \Phi_0\circ\Phi_1\circ...\circ\Phi_m.
		\end{equation} 
		The identity ${\Phi}_{m}=\hbox{Id}+\Psi_m$ gives 
		$$\widehat\Phi_{m+1}=\widehat\Phi_{m}+\widehat\Phi_{m}\Psi_{m+1}.$$
		Using \eqref{e-Psi-scrR-hyb} and \eqref{hyp-ind dltp}, a completeness argument implies that the series $\sum(\widehat\Phi_{m+1}-\widehat\Phi_{m})$ converges to an element $\Phi_\infty$ still close to the identity, so invertible and which satisfies
		\begin{equation}\label{Cv-hyp-Phin}
			\interleave \Phi_{\infty}^{-1}-\widehat{\Phi}_{n}^{-1}\interleave_{s_0+1}^{q,\gamma,\mathtt{m}}+\interleave\Phi_{\infty}-\widehat{\Phi}_{n}\interleave_{s_0+1}^{q,\gamma,\mathtt{m}}\lesssim\delta_{0}(s_h)N_0^{\mu_2}N_{n+1}^{-\mu_2}
		\end{equation}
		and \eqref{cont-Phifty}. We refer the reader to \cite[Prop. 6.5]{HR21} for the complete computations up to slight modifications corresponding to the hybrid norm. By construction \eqref{Def-Phi} and \eqref{Op-Lm1}, we have in $\mathscr{O}_{n+1}^{\gamma}$ the following identity
		\begin{align*}
			\widehat{\Phi}_{n}^{-1}\mathscr{L}_{0}\widehat{\Phi}_{n}&=\omega\cdot\partial_{\varphi}\mathbf{I}_{\mathtt{m},\perp}+\mathscr{D}_{n+1}+\mathscr{R}_{n+1}\\
			&=\mathscr{L}_{\infty}+\mathscr{D}_{n+1}-\mathscr{D}_{\infty}+\mathscr{R}_{n+1}.
		\end{align*}
		Assume for a while that the set $\mathscr{O}_{\infty,n}^{\gamma,\tau_1,\tau_2}(i_{0})$ described in Proposition \ref{prop RR} satisfies the following inclusion property with respect to the  intermediate Cantor sets given by \eqref{Cantor-SX},
		\begin{equation}\label{incl Crr}
			\mathscr{O}_{\infty,n}^{\gamma,\tau_1,\tau_2}(i_{0})\subset\bigcap_{m=0}^{n+1}\mathscr{O}_{m}^{\gamma}=\mathscr{O}_{n+1}^{\gamma}.
		\end{equation}
		Hence, in restriction to $\mathscr{O}_{\infty,n}^{\gamma,\tau_1,\tau_2}(i_0)\subset\mathscr{O}_{n+1}^{\gamma}$, we obtain 	
		\begin{align*}
			\Phi_{\infty}^{-1}\mathscr{L}_{0}\Phi_{\infty}&=\mathscr{L}_{\infty}+\big(\mathscr{D}_{n+1}-\mathscr{D}_{\infty}+\mathscr{R}_{n+1}\big)\Pi_{\overline{\mathbb{S}}_{0}}^{\perp}\\
			&\quad+\Phi_{\infty}^{-1}\mathscr{L}_{0}\left(\Phi_{\infty}-\widehat{\Phi}_{n}\right)+\left(\Phi_{\infty}^{-1}-\widehat{\Phi}_{n}^{-1}\right)\mathscr{L}_{0}\widehat{\Phi}_{n}\\
			&\triangleq \mathscr{L}_{\infty}+\mathscr{E}_{n}^1.
		\end{align*}
		The estimate \eqref{e-scrEn1} is obtained by using \eqref{hyb nor}, Lemma \ref{properties of Toeplitz in time operators}-(ii)-(iii), \eqref{Cv-od-scrDn}, \eqref{e-scrL0} and \eqref{Cv-hyp-Phin} combined with \eqref{def dltm dltmh}, \eqref{hyp-ind dltp}, \eqref{edlt0s}, \eqref{cont-Phifty} and \eqref{sml-RR}. Now it remains to prove \eqref{incl Crr}. This is done by a finite induction on $m$ with $n$ fixed. First, by definition we have
		$\mathscr{O}_{\infty,n}^{\gamma,\tau_1,\tau_2}(i_{0})\subset\mathcal{O}\triangleq \mathscr{O}_{0}^{\gamma }.$ 
		Now suppose that $\mathscr{O}_{\infty,n}^{\gamma,\tau_1,\tau_2}(i_{0})\subset\mathscr{O}_{m}^{\gamma}$ for $m\leqslant n$ and let us prove that 
		\begin{align}\label{Inc-Orr in Om+1}
			\mathscr{O}_{\infty,n}^{\gamma,\tau_1,\tau_2}(i_{0})\subset\mathscr{O}_{m+1}^{\gamma}.
		\end{align}
		Let  $(b,\omega)\in\mathscr{O}_{\infty,n}^{\gamma,\tau_1,\tau_2}(i_{0}).$ For $(l,j,j_{0})\in\mathbb{Z}^{d }\times(\mathbb{Z}_{\mathtt{m}}\setminus\overline{\mathbb{S}}_{0,1})\times(\mathbb{Z}_{\mathtt{m}}\setminus\overline{\mathbb{S}}_{0,2})$ such that $0\leqslant \langle l,j,j_0\rangle\leqslant N_{m}$, the triangle inequality, \eqref{Cv-od-scrDn}, \eqref{p-RR} and \eqref{e-dlt0sh init} imply
		\begin{align*}
			\left|\omega\cdot l+\mu_{j,1}^{(m)}(b,\omega)-\mu_{j_{0},2}^{(m)}(b,\omega)\right| & \geqslant\left|\omega\cdot l+\mu_{j,1}^{(\infty)}(b,\omega)-\mu_{j_{0},2}^{(\infty)}(b,\omega)\right|-2\max_{k\in\{1,2\}}\sup_{j\in\mathbb{Z}_{\mathtt{m}}\setminus\overline{\mathbb{S}}_{0,k}}\left\|\mu_{j,k}^{(m)}-\mu_{j,k}^{(\infty)}\right\|^{q,\gamma}\\
			& \geqslant\tfrac{2\gamma}{\langle l,j,j_0\rangle^{\tau_2}}-2C\gamma\delta_{0}(s_{h})N_{0}^{\mu_{2}} N_{m}^{-\mu_{2}}\\
			& \geqslant\tfrac{\gamma}{\langle l,j,j_0\rangle^{\tau_2}}\left(2-2C\gamma\varepsilon_{0}\langle l,j,j_0\rangle^{\tau_2-\mu_2}\right).
		\end{align*}
		Thus for $\varepsilon_0$ small enough and by  \eqref{p-RR}(implying that $\mu_2\geqslant \tau_2$) we get
		$$\left|\omega\cdot l+\mu_{j,1}^{(m)}(b,\omega)-\mu_{j_{0},2}^{(m)}(b,\omega)\right| > \tfrac{\gamma}{\langle l,j,j_0\rangle^{\tau_2}}\cdot$$
		Now for $k\in\{1,2\}$ and $(l,j,j_0)\in\mathbb{Z}^{d}\times(\mathbb{Z}_{\mathtt{m}}\setminus\overline{\mathbb{S}}_{0,k})^2$ with $(l,j)\neq (0,j_0)$ and $|l|\leqslant N_m,$ we get
		\begin{align*}
			\left|\omega\cdot l+\mu_{j,k}^{(m)}(b,\omega)-\mu_{j_{0},k}^{(m)}(b,\omega)\right| & \geqslant\left|\omega\cdot l+\mu_{j,k}^{(\infty)}(b,\omega)-\mu_{j_{0},k}^{(\infty)}(b,\omega)\right|-2\sup_{j\in\mathbb{Z}_{\mathtt{m}}\setminus\overline{\mathbb{S}}_{0,k}}\left\|\mu_{j,k}^{(m)}-\mu_{j,k}^{(\infty)}\right\|^{q,\gamma}\\
			& \geqslant\tfrac{2\gamma\langle j-j_0\rangle}{\langle l\rangle^{\tau_2}}-2C\gamma\delta_{0}(s_{h})N_{0}^{\mu_{2}} N_{m}^{-\mu_{2}}\\
			& \geqslant\tfrac{\gamma\langle j-j_0\rangle}{\langle l\rangle^{\tau_2}}\left(2-2C\gamma\varepsilon_{0}\langle l\rangle^{\tau_2-\mu_2}\right).
		\end{align*}
		Hence, taking $\varepsilon_0$ small enough, we obtain $k\in\{1,2\},$
		$$\left|\omega\cdot l+\mu_{j,k}^{(m)}(b,\omega)-\mu_{j_{0},k}^{(m)}(b,\omega)\right|>\tfrac{\gamma\langle j-j_0\rangle}{\langle l\rangle^{\tau_2}}\cdot$$
		Hence, $(b,\omega)\in\mathscr{O}_{m+1}^{\gamma}$ which proves \eqref{Inc-Orr in Om+1}.\\
		\textbf{(ii)} One can get the estimates \eqref{ed-rjkfty} and \eqref{ed-mujkfty} by a similar induction procedure as above starting with \eqref{diff Vpm} and \eqref{ed-hyb-scrR0} applied with $\mathtt{p}=4\tau_2q+4\tau_2.$ For more details, we refer the reader to \cite[Prop. 6.5]{HR21}.
	\end{proof}
	We  end this section with the effective construction of the approximate right inverse of the linearized operator in the normal directions. Since we have constructed a  diagonal operator $\mathscr{L}_{\infty}$ with Fourier multiplier entries, the situation is brought back to two decoupled scalar studies. Therefore, we can copy the proof done in \cite[Prop. 6.6]{HR21} with small adaptations and obtain the following result.
	\begin{prop}\label{prop inv linfty}
		Let $(\gamma,q,d,\tau_{1},s_{0},\mu_2,s_h,S,\mathtt{m})$ satisfying \eqref{setting tau1 and tau2}--\eqref{init Sob cond}, \eqref{ouvert-sym}  and \eqref{p-RR}--\eqref{sml-RR}. There exists $\sigma_5\triangleq \sigma_5(\tau_1,\tau_2,q,d)\geqslant\sigma_{4}$ such that if 
		\begin{equation}\label{bnd frkIn-1}
			\|\mathfrak{I}_0\|_{s_h+\sigma_5}^{q,\gamma,\mathtt{m}}\leqslant 1,
		\end{equation}
		then the following assertions hold true.
		\begin{enumerate}[label=(\roman*)]
			\item Consider the operator  $\mathscr{L}_{\infty}$ defined in Proposition $\ref{prop RR},$ then there exists a family of  linear operators $\big(\mathtt{T}_n\big)_{n\in\mathbb{N}}$  defined in $\mathcal{O}$ satisfying the estimate
			$$
			\forall s\in[s_0,S],\quad \sup_{n\in\mathbb{N}}\|\mathtt{T}_{n}\rho\|_{s}^{q,\gamma ,\mathtt{m}}\lesssim \gamma ^{-1}\|\rho\|_{s+\tau_{1}q+\tau_{1}}^{q,\gamma ,\mathtt{m}}$$
			and such that for any $n\in\mathbb{N}$, in the Cantor set
			$$\Lambda_{\infty,n}^{\gamma,\tau_{1}}(i_{0})\triangleq\bigcap_{\underset{(l,j)\in\mathbb{Z}^{d }\times\mathbb{S}_{0}^{c}\atop |l|\leqslant N_{n}}{k\in\{1,2\}}}\left\lbrace(b,\omega)\in\mathcal{O}\quad\textnormal{s.t.}\quad\left|\omega\cdot l+\mu_{j,k}^{(\infty)}(b,\omega,i_{0})\right|>\tfrac{\gamma \langle j\rangle }{\langle l\rangle^{\tau_{1}}}\right\rbrace,
			$$
			we have
			$$
			\mathscr{L}_{\infty}\mathtt{T}_n=\mathbf{I}_{\mathtt{m},\perp}+{\mathscr{E}}_{n}^2,
			$$
			with 
			$$
			\forall s_{0}\leqslant s\leqslant\overline{s}\leqslant S, \quad \|{\mathscr{E}}_{n}^2\rho\|_{s}^{q,\gamma ,\mathtt{m}} \lesssim 
			N_n^{s-\overline{s}}\gamma^{-1}\|\rho\|_{\overline{s}+1+\tau_{1}q+\tau_{1}}^{q,\gamma ,\mathtt{m}}.
			$$
			\item 
			There exists a family of linear  operators $\big(\widehat{\mathtt{T}}_{n}\big)_{n\in\mathbb{N}}$ satisfying
			\begin{equation*}
				\forall \, s\in\,[ s_0, S],\quad\sup_{n\in\mathbb{N}}\|\widehat{\mathtt{T}}_{n}\rho\|_{s}^{q,\gamma ,\mathtt{m}}\lesssim\gamma^{-1}\left(\|\rho\|_{s+\sigma_5}^{q,\gamma ,\mathtt{m}}+\| \mathfrak{I}_{0}\|_{s+\sigma_5}^{q,\gamma ,\mathtt{m}}\|\rho\|_{s_{0}+\sigma_5}^{q,\gamma,\mathtt{m}}\right)
			\end{equation*}
			and such that in the Cantor set
			\begin{equation}\label{ttGn}
				\mathtt{G}_n(\gamma,\tau_{1},\tau_{2},i_{0})\triangleq \mathcal{O}_{\infty,n}^{\gamma,\tau_{1}}(i_{0})\cap\mathscr{O}_{\infty,n}^{\gamma,\tau_{1},\tau_{2}}(i_{0})\cap\Lambda_{\infty,n}^{\gamma,\tau_{1}}(i_{0}),
			\end{equation}
			we have
			$$
			\widehat{\mathcal{L}}\,\widehat{\mathtt{T}}_{n}=\mathbf{I}_{\mathtt{m},\perp}+\mathtt{E}_n,
			$$
			where $\mathtt{E}_n$ satisfies the following estimate
			\begin{align*}
				\forall\, s\in [s_0,S],\quad  &\|\mathtt{E}_n\rho\|_{s_0}^{q,\gamma ,\mathtt{m}}
				\nonumber\lesssim N_n^{s_0-s}\gamma^{-1}\Big( \|\rho\|_{s+\sigma_5}^{q,\gamma,\mathtt{m}}+\varepsilon\gamma^{-2}\| \mathfrak{I}_{0}\|_{s+\sigma_5}^{q,\gamma,\mathtt{m}}\|\rho\|_{s_{0}+\sigma_5}^{q,\gamma,\mathtt{m}} \Big)\\
				&\qquad\qquad\qquad\quad+ \varepsilon\gamma^{-3}N_{0}^{{\mu}_{2}}N_{n+1}^{-\mu_{2}} \|\rho\|_{s_0+\sigma_5}^{q,\gamma,\mathtt{m}}.
			\end{align*}
			Recall that  $\widehat{\mathcal{L}},$ $  \mathcal{O}_{\infty,n}^{\gamma,\tau_{1}}(i_{0})$ and $\mathscr{O}_{\infty,n}^{\gamma,\tau_{1},\tau_{2}}(i_{0})$ are given in Propositions $\ref{lemma-normal-s}$, $\ref{prop strighten}$ and $\ref{prop RR}$, respectively.
			\item In the Cantor set $\mathtt{G}_{n}(\gamma,\tau_{1},\tau_{2},i_{0})$, we have the following splitting
			$$\widehat{\mathcal{L}}=\widehat{\mathtt{L}}_{n}+\widehat{\mathtt{R}}_{n},\qquad\textnormal{with}\qquad\widehat{\mathtt{L}}_{n}\widehat{\mathtt{T}}_{n}=\textnormal{Id}\qquad\textnormal{and}\qquad\widehat{\mathtt{R}}_{n}=\mathtt{E}_{n}\widehat{\mathtt{L}}_{n},$$
			where the operators $\widehat{\mathtt{L}}_{n}$ and $\widehat{\mathtt{R}}_{n}$ are defined in  $\mathcal{O}$ and satisfy the following estimates
			\begin{align*}
				\forall s\in[s_{0},S],\quad& \sup_{n\in\mathbb{N}}\|\widehat{\mathtt{L}}_{n}\rho\|_{s}^{q,\gamma,\mathtt{m}}\lesssim\|\rho\|_{s+1}^{q,\gamma,\mathtt{m}}+\varepsilon\gamma^{-2}\|\mathfrak{I}_{0}\|_{s+\sigma_5}^{q,\gamma,\mathtt{m}}\|\rho\|_{s_{0}+1}^{q,\gamma,\mathtt{m}},\\
				\forall s\in[s_{0},S],\quad &\|\widehat{\mathtt{R}}_{n}\rho\|_{s_{0}}^{q,\gamma,\mathtt{m}}\lesssim N_{n}^{s_{0}-s}\gamma^{-1}\left(\|\rho\|_{s+\sigma_5}^{q,\gamma,\mathtt{m}}+\varepsilon\gamma^{-2}\|\mathfrak{I}_{0}\|_{s+\sigma_5}^{q,\gamma,\mathtt{m}}\|\rho\|_{s_{0}+\sigma_5}^{q,\gamma,\mathtt{m}}\right)\\
				&\qquad\qquad\qquad\quad+\varepsilon\gamma^{-3}N_{0}^{\mu_{2}}N_{n+1}^{-\mu_{2}}\|\rho\|_{s_{0}+\sigma_5}^{q,\gamma,\mathtt{m}}.
			\end{align*}
		\end{enumerate}
	\end{prop}
	\section{Construction of quasi-periodic solutions }
	We provide, in this last section, a construction of a non-trivial solution to the equation \eqref{operatorF}. This is done in two steps. First, we implement a Nash-Moser iteration, where we find a solution provided that the parameters $(b,\omega)$ belong to a suitable Borel set. The latter is constructed as the intersection of the Cantor sets required to invert the linearized operator in the normal modes for all the steps of the procedure. Then we rigidified  the frequencies in order to get a solution for the original problem where $\alpha=-\mathtt{J}\omega_{\textnormal{Eq}}(b).$ This gives rise to a final set  described in terms of  $b$ that we should estimate its Lebesgue  measure. Actually, we prove that it has asymptotically full measure as the parameter $\varepsilon$ vanishes.
	\subsection{Nash-Moser iteration}
	Here, we perform the Nash-Moser scheme which allows to find a solution of
	$$\mathcal{F}\big(i,\alpha,b,\omega\big)\triangleq\mathcal{F}\big(i,\alpha,b,\omega,\varepsilon\big)=0,$$ 
	with $\mathcal{F}$ as in \eqref{operatorF}.
	This method is classical and has been used in several papers, see for instance \cite{BBMH18,BBM14,BB15}. 
	 The iterative construction of the approximate solutions is summarized in the following proposition. The proof is a slight modification of the one exposed in \cite{BM18,HHM21,HR21}.
	\begin{prop}\label{Nash-Moser}
		\textbf{(Nash-Moser scheme)}\\
		Let $(\tau_{1},\tau_{2},q,d,s_{0})$ satisfy \eqref{setting tau1 and tau2}--\eqref{init Sob cond} and $\mathtt{m}\geqslant \mathtt{m}^*,$ where $\mathtt{m}^*$ is defined in Corollary $\ref{coro-equilib-freq}.$ We consider the following parameters
		\begin{equation}\label{param NM}
			\left\lbrace\begin{array}{rcl}
				\overline{a} & = & \tau_{2}+{3}\\
				\mu_1 & = & 3q(\tau_{2}+{3})+6\overline{\sigma}+6\\
				a_{1} & = & 6q(\tau_{2}+{3})+12\overline{\sigma}+15\\
				a_{2} & = & 3q(\tau_{2}+{3})+6\overline{\sigma}+9\\
				\mu_{2} & = & 2q(\tau_{2}+{3})+5\overline{\sigma}+7\\
				s_{h} & = & s_{0}+4q(\tau_{2}+{3})+9\overline{\sigma}+11\\
				\kappa_{1} & = & 2s_{h}-s_{0}
			\end{array}\right.
		\end{equation}
		where the number $\overline{\sigma}=\overline{\sigma}(\tau_{1},\tau_{2},d)$ is the total loss of regularity given by Theorem $\ref{theo appr inv}.$
		There exist $C_{\ast}>0$ and $\varepsilon_{0}>0$ such that for any $\varepsilon\in[0,\varepsilon_{0}]$ we impose the constraint relating $\gamma$ and $N_{0}$ to $\varepsilon$,
		\begin{equation}\label{rigidity gam-N0}
			0<a<\tfrac{1}{\mu_{2}+q+2},\qquad \gamma\triangleq\varepsilon^{a}, \qquad N_{0}\triangleq\gamma^{-1}.
		\end{equation}
		Let $n\in\mathbb{N}.$ We introduce the finite dimensional subspace $E_{n,\mathtt{m}}$ defined by
		$$E_{n,\mathtt{m}}\triangleq\Big\{\mathfrak{I}=(\Theta,I,z)\in\mathbb{T}^d\times\mathbb{R}^d\times\mathbf{H}_{\perp,\mathtt{m}}^{\infty}\quad
		\textnormal{s.t.}\quad\Theta=\Pi_{N_n}\Theta,\quad I=\Pi_{N_n}I\quad\textnormal{and}\quad z=\Pi_{N_n}z\Big\},$$
		where $\Pi_{N_n}$ is the projector defined through \eqref{def projectors PiN}. Then, the following properties hold true.
		\begin{itemize}
			\item [$(\mathcal{P}1)_{n}$] There exists a $q$-times differentiable application
			$$\mathtt{W}_{n}:\begin{array}[t]{rcl}
				\mathcal{O} & \rightarrow &  E_{n-1,\mathtt{m}}\times\mathbb{R}^{d}\times\mathbb{R}^{d+1}\\
				(b,\omega) & \mapsto & \big(\mathfrak{I}_{n},\alpha_{n}-\mathtt{J}\omega,0\big)
			\end{array}$$
			satisfying $\mathtt{W}_{0}=0$ and for $n\in\mathbb{N}^*,$
			\begin{equation}\label{e-ttWn}
				\|\mathtt{W}_{n}\|_{s_{0}+\overline{\sigma}}^{q,\gamma,\mathtt{m}}\leqslant C_{\ast}\varepsilon\gamma^{-1}N_{0}^{q\overline{a}}.
			\end{equation}
			We set
			\begin{equation}\label{ttU0}
				\mathtt{U}_0\triangleq\Big((\varphi,0,0),\mathtt{J}\omega,(b,\omega)\Big)
			\end{equation}
			and for $n\in\mathbb{N}^*,$
			\begin{equation}\label{ttUn}
				\mathtt{U}_{n}\triangleq\mathtt{U}_{0}+\mathtt{W}_{n}\qquad \textnormal{and}\qquad \mathtt{H}_{n} \triangleq\mathtt{U}_{n}-\mathtt{U}_{n-1},.
			\end{equation}
			Then 
			\begin{align}
				\forall s\in[s_{0},S],\quad\|\mathtt{H}_{1}\|_{s}^{q,\gamma,\mathtt{m}}&\leqslant \tfrac{1}{2}C_{\ast}\varepsilon\gamma^{-1}N_{0}^{q\overline{a}},\label{ttH1s}\\
				\forall\, 2\leqslant m\leqslant n,\quad\|\mathtt{H}_{m}\|_{s_{0}+\overline{\sigma}}^{q,\gamma,\mathtt{m}}&\leqslant C_{\ast}\varepsilon\gamma^{-1}N_{m-1}^{-a_{2}},\label{ttHks0sig}\\
				\forall n\geqslant 2,\quad\|\mathtt{H}_{n}\|_{\overline{s}_h+\overline{\sigma}}^{q,\gamma,\mathtt{m}}&\leqslant C_{\ast}\varepsilon\gamma^{-1}N_{n-1}^{-a_{2}}.\label{ttHn shbsig4} 
			\end{align}
			\item [$(\mathcal{P}2)_{n}$] Set 
			\begin{equation}\label{in gamn}
				i_{n}\triangleq(\varphi,0,0)+\mathfrak{I}_{n},\qquad \gamma_{n}\triangleq\gamma(1+2^{-n})\in[\gamma,2\gamma].
			\end{equation}
			The torus $i_n$ is reversible and $\mathtt{m}$-fold, that is
			\begin{equation}\label{reversibility in}
				\mathfrak{S}i_n(\varphi)=i_n(-\varphi){\qquad\textnormal{and}\qquad \mathfrak{T}_{\mathtt{m}}i_n(\varphi)=i_n(\varphi),}
			\end{equation}
			with $\mathfrak{S}$ and $\mathfrak{T}_{\mathtt{m}}$ as in \eqref{rev th I z}-\eqref{mfold th I z}. Define also
			$$\mathtt{A}_{0}^{\gamma}\triangleq\mathcal{O}\qquad\mbox{and}\qquad \mathtt{A}_{n+1}^{\gamma}\triangleq\mathtt{A}_{n}^{\gamma}\cap\mathtt{G}_{n}(\gamma_{n+1},\tau_{1},\tau_{2},i_{n})$$
			where $\mathtt{G}_{n}(\gamma_{n+1},\tau_{1},\tau_{2},i_{n})$ is given through \eqref{ttGn}. Consider the open sets 
			$$
			\forall \mathtt{v}>0,\quad \mathrm{O}_{n}^\mathtt{v}\triangleq\Big\{(b,\omega)\in\mathcal{O}\quad\textnormal{s.t.}\quad {\mathtt{dist}}\big((b,\omega),\mathtt{A}_{n}^{2\gamma}\big)< \mathtt{v} N_{n}^{-\overline{a}}\Big\},\qquad\mathtt{dist}(x,\mathtt{A})\triangleq\inf_{y\in\mathtt{A}}\|x-y\|.$$ 
			Then we have the following estimate 
			\begin{equation}\label{decay FttUn}
				\|\mathcal{F}(\mathtt{U}_{n})\|_{s_{0}}^{q,\gamma,\mathtt{m},\mathrm{O}_{n}^{2\gamma}}\triangleq\sum_{\underset{|\alpha|\leqslant q}{\alpha\in\mathbb{N}^{d+1}}}\gamma^{|\alpha|}\sup_{(b,\omega)\in{\mathrm{O}_{n}^{2\gamma}}}\|\partial_{(b,\omega)}^{\alpha}\mathcal{F}(\mathtt{U}_{n})(b,\omega,\cdot)\|_{H^{s_0-|\alpha|}_{\mathtt{m}}}\leqslant C_{\ast}\varepsilon N_{n-1}^{-a_{1}}.
			\end{equation}
			\item[$(\mathcal{P}3)_{n}$] We have the following growth in high regularity norm 
			\begin{equation}\label{growth ttWn}
				\|\mathtt{W}_{n}\|_{\kappa_{1}+\overline{\sigma}}^{q,\gamma,\mathtt{m}}\leqslant C_{\ast}\varepsilon\gamma^{-1}N_{n-1}^{\mu_1}.
			\end{equation}
		\end{itemize}
	\end{prop}
	
	\begin{proof}
		We follow closely \cite[Prop. 7.1]{HR21}.
		First notice that the initial guess $\mathtt{U}_0$ is associated to a reversible flat torus and satisfies by virtue of \eqref{operatorF} and Lemma \ref{tame X per} the following estimate for some large enough constant $C_{\ast}$
		\begin{equation}\label{e-F(ttU0)}
			\forall s\geqslant 0,\quad \|\mathcal{F}(\mathtt{U}_{0})\|_{s}^{q,\gamma,\mathtt{m}}\leqslant C_{\ast}\varepsilon.
		\end{equation}
		The properties $(\mathcal{P}1)_{0},$ $(\mathcal{P}2)_{0}$ and $(\mathcal{P}3)_{0}$ follow immediately since $N_{-1}=1$ and $\mathrm{O}_{0}^{\gamma}=\mathcal{O}$ and by setting $\mathtt{W}_0=0.$ Now, let us turn to the induction step. Fix $n\in\mathbb{N}^*$ and suppose that $(\mathcal{P}1)_{\ell},$ $(\mathcal{P}2)_{\ell}$ and $(\mathcal{P}3)_{\ell}$ hold for any $\ell\in\llbracket 0,n\rrbracket.$ The purpose is to verify that these properties also hold at the order $n+1$. We denote by
		$$L_{n}\triangleq L_{n}(b,\omega)\triangleq d_{i,\alpha}\mathcal{F}(i_{n}\big(b,\omega),\alpha_{n}(b,\omega),(b,\omega)\big)$$
		the linearized operator of $\mathcal{F}$ at the state $(i_n,\alpha_n)$. As we shall see later, the next approximation $\mathtt{U}_{n+1}$ can be obtained through  the construction of a reversible and $\mathtt{m}$-fold preserving approximate right inverse for $L_n,$  which is the subject of Theorem \ref{theo appr inv}. To apply this result and get some bounds on $\mathtt{U}_{n+1}$ we need to establish first some intermediate results connected to the smallness condition and to some Cantor set inclusions.

		\noindent $\blacktriangleright$ \textbf{Smallness/boundedness properties.} First observe that \eqref{param NM} implies \eqref{p-RR}. Thus, to apply Theorem \ref{theo appr inv}, we need to check the smallness \eqref{sml-RR} and boundedness \eqref{bnd frkIn-final} properties. According to \eqref{rigidity gam-N0}, a small enough choice of $\varepsilon$ leads, for some a priori fixed $\varepsilon_0>0,$ to
		\begin{equation}\label{sml NM}
			\varepsilon\gamma^{-2-q}N_0^{\mu_{2}}=\varepsilon^{1-a(\mu_2+q+2)}\leqslant\varepsilon_0,
		\end{equation}
		which is exactly \eqref{sml-RR}. Now, since from \eqref{param NM} $\kappa_{1}=2s_h-s_0,$ then by interpolation inequality in Lemma \ref{lem funct prop}-(v), we have for $n\geqslant2,$
		\begin{equation}\label{interp NM}
			\|\mathtt{H}_{n}\|_{s_{h}+\overline{\sigma}}^{q,\gamma,\mathtt{m}}\lesssim\left(\|\mathtt{H}_{n}\|_{s_{0}+\overline{\sigma}}^{q,\gamma,\mathtt{m}}\right)^{\frac{1}{2}}\left(\|\mathtt{H}_{n}\|_{\kappa_{1}+\overline{\sigma}}^{q,\gamma,\mathtt{m}}\right)^{\frac{1}{2}}.		\end{equation}
		The property \eqref{growth ttWn} applied with the indices $n$ and $n-1$ gives
		\begin{align*}
			\|\mathtt{H}_{n}\|_{\kappa_{1}+\overline{\sigma}}^{q,\gamma,\mathtt{m}}&=\|\mathtt{U}_{n}-\mathtt{U}_{n-1}\|_{\kappa_{1}+\overline{\sigma}}^{q,\gamma,\mathtt{m}}\\
			&=\|\mathtt{W}_{n}-\mathtt{W}_{n-1}\|_{\kappa_{1}+\overline{\sigma}}^{q,\gamma,\mathtt{m}}\\
			&\leqslant\|\mathtt{W}_{n}\|_{\kappa_{1}+\overline{\sigma}}^{q,\gamma,\mathtt{m}}+\|\mathtt{W}_{n-1}\|_{\kappa_{1}+\overline{\sigma}}^{q,\gamma,\mathtt{m}}\\
			&\leqslant 2C_{\ast}\varepsilon\gamma^{-1}N_{n-1}^{\mu_{1}}.
		\end{align*}
		Inserting the last estimate together with \eqref{ttHks0sig} into \eqref{interp NM} leads to
		\begin{equation}\label{est ttHn sh+sigma}
			\forall n\geqslant 2,\quad\|\mathtt{H}_{n}\|_{s_{h}+\overline{\sigma}}^{q,\gamma,\mathtt{m}}\leqslant CC_{\ast}\varepsilon\gamma^{-1}N_{n-1}^{\frac{1}{2}(\mu_{1}-a_{2})}.
		\end{equation}
		Observe that \eqref{param NM} implies in particular $a_{2}\geqslant \mu_{1}+2.$ Hence, by \eqref{rigidity gam-N0}, \eqref{ttH1s} and \eqref{est ttHn sh+sigma}, we infer
		\begin{align*}
			\|\mathtt{W}_{n}\|_{s_{h}+\overline{\sigma}}^{q,\gamma,\mathtt{m}}&\leqslant\|\mathtt{H}_{1}\|_{s_{h}+\overline{\sigma}}^{q,\gamma,\mathtt{m}}+\sum_{k=2}^{n}\|\mathtt{H}_{k}\|_{s_{h}+\overline{\sigma}}^{q,\gamma,\mathtt{m}}\\
			&\leqslant \tfrac{1}{2}C_{\ast}\varepsilon\gamma^{-1}N_{0}^{q\overline{a}}+ CN_{0}^{-1}C_{\ast}\varepsilon\gamma^{-1}\\
			&\leqslant C_{\ast}\varepsilon^{1-a(1+q\overline{a})}.
		\end{align*}
		Remark that  \eqref{param NM} and \eqref{rigidity gam-N0} provide $a\leqslant\tfrac{1}{2(1+q\overline{a})}.$ Thus, taking $\varepsilon$ small enough and $\overline{\sigma}\geqslant\sigma_5$ with $\sigma_5$ as in Proposition \eqref{prop inv linfty}, we get
		\begin{align}\label{bnd frkIn}
			\|\mathfrak{I}_{n}\|_{s_{h}+\sigma_5}^{q,\gamma,\mathtt{m}}&\leqslant\|\mathtt{W}_{n}\|_{s_{h}+\overline{\sigma}}^{q,\gamma,\mathtt{m}}\nonumber\\
			&\leqslant C_{\ast}\varepsilon^{\frac{1}{2}}\nonumber\\
			&\leqslant 1,
		\end{align}
		which corresponds to \eqref{bnd frkIn-1}. Up to increase the value of $\overline{\sigma},$ we can always assume that $s_0+\overline{\sigma}\geqslant\overline{s}_h+\sigma_{4}$ where $\overline{s}_h$ and $\sigma_{4}$ are respectively given by \eqref{param} and Proposition \ref{prop RR}. Consequently \eqref{ttHks0sig} gives \eqref{ttHn shbsig4}.\\
		
		\noindent $\blacktriangleright$ \textbf{Set inclusions.} The properties \eqref{sml NM} and \eqref{bnd frkIn} allow to apply Theorem \ref{theo appr inv}. Hence, we can reduce the linearized operator $L_{n}$ at the current step. Therefore, the sets $\mathtt{A}_{\ell}^{\gamma}$ for $\ell\leqslant n+1$ and $\gamma\in(0,1)$ are well-defined. Our next purpose is to check some suitable inclusions required later for defining the extensions of our quantities outside the constructed Cantor sets. More precisely, we shall verify the following
		\begin{equation*}
			\mathtt{A}_{n+1}^{2\gamma}\subset\mathrm{O}_{n+1}^{4\gamma}\subset\left(\mathtt{A}_{n+1}^{\gamma}\cap\mathrm{O}_{n}^{2\gamma}\right).
		\end{equation*}
		Obviously, by construction, the first inclusion is trivial. Hence, we are left to prove the last one. Observe that by construction $\mathtt{A}_{\ell+1}^{2\gamma}\subset\mathtt{A}_{\ell}^{2\gamma}.$ Then for $(b,\omega)\in \mathrm{O}_{\ell+1}^{4\gamma},$ we have 
		\begin{align*}
			\mathtt{dist}\big((b,\omega),\mathtt{A}_{\ell}^{2\gamma}\big)&\leqslant  \mathtt{dist}\big((b,\omega),\mathtt{A}_{\ell+1}^{2\gamma}\big)\\
			&<4\gamma N_{\ell+1}^{-\overline{a}}= 4\gamma N_{\ell}^{-\overline{a}}N_{0}^{-\frac{1}{2}\overline{a}}\\
			&<2\gamma N_{\ell}^{-\overline{a}}.
		\end{align*}
		The last estimate is true provided that $2N_0^{-\frac12\overline a}<1,$ which is obtained taking $\varepsilon$ small enough according to \eqref{rigidity gam-N0}. Thus, we have proved
		\begin{equation}\label{O2gm in Ogm}
			\forall \ell\in \llbracket 0,n\rrbracket,\quad  \mathrm{O}_{\ell+1}^{4\gamma}\subset \mathrm{O}_{\ell}^{2\gamma}.
		\end{equation}
		Now we prove by induction in $\ell$ the following inclusion
		\begin{equation}\label{hyprec O in A}
			\forall \ell\in \llbracket 0,n+1\rrbracket,\quad  \mathrm{O}_{\ell}^{4\gamma}\subset\mathtt{A}_{\ell}^{\gamma}.
		\end{equation}
		The case $\ell=0$ is obvious because $\mathrm{O}_{0}^{4\gamma}=\mathcal{O}=\mathtt{A}_{0}^{\gamma}.$ Now suppose that \eqref{hyprec O in A} is true for some $\ell\in \llbracket 0,n\rrbracket$ and let us check the inclusion property \eqref{hyprec O in A} at the next order $\ell+1.$ Putting together \eqref{O2gm in Ogm} and \eqref{hyprec O in A}, we get
		$$\mathrm{O}_{\ell+1}^{4\gamma}\subset\mathrm{O}_{\ell}^{2\gamma}\subset\mathrm{O}_{\ell}^{4\gamma}\subset\mathtt{A}_{\ell}^{\gamma}.$$
		Hence, it remains to verify that 
		$$\mathrm{O}_{\ell+1}^{4\gamma}\subset\mathtt{G}_{\ell}\Big(\gamma_{\ell+1},\tau_{1},\tau_{2},i_{\ell}\Big).$$
		Let $(b,\omega)\in\mathrm{O}_{\ell+1}^{4\gamma},$ then by construction, one can find $(b',\omega')\in\mathtt{A}_{\ell+1}^{2\gamma}$ with 
		\begin{equation}\label{dist b b'}
			\mathtt{dist}\left((b,\omega),(b',\omega')\right)<4\gamma N_{\ell+1}^{-\overline{a}}.
		\end{equation}
		Let us start proving that $(b,\omega)\in\mathcal{O}_{\infty,\ell}^{\gamma_{\ell+1},\tau_{1}}(i_{\ell}).$ For all $k\in\{1,2\}$ and $(l,j)\in\mathbb{Z}^{d}\times\mathbb{Z}_{\mathtt{m}}\setminus\{(0,0)\}$ with $|l|\leqslant N_{\ell},$ we have by triangle and Cauchy-Schwarz inequalities together with \eqref{dist b b'} and the fact $(b',\omega')\in\mathcal{O}_{\infty,\ell}^{2\gamma_{\ell+1},\tau_{1}}(i_{\ell})$,
		\begin{align*}
			\left|\omega\cdot l+j{c_{k}(b,\omega,i_{\ell})}\right|&\geqslant\left|\omega'\cdot l+j{c_{k}(b',\omega',i_{\ell})}\right|-|\omega-\omega'||l|-|j|\left|c_{k}(b,\omega,i_{\ell})-c_{k}(b',\omega',i_{\ell})\right|\\
			&>\tfrac{4\gamma_{\ell+1}^{\upsilon}2^{\upsilon}\langle j\rangle}{\langle l\rangle^{\tau_{1}}}-4\gamma N_{\ell+1}^{1-\overline{a}}-\langle j\rangle\left|c_{k}(b,\omega,i_{\ell})-c_{k}(b',\omega',i_{\ell})\right|.
		\end{align*}
		Now the Mean Value Theorem and the definition of $\mathrm{O}_{\ell+1}^{4\gamma}$ imply
		$$\left|c_{k}(b,\omega,i_{\ell})-c_{k}(b',\omega',i_{\ell})\right|\leqslant CN_{\ell+1}^{-\overline{a}}\|c_{k}(i_{\ell})\|^{q,\gamma}.$$
		From \eqref{sml-r0}, we deduce 
		\begin{align*}
			\|c_{k}(i_{\ell})\|^{q,\gamma}&\leqslant\|c_{k}(i_{\ell})-\mathtt{v}_{k}\|^{q,\gamma}+\|\mathtt{v}_{k}\|^{q,\gamma}\\
			&\leqslant C.
		\end{align*}
		Combining the last two estimates gives
		$$\left|c_{k}(b,\omega,i_{\ell})-c_{k}(b',\omega',i_{\ell})\right|\leqslant C\gamma\gamma^{-1}N_{\ell+1}^{-\overline{a}}\leqslant C\gamma N_{\ell+1}^{1-\overline{a}}.$$
		Consequently, using the facts that $\gamma_{\ell}\geqslant\gamma$ and $\upsilon\in(0,1),$ we get
		\begin{align*}
			\left|\omega\cdot l+jc_{k}(b,\omega,i_{\ell})\right|&>\tfrac{4\gamma_{\ell+1}^{\upsilon}2^{\upsilon}\langle j\rangle}{\langle l\rangle^{\tau_{1}}}-C\gamma\langle j\rangle N_{\ell+1}^{1-\overline{a}}\\
			&\geqslant\tfrac{4\gamma_{\ell+1}^{\upsilon}\langle j\rangle}{\langle l\rangle^{\tau_{1}}}\left(2^{\upsilon}-CN_{\ell+1}^{\tau_{1}+1-\overline{a}}\right).
		\end{align*}
		Our choice of parameters \eqref{param NM} and \eqref{setting tau1 and tau2} implies in particular
		\begin{equation}\label{cond-abarre-1}
			\overline{a}=\tau_{2}+{3}\geqslant \tau_{1}+2.
		\end{equation}
		Therefore, taking $N_{0}$ sufficiently large, we obtain
		$$2^{\upsilon}-CN_{\ell+1}^{\tau_{1}+1-\overline{a}}\geqslant 2^{\upsilon}-CN_{0}^{-1}>1,$$
		which implies in turn
		$$\left|\omega\cdot l+jc_{k}(b,\omega,i_{\ell})\right|>\tfrac{4\gamma_{\ell+1}^{\upsilon}\langle j\rangle}{\langle l\rangle^{\tau_{1}}}\cdot
		$$
		This proves that $(b,\omega)\in\mathcal{O}_{\infty,\ell}^{\gamma_{\ell+1},\tau_{1}}(i_{\ell}).$ Now, let us check that $(b,\omega)\in\mathscr{O}_{\infty,\ell}^{\gamma_{\ell+1},\tau_{1},\tau_{2}}(i_{\ell}).$ For all $k\in\{1,2\}$ and $(l,j,j_0)\in\mathbb{Z}^{d}\times(\mathbb{Z}_{\mathtt{m}}\setminus\overline{\mathbb{S}}_{0,k})^2$ with $|l|\leqslant N_{\ell},$ using the triangle and Cauchy-Schwarz inequalities together with \eqref{dist b b'} and the fact that $(b',\omega')\in\mathscr{O}_{\infty,\ell}^{2\gamma_{\ell+1},\tau_{1},\tau_{2}}(i_{\ell})$
		\begin{align*}
			\left|\omega\cdot l+\mu_{j,k}^{(\infty)}(b,\omega,i_{\ell})-\mu_{j_{0},k}^{(\infty)}(b,\omega,i_{\ell})\right|&\geqslant\left|\omega'\cdot l+\mu_{j,k}^{(\infty)}(b',\omega',i_{\ell})-\mu_{j_{0},k}^{(\infty)}(b',\omega',i_{\ell})\right|-|\omega-\omega'||l|\\
			&-\left|\mu_{j,k}^{(\infty)}(b,\omega,i_{\ell})-\mu_{j_{0},k}^{(\infty)}(b,\omega,i_{\ell})+\mu_{j_{0},k}^{(\infty)}(b',\omega',i_{\ell})-\mu_{j,k}^{(\infty)}(b',\omega',i_{\ell})\right|\\
			&>\tfrac{4\gamma_{\ell+1}\langle j-j_{0}\rangle}{\langle l\rangle^{\tau_{2}}}-4\gamma N_{\ell+1}^{1-\overline{a}}\\
			&-\left|\mu_{j,k}^{(\infty)}(b,\omega,i_{\ell})-\mu_{j_{0},k}^{(\infty)}(b,\omega,i_{\ell})+\mu_{j_{0},k}^{(\infty)}(b',\omega',i_{\ell})-\mu_{j,k}^{(\infty)}(b',\omega',i_{\ell})\right|.
		\end{align*}
		We remind from \eqref{def mu lim} that the perturbed eigenvalues admit the following structure
		\begin{equation}\label{dec mjkfty}
			\mu_{j,k}^{(\infty)}(b,\omega,i_{\ell})=\mu_{j,k}^{(0)}(b,\omega,i_{\ell})+r_{j,k}^{(\infty)}(b,\omega,i_{\ell}).
		\end{equation}
		Hence,
		\begin{align*}
			&\left|\mu_{j,k}^{(\infty)}(b,\omega,i_{\ell})-\mu_{j_{0},k}^{(\infty)}(b,\omega,i_{\ell})+\mu_{j_{0},k}^{(\infty)}(b',\omega',i_{\ell})-\mu_{j,k}^{(\infty)}(b',\omega',i_{\ell})\right|\\
			&\leqslant\left|\mu_{j,k}^{(0)}(b,\omega,i_{\ell})-\mu_{j_{0},k}^{(0)}(b,\omega,i_{\ell})+\mu_{j_{0},k}^{(0)}(b',\omega',i_{\ell})-\mu_{j,k}^{(0)}(b',\omega',i_{\ell})\right|\\
			&\quad+\left|r_{j,k}^{(\infty)}(b,\omega,i_{\ell})-r_{j,k}^{(\infty)}(b',\omega',i_{\ell})\right|+\left|r_{j_{0},k}^{(\infty)}(b,\omega,i_{\ell})-r_{j_{0},k}^{(\infty)}(b',\omega',i_{\ell})\right|.
		\end{align*}
		The  Mean Value Theorem, \eqref{dist b b'} and the definition of $\mathrm{O}_{\ell+1}^{4\gamma}$ allow to write
		\begin{equation}\label{im00}
			\left|\mu_{j,k}^{(0)}(b,\omega,i_{\ell})-\mu_{j_{0},k}^{(0)}(b,\omega,i_{\ell})+\mu_{j_{0},k}^{(0)}(b',\omega',i_{\ell})-\mu_{j,k}^{(0)}(b',\omega',i_{\ell})\right|\leqslant\gamma CN_{\ell+1}^{1-\overline{a}}\langle j-j_{0}\rangle.
		\end{equation}
		Similarly, using in particular  \eqref{e-rjfty}, \eqref{sml NM} and the definition of $\mathrm{O}_{\ell+1}^{4\gamma}$ we get
		\begin{equation}\label{ir00}
			\left|r_{j,k}^{(\infty)}(b,\omega,i_{\ell})-r_{j,k}^{(\infty)}(b',\omega',i_{\ell})\right|\leqslant C\gamma N_{\ell+1}^{-\overline{a}}\varepsilon \gamma^{-2}\leqslant \gamma CN_{\ell+1}^{1-\overline{a}}\langle j-j_{0}\rangle.
		\end{equation}
		Gathering the previous inequalities and using the facts that $|l|\leqslant N_{\ell}$ and $\gamma_{\ell+1}\geqslant\gamma$ we deduce
		\begin{align*}
			\left|\omega\cdot l+\mu_{j,k}^{(\infty)}(b,\omega,i_{\ell})-\mu_{j_{0},k}^{(\infty)}(b,\omega,i_{\ell})\right|&\geqslant\tfrac{\gamma_{\ell+1}\langle j-j_{0}\rangle}{\langle l\rangle^{\tau_{2}}}\left(4-CN_{\ell+1}^{\tau_{2}+1-\overline{a}}\right).
		\end{align*}
		By the choice \eqref{cond-abarre-1}, if $N_{0}$ is large enough, then we obtain
		$$CN_{\ell+1}^{\tau_{2}+1-\overline{a}}\leqslant CN_0^{-1}<1,$$
		which implies in turn
		$$\left|\omega\cdot l+\mu_{j,k}^{(\infty)}(b,\omega,i_{\ell})-\mu_{j_{0},k}^{(\infty)}(b,\omega,i_{\ell})\right|>\tfrac{2\gamma_{\ell+1}\langle j-j_{0}\rangle}{\langle l\rangle^{\tau_{2}}}\cdot$$
		For all $(l,j,j_0)\in\mathbb{Z}^{d}\times(\mathbb{Z}_{\mathtt{m}}\setminus\overline{\mathbb{S}}_{0,1})\times(\mathbb{Z}_{\mathtt{m}}\setminus\overline{\mathbb{S}}_{0,2})$ with $\langle l,j,j_0\rangle\leqslant N_{\ell},$ using the triangle and Cauchy-Schwarz inequalities together with \eqref{dist b b'} and the fact that $(b',\omega')\in\mathscr{O}_{\infty,\ell}^{2\gamma_{\ell+1},\tau_{1},\tau_{2}}(i_{\ell})$
		\begin{align*}
			\left|\omega\cdot l+\mu_{j,1}^{(\infty)}(b,\omega,i_{\ell})-\mu_{j_{0},2}^{(\infty)}(b,\omega,i_{\ell})\right|&\geqslant\left|\omega'\cdot l+\mu_{j,1}^{(\infty)}(b',\omega',i_{\ell})-\mu_{j_{0},2}^{(\infty)}(b',\omega',i_{\ell})\right|-|\omega-\omega'||l|\\
			&-\left|\mu_{j,1}^{(\infty)}(b,\omega,i_{\ell})-\mu_{j_{0},2}^{(\infty)}(b,\omega,i_{\ell})+\mu_{j_{0},2}^{(\infty)}(b',\omega',i_{\ell})-\mu_{j,1}^{(\infty)}(b',\omega',i_{\ell})\right|\\
			&>\tfrac{4\gamma_{\ell+1}}{\langle l,j,j_{0}\rangle^{\tau_{2}}}-4\gamma N_{\ell+1}^{1-\overline{a}}\\
			&-\left|\mu_{j,1}^{(\infty)}(b,\omega,i_{\ell})-\mu_{j_{0},2}^{(\infty)}(b,\omega,i_{\ell})+\mu_{j_{0},2}^{(\infty)}(b',\omega',i_{\ell})-\mu_{j,1}^{(\infty)}(b',\omega',i_{\ell})\right|.
		\end{align*}
		Similarly to \eqref{im00} and \eqref{ir00}, we get
		\begin{align*}
			\left|\mu_{j,1}^{(0)}(b,\omega,i_{\ell})-\mu_{j_{0},2}^{(0)}(b,\omega,i_{\ell})+\mu_{j_{0},2}^{(0)}(b',\omega',i_{\ell})-\mu_{j,1}^{(0)}(b',\omega',i_{\ell})\right|\leqslant\gamma CN_{\ell+1}^{1-\overline{a}}\langle j,j_{0}\rangle\leqslant\gamma CN_{\ell+1}^{2-\overline{a}},\\
			\left|r_{j,1}^{(\infty)}(b,\omega,i_{\ell})-r_{j,1}^{(\infty)}(b',\omega',i_{\ell})\right|+\left|r_{j_0,2}^{(\infty)}(b,\omega,i_{\ell})-r_{j_0,2}^{(\infty)}(b',\omega',i_{\ell})\right|\leqslant \gamma CN_{\ell+1}^{1-\overline{a}}\langle j,j_{0}\rangle\leqslant\gamma CN_{\ell+1}^{2-\overline{a}}.
		\end{align*}
		Therefore,
		$$\left|\omega\cdot l+\mu_{j,1}^{(\infty)}(b,\omega,i_{\ell})-\mu_{j_{0},2}^{(\infty)}(b,\omega,i_{\ell})\right|\geqslant\tfrac{\gamma_{\ell+1}}{\langle l,j,j_0\rangle^{\tau_{2}}}\left(4-CN_{\ell+1}^{\tau_{2}+2-\overline{a}}\right).$$
		By the choice \eqref{cond-abarre-1}, if $N_{0}$ is large enough, then we obtain
		$${CN_{\ell+1}^{\tau_{2}+2-\overline{a}}\leqslant CN_0^{-1}<1},$$
		and then
		$$\left|\omega\cdot l+\mu_{j,1}^{(\infty)}(b,\omega,i_{\ell})-\mu_{j_{0},2}^{(\infty)}(b,\omega,i_{\ell})\right|>\tfrac{2\gamma_{\ell+1}}{\langle l,j,j_0\rangle^{\tau_{2}}}\cdot$$
		This proves that $(b,\omega)\in\mathscr{O}_{\infty,\ell}^{\gamma_{\ell+1},\tau_{1},\tau_{2}}(i_{\ell}).$ It remains to check that $(b,\omega)\in\Lambda_{\infty,\ell}^{\gamma_{\ell+1},\tau_1}(i_{\ell}).$ For all $k\in\{1,2\}$ and  $(l,j)\in\mathbb{Z}^{d}\times(\mathbb{Z}_{\mathtt{m}}\setminus\overline{\mathbb{S}}_{0,k})$ with $|l|\leqslant N_{\ell},$ we have by left triangle and Cauchy-Schwarz inequalities together with \eqref{dist b b'} and the fact $(b',\omega')\in\Lambda_{\infty,\ell}^{2\gamma_{\ell+1},\tau_{1}}(i_{\ell})$
		\begin{align*}
			\left|\omega\cdot l+\mu_{j,k}^{(\infty)}(b,\omega,i_{\ell})\right|&\geqslant\left|\omega'\cdot l+\mu_{j,k}^{(\infty)}(b',\omega',i_{\ell})\right|-|\omega-\omega'||l|-\left|\mu_{j,k}^{(\infty)}(b,\omega,i_{k})-\mu_{j,k}^{(\infty)}(b',\omega',i_{\ell})\right|\\
			&>\tfrac{2\gamma_{\ell+1}\langle j\rangle}{\langle l\rangle^{\tau_{1}}}-4\gamma N_{\ell}N_{\ell+1}^{-\overline{a}}-\left|\mu_{j,k}^{(\infty)}(b,\omega,i_{\ell})-\mu_{j,k}^{(\infty)}(b',\omega',i_{\ell})\right|\\
			&>\tfrac{2\gamma_{\ell+1}\langle j\rangle}{\langle l\rangle^{\tau_{1}}}-4\gamma N_{\ell+1}^{1-\overline{a}}-\left|\mu_{j,k}^{(\infty)}(b,\omega,i_{\ell})-\mu_{j,k}^{(\infty)}(b',\omega',i_{\ell})\right|.
		\end{align*}
		The Mean Value Theorem and the definition of $\mathrm{O}_{\ell+1}^{4\gamma}$ give
		\begin{align*}
			\left|\mu_{j,k}^{(\infty)}(b,\omega,i_{\ell})-\mu_{j,k}^{(\infty)}(b',\omega',i_{\ell})\right|&\leqslant|(b,\omega)-(b',\omega')|\gamma^{-1}\|\mu_{j,k}^{(\infty)}(i_{\ell})\|^{q,\gamma}\\
			&\leqslant 4N_{\ell+1}^{-\overline{a}}\|\mu_{j,k}^{(\infty)}(i_{\ell})\|^{q,\gamma}.
		\end{align*}
		Now, by triangle inequlity
		$$\forall j\in\mathbb{Z}_{\mathtt{m}}\setminus\overline{\mathbb{S}}_{0,k},\quad \|\mu_{j,k}^{(\infty)}(i_{\ell})\|^{q,\gamma}\leqslant\|\mu_{j,k}^{(\infty)}(i_{\ell})-\Omega_{j,k}\|^{q,\gamma}+\|\Omega_{j,k}\|^{q,\gamma}.$$
		From \eqref{lim omega jk} one has for all $|j|\geqslant\mathtt{m}^{*}$, 
		$$\|\Omega_{j,k}\|^{q,\gamma}\leqslant C|j|.$$
		Besides, \eqref{def mu lim}, \eqref{mu0 r0}, \eqref{e-ed-r0} and \eqref{e-rjfty} imply
		$$\forall j\in\mathbb{Z}_{\mathtt{m}}\setminus\overline{\mathbb{S}}_{0,k},\quad \|\mu_{j,k}^{(\infty)}(i_{\ell})-\Omega_{j,k}\|^{q,\gamma}\leqslant C|j|.$$
		Putting together the preceding three estimates gives
		$$\forall j\in\mathbb{Z}_{\mathtt{m}}\setminus\overline{\mathbb{S}}_{0,k},\quad \|\mu_{j,k}^{(\infty)}(i_{\ell})\|^{q,\gamma}\leqslant C|j|.$$
		As a consequence, we have
		$$\left|\mu_{j,k}^{(\infty)}(b,\omega,i_{\ell})-\mu_{j,k}^{(\infty)}(b',\omega',i_{\ell})\right|\leqslant C\langle j\rangle N_{\ell+1}^{-\overline{a}}\leqslant C\gamma\langle j\rangle N_{\ell+1}^{1-\overline{a}}.$$
		Usint that $|l|\leqslant N_{\ell}\leqslant N_{\ell+1}$ and $\gamma_{\ell+1}\geqslant \gamma$, we get
		\begin{align*}
			\left|\omega\cdot l+\mu_{j,k}^{(\infty)}(b,\omega,i_{\ell})\right|&\geqslant\tfrac{2\gamma_{\ell+1}\langle j\rangle}{\langle l\rangle^{\tau_{1}}}-C\gamma\langle j\rangle N_{\ell+1}^{1-\overline{a}}\\
			&\geqslant\tfrac{\gamma_{\ell+1}\langle j\rangle}{\langle l\rangle^{\tau_{1}}}\left(2-CN_{\ell+1}^{\tau_{1}+1-\overline{a}}\right).
		\end{align*}
		Now, we choose $N_{0}$ sufficiently large so that
		$$CN_{\ell+1}^{\tau_{1}+1-\overline{a}}\leqslant CN_{0}^{-1}<1,$$
		and then
		$$\left|\omega\cdot l+\mu_{j,k}^{(\infty)}(b,\omega,i_{\ell})\right|>\tfrac{\gamma_{\ell+1}\langle j\rangle}{\langle l\rangle^{\tau_{1}}}\cdot
		$$
		This shows that, $(b,\omega)\in\Lambda_{\infty,\ell}^{\gamma_{\ell+1},\tau_{1}}(i_{\ell})$ and finally $(b,\omega)\in\mathtt{G}_{\ell}\big(\gamma_{\ell+1},\tau_{1},\tau_{2},i_{\ell}\big).$ Therefore $(b,\omega)\in\mathtt{A}_{\ell+1}^{\gamma}.$ This achieves the induction proof of \eqref{hyprec O in A}.\\
		
		\noindent $\blacktriangleright$ \textbf{Construction of the next approximation.} Our next task is to construct the next approximate solution denoted $\mathtt{U}_{n+1}.$ Observe that according to Theorem \ref{theo appr inv}, the properties \eqref{sml NM} and \eqref{bnd frkIn} allow to construct a reversible approximate right inverse $\mathrm{T}_n\triangleq\mathrm{T}_{n}(b,\omega)$ of the linearized operator $L_{n}.$ Recall that the operator $\mathrm{T}_n$ is well-defined on the whole set of parameters $\mathcal{O}$ and satisfies, by virtue of \eqref{tame T0}, the following tame estimate
		\begin{equation}\label{eari-NM}
			\forall s\in[s_{0},S],\quad\|\mathrm{T}_{n}\rho\|_{s}^{q,\gamma,\mathtt{m}}\lesssim\gamma^{-1}\left(\|\rho\|_{s+\overline{\sigma}}^{q,\gamma,\mathtt{m}}+\|\mathfrak{I}_{n}\|_{s+\overline{\sigma}}^{q,\gamma,\mathtt{m}}\|\rho\|_{s_{0}+\overline{\sigma}}^{q,\gamma,\mathtt{m}}\right).
		\end{equation}
		In addition it is an approximate right inverse of $L_n$ when restricted to $\mathtt{G}_{n}(\gamma_{n+1},\tau_{1},\tau_{2},i_{n}).$ More precisely, according to \eqref{splitting of approximate inverse} we have in $\mathtt{G}_{n}(\gamma_{n+1},\tau_{1},\tau_{2},i_{n})$
		\begin{equation}\label{approx Ln}
			L_n{\rm T}_n-\textnormal{Id} = \mathcal{E}^{(n)}_1+\mathcal{E}^{(n)}_2+\mathcal{E}^{(n)}_3,
		\end{equation}
		where the error terms in the right hand-side satisfy the estimates \eqref{calE1}, \eqref{calE2} and \eqref{calE3}. The next approximation is defined as follows,
		$$\widetilde{\mathtt{U}}_{n+1}\triangleq\mathtt{U}_{n}+\widetilde{\mathtt{H}}_{n+1},\qquad \widetilde{\mathtt{H}}_{n+1}\triangleq(\widehat{\mathfrak{I}}_{n+1},\widehat{\alpha}_{n+1},0)\triangleq-\mathbf{\Pi}_{N_n}\mathrm{T}_{n}\Pi_{N_n}\mathcal{F}(\mathtt{U}_{n})\in E_{n,\mathtt{m}}\times\mathbb{R}^{d}\times\mathbb{R}^{d+1},
		$$
		where the projector $\mathbf{\Pi}_{N_n}$ and its orthogonal are defined by
		\begin{equation}\label{proj-NM}
			\mathbf{\Pi}_{N_n}(\mathfrak{I},\alpha,0)=(\Pi_{N_n}\mathfrak{I},\alpha,0)\qquad \textnormal{and }\qquad\mathbf{\Pi}_{N_n}^{\perp}(\mathfrak{I},\alpha,0)=(\Pi_{N_n}^{\perp}\mathfrak{I},0,0).
		\end{equation}
		Then, applying Taylor formula yields
		\begin{align}\label{Fnext-NM}
			\nonumber \mathcal{F}(\widetilde{\mathtt{U}}_{n+1})& =  \mathcal{F}(\mathtt{U}_{n})-L_{n}\mathbf{\Pi}_{N_n}\mathrm{T}_{n}\Pi_{N_n}\mathcal{F}(\mathtt{U}_{n})+Q_{n}\\
			\nonumber& =  \mathcal{F}(\mathtt{U}_{n})-L_{n}\mathrm{T}_{n}\Pi_{N_n}\mathcal{F}(\mathtt{U}_{n})+L_{n}\mathbf{\Pi}_{N_n}^{\perp}\mathrm{T}_{n}\Pi_{N_n}\mathcal{F}(\mathtt{U}_{n})+Q_{n}\\
			\nonumber& =  \mathcal{F}(\mathtt{U}_{n})-\Pi_{N_n}L_{n}\mathrm{T}_{n}\Pi_{n}\mathcal{F}(\mathtt{U}_{n})+(L_{n}\mathbf{\Pi}_{N_n}^{\perp}-\Pi_{N_n}^{\perp}L_{n})\mathrm{T}_{n}\Pi_{N_n}\mathcal{F}(\mathtt{U}_{n})+Q_{n}\\
			& =  \Pi_{N_n}^{\perp}\mathcal{F}(\mathtt{U}_{n})-\Pi_{N_n}(L_{n}\mathrm{T}_{n}-\textnormal{Id})\Pi_{N_n}\mathcal{F}(\mathtt{U}_{n})+(L_{n}\mathbf{\Pi}_{N_n}^{\perp}-\Pi_{N_n}^{\perp}L_{n})\mathrm{T}_{n}\Pi_{N_n}\mathcal{F}(\mathtt{U}_{n})+Q_{n},
		\end{align}
		where $Q_n$ denotes the quadratic part given by
		\begin{align}\label{quad-NM}
			Q_{n}=\mathcal{F}(\mathtt{U}_{n}+\widetilde{\mathtt{H}}_{n+1})-\mathcal{F}(\mathtt{U}_{n})-L_{n}\widetilde{\mathtt{H}}_{n+1}.
		\end{align}
		Now, we shall prove \eqref{decay FttUn} at the order $n+1$ for a suitable extension $\mathtt{U}_{n+1}$ of $\widetilde{\mathtt{U}}_{n+1}|_{\mathrm{O}_{n+1}^{2\gamma}}.$ This is done in two steps. The first one is to prove that
		\begin{align}\label{decay ttUn inter}
			\|\mathcal{F}(\widetilde{\mathtt{U}}_{n+1})\|_{s_{0}}^{q,\gamma,\mathtt{m},\mathrm{O}_{n+1}^{4\gamma}}\leqslant C_{\ast}\varepsilon N_{n}^{-a_{1}}.
		\end{align}
		The second step is to construct the classical extension of $\mathtt{U}_{n+1}$ which fulfills the desired estimate \eqref{decay FttUn}.\\
		\ding{226} \textit{Proof of \eqref{decay ttUn inter}.} We estimate each one of the four terms in the right hand-side of \eqref{Fnext-NM}. Let us begin with the first one. Applying Lemma \ref{lem funct prop}-(i) and using the inclusion \eqref{O2gm in Ogm}, we obtain
		\begin{equation}\label{e-first term NM}
			\|\Pi_{N_n}^{\perp}\mathcal{F}(\mathtt{U}_{n})\|_{s_{0}}^{q,\gamma,\mathtt{m},\mathrm{O}_{n+1}^{4\gamma}}\leqslant N_{n}^{s_{0}-\kappa_{1}}\|\mathcal{F}(\mathtt{U}_{n})\|_{\kappa_{1}}^{q,\gamma,\mathtt{m},\mathrm{O}_{n}^{2\gamma}}.
		\end{equation}
		Now, Taylor formula together with \eqref{operatorF}, Lemma \ref{tame X per}, \eqref{e-F(ttU0)} and \eqref{ttUn} imply
		\begin{align}\label{e-fttUn-Wn}
			\nonumber \forall s\geqslant s_{0},\quad\|\mathcal{F}(\mathtt{U}_{n})\|_{s}^{q,\gamma,\mathtt{m},\mathrm{O}_{n}^{2\gamma}}&\leqslant\|\mathcal{F}(\mathtt{U}_{0})\|_{s}^{q,\gamma,\mathtt{m}}+\|\mathcal{F}(\mathtt{U}_{n})-\mathcal{F}(\mathtt{U}_{0})\|_{s}^{q,\gamma,\mathtt{m},\mathrm{O}_{n}^{2\gamma}}\\
			&\lesssim\varepsilon+\| \mathtt{W}_{n}\|_{s+\overline{\sigma}}^{q,\gamma,\mathtt{m}}.
		\end{align}
		Besides, \eqref{growth ttWn}, \eqref{def geo Nn} and \eqref{rigidity gam-N0} together give 
		\begin{align}\label{eps+ttWn}
			\nonumber\varepsilon+\|\mathtt{W}_{n}\|_{\kappa_{1}+\overline{\sigma}}^{q,\gamma,\mathtt{m}}&\leqslant\varepsilon\left(1+C_{\ast}\gamma^{-1}N_{n-1}^{\mu_{1}}\right)\\
			&\leqslant 2C_{\ast}\varepsilon N_{n}^{\frac{2}{3}\mu_{1}+1}.
		\end{align}
		Inserting \eqref{e-fttUn-Wn} and \eqref{eps+ttWn} into \eqref{e-first term NM} yields
		\begin{equation}\label{e-1-NM}
			\|\Pi_{N_n}^{\perp}\mathcal{F}(\mathtt{U}_{n})\|_{s_{0}}^{q,\gamma,\mathtt{m},\mathrm{O}_{n+1}^{4\gamma}}\lesssim C_{\ast}\varepsilon N_{n}^{s_{0}+\frac{2}{3}\mu_{1}+1-\kappa_{1}}.
		\end{equation}
		Let us move on to the second term. According to \eqref{hyprec O in A}, we have the following inclusions
		$$\mathrm{O}_{n+1}^{4\gamma}\subset\mathtt{A}_{n+1}^{\gamma}\subset\mathtt{G}_{n}\Big(\gamma_{n+1},\tau_{1},\tau_{2},i_{n}\Big).$$
		Hence, the decomposition \eqref{approx Ln} holds on $\mathrm{O}_{n+1}^{4\gamma}$ and we can write
		$$\Pi_{N_n}(L_{n}\mathrm{T}_{n}-\textnormal{Id})\Pi_{N_n}\mathcal{F}(\mathtt{U}_{n})=\mathfrak{E}_{1,n}+\mathfrak{E}_{2,n}+\mathfrak{E}_{3,n},$$
		with for any $k\in\{1,2,3\},$
		$$\mathfrak{E}_{k,n}\triangleq\Pi_{N_n}\mathcal{E}_{k}^{(n)}\Pi_{N_n}\mathcal{F}(\mathtt{U}_{n}).$$
		Thus, we need to estimate each one of the error terms $\mathfrak{E}_{k,n}$.  
		We begin with $\mathfrak{E}_{1,n}$ for which we need the following interpolation-type inequality
		\begin{align}\label{interp NM2}
			\|\mathcal{F}(\mathtt{U}_{n})\|_{q,s_{0}+\overline{\sigma}}^{\gamma,\mathtt{m},\mathrm{O}_{n}^{2\gamma}}&\leqslant\|\Pi_{N_n}\mathcal{F}(\mathtt{U}_n)\|_{q,s_0+\overline{\sigma}}^{\gamma,\mathtt{m},\mathrm{O}_{n}^{2\gamma}}+\|\Pi_{N_n}^{\perp}\mathcal{F}(\mathtt{U}_n)\|_{q,s_0+\overline{\sigma}}^{\gamma,\mathtt{m},\mathrm{O}_{n}^{2\gamma}}\nonumber\\
			&\leqslant N_{n}^{\overline{\sigma}}\|\mathcal{F}(\mathtt{U}_n)\|_{q,s_0}^{\gamma,\mathtt{m},\mathrm{O}_{n}^{2\gamma}}+N_{n}^{s_0-\kappa_{1}}\|\mathcal{F}(\mathtt{U}_n)\|_{q,\kappa_{1}+\overline{\sigma}}^{\gamma,\mathtt{m},\mathrm{O}_{n}^{2\gamma}}.
		\end{align}
	Now, putting together \eqref{e-fttUn-Wn} and \eqref{eps+ttWn}, we infer
	\begin{equation}\label{FttUn high}
		\|\mathcal{F}(\mathtt{U}_{n})\|_{\kappa_{1}+\overline{\sigma}}^{q,\gamma,\mathtt{m},\mathrm{O}_{n}^{2\gamma}}  \leqslant  C_{\ast}\varepsilon N_{n}^{\overline{\sigma}+\frac{2}{3}\mu_{1}+1}.
	\end{equation}
		Combining \eqref{calE1}, \eqref{interp NM2}, \eqref{e-ttWn}, \eqref{sml NM} and \eqref{FttUn high}, we obtain
		\begin{align}\label{Efrk1n}
			\|&\mathfrak{E}_{1,n}\|_{s_0}^{q,\gamma,\mathtt{m},\mathrm{O}_{n}^{2\gamma}}\lesssim\gamma^{-1}\|\mathcal{F}(\mathtt{U}_n)\|_{s_0+\overline{\sigma}}^{q,\gamma,\mathtt{m},\mathrm{O}_{n}^{2\gamma}}\|\Pi_{N_n}\mathcal{F}(\mathtt{U}_n)\|_{s_0+\overline{\sigma}}^{q,\gamma,\mathtt{m},\mathrm{O}_{n}^{2\gamma}}\left(1+\|\mathfrak{I}_n\|_{s_0+\overline{\sigma}}^{q,\gamma,\mathtt{m}}\right)\nonumber\\
			&\lesssim \gamma^{-1}N_n^{\overline{\sigma}}\left(N_n^{\overline{\sigma}}\|\mathcal{F}(\mathtt{U}_n)\|_{s_0}^{q,\gamma,\mathtt{m},\mathrm{O}_{n}^{2\gamma}}+N_{n}^{s_0-\kappa_{1}}\|\mathcal{F}(\mathtt{U}_n)\|_{\kappa_{1}+\overline{\sigma}}^{q,\gamma,\mathtt{m},\mathrm{O}_{n}^{2\gamma}}\right)\|\mathcal{F}(\mathtt{U}_n)\|_{s_0}^{q,\gamma,\mathtt{m},\mathrm{O}_{n}^{2\gamma}}\left(1+\|\mathtt{W}_{n}\|_{s_0+\overline{\sigma}}^{q,\gamma,\mathtt{m}}\right)\nonumber\\
			&\lesssim C_{\ast}\varepsilon\left(N_{n}^{2\overline{\sigma}-\frac{4}{3}a_1}+N_n^{s_0+2\overline{\sigma}+\frac{2}{3}\mu_{1}+1-\frac{2}{3}a_1-\kappa_{1}}\right).
		\end{align}
		As for $\mathfrak{E}_{2,n}$, we apply \eqref{calE2} with $\mathtt{b}=\kappa_{1}-s_0$ and use \eqref{sml NM}, \eqref{decay FttUn} and \eqref{growth ttWn} in order to find
		\begin{align}\label{Efrk2n}
			\|\mathfrak{E}_{2,n}\|_{s_0}^{q,\gamma,\mathtt{m},\mathrm{O}_{n}^{2\gamma}}&\lesssim\gamma^{-1}N_n^{s_0-\kappa_{1}}\left(\|\Pi_{N_n}\mathcal{F}(\mathtt{U}_n)\|_{\kappa_{1}+\overline{\sigma}}^{q,\gamma,\mathtt{m},\mathrm{O}_{n}^{2\gamma}}+\varepsilon\|\mathfrak{I}_n\|_{\kappa_{1}+\overline{\sigma}}^{q,\gamma,\mathtt{m}}\|\Pi_{N_n}\mathcal{F}(\mathtt{U}_n)\|_{s_0+\overline{\sigma}}^{q,\gamma,\mathtt{m},\mathrm{O}_{n}^{2\gamma}}\right)\nonumber\\
			&\lesssim \gamma^{-1}N_n^{s_0-\kappa_{1}}\left(\|\mathcal{F}(\mathtt{U}_n)\|_{\kappa_{1}+\overline{\sigma}}^{q,\gamma,\mathtt{m},\mathrm{O}_{n}^{2\gamma}}+\varepsilon N_n^{\overline{\sigma}}\|\mathtt{W}_n\|_{\kappa_{1}+\overline{\sigma}}^{q,\gamma,\mathtt{m}}\|\mathcal{F}(\mathtt{U}_n)\|_{s_0}^{q,\gamma,\mathtt{m},\mathrm{O}_{n}^{2\gamma}}\right)\nonumber\\
			&\lesssim C_{\ast}\varepsilon N_n^{s_0+\overline{\sigma}+\frac{2}{3}\mu_{1}+2-\kappa_{1}}+C_{\ast}\varepsilon N_n^{s_0+\overline{\sigma}+\frac{2}{3}\mu_{1}+2-\frac{2}{3}a_{1}-\kappa_{1}}\nonumber\\
			&\lesssim C_{\ast}\varepsilon N_n^{s_0+\overline{\sigma}+\frac{2}{3}\mu_{1}+2-\kappa_{1}}.
		\end{align}
		Similarly, putting together \eqref{calE3}, \eqref{def geo Nn}, \eqref{rigidity gam-N0} and \eqref{sml NM}, we infer
		\begin{align}\label{Efrk3n}
			\|\mathfrak{E}_{3,n}\|_{s_0}^{q,\gamma,\mathtt{m},\mathrm{O}_{n}^{2\gamma}}&\lesssim N_n^{s_0-\kappa_{1}}\gamma^{-2}\left(\|\Pi_{N_n}\mathcal{F}(\mathtt{U}_n)\|_{\kappa_{1}+\overline{\sigma}}^{q,\gamma,\mathtt{m},\mathrm{O}_{n}^{2\gamma}}+\varepsilon\gamma^{-2}\|\mathfrak{I}_n\|_{\kappa_{1}+\overline{\sigma}}^{q,\gamma,\mathtt{m}}\|\Pi_{N_n}\mathcal{F}(\mathtt{U}_n)\|_{s_0+\overline{\sigma}}^{q,\gamma,\mathtt{m},\mathrm{O}_{n}^{2\gamma}}\right)\nonumber\\
			&\quad+\varepsilon\gamma^{-4}N_{0}^{\mu_{2}}N_n^{-\mu_{2}}\|\Pi_{N_n}\mathcal{F}(\mathtt{U}_n)\|_{s_0+\overline{\sigma}}^{q,\gamma,\mathtt{m},\mathrm{O}_{n}^{2\gamma}}\nonumber\\
			&\lesssim C_{\ast}\varepsilon\left(N_n^{s_0+\overline{\sigma}+\frac{2}{3}\mu_{1}+2-\kappa_{1}}+N_n^{\overline{\sigma}+1-\mu_{2}-\frac{2}{3}a_1}\right).
		\end{align}
		Gathering, \eqref{Efrk1n}, \eqref{Efrk2n} and \eqref{Efrk3n}, we deduce
		\begin{equation}\label{second term NMn}
			\|\Pi_{N_n}(L_{n}\mathrm{T}_{n}-\textnormal{Id})\Pi_{N_n}\mathcal{F}(\mathtt{U}_{n})\|_{s_0}^{q,\gamma,\mathtt{m},\mathrm{O}_{n+1}^{4\gamma}}\leqslant CC_{\ast}\varepsilon\left(N_{n}^{2\overline{\sigma}-\frac{4}{3}a_1}+N_n^{s_0+2\overline{\sigma}+\frac{2}{3}\mu_{1}+1-\kappa_{1}}+N_n^{\overline{\sigma}+1-\mu_{2}-\frac{2}{3}a_1}\right).
		\end{equation}
		For $n=0$, we deduce from \eqref{e-F(ttU0)}, \eqref{sml NM} and by slight modifications of the preceding  computations 
		\begin{align}\label{second term NM0}
			\|\Pi_{N_0}(L_{0}\mathrm{T}_{0}-\textnormal{Id})\Pi_{N_0}\mathcal{F}(\mathtt{U}_{0})\|_{s_{0}}^{q,\gamma,\mathtt{m},\mathrm{O}_{1}^{4\gamma}}&\leqslant\|\mathfrak{E}_{1,0}\|_{s_0}^{q,\gamma,\mathtt{m}}+\|\mathfrak{E}_{2,0}\|_{s_0}^{q,\gamma,\mathtt{m}}+\|\mathfrak{E}_{3,0}\|_{s_0}^{q,\gamma,\mathtt{m}}\nonumber\\
			& \lesssim \varepsilon^{2}\gamma^{-1}+\varepsilon\gamma^{-1}+\big(\varepsilon \gamma^{-2}N_{0}^{s_{0}-\kappa_{1}}+\varepsilon^{2}\gamma^{-4}\big)\nonumber\\
			&\lesssim\varepsilon\gamma^{-2}.
		\end{align}
		Now, we turn to the estimate corresponding to the third term in \eqref{Fnext-NM}. In view of  \eqref{Linearized-op-F-DC}, we have for $\mathtt{H}=(\widehat{\mathfrak{I}},\widehat{\alpha})$ with $\widehat{\mathfrak{I}}=(\widehat{\Theta},\widehat{I},\widehat{z}),$
		\begin{equation}\label{linH00}
			L_{n}\mathtt{H}=\omega\cdot\partial_{\varphi}\widehat{\mathfrak{I}}-(0,0,\mathcal{J}\mathbf{L}_0(b)\widehat{z})-\varepsilon d_{i}X_{\mathcal{P}_{\varepsilon}}(i_{n})\widehat{\mathfrak{I}}-(\mathtt{J}\widehat{\alpha},0,0).
		\end{equation}
		Now, \eqref{proj-NM} and the fact that $\omega\cdot\partial_{\varphi}$ and $\mathcal{J}\mathbf{L}_0(b)$ are diagonal yield
		$$\left(L_{n}\mathbf{\Pi}_{N_n}^{\perp}-\Pi_{N_n}^{\perp}L_{n}\right)\mathtt{H}=-\varepsilon[d_{i}X_{\mathcal{P}_{\varepsilon}}(i_{n}),\Pi_{N_n}^{\perp}]\widehat{\mathfrak{I}}.$$
		Applying Lemma \ref{tame X per}-{(ii)} together with Lemma \ref{lem funct prop}-(i) and \eqref{O2gm in Ogm}, we infer
		$$
		\left\|\left(L_{n}\mathbf{\Pi}_{N_n}^{\perp}-\Pi_{N_n}^{\perp}L_{n}\right)\mathtt{H}\right\|_{s_{0}}^{q,\gamma,\mathtt{m},\mathrm{O}_{n+1}^{4\gamma}}\lesssim\varepsilon N_{n}^{s_{0}-\kappa_{1}}\left(\|\widehat{\mathfrak{I}}\|_{\kappa_{1}+1}^{q,\gamma,\mathtt{m},\mathrm{O}_{n}^{2\gamma}}+\|\mathfrak{I}_{n}\|_{\kappa_{1}+\overline{\sigma}}^{q,\gamma,\mathtt{m}}\|\widehat{\mathfrak{I}}\|_{s_{0}+1}^{q,\gamma,\mathtt{m},\mathrm{O}_{n}^{2\gamma}}\right).
		$$
		Hence,
		\begin{align*}
			\left\|\left(L_{n}\mathbf{\Pi}_{N_n}^{\perp}-\Pi_{N_n}^{\perp}L_{n}\right)\mathrm{T}_{n}\Pi_{n}\mathcal{F}(\mathtt{U}_{n})\right\|_{s_{0}}^{q,\gamma,\mathtt{m},\mathrm{O}_{n+1}^{4\gamma}}&\lesssim\varepsilon N_{n}^{s_{0}-\kappa_{1}}\|\mathrm{T}_{n}\Pi_{N_n}\mathcal{F}(\mathtt{U}_{n})\|_{\kappa_{1}+1}^{q,\gamma,\mathtt{m},\mathrm{O}_{n}^{2\gamma}}\\
			&\quad +\varepsilon N_{n}^{s_{0}-\kappa_{1}}\|\mathfrak{I}_{n}\|_{\kappa_{1}+\overline{\sigma}}^{q,\gamma,\mathtt{m}}\|\mathrm{T}_{n}\Pi_{N_n}\mathcal{F}(\mathtt{U}_{n})\|_{s_{0}+1}^{q,\gamma,\mathtt{m},\mathrm{O}_{n}^{2\gamma}}.
		\end{align*}
		Now, using \eqref{eari-NM}, \eqref{O2gm in Ogm}, Lemma \ref{lem funct prop}-(i), Sobolev embeddings, \eqref{sml NM} and \eqref{rigidity gam-N0}, we get
		\begin{align*}
			&\left\|\left(L_{n}\mathbf{\Pi}_{N_n}^{\perp}-\Pi_{N_n}^{\perp}L_{n}\right)\mathrm{T}_{n}\Pi_{n}\mathcal{F}(\mathtt{U}_{n})\right\|_{s_{0}}^{q,\gamma,\mathtt{m},\mathrm{O}_{n+1}^{4\gamma}}\\
			&\lesssim \varepsilon\gamma^{-1}N_{n}^{s_{0}-\kappa_{1}}\|\Pi_{N_n}\mathcal{F}(\mathtt{U}_{n})\|_{\kappa_{1}+\overline{\sigma}+1}^{q,\gamma,\mathtt{m},\mathrm{O}_{n}^{2\gamma}}+\|\mathfrak{I}_{n}\|_{\kappa_{1}+\overline{\sigma}+1}^{q,\gamma,\mathtt{m}}\|\Pi_{N_n}\mathcal{F}(\mathtt{U}_{n})\|_{s_{0}+\overline{\sigma}}^{q,\gamma,\mathtt{m},\mathrm{O}_{n}^{2\gamma}}\\
			&\quad+\varepsilon\gamma^{-1}N_{n}^{s_{0}-\kappa_{1}}\|\mathfrak{I}_{n}\|_{\kappa_{1}+\overline{\sigma}}^{q,\gamma,\mathtt{m}}\left(\|\Pi_{N_n}\mathcal{F}(\mathtt{U}_{n})\|_{s_{0}+\overline{\sigma}+1}^{q,\gamma,\mathtt{m},\mathrm{O}_{n}^{2\gamma}}+\|\mathfrak{I}_{n}\|_{s_{0}+\overline{\sigma}+1}^{q,\gamma,\mathtt{m}}\|\Pi_{N_n}\mathcal{F}(\mathtt{U}_{n})\|_{s_{0}+\overline{\sigma}}^{q,\gamma,\mathtt{m},\mathrm{O}_{n}^{2\gamma}}\right)\\
			&\lesssim \varepsilon N_{n}^{s_{0}+2-\kappa_{1}}\left(\|\mathcal{F}(\mathtt{U}_{n})\|_{\kappa_{1}+\overline{\sigma}}^{q,\gamma,\mathtt{m},\mathrm{O}_{n}^{2\gamma}}+\| \mathtt{W}_{n}\|_{\kappa_{1}+\overline{\sigma}}^{q,\gamma,\mathtt{m}}\|\Pi_{N_n}\mathcal{F}(\mathtt{U}_{n})\|_{s_{0}+\overline{\sigma}}^{q,\gamma,\mathtt{m},\mathrm{O}_{n}^{2\gamma}}\right).
		\end{align*}
		Besides, from Lemma \ref{lem funct prop}-(ii), \eqref{decay FttUn} and \eqref{def geo Nn}, we obtain
		\begin{align*}
			\|\Pi_{N_n}\mathcal{F}(\mathtt{U}_{n})\|_{s_{0}+\overline{\sigma}}^{q,\gamma,\mathtt{m},\mathrm{O}_{n}^{2\gamma}}&\leqslant N_{n}^{\overline{\sigma}}\|\mathcal{F}(\mathtt{U}_{n})\|_{s_{0}}^{q,\gamma,\mathtt{m},\mathrm{O}_{n}^{2\gamma}}\\
			&\leqslant C_{\ast}\varepsilon N_{n}^{\overline{\sigma}}N_{n-1}^{-a_{1}}\\
			&\leqslant C_{\ast}\varepsilon N_{n}^{\overline{\sigma}-\frac{2}{3}a_{1}}.
		\end{align*}
		Added to \eqref{param NM}, \eqref{FttUn high} and \eqref{growth ttWn}, we deduce that 
		\begin{equation}\label{final estimate commutator}
			\|(L_{n}\mathbf{\Pi}_{n}^{\perp}-\Pi_{n}^{\perp}L_{n})\mathrm{T}_{n}\Pi_{n}\mathcal{F}(\mathtt{U}_{n})\|_{s_{0}}^{q,\gamma,\mathtt{m},\mathrm{O}_{n+1}^{4\gamma}}\leqslant CC_{\ast}\varepsilon N_{n}^{s_{0}+\overline{\sigma}+\frac{2}{3}\mu_{1}+3-\kappa_{1}}.
		\end{equation}
		We are left with the quadratic term in \eqref{Fnext-NM}. Another application of Taylor formula with \eqref{quad-NM} leads to
		$$Q_{n}=\int_{0}^{1}(1-t)d_{i,\alpha}^{2}\mathcal{F}(\mathtt{U}_{n}+t\widetilde{\mathtt{H}}_{n+1})[\widetilde{\mathtt{H}}_{n+1},\widetilde{\mathtt{H}}_{n+1}]dt.$$
		Now, \eqref{linH00} and Lemma \ref{tame X per}-{(iii)} give
		\begin{align}\label{quad-e-11}
			\| Q_{n}\|_{s_{0}}^{q,\gamma,\mathtt{m},\mathrm{O}_{n+1}^{4\gamma}}\lesssim\varepsilon\left(1+\|\mathtt{W}_{n}\|_{s_{0}+2}^{q,\gamma,\mathtt{m}}+\| \widetilde{\mathtt{H}}_{n+1}\|_{s_{0}+2}^{q,\gamma,\mathtt{m},\mathrm{O}_{n+1}^{4\gamma}}\right)\left(\| \widetilde{\mathtt{H}}_{n+1}\|_{s_{0}+2}^{q,\gamma,\mathtt{m},\mathrm{O}_{n+1}^{4\gamma}}\right)^{2}.
		\end{align}
	Observe that, \eqref{e-fttUn-Wn}, \eqref{rigidity gam-N0} and \eqref{e-ttWn} imply
	\begin{equation}\label{sml FttU0-00}
		\gamma^{-1}\|\mathcal{F}(\mathtt{U}_{n})\|_{s_{0}}^{q,\gamma,\mathtt{m},\mathrm{O}_{n}^{2\gamma}}\leqslant 1.
	\end{equation}
		Gathering \eqref{O2gm in Ogm}, \eqref{hyprec O in A}, \eqref{eari-NM}, \eqref{e-fttUn-Wn} and \eqref{sml FttU0-00}, we obtain for all $s\in[s_{0},S]$ 
		\begin{align}\label{tttHm+1 and ttWn}
			\| \widetilde{\mathtt{H}}_{n+1}\|_{s}^{q,\gamma,\mathtt{m},\mathrm{O}_{n+1}^{4\gamma}}&=  \|\mathbf{\Pi}_{N_n}\mathrm{T}_{n}\Pi_{N_n}\mathcal{F}(\mathtt{U}_{n})\|_{s}^{q,\gamma,\mathtt{m},\mathrm{O}_{n+1}^{4\gamma}}\nonumber\\
			& \lesssim  \gamma^{-1}\left(\|\Pi_{N_n}\mathcal{F}(\mathtt{U}_{n})\|_{s+\overline{\sigma}}^{q,\gamma,\mathtt{m},\mathrm{O}_{n}^{2\gamma}}+\|\mathfrak{I}_{n}\|_{s+\overline{\sigma}}^{q,\gamma,\mathtt{m}}\|\Pi_{N_n}\mathcal{F}(\mathtt{U}_{n})\|_{s_{0}+\overline{\sigma}}^{q,\gamma,\mathtt{m},\mathrm{O}_{n}^{2\gamma}}\right)\nonumber\\
			& \lesssim  \gamma^{-1}\left(N_{n}^{\overline{\sigma}}\|\mathcal{F}(\mathtt{U}_{n})\|_{s}^{q,\gamma,\mathtt{m},\mathrm{O}_{n}^{2\gamma}}+N_{n}^{2\overline{\sigma}}\|\mathfrak{I}_{n}\|_{s}^{q,\gamma,\mathtt{m}}\|\mathcal{F}(\mathtt{U}_{n})\|_{s_{0}}^{q,\gamma,\mathtt{m},\mathrm{O}_{n}^{2\gamma}}\right)\nonumber\\
			& \lesssim  \gamma^{-1}N_{n}^{2\overline{\sigma}}\left(\varepsilon+\| \mathtt{W}_{n}\|_{s}^{q,\gamma,\mathtt{m}}\right).
		\end{align}
		Similarly, \eqref{eari-NM}, \eqref{O2gm in Ogm}, \eqref{e-ttWn}, \eqref{sml NM} and \eqref{decay FttUn}, we infer
		\begin{align}\label{decay tttHm+1 s0}
			\nonumber \| \widetilde{\mathtt{H}}_{n+1}\|_{s_{0}}^{q,\gamma,\mathtt{m},\mathrm{O}_{n+1}^{4\gamma}}&\lesssim\gamma^{-1}N_{n}^{\overline{\sigma}}\|\mathcal{F}(\mathtt{U}_{n})\|_{s_{0}}^{q,\gamma,\mathtt{m},\mathrm{O}_{n}^{2\gamma}}\\
			&\lesssim C_{\ast}\varepsilon\gamma^{-1} N_{n}^{\overline{\sigma}}N_{n-1}^{-a_{1}}.
		\end{align}
		For $\varepsilon$ sufficiently small, \eqref{e-ttWn} and \eqref{decay tttHm+1 s0} imply
		\begin{align*}
			\|\mathtt{W}_{n}\|_{s_{0}+2}^{q,\gamma,\mathtt{m}}+\| \widetilde{\mathtt{H}}_{n+1}\|_{s_{0}+2}^{q,\gamma,\mathtt{m},\mathrm{O}_{n+1}^{4\gamma}} & \leqslant C_{\ast}\varepsilon\gamma^{-1}N_0^{q\overline{a}}+N_{n}^{2}\| \widetilde{\mathtt{H}}_{n+1}\|_{s_{0}}^{q,\gamma,\mathtt{m},\mathrm{O}_{n+1}^{4\gamma}}\\
			& \leqslant 1+C\varepsilon\gamma^{-1}N_{n}^{\overline{\sigma}+2}N_{n-1}^{-a_{1}}\\
			& \leqslant 1+C\varepsilon\gamma^{-1}N_{n-1}^{3+\frac{3}{2}\overline{\sigma}-a_{1}}.
		\end{align*}
		But \eqref{param NM} gives in particular $a_{1}\geqslant 3+\tfrac{3}{2}\overline{\sigma}.$ Thus,
		\begin{equation}\label{bnd ttWn and tttHn+1}
			\|\mathtt{W}_{n}\|_{s_{0}+2}^{q,\gamma,\mathtt{m}}+\| \widetilde{\mathtt{H}}_{n+1}\|_{s_{0}+2}^{q,\gamma,\mathtt{m},\mathrm{O}_{n+1}^{4\gamma}}\leqslant 2.
		\end{equation}
		Therefore, inserting \eqref{bnd ttWn and tttHn+1} and \eqref{decay tttHm+1 s0} into \eqref{quad-e-11} and using \eqref{rigidity gam-N0} and \eqref{sml NM}, we get 
		\begin{align*}
			\| Q_{n}\|_{s_{0}}^{q,\gamma,\mathtt{m},\mathrm{O}_{n+1}^{4\gamma}} & \lesssim \varepsilon\left(\| \widetilde{\mathtt{H}}_{n+1}\|_{s_{0}+2}^{q,\gamma,\mathtt{m},\mathrm{O}_{n+1}^{4\gamma}}\right)^{2}\\
			& \leqslant \varepsilon N_{n}^{4}\left(\| \widetilde{\mathtt{H}}_{n+1}\|_{s_{0}}^{q,\gamma,\mathtt{m},\mathrm{O}_{n+1}^{4\gamma}}\right)^{2}\\
			& \lesssim \varepsilon C_{\ast}N_{n}^{2\overline{\sigma}+4}N_{n-1}^{-2a_1}.
		\end{align*}
		Using \eqref{def geo Nn}, we finally obtain for $n\in\mathbb{N}^*$,
		\begin{equation}\label{fourth term NM}
			\|Q_{n}\|_{s_{0}}^{q,\gamma,\mathtt{m},\mathcal{O}_{n+1}^{4\gamma}}\leqslant CC_{\ast}\varepsilon N_{n}^{2\overline{\sigma}+4-\frac{4}{3}a_{1}}.
		\end{equation}
		As for $n=0$, we come back to \eqref{tttHm+1 and ttWn} and \eqref{e-F(ttU0)} to get for all $s\in[s_{0},S]$
		\begin{align}\label{ttH1}
			\|\widetilde{\mathtt{H}}_{1}\|_{s}^{q,\gamma,\mathtt{m},\mathrm{O}_{1}^{4\gamma}}&\lesssim\gamma^{-1}\|\Pi_{0}\mathcal{F}(\mathtt{U}_{0})\|_{s+\overline{\sigma}}^{q,\gamma,\mathtt{m}}\nonumber\\
			&\lesssim C_{\ast}\varepsilon\gamma^{-1}.
		\end{align}
		Finally, the inequality \eqref{fourth term NM} becomes for $n=0$,
		\begin{equation}\label{e-quad0-NM}
			\|Q_{0}\|_{s_{0}}^{q,\gamma,\mathtt{m},\mathrm{O}_{0}^{4\gamma}}\lesssim C_{\ast}\varepsilon^{3}\gamma^{-2}.
		\end{equation}
		Plugging \eqref{e-1-NM}, \eqref{second term NMn}, \eqref{final estimate commutator} and \eqref{fourth term NM}, into \eqref{Fnext-NM} gives for $n\in\mathbb{N}^{*}$,
		\begin{align*}
			\|\mathcal{F}(\widetilde{\mathtt{U}}_{n+1})\|_{s_{0}}^{q,\gamma,\mathtt{m},\mathrm{O}_{n+1}^{4\gamma}}&\leqslant CC_{\ast}\varepsilon\left(N_{n}^{s_{0}+2\overline{\sigma}+\frac{2}{3}\mu_{1}+1-\kappa_{1}}+N_{n}^{\overline{\sigma}+1-\mu_{2}-\frac{2}{3}a_{1}}+N_{n}^{2\overline{\sigma}+4-\frac{4}{3}a_{1}}\right).
		\end{align*} 
		Now, our choice of parameters in \eqref{param NM} implies
		$$s_{0}+2\overline{\sigma}+\tfrac{2}{3}\mu_{1}+2+a_{1}\leqslant\kappa_{1},\qquad\overline{\sigma}+\tfrac{1}{3}a_{1}+2\leqslant\mu_{2},\qquad2\overline{\sigma}+5\leqslant\tfrac{1}{3}a_{1}.$$
		Consequently, for $N_{0}$ large enough, that is $\varepsilon$ small enough, we can obtain for any $n\in\mathbb{N},$
		\begin{equation}\label{suitable constraints NM}
			\max\Big(CN_{n}^{s_{0}+2\overline{\sigma}+\frac{2}{3}\mu_{1}+1-\kappa_{1}}\,,\, CN_{n}^{\overline{\sigma}+1-\mu_{2}-\frac{2}{3}a_{1}}\, , \, CN_{n}^{2\overline{\sigma}+4-\frac{4}{3}a_{1}}\Big) \leqslant \tfrac{1}{3}N_{n}^{-a_{1}},
		\end{equation}
		which implies in turn \eqref{decay ttUn inter} for $n\in\mathbb{N}^*.$ As for the case $n=0$, we insert \eqref{e-1-NM}, \eqref{second term NM0}, \eqref{final estimate commutator} and \eqref{e-quad0-NM} into \eqref{Fnext-NM} to get
		$$\|\mathcal{F}(\widetilde{\mathtt{U}}_{1})\|_{q,s_{0}}^{\gamma,\mathtt{m},\mathrm{O}_{1}^{4\gamma}}\leqslant CC_{\ast}\varepsilon\left(N_{0}^{s_{0}+2\overline{\sigma}+\frac{3}{2}\mu_{1}+1-\kappa_{1}}+\varepsilon\gamma^{-2}+\varepsilon^{2}\gamma^{-2}\right).$$
		Hence, using \eqref{suitable constraints NM} and the fact that \eqref{rigidity gam-N0} and \eqref{param NM} imply $0<a<\tfrac{1}{2+a_{1}}\cdot$, then taking $\varepsilon$ small enough, we infer
		$$C\left(\varepsilon\gamma^{-2}+\varepsilon^{2}\gamma^{-2}\right)\leqslant \tfrac{2}{3}N_{0}^{-a_{1}}.$$
		As a consequence, \eqref{decay ttUn inter} occurs also for $n=0.$\\
		\ding{226} \textit{Construction of the extension.}
		\noindent The next goal is to construct an extention of $\widetilde{\mathtt{H}}_{n+1}$ to the whole set of parameters $\mathcal{O}$ and still satisfying nice decay properties. For this aim, we introduce the following cut-off function $\chi_{n+1}\in C^{\infty}\big(\mathcal{O},[0,1]\big)$ given by
		$$\chi_{n+1}=\begin{cases}
			1 & \textnormal{in }\mathrm{O}_{n+1}^{2\gamma}\\
			0 & \textnormal{in }\mathcal{O}\setminus \mathrm{O}_{n+1}^{4\gamma}
		\end{cases}$$
		and such that
		\begin{equation}\label{groth deriv chin+1}
			\forall \alpha\in\mathbb{N}^{d},\quad |\alpha|\in\llbracket 0,q\rrbracket,\quad  \|\partial_{b,\omega}^{\alpha}\chi_{n+1}\|_{L^{\infty}(\mathcal{O})}\lesssim\left(\gamma^{-1}N_{n}^{\overline{a}}\right)^{|\alpha|}.
		\end{equation}
		Therefore, we can define the extension $\mathtt{H}_{n+1}$ of $\widetilde{\mathtt{H}}_{n+1}$ as follows
		\begin{equation}\label{def extension ttHn+1}
			\mathtt{H}_{n+1}\triangleq\begin{cases}
				\chi_{n+1}\,\widetilde{\mathtt{H}}_{n+1} & \textnormal{in } \mathrm{O}_{n+1}^{4\gamma},\\
				0 &\textnormal{in }\mathcal{O}\setminus\mathrm{O}_{n+1}^{4\gamma}.
			\end{cases}
		\end{equation}
		We also define
		\begin{equation}\label{def ttUn+1}
			\mathtt{W}_{n+1}\triangleq\mathtt{W}_{n}+\mathtt{H}_{n+1},\qquad\mathtt{U}_{n+1}\triangleq\mathtt{U}_0+\mathtt{W}_{n+1}=\mathtt{U}_{n}+\mathtt{H}_{n+1}.
		\end{equation}
		We can observe that in restriction to $\mathrm{O}_{n+1}^{2\gamma}$, we have
		$$\mathtt{H}_{n+1}=\widetilde{\mathtt{H}}_{n+1},\qquad\mathtt{U}_{n+1}=\widetilde{\mathtt{U}}_{n+1}\qquad\textnormal{and}\qquad\mathcal{F}(\mathtt{U}_{n+1})=\mathcal{F}(\widetilde{\mathtt{U}}_{n+1}).$$
		The last identity together with \eqref{decay ttUn inter} and the fact that $\mathrm{O}_{n+1}^{2\gamma}\subset\mathrm{O}_{n+1}^{4\gamma}$ imply \eqref{decay FttUn}.
		Now, the product laws in Lemma \ref{lem funct prop} together with \eqref{def extension ttHn+1} and \eqref{groth deriv chin+1} provide the following estimate
		\begin{equation}\label{control ext ttHn+1}
			\forall s\geqslant s_{0},\quad\|\mathtt{H}_{n+1}\|_{s}^{q,\gamma,\mathtt{m}}\lesssim N_{n}^{q\overline{a}}\|\widetilde{\mathtt{H}}_{n+1}\|_{s}^{q,\gamma,\mathtt{m},\mathrm{O}_{n+1}^{4\gamma}}.
		\end{equation}
		Then, gathering \eqref{control ext ttHn+1} and \eqref{decay tttHm+1 s0}, implies for any $n\in\mathbb{N}^{*},$
		\begin{align*}
			\| \mathtt{H}_{n+1}\|_{s_{0}+\overline{\sigma}}^{q,\gamma,\mathtt{m}}&\leqslant C N_{n}^{q\overline{a}+\overline\sigma}\|\widetilde{\mathtt{H}}_{n+1}\|_{s_{0}}^{q,\gamma,\mathtt{m},\mathrm{O}_{n+1}^{4\gamma}}\\
			&\leqslant C C_{\ast}\varepsilon\gamma^{-1} N_{n}^{q\overline{a}+2\overline{\sigma}-\frac{2}{3}a_{1}}.
		\end{align*}
		The constraint \eqref{param NM} implies in particular $a_{2}=\tfrac{2}{3}a_{1}-q\overline{a}-2\overline{\sigma}-1\geqslant 1,$ we get for $\varepsilon$ small enough
		\begin{align}\label{estim Hm+1 tilde}
			\|\mathtt{H}_{n+1}\|_{s_{0}+\overline{\sigma}}^{q,\gamma,\mathtt{m}}&\leqslant CN_0^{-1}C_{\ast}\varepsilon\gamma^{-1}N_{n}^{-a_{2}}\nonumber\\
			&\leqslant C_{\ast}\varepsilon\gamma^{-1}N_{n}^{-a_{2}}.
		\end{align}
		As for the case $n=0$, we combine \eqref{control ext ttHn+1} and \eqref{ttH1} to get, up to take $C_{\ast}$ sufficiently large,
		\begin{equation}\label{e-ttH1-01}
			\|\mathtt{H}_{1}\|_{s}^{q,\gamma,\mathtt{m}}\leqslant \tfrac{1}{2}C_{\ast}\varepsilon\gamma^{-1}N_{0}^{q\overline{a}}.
		\end{equation}
		Putting together \eqref{ttUn}, \eqref{e-ttH1-01} and \eqref{estim Hm+1 tilde}, we deduce
		\begin{align*}
			\|\mathtt{W}_{n+1}\|_{s_{0}+\overline{\sigma}}^{q,\gamma,\mathtt{m}}&\leqslant\|\mathtt{H}_{1}\|_{s_{0}+\overline{\sigma}}^{q,\gamma,\mathtt{m}}+\sum_{k=2}^{n+1}\|\mathtt{H}_{k}\|_{s_{0}+\overline{\sigma}}^{q,\gamma,\mathtt{m}}\\
			&\leqslant \tfrac{1}{2}C_{\ast}\varepsilon\gamma^{-1}N_{0}^{q\overline{a}}+CN_{0}^{-1}C_{\ast}\varepsilon\gamma^{-1}\\
			&\leqslant C_{\ast}\varepsilon\gamma^{-1}N_{0}^{q\overline{a}}.
		\end{align*}
		This proves \eqref{e-ttWn} at the order $n+1.$ Now \eqref{tttHm+1 and ttWn}, \eqref{control ext ttHn+1} and \eqref{growth ttWn} all together yield
		\begin{align*}
			\|\mathtt{W}_{n+1}\|_{\kappa_{1}+\overline{\sigma}}^{q,\gamma,\mathtt{m}} & \leqslant  \| \mathtt{W}_{n}\|_{\kappa_{1}+\overline{\sigma}}^{q,\gamma,\mathtt{m}}+CN_{n}^{q\overline{a}}\|\mathtt{H}_{n+1}\|_{\kappa_{1}+\overline{\sigma}}^{q,\gamma,\mathtt{m}}\\
			& \leqslant C_{\ast} \varepsilon\gamma^{-1}N_{n-1}^{\mu_{1}}+CC_{\ast}\gamma^{-1}N_{n}^{q\overline{a}+2\overline{\sigma}}\left(\varepsilon+\| \mathtt{W}_{n}\|_{\kappa_{1}+\overline{\sigma}}^{q,\gamma,\mathtt{m}}\right)\\
			& \leqslant  CC_{\ast}\varepsilon\gamma^{-1}N_{n}^{q\overline{a}+2\overline{\sigma}+1+\frac{2}{3}\mu_{1}}.
		\end{align*}
		By \eqref{param NM} we get $q\overline{a}+2\overline{\sigma}+2=\tfrac{\mu_{1}}{3},$ which implies that for $\varepsilon$ small enough we have
		\begin{align*}
			\|\mathtt{W}_{n+1}\|_{\kappa_{1}+\overline{\sigma}}^{q,\gamma,\mathtt{m}}&\leqslant CN_{0}^{-1}C_{\ast}\varepsilon\gamma^{-1}N_{n}^{\mu_{1}}\\
			&\leqslant C_{\ast}\varepsilon\gamma^{-1}N_{n}^{\mu_{1}}.
		\end{align*}
		This proves \eqref{growth ttWn} at the order $n+1.$\\
		\ding{226} \textit{Reversibility preserving property of the scheme.} Form $(\mathcal{P}2)_n$, we know that the torus $i_n$ is reversible. Observe that the projectors $\Pi_{N_n}$ are reversibility preserving thanks to the symmetry with respect to the Fourier modes. Now, using the reversibility property of the operators $\mathrm{T}_n$ and $\Pi_{N_n}$ we have that the torus component of $\widehat{\mathfrak{I}}_{n+1}$ of $\widetilde{\mathtt{H}}_{n+1}$ is reversible. Since the cut-off function $\chi_{n+1}$ only depends on the variables $(b,\omega)$, then the reversibility property is also preserved  for  the torus component $\mathfrak{I}_{n+1}$ of $\mathtt{H}_{n+1}$. Looking at the first component of \eqref{def ttUn+1}, we have
		$$i_{n+1}=i_n+\mathfrak{I}_{n+1}.$$
		Hence, the reversibility property \eqref{reversibility in} also holds at the order $n+1$.
	\end{proof}
	The previous iteration procedure converges and allows to find a non trivial reversible quasi-periodic solution of our problem provided some restriction on the internal radius $b$. More precisely, we have the following result.
	\begin{cor}\label{Corollary NM}
		There exists $\varepsilon_0>0$ such that, for all $\varepsilon\in(0,\varepsilon_0),$ the following assertions hold true. There exists a $q$-times differentiable function
		$$\mathtt{U}_{\infty}:\begin{array}[t]{rcl}
			\mathcal{O} & \rightarrow & \left(\mathbb{T}^{d}\times\mathbb{R}^{d}\times \mathbf{H}^{s_0}_{\bot,\mathtt{m}}\right)\times\mathbb{R}^{d}\times\mathbb{R}^{d+1}\\
			(b,\omega) & \mapsto & \big(i_{\infty}(b,\omega),\alpha_{\infty}(b,\omega),(b,\omega)\big)
		\end{array}$$
		such that in restriction to the Cantor set $\mathtt{G}_{\infty}^{\gamma}$ defined by
		\begin{equation}\label{Canbomgfin}
			\mathtt{G}_{\infty}^{\gamma}\triangleq\bigcap_{n\in\mathbb{N}}\mathtt{A}_{n}^{\gamma},
		\end{equation}
		we have
		\begin{equation}\label{solution NM non rigidified}
			\forall(b,\omega)\in\mathtt{G}_{\infty}^{\gamma},\quad\mathcal{F}\big(\mathtt{U}_{\infty}(b,\omega)\big)=0.
		\end{equation}
		The torus $i_{\infty}$ is $\mathtt{m}$-fold and reversible. The vector $\alpha_{\infty}\in W^{q,\infty,\gamma}(\mathcal{O},\mathbb{R}^d)$ satisfies
		\begin{equation}\label{limt alf}
			\alpha_{\infty}(b,\omega)=\mathtt{J}\omega+\mathrm{r}_{\varepsilon}(b,\omega),\qquad\|\mathrm{r}_{\varepsilon}\|^{q,\gamma}\lesssim\varepsilon\gamma^{-1}N_{0}^{q\overline{a}}.
		\end{equation}
		In addition, there exists a $q$-times differentiable function $b\in(b_{*},b^{*})\mapsto\omega(b,\varepsilon)$ implicitly defined by 
		\begin{equation}\label{implicit omegaper}
			\alpha_{\infty}\big(b,\omega(b,\varepsilon)\big)=-\mathtt{J}\omega_{\textnormal{Eq}}(b)
		\end{equation}
		with
		\begin{equation}\label{pert freq}
			\omega(b,\varepsilon)=-\omega_{\textnormal{Eq}}(b)+\bar{r}_{\varepsilon}(b),\qquad \|\bar{r}_{\varepsilon}\|^{q,\gamma}\lesssim\varepsilon\gamma^{-1}N_{0}^{q\overline{a}},
		\end{equation}
		such that
						\begin{equation}\label{solution NM rigidified}
			\forall b\in \mathcal{C}_{\infty}^{\varepsilon},\quad \mathcal{F}\Big(\mathtt{U}_{\infty}\big(b,\omega(b,\varepsilon)\big)\Big)=0,
		\end{equation}
where
		\begin{equation}\label{FCS}
			\mathcal{C}_{\infty}^{\varepsilon}\triangleq\Big\{b\in(b_{*},b^{*})\quad\textnormal{s.t.}\quad\big(b,\omega(b,\varepsilon)\big)\in\mathtt{G}_{\infty}^{\gamma}\Big\}.
		\end{equation}
	\end{cor}
	\begin{proof}
		We deduce from \eqref{ttUn} and \eqref{estim Hm+1 tilde} that
		$$\|\mathtt{W}_{n+1}-\mathtt{W}_{n}\|_{s_{0}}^{q,\gamma,\mathtt{m}}=\|\mathtt{H}_{n+1}\|_{s_{0}}^{q,\gamma,\mathtt{m}}\leqslant\|\mathtt{H}_{n+1}\|_{s_{0}+\overline{\sigma}}^{q,\gamma,\mathtt{m}}\leqslant C_{\ast}\varepsilon\gamma^{-1}N_{n}^{-a_{2}}.$$
		Consequently, the sequence $\left(\mathtt{W}_{n}\right)_{n\in\mathbb{N}}$ converges and we denote
		$$\mathtt{W}_{\infty}\triangleq\lim_{n\rightarrow\infty}\mathtt{W}_{n}\triangleq(\mathfrak{I}_{\infty},\alpha_{\infty}-\mathtt{J}\omega,0,0),\qquad\mathtt{U}_{\infty}\triangleq\big(i_{\infty},\alpha_{\infty},(b,\omega)\big)=\mathtt{U}_0+\mathtt{W}_{\infty}.$$
		The reversibility and $\mathtt{m}$-fold properties of $i_{\infty}$ are obtained as the pointwise limit in \eqref{reversibility in}. Now, for $\varepsilon$ small enough, we get \eqref{solution NM non rigidified} from \eqref{decay FttUn}. The identity \eqref{limt alf} follows from the previous construction and the corresponding estimate is obtained by taking the limit in \eqref{e-ttWn}. We recall that the open set $\mathcal{O}$ is defined in \eqref{ouvert-sym}-\eqref{def scrU} by
		$$\mathcal{O}=(b_{*},b^{*})\times\mathscr{U},\qquad\mathscr{U}=B(0,R_{0}),\qquad\omega_{\textnormal{Eq}}\big([b_*,b^*]\big)\subset B\big(0,\tfrac{R_0}{2}\big).$$
		According to \eqref{limt alf}, we have that for any $b\in(b_{*},b^{*}),$ the mapping $\omega\in\mathscr{U}\mapsto\alpha_{\infty}(b,\omega)\in\alpha_{\infty}(b,\mathscr{U})$ is invertible, implying that
		$$\widehat{\omega}=\alpha_{\infty}(b,\omega)=\mathtt{J}\omega+\mathrm{r}_{\varepsilon}(b,\omega)\qquad\Leftrightarrow\qquad\omega=\alpha_{\infty}^{-1}(b,\widehat{\omega})=\mathtt{J}\widehat{\omega}+\widehat{\mathrm{r}}_{\varepsilon}(b,\widehat{\omega}).$$
		In particular, 
		$$\widehat{\mathrm{r}}_{\varepsilon}(b,\widehat{\omega})=-\mathtt{J}\mathrm{r}_{\varepsilon}(b,\omega).$$
		Then, differentiating the previous relation and using \eqref{limt alf}, we get 
		\begin{equation}\label{estimate mathrm repsilon}
			\|\widehat{\mathrm{r}}_{\varepsilon}\|^{q,\gamma}\lesssim\varepsilon\gamma^{-1}N_{0}^{q\overline{a}}.
		\end{equation}
		Finally, if we denote 
		$$\omega(b,\varepsilon)\triangleq\alpha_{\infty}^{-1}(b,-\mathtt{J}\omega_{\textnormal{Eq}}(b))=-\omega_{\textnormal{Eq}}(b)+\overline{\mathrm{r}}_{\varepsilon}(b),\qquad\overline{\mathrm{r}}_{\varepsilon}(b)\triangleq\widehat{\mathrm{r}}_{\varepsilon}\big(b,-\mathtt{J}\omega_{\textnormal{Eq}}(b)\big),$$
		then we have in particular \eqref{implicit omegaper} and the  estimate \eqref{pert freq} follows from \eqref{estimate mathrm repsilon}. In addition, combining \eqref{solution NM non rigidified}, \eqref{implicit omegaper} and \eqref{FCS}, the indentity \eqref{solution NM rigidified} holds. 
		 The proof of Corollary \ref{Corollary NM} is now complete.
	\end{proof}
	\subsection{Measure of the final Cantor set}
	In this last section, we check that the final Cantor set $\mathcal{C}_{\infty}^{\varepsilon}$ in the variable $b$ given by \eqref{FCS} is massive set, which proves the existence of non-trivial quasi-periodic solution to our problem. Actually, we prove that the measure of $\mathcal{C}_{\infty}^{\varepsilon}$ is $\varepsilon$-close to $(b^*-b_*).$ One of the main technical ingredient is  the following  R\"{u}ssmann Lemma \cite[Thm. 17.1]{R01}. 
	\begin{lem}\label{lemma Russmann book}
		Let $q_{0}\in\mathbb{N}^{*}$ and $\alpha,\beta\in\mathbb{R}_{+}^*.$ Let $f\in C^{q_{0}}([a,b],\mathbb{R})$ such that
		$$\inf_{x\in[a,b]}\max_{k\in\llbracket 0,q_{0}\rrbracket}|f^{(k)}(x)|\geqslant\beta.$$
		Then, there exists $C=C(a,b,q_{0},\| f\|_{C^{q_{0}}([a,b],\mathbb{R})})>0$ such that 
		$$\Big|\big\lbrace x\in[a,b]\quad\textnormal{s.t.}\quad |f(x)|\leqslant\alpha\big\rbrace\Big|\leqslant C\tfrac{\alpha^{\frac{1}{q_{0}}}}{\beta^{1+\frac{1}{q_{0}}}},$$
		where the notation  $|A|$  corresponds to the Lebesgue measure of a given measurable set $A.$
	\end{lem}
	Our main result is stated in the next proposition.
	\begin{prop}\label{lem-meas-es1}
		Let $q_{0}$ be defined as in Lemma $\ref{lemma transversalityE}$ and assume that \eqref{param NM} and \eqref{rigidity gam-N0} hold with $q=q_0+1.$ Assume the additional conditions
		\begin{equation}\label{selec tau12-upsilon}
			\tau_1>dq_{0},\qquad\tau_2>\tau_1+dq_0,\qquad\upsilon\triangleq \frac{1}{q_0+3}\cdot
		\end{equation}
		Then there exists $C>0$ such that 
		$$(b^{*}-b_{*})-C \varepsilon^{\frac{a\upsilon}{q_{0}}}\leqslant\big|\mathcal{C}_{\infty}^{\varepsilon}\big|\leqslant b^*-b_* .$$
	\end{prop}
	\begin{proof}
		The identities \eqref{FCS} and \eqref{Canbomgfin} provide the following decomposition of the final Cantor set
		\begin{equation}\label{Cinftyepsilon}
			\mathcal{C}_{\infty}^{\varepsilon}=\bigcap_{n\in\mathbb{N}}\mathcal{C}_{n}^{\varepsilon},\qquad\textnormal{where}\qquad \mathcal{C}_{n}^{\varepsilon}\triangleq \Big\{b\in(b_{*},b^{*})\quad\hbox{s.t}\quad \big(b,{\omega}(b,\varepsilon)\big)\in\mathtt{A}_n^{\gamma}\Big\},
		\end{equation}
		with $\mathtt{A}_n^{\gamma}$ and $\omega(b,\varepsilon)$ as in Proposition \ref{Nash-Moser} and \eqref{pert freq}. We can write
		\begin{equation}\label{complement finCantset}
			(b_{*},b^{*})\setminus\mathcal{C}_{\infty}^{\varepsilon}=\big((b_{*},b^{*})\setminus\mathcal{C}_{0}^{\varepsilon}\big)\sqcup\bigsqcup_{n=0}^{\infty}\big(\mathcal{C}_{n}^{\varepsilon}\setminus\mathcal{C}_{n+1}^{\varepsilon}\big).
		\end{equation}
		First, let us prove that
		\begin{equation}\label{triv emb}
			(b_{*},b^{*})\setminus\mathcal{C}_{0}^{\varepsilon}=\varnothing,\qquad\textnormal{i.e.}\qquad\mathcal{C}_{0}^{\varepsilon}=(b_*,b^*).
		\end{equation}
		For this purpose, notice that \eqref{pert freq} and \eqref{rigidity gam-N0} imply
		$$\sup_{b\in(b_{*},b^{*})}\big|\omega(b,\varepsilon)+\omega_{\textnormal{Eq}}(b)\big|\leqslant\|\overline{\mathrm{r}}_{\varepsilon}\|^{q,\gamma}\leqslant C\varepsilon\gamma^{-1}N_{0}^{q\overline{a}}=C\varepsilon^{1-a(1+q\overline{a})}.$$
		But \eqref{param NM} and \eqref{rigidity gam-N0} give in particular
		$$0<a<\frac{1}{1+q\overline{a}}\cdot$$
		Hence, in view of \eqref{def scrU}, for $\varepsilon$ sufficiently small we can ensure
		$$\forall b\in(b_{*},b^{*}),\quad \omega(b,\varepsilon)\in \mathscr{U}=B(0,R_{0}).$$
		By construction of $\mathtt{A}_{0}^{\gamma}$ and $\mathcal{O}$, we deduce \eqref{triv emb}. Coming back to \eqref{complement finCantset}, we find 
		\begin{align}\label{sum meas CnmCn+1}
			\nonumber \Big|(b_{*},b^{*})\setminus\mathcal{C}_{\infty}^{\varepsilon}\Big|&\leqslant\sum_{n=0}^{\infty}\Big|\mathcal{C}_{n}^{\varepsilon}\setminus\mathcal{C}_{n+1}^{\varepsilon}\Big|\\
			&\triangleq \sum_{n=0}^{\infty}\mathcal{S}_{n}.
		\end{align}
		Using the notations of Propositions \ref{prop RR} and \ref{prop proj nor dir}, we denote the perturbed frequencies associated with the reduced linearized operator at state $i_n$ in the following way  
		\begin{align}\label{def mujknfty}
			\nonumber \mu_{j,k,n}^{(\infty)}(b,\varepsilon)&\triangleq \mu_{j,k,n}^{(\infty)}\big(b,{\omega}(b,\varepsilon),i_{n}\big)\\
			&=\Omega_{j,k}(b)+jr_{k,n}^{(0)}(b,\varepsilon)+r_{j,k,n}^{(\infty)}(b,\varepsilon),
		\end{align}
		where
		$$r_{k,n}^{(0)}(b,\varepsilon)\triangleq \mathtt{c}_{k,n}(b,\varepsilon)-\mathtt{v}_k(b),\qquad \mathtt{c}_{k,n}(b,\varepsilon)\triangleq \mathtt{c}_{k}\big(b,{\omega}(b,\varepsilon),i_{n}\big),\qquad r_{j,k,n}^{(\infty)}(b,\varepsilon)\triangleq r_{j,k}^{(\infty)}\big(b,{\omega}(b,\varepsilon),i_{n}\big).$$
		Now, according to \eqref{Cinftyepsilon} and Propositions \ref{prop strighten}, \ref{prop RR} and \ref{prop inv linfty}, one has by construction that for any $n\in\mathbb{N}$, 
		\begin{align}\label{Set CnmCn+1}
			\mathcal{C}_{n}^{\varepsilon}\setminus\mathcal{C}_{n+1}^{\varepsilon}&=\bigcup_{\underset{(l,j)\in\mathbb{Z}^{d}\times\mathbb{Z}_{\mathtt{m}}\setminus\{(0,0)\}\atop |l|\leqslant N_{n}}{k\in\{1,2\}}}\mathcal{R}_{l,j}^{(0,k)}(i_{n})\cup\bigcup_{\underset{(l,j)\in\mathbb{Z}^{d}\times(\mathbb{Z}_{\mathtt{m}}\setminus\overline{\mathbb{S}}_{0,k})\atop |l|\leqslant N_{n}}{k\in\{1,2\}}}\mathcal{R}_{l,j}^{(1,k)}(i_{n})\nonumber\\
			&\quad\cup\bigcup_{\underset{\underset{\langle l\rangle\leqslant N_{n}}{(l,j,j_{0})\in\mathbb{Z}^{d}\times(\mathbb{Z}_{\mathtt{m}}\setminus\overline{\mathbb{S}}_{0,k})^{2}}\atop(l,j)\neq(0,j_{0})}{k\in\{1,2\}}}\mathcal{R}_{l,j,j_{0}}^{k}(i_{n})\cup\bigcup_{(l,j,j_{0})\in\mathbb{Z}^{d}\times(\mathbb{Z}_{\mathtt{m}}\setminus\overline{\mathbb{S}}_{0,1})\times(\mathbb{Z}_{\mathtt{m}}\setminus\overline{\mathbb{S}}_{0,2})\atop\langle l,j,j_0\rangle\leqslant N_{n}}\mathcal{R}_{l,j,j_{0}}^{1,2}(i_{n}),
		\end{align}
		where we denote for $k\in\{1,2\},$
		\begin{align*}
			\mathcal{R}_{l,j}^{(0,k)}(i_{n})&\triangleq \left\lbrace b\in\mathcal{C}_{n}^{\varepsilon}\quad\textnormal{s.t.}\quad\Big|{\omega}(b,\varepsilon)\cdot l+j\mathtt{c}_{k,n}(b,\varepsilon)\Big|\leqslant\tfrac{4\gamma_{n+1}^{\upsilon}\langle j\rangle}{\langle l\rangle^{\tau_1}}\right\rbrace,\\
			\mathcal{R}_{l,j,j_{0}}^{k}(i_{n})&\triangleq \left\lbrace b\in\mathcal{C}_{n}^{\varepsilon}\quad\textnormal{s.t.}\quad\Big|{\omega}(b,\varepsilon)\cdot l+\mu_{j,k,n}^{(\infty)}(b,\varepsilon)-\mu_{j_{0},k,n}^{(\infty)}(b,\varepsilon)\Big|\leqslant\tfrac{2\gamma_{n+1}\langle j-j_0\rangle}{\langle l\rangle^{\tau_2}}\right\rbrace,\\
			\mathcal{R}_{l,j,j_{0}}^{1,2}(i_{n})&\triangleq \left\lbrace b\in\mathcal{C}_{n}^{\varepsilon}\quad\textnormal{s.t.}\quad\Big|{\omega}(b,\varepsilon)\cdot l+\mu_{j,1,n}^{(\infty)}(b,\varepsilon)-\mu_{j_{0},2,n}^{(\infty)}(b,\varepsilon)\Big|\leqslant\tfrac{2\gamma_{n+1}}{\langle l,j,j_0\rangle^{\tau_2}}\right\rbrace,\\
			\mathcal{R}_{l,j}^{(1,k)}(i_{n})&\triangleq \left\lbrace b\in\mathcal{C}_{n}^{\varepsilon}\quad\textnormal{s.t.}\quad\Big|{\omega}(b,\varepsilon)\cdot l+\mu_{j,k,n}^{(\infty)}(b,\varepsilon)\Big|\leqslant\tfrac{\gamma_{n+1}\langle j\rangle}{\langle l\rangle^{\tau_1}}\right\rbrace.
		\end{align*}
		Since 
		$$W^{q,\infty,\gamma}(\mathcal{O},\mathbb{C})\hookrightarrow C^{q-1}(\mathcal{O},\mathbb{C})\qquad\textnormal{and}\qquad q=q_0+1,$$
		one obtains for any $n\in\mathbb{N}$ and for any $(k,\ell)\in\{1,2\}^2$, the $C^{q_0}$ regularity of the curves
		$$\begin{array}{ll}
			\displaystyle b\mapsto\omega(b,\varepsilon)\cdot l+j\mathtt{c}_{k,n}(b,\varepsilon),&\quad (l,j)\in\mathbb{Z}^{d}\times\mathbb{Z}_{\mathtt{m}}\backslash\{(0,0)\},\vspace{0.1cm}\\
			\displaystyle b\mapsto\omega(b,\varepsilon)\cdot l+\mu_{j,k,n}^{(\infty)}(b,\varepsilon)-\mu_{j_0,\ell,n}^{(\infty)}(b,\varepsilon),&\quad (l,j,j_0)\in\mathbb{Z}^{d}\times(\mathbb{Z}_{\mathtt{m}}\setminus\overline{\mathbb{S}}_{0,k})\times (\mathbb{Z}_{\mathtt{m}}\setminus\overline{\mathbb{S}}_{0,\ell}),\vspace{0.1cm}\\
			\displaystyle b\mapsto\omega(b,\varepsilon)\cdot l+\mu_{j,k,n}^{(\infty)}(b,\varepsilon),&\quad (l,j)\in\mathbb{Z}^{d}\times(\mathbb{Z}_{\mathtt{m}}\setminus\overline{\mathbb{S}}_{0,k}).	
		\end{array}$$
		Therefore, applying Lemma \ref{lemma Russmann book} together with Lemma \ref{lem Ru-pert} yields that for all $n\in\mathbb{N}$, 
		\begin{align}\label{e-Russ- Rk12}
			\forall |j|\leqslant C_{0}\langle l\rangle,\qquad \Big|\mathcal{R}_{l,j}^{(0,k)}(i_{n})\Big|&\lesssim\gamma^{\frac{\upsilon}{q_{0}}}\langle j\rangle^{\frac{1}{q_{0}}}\langle l\rangle^{-1-\frac{\tau_1+1}{q_{0}}},\nonumber\\
			\forall |j|\leqslant C_{0}\langle l\rangle,\qquad \Big|\mathcal{R}_{l,j}^{(1,k)}(i_{n})\Big|&\lesssim\gamma^{\frac{1}{q_{0}}}\langle j\rangle^{\frac{1}{q_{0}}}\langle l\rangle^{-1-\frac{\tau_1+1}{q_{0}}},\\
			\forall |j-j_0|\leqslant C_{0}\langle l\rangle,\qquad \Big|\mathcal{R}_{l,j,j_{0}}^{k}(i_{n})\Big|&\lesssim\gamma^{\frac{1}{q_{0}}}\langle j-j_0\rangle^{\frac{1}{q_{0}}}\langle l\rangle^{-1-\frac{\tau_2+1}{q_{0}}},\nonumber\\
			\Big|\mathcal{R}_{l,j,j_{0}}^{1,2}(i_{n})\Big|&\lesssim\gamma^{\frac{1}{q_{0}}}\langle l,j,j_0\rangle^{-1-\frac{\tau_2+1}{q_{0}}}.\nonumber
		\end{align}
		We first estimate the measure of $\mathcal{S}_0$ and $\mathcal{S}_{1}$ defined in \eqref{sum meas CnmCn+1}. Their estimation cannot be done in a similar way to the other terms due to the range of validity of the estimate \eqref{ediff in in-1 norm sh+sigma4} obtained later in the proof of Lemma \ref{lemm-dix1}. From Lemma  \ref{some cantor set are empty}, we have some trivial inclusions allowing us to write for $n\in\{0,1\}$,
		\begin{align}\label{e-cal S01}
			\mathcal{S}_{n}&\lesssim  \sum_{k\in\{1,2\}\atop\underset{|j|\leqslant C_{0}\langle l\rangle, |l|\leqslant N_{n}}{(l,j)\in\mathbb{Z}^{d}\times\mathbb{Z}_{\mathtt{m}}\setminus\{(0,0)\}}}\Big|\mathcal{R}_{l,j}^{(0,k)}(i_{n})\Big|+\sum_{k\in\{1,2\}\atop\underset{|j|\leqslant C_{0}\langle l\rangle, |l|\leqslant N_{n}}{(l,j)\in\mathbb{Z}^{d}\times(\mathbb{Z}_{\mathtt{m}}\setminus\overline{\mathbb{S}}_{0,k})}}\Big|\mathcal{R}_{l,j}^{(1,k)}(i_{n})\Big|\\
			&\quad+\sum_{k\in\{1,2\}\atop\underset{\underset{(l,j)\neq(0,j_0),\min(|j|,|j_0|)\leqslant c_2\gamma_{1}^{-\upsilon}\langle l\rangle^{\tau_1}}{|j-j_0|\leqslant C_0\langle l\rangle,|l|\leqslant N_n}}{(l,j,j_{0})\in\mathbb{Z}^{d}\times(\mathbb{Z}_{\mathtt{m}}\setminus\overline{\mathbb{S}}_{0,k})^2}}\Big|\mathcal{R}_{l,j,j_{0}}^{k}(i_{n})\Big|+\sum_{(l,j,j_{0})\in\mathbb{Z}^{d}\times(\mathbb{Z}_{\mathtt{m}}\setminus\overline{\mathbb{S}}_{0,1})\times(\mathbb{Z}_{\mathtt{m}}\setminus\overline{\mathbb{S}}_{0,2})\atop\langle l,j,j_0\rangle\leqslant N_{n}}\Big|\mathcal{R}_{l,j,j_{0}}^{1,2}(i_{n})\Big|.\nonumber
		\end{align}
		Inserting \eqref{e-Russ- Rk12} into \eqref{e-cal S01} implies that for $n\in\{0,1\}$,
		\begin{align*}
			\mathcal{S}_{n}&\lesssim \gamma^{\frac{\upsilon}{q_0}}\sum_{(l,j)\in\mathbb{Z}^{d+1}\atop|j|\leqslant C_{0}\langle l\rangle}\langle j\rangle^{\frac{1}{q_{0}}}\langle l\rangle^{-1-\frac{\tau_1+1}{q_{0}}}+ \gamma^{\frac{1}{q_{0}}}\sum_{(l,j)\in\mathbb{Z}^{d+1}\atop|j|\leqslant C_{0}\langle l\rangle}\langle j\rangle^{\frac{1}{q_{0}}}\langle l\rangle^{-1-\frac{\tau_1+1}{q_{0}}}\\
			&\quad+\gamma^{\frac{1}{q_{0}}}\sum_{(l,j,j_0)\in\mathbb{Z}^{d+2}\atop\underset{\min(|j|,|j_0|)\leqslant c_2\gamma_1^{\upsilon}\langle l\rangle^{\tau_1}}{|j-j_0|\leqslant C_{0}\langle l\rangle}}\langle j-j_0\rangle^{\frac{1}{q_{0}}}\langle l\rangle^{-1-\frac{\tau_2+1}{q_{0}}}+\gamma^{\frac{1}{q_{0}}}\sum_{(l,j,j_0)\in\mathbb{Z}^{d+2}}\langle l,j,j_0\rangle^{-1-\frac{\tau_2+1}{q_{0}}}.
		\end{align*}
		Notice that the conditions $|j-j_0|\leqslant C_0\langle l\rangle$ and $\min(|j|,|j_0|)\leqslant c_2\gamma_1^{-\upsilon}\langle l\rangle^{\tau_1}$ imply
		\begin{equation}\label{Majo-max j j0}
			\max(|j|,|j_0|)\leqslant\min(|j|,|j_0|)+|j-j_0|\leqslant c_2\gamma_1^{-\upsilon}\langle l\rangle^{\tau_1}+C_0\langle l\rangle\lesssim\gamma^{-\upsilon}\langle l\rangle^{\tau_1}.
		\end{equation}
		Consequently, we have
		\begin{align}\label{e-calSn00}
			\max_{n\in\{0,1\}}\mathcal{S}_{n}&\lesssim  \gamma^{\frac{1}{q_{0}}}\Bigg(\sum_{l\in\mathbb{Z}^d}\langle l\rangle^{-\frac{\tau_1}{q_{0}}}+\gamma^{-\upsilon}\sum_{\l\in\mathbb{Z}^d}\langle l\rangle^{\tau_1-1-\frac{\tau_2}{q_0}}+\sum_{(l,j,j_0)\in\mathbb{Z}^{d+2}}\langle l,j,j_0\rangle^{-1-\frac{\tau_2+1}{q_{0}}}\Bigg)+\gamma^{\frac{\upsilon}{q_0}}\sum_{l\in\mathbb{Z}^d}\langle l\rangle^{-\frac{\tau_1}{q_0}}.
		\end{align}
		Observe that \eqref{selec tau12-upsilon} implies 
		$$\min\big(\tfrac{\upsilon}{q_0}\,\tfrac{1}{q_0}-\upsilon\big)=\tfrac{\upsilon}{q_0}.$$
		Now the constraints on $\tau_1$ and $\tau_2$ listed in \eqref{selec tau12-upsilon} allow to make the series in \eqref{e-calSn00} convergent and we get
		\begin{equation}\label{e-calSn 01}
			\max_{n\in\{0,1\}}\mathcal{S}_n\lesssim\gamma^{\min\big(\frac{\upsilon}{q_0},\tfrac{1}{q_0}-\upsilon\big)}=\gamma^{\frac{\upsilon}{q_0}}.
		\end{equation}
		Let us now move to the estimate of $\mathcal{S}_{n}$ for $n\geqslant 2$ defined by \eqref{sum meas CnmCn+1}. Using Lemma \ref{lemm-dix1} and Lemma \ref{some cantor set are empty}, we infer
		\begin{align*}
			\mathcal{S}_{n}&\lesssim  \sum_{k\in\{1,2\}\atop\underset{|j|\leqslant C_{0}\langle l\rangle, N_{n-1}<|l|\leqslant N_{n}}{(l,j)\in\mathbb{Z}^{d}\times\mathbb{Z}_{\mathtt{m}}\setminus\{(0,0)\}}}\Big|\mathcal{R}_{l,j}^{(0,k)}(i_{n})\Big|+\sum_{k\in\{1,2\}\atop\underset{|j|\leqslant C_{0}\langle l\rangle, N_{n-1}<|l|\leqslant N_{n}}{(l,j)\in\mathbb{Z}^{d}\times(\mathbb{Z}_{\mathtt{m}}\setminus\overline{\mathbb{S}}_{0,k})}}\Big|\mathcal{R}_{l,j}^{(1,k)}(i_{n})\Big|\\
			&\quad+\sum_{k\in\{1,2\}\atop\underset{\underset{(l,j)\neq(0,j_0),\min(|j|,|j_0|)\leqslant c_2\gamma_{n+1}^{-\upsilon}\langle l\rangle^{\tau_1}}{|j-j_0|\leqslant C_0\langle l\rangle,N_{n-1}<|l|\leqslant N_n}}{(l,j,j_{0})\in\mathbb{Z}^{d}\times(\mathbb{Z}_{\mathtt{m}}\setminus\overline{\mathbb{S}}_{0,k})^2}}\Big|\mathcal{R}_{l,j,j_{0}}^{k}(i_{n})\Big|+\sum_{(l,j,j_{0})\in\mathbb{Z}^{d}\times(\mathbb{Z}_{\mathtt{m}}\setminus\overline{\mathbb{S}}_{0,1})\times(\mathbb{Z}_{\mathtt{m}}\setminus\overline{\mathbb{S}}_{0,2})\atop N_{n-1}<\langle l,j,j_0\rangle\leqslant N_{n}}\Big|\mathcal{R}_{l,j,j_{0}}^{1,2}(i_{n})\Big|.
		\end{align*}
		Similarly to \eqref{Majo-max j j0}, we have the implication
		$$\Big(\min(|j|,|j_0|)\leqslant c_2\gamma_{n+1}^{-\upsilon}\langle l\rangle^{\tau_1}\quad\textnormal{and}\quad|j-j_0|\leqslant C_0\langle l\rangle\Big)\quad\Rightarrow\quad\max(|j|,|j_0|)\lesssim\gamma^{-\upsilon}\langle l\rangle^{\tau_1}.$$
		Hence, we deduce from \eqref{e-Russ- Rk12} that for any $n\geqslant2$
		\begin{align*}
			\mathcal{S}_{n}&\lesssim  \gamma^{\frac{1}{q_{0}}}\Bigg(\sum_{l\in\mathbb{Z}^d\atop|l|> N_{n-1}}\langle l\rangle^{-\frac{\tau_1}{q_{0}}}+\gamma^{-\upsilon}\sum_{\l\in\mathbb{Z}^d\atop|l|>N_{n-1}}\langle l\rangle^{\tau_1-1-\frac{\tau_2}{q_0}}+\sum_{(l,j,j_0)\in\mathbb{Z}^{d+2}\atop\langle l,j,j_0\rangle>N_{n-1}}\langle l,j,j_0\rangle^{-1-\frac{\tau_2+1}{q_{0}}}\Bigg)+\gamma^{\frac{\upsilon}{q_0}}\sum_{l\in\mathbb{Z}^d\atop|l|>N_{n-1}}\langle l\rangle^{-\frac{\tau_1}{q_0}}.
		\end{align*}
		We deduce that the series of general term $\mathcal{S}_n$ converges and
		\begin{align}\label{esti cal Snb2}
			\sum_{n=2}^\infty \mathcal{S}_{n}&\lesssim  \gamma^{\frac{\upsilon}{q_{0}}}=\varepsilon^{\frac{a\upsilon}{q_0}}.
		\end{align}
		Inserting \eqref{e-calSn 01} and \eqref{esti cal Snb2} into \eqref{sum meas CnmCn+1} yields
		$$\Big|(b_{*},b^{*})\setminus\mathcal{C}_{\infty}^{\varepsilon}\Big|\lesssim  \varepsilon^{\frac{a\upsilon}{q_0}}.$$
		This proves the Proposition \ref{lem-meas-es1}. 
	\end{proof}
	We shall now prove Lemma \ref{lemm-dix1} and Lemma \ref{some cantor set are empty} used in the proof of Proposition \ref{lem-meas-es1}.
	\begin{lem}\label{lemm-dix1}
		Let $n\in\mathbb{N}\setminus\{0,1\}$ and $k\in\{1,2\}.$ Then the following assertions hold true.
		\begin{enumerate}[label=(\roman*)]
			\item For $(l,j)\in\mathbb{Z}^d\times\mathbb{Z}_{\mathtt{m}}$ with $(l,j)\neq(0,0)$ and $|l|\leqslant N_{n-1}$, we get  $\,\,\mathcal{R}_{l,j}^{(0,k)}(i_{n})=\varnothing.$
			\item For  $(l,j)\in\mathbb{Z}^d\times(\mathbb{Z}_{\mathtt{m}}\setminus\overline{\mathbb{S}}_{0,k})$ with $|l|\leqslant N_{n-1}$, we get $\,\,\mathcal{R}_{l,j}^{(1,k)}(i_{n})=\varnothing.$
			\item For $(l,j,j_0)\in\mathbb{Z}^d\times(\mathbb{Z}_{\mathtt{m}}\setminus\overline{\mathbb{S}}_{0,k})^2$ with $|l|\leqslant N_{n-1}$ and $(l,j)\neq(0,j_0),$ we get $\,\,\mathcal{R}_{l,j,j_0}^{k}(i_n)=\varnothing.$
			\item For $(l,j,j_0)\in\mathbb{Z}^d\times(\mathbb{Z}_{\mathtt{m}}\setminus\overline{\mathbb{S}}_{0,1})\times(\mathbb{Z}_{\mathtt{m}}\setminus\overline{\mathbb{S}}_{0,2})$ with $\langle l,j,j_0\rangle\leqslant N_{n-1},$ we get $\,\,\mathcal{R}_{l,j,j_0}^{1,2}(i_n)=\varnothing.$
			\item For any $n\in\mathbb{N}\setminus\{0,1\},$
			\begin{align*}
				\mathcal{C}_{n}^{\varepsilon}\setminus\mathcal{C}_{n+1}^{\varepsilon}&=\bigcup_{\underset{\underset{N_{n-1}<|l|\leqslant N_{n}}{(l,j)\in\mathbb{Z}^{d}\times\mathbb{Z}_{\mathtt{m}}\setminus\{(0,0)\}}}{k\in\{1,2\}}}\mathcal{R}_{l,j}^{(0,k)}(i_{n})\cup\bigcup_{\underset{\underset{N_{n-1}<|l|\leqslant N_{n}}{(l,j)\in\mathbb{Z}^{d}\times(\mathbb{Z}_{\mathtt{m}}\setminus\overline{\mathbb{S}}_{0,k})}}{k\in\{1,2\}}}\mathcal{R}_{l,j}^{(1,k)}(i_{n})\\
				&\quad\cup\bigcup_{\underset{\underset{N_{n-1}<|l|\leqslant N_{n}}{(l,j,j_{0})\in\mathbb{Z}^{d}\times(\mathbb{Z}_{\mathtt{m}}\setminus\overline{\mathbb{S}}_{0,k})^{2}}\atop(l,j)\neq(0,j_{0})}{k\in\{1,2\}}}\mathcal{R}_{l,j,j_{0}}^{k}(i_{n})\cup\bigcup_{(l,j,j_{0})\in\mathbb{Z}^{d}\times(\mathbb{Z}_{\mathtt{m}}\setminus\overline{\mathbb{S}}_{0,1})\times(\mathbb{Z}_{\mathtt{m}}\setminus\overline{\mathbb{S}}_{0,2})\atop N_{n-1}<\langle l,j,j_0\rangle\leqslant N_{n}}\mathcal{R}_{l,j,j_{0}}^{1,2}(i_{n}).
			\end{align*}
		\end{enumerate}
	\end{lem}
	\begin{proof}
		Observe that the point (v) follows immediately from \eqref{Set CnmCn+1} and the points (i), (ii), (iii) and (iv). The points (i), (ii) and (iii) can be proved similarly to \cite[Lem. 7.1-(i)-(ii)-(iii)]{HR21} based on the following estimate, obtained from \eqref{ttHn shbsig4}.
		\begin{align}\label{ediff in in-1 norm sh+sigma4}
			\forall n\geqslant 2,\qquad\|i_{n}-i_{n-1}\|_{\overline{s}_{h}+\overline{\sigma}}^{q,\gamma,\mathtt{m}}&\leqslant\|\mathtt{U}_{n}-\mathtt{U}_{n-1}\|_{\overline{s}_{h}+\overline{\sigma}}^{q,\gamma,\mathtt{m}}\nonumber\\
			&\leqslant\| \mathtt{H}_{n}\|_{\overline{s}_{h}+\overline{\sigma}}^{q,\gamma,\mathtt{m}}\nonumber\\
			&\leqslant C_{\ast}\varepsilon\gamma^{-1}N_{n-1}^{-a_{2}}.
		\end{align}
		We mention that the required constraint on $\upsilon$ stated in \eqref{selec tau12-upsilon} appears in the skipped proofs. Now it remains to prove the point (iv).\\
		\noindent \textbf{(iv)} Let $(l,j,j_{0})\in\mathbb{Z}^d\times(\mathbb{Z}_{\mathtt{m}}\setminus\overline{\mathbb{S}}_{0,1})\times(\mathbb{Z}_{\mathtt{m}}\setminus\overline{\mathbb{S}}_{0,2})$ such that $\langle l,j,j_0\rangle\leqslant N_{n-1}.$ It is sufficient to prove that 
		$$\mathcal{R}_{l,j,j_{0}}^{1,2}(i_{n})\subset \mathcal{R}_{l,j,j_{0}}^{1,2}(i_{n-1}).$$
		Indeed, if this inclusion holds, then by construction
		$$\mathcal{R}_{l,j,j_0}^{1,2}(i_n)\subset\big(\mathcal{C}_{n}^{\varepsilon}\setminus\mathcal{C}_{n+1}^{\varepsilon}\big)\cap\big(\mathcal{C}_{n-1}^{\varepsilon}\setminus\mathcal{C}_{n}^{\varepsilon}\big)=\varnothing.$$
		Take $b\in\mathcal{R}_{l,j,j_{0}}^{1,2}(i_{n})\subset\mathcal{C}_{n}^{\varepsilon}\subset \mathcal{C}_{n-1}^{\varepsilon}.$ Then coming back to \eqref{Set CnmCn+1}, we deduce from the triangle inequality that
		\begin{equation}\label{triv121}
			\left|{\omega}(b,\varepsilon)\cdot l+\mu_{j,1,n-1}^{(\infty)}(b,\varepsilon)-\mu_{j_{0},2,n-1}^{(\infty)}(b,\varepsilon)\right|\leqslant\tfrac{2\gamma_{n+1}}{\langle l,j,j_0\rangle^{\tau_2}}+\varrho_{j,j_0,n}(b,\varepsilon),
		\end{equation}
		where 
		$$\varrho_{j,j_0,n}(b,\varepsilon)\triangleq \left|\mu_{j,1,n}^{(\infty)}(b,\varepsilon)-\mu_{j_0,2,n}^{(\infty)}(b,\varepsilon)-\mu_{j,1,n-1}^{(\infty)}(b,\varepsilon)+\mu_{j_{0},2,n-1}^{(\infty)}(b,\varepsilon)\right|.$$ 
		Using the decomposition \eqref{def mujknfty}, we infer
		\begin{align}\label{triv122}
			\nonumber \varrho_{j,j_0,n}(b,\varepsilon)&\leqslant |j|\left|r_{1,n}^{(0)}(b,\varepsilon)-r_{1,n-1}^{(0)}(b,\varepsilon)\right|+|j_0|\left|r_{2,n}^{(0)}(b,\varepsilon)-r_{2,n-1}^{(0)}(b,\varepsilon)\right|\\
			&\quad+\left|r_{j,1,n}^{(\infty)}(b,\varepsilon)-r_{j,1,n-1}^{(\infty)}(b,\varepsilon)\right|+\left|r_{j_0,2,n}^{(\infty)}(b,\varepsilon)-r_{j_0,2,n-1}^{(\infty)}(b,\varepsilon)\right|.
		\end{align}
		Applying  \eqref{e-ed-r0} together with \eqref{ediff in in-1 norm sh+sigma4}, \eqref{rigidity gam-N0} and the fact that $\sigma_{4}\geqslant\sigma_{3},$ we obtain
		\begin{align}\label{triv123}
			\left|r_{1,n}^{(0)}(b,\varepsilon)-r_{1,n-1}^{(0)}(b,\varepsilon)\right|+\left|r_{2,n}^{(0)}(b,\varepsilon)-r_{2,n-1}^{(0)}(b,\varepsilon)\right|&\lesssim \varepsilon\| i_{n}-i_{n-1}\|_{\overline{s}_{h}+\sigma_{3}}^{q,\gamma,\mathtt{m}}\nonumber\\
			&\lesssim \varepsilon^{2}\gamma^{-1}N_{n-1}^{-a_2}\nonumber\\
			&\lesssim \varepsilon^{2-a}N_{n-1}^{-a_2}.
		\end{align}
		In the same way, we can apply \eqref{e-rjfty} together with \eqref{ediff in in-1 norm sh+sigma4} and \eqref{rigidity gam-N0} to deduce
		\begin{align}\label{triv124}
			\left|r_{j,1,n}^{(\infty)}(b,\varepsilon)-r_{j,1,n-1}^{(\infty)}(b,\varepsilon)\right|+\left|r_{j_0,2,n}^{(\infty)}(b,\varepsilon)-r_{j_0,2,n-1}^{(\infty)}(b,\varepsilon)\right|&\lesssim \varepsilon\gamma^{-1}\| i_{n}-i_{n-1}\|_{\overline{s}_{h}+\sigma_{4}}^{q,\gamma,\mathtt{m}}\nonumber\\
			&\lesssim \varepsilon^{2}\gamma^{-2}N_{n-1}^{-a_2}\nonumber\\
			&\lesssim \varepsilon^{2(1-a)}\langle l,j,j_0\rangle N_{n-1}^{-a_2}.
		\end{align}
		Inserting \eqref{triv123} and \eqref{triv124} into \eqref{triv122} gives
		\begin{align}\label{sml-rhojj0n}
			\varrho_{j,j_0,n}(b,\varepsilon)\lesssim \varepsilon^{2(1-a)}\langle l,j,j_0\rangle N_{n-1}^{-a_{2}}.
		\end{align}
		Now putting together \eqref{triv121},  \eqref{sml-rhojj0n} and the fact that $\gamma_{n+1}=\gamma_{n}-\varepsilon^a 2^{-n-1},$ we get
		\begin{align*}
			\left|{\omega}(b,\varepsilon)\cdot l+\mu_{j,1,n-1}^{(\infty)}(b,\varepsilon)-\mu_{j_{0},2,n-1}^{(\infty)}(b,\varepsilon)\right| 
			&\leqslant \displaystyle\tfrac{2\gamma_{n}}{\langle l,j,j_0\rangle^{\tau_2}}-{\varepsilon^a}2^{-n}\langle l,j,j_0\rangle ^{-\tau_2}+C\varepsilon^{2(1-a)}\langle l,j,j_0\rangle N_{n-1}^{-a_{2}}.
		\end{align*} 
		Added to the constraint $\langle l,j,j_0\rangle\leqslant N_{n-1}$, we obtain
		\begin{align*}
			-{\varepsilon^a }2^{-n}\langle l,j,j_0\rangle ^{-\tau_2}+C\varepsilon^{2(1-a)}N_{n-1}^{-a_{2}}\leqslant {\varepsilon^a}2^{-n}\langle l,j,j_0\rangle ^{-\tau_2}\Big(-1+C\varepsilon^{2-3a}2^{n}N_{n-1}^{-a_{2}+\tau_2+1}\Big).
		\end{align*}
		Observe that our choice of parameters \eqref{param NM} and \eqref{rigidity gam-N0} gives in particular
		\begin{align*}
			a_2>\tau_2+1\qquad\textnormal{and}\qquad a<\tfrac{2}{3}\cdot
		\end{align*}
		Hence, up to taking $\varepsilon$ small enough, we infer
		\begin{align*}
			\forall\, n\in\mathbb{N},\quad -1+C\varepsilon^{2-3a}2^{n}N_{n-1}^{-a_{2}+\tau_2+1}\leqslant 0.
		\end{align*}
		Consequently,
		$$\left|{\omega}(b,\varepsilon)\cdot l+\mu_{j,1,n-1}^{(\infty)}(b,\varepsilon)-\mu_{j_{0},2,n-1}^{(\infty)}(b,\varepsilon)\right| \leqslant\tfrac{2\gamma_{n}}{\langle l,j,j_0\rangle^{\tau_2}}\cdot$$
		This means that $b\in  \mathcal{R}_{l,j,j_{0}}^{1,2}(i_{n-1}).$ This concludes the proof of Lemma \ref{lemm-dix1}.
	\end{proof}
	The following lemma provides necessary constraints between time and space Fourier modes so that the sets in \eqref{Set CnmCn+1} are not void.
	\begin{lem}\label{some cantor set are empty}
		Let $k\in\{1,2\}.$ There exists $\varepsilon_0$ such that for any $\varepsilon\in[0,\varepsilon_0]$ and $n\in\mathbb{N}$ the following assertions hold true. 
		\begin{enumerate}[label=(\roman*)]
			\item Let $(l,j)\in\mathbb{Z}^{d}\times\mathbb{Z}_{\mathtt{m}}\setminus\{(0,0)\}.$ If $\,\mathcal{R}_{l,j}^{(0,k)}(i_{n})\neq\varnothing,$ then $|j|\leqslant C_{0}\langle l\rangle.$
			\item Let $(l,j)\in\mathbb{Z}^{d}\times\left(\mathbb{Z}_{\mathtt{m}}\setminus\overline{\mathbb{S}}_{0,k}\right).$ If $\, \mathcal{R}_{l,j}^{(1,k)}(i_{n})\neq\varnothing,$ then $|j|\leqslant C_{0}\langle l\rangle.$
			\item Let $(l,j,j_0)\in\mathbb{Z}^d\times(\mathbb{Z}_{\mathtt{m}}\setminus\overline{\mathbb{S}}_{0,k})^2.$ If $\mathcal{R}_{l,j,j_0}^{k}(i_n)\neq\varnothing$, then $|j-j_0|\leqslant C_0\langle l\rangle.$
			\item Let $(l,j,j_0)\in\mathbb{Z}^d\times(\mathbb{Z}_{\mathtt{m}}\setminus\overline{\mathbb{S}}_{0,k})^2.$ There exists $c_2>0$ such that if $\min(|j|,|j_0|)\geqslant c_2\gamma_{n+1}^{-\upsilon}\langle l\rangle^{\tau_1},$ then
			$$\mathcal{R}_{l,j,j_0}^{k}(i_n)\subset\mathcal{R}_{l,j-j_0}^{(0,k)}(i_n).$$
		\end{enumerate}
	\end{lem}
	\begin{proof}
		\textbf{(i)} Observe that the case $j=0$ is trivial. Now, for $j\neq0$ we assume that $\mathcal{R}_{l,j}^{(0,\epsilon)}(i_{n})\neq\varnothing.$ Then, there exists $b\in(b_{*},b^{*})$ such that
		$$|\omega(b,\varepsilon)\cdot l+j \mathtt{c}_{k,n}(b,\varepsilon)|\leqslant\tfrac{4\gamma_{n+1}^{\upsilon}|j|}{\langle l\rangle^{\tau_1}}\cdot$$ 
		From triangle and Cauchy-Schwarz inequalities, \eqref{in gamn}, \eqref{rigidity gam-N0} and the fact that $(b,\varepsilon)\mapsto \omega(b, \varepsilon)$ is bounded, we deduce 
		\begin{align}\label{maj cknj}
			|\mathtt{c}_{k,n}(b,\varepsilon)||j|&\leqslant 4|j|\gamma_{n+1}^{\upsilon}\langle l\rangle^{-\tau_1}+|{\omega}(b,\varepsilon)\cdot l|\nonumber\\
			&\leqslant 4|j|\gamma_{n+1}^{\upsilon}+C\langle l\rangle\nonumber\\
			&\leqslant 8\varepsilon^{a\upsilon}|j|+C\langle l\rangle.
		\end{align}
		Now by construction \eqref{mu0 r0}, we can write
		$$\mathtt{c}_{k,n}(b,\varepsilon)=\mathtt{v}_k(b)+r_{k,n}^{(0)}(b,\varepsilon).$$
		Remark that by definition \eqref{def V10 V20}, we have
		$$\inf_{k\in\{1,2\}}\inf_{b\in(b_{*},b^{*})}\big|\mathtt{v}_{k}(b)\big|\geqslant\Omega.$$
		Together with \eqref{sml-r0} and Proposition \ref{Nash-Moser}-$(\mathcal{P}1)_{n}$, this implies
		\begin{align}\label{e-uni r0}
			\forall q'\in\llbracket 0,q\rrbracket,\quad\sup_{n\in\mathbb{N}}\sup_{b\in(b_{*},b^{*})}\big|\partial_{b}^{q'}r_{k,n}^{(0)}(b,\varepsilon)\big|&\leqslant\gamma^{-q'}\sup_{n\in\mathbb{N}}\| r_{k,n}^{(0)}\|^{q,\gamma}\nonumber\\
			&\lesssim\varepsilon\gamma^{-q'}\nonumber\\
			&\lesssim\varepsilon^{1-aq'}.
		\end{align}
		Hence, choosing $\varepsilon$ small enough implies
		\begin{equation}\label{low bnd ckn}
			\inf_{k\in\{1,2\}}\inf_{n\in\mathbb{N}}\inf_{b\in(b_{*},b^{*})}|\mathtt{c}_{k,n}(b,\varepsilon)|\geqslant \tfrac{\Omega}{2}\cdot
		\end{equation}
		Inserting \eqref{low bnd ckn} into \eqref{maj cknj} yields
		$$\big(\tfrac{\Omega}{2}-8\varepsilon^{a\upsilon}\big)|j|\leqslant C\langle l\rangle.$$
		Thus, selecting $\varepsilon$ small enough ensures that  $|j|\leqslant C_{0}\langle l\rangle$ for some $C_{0}>0.$\\
		
		\noindent \textbf{(ii)} The case $j=0$ is obvious so we may treat the case where $j\neq 0.$ Assume that  $\mathcal{R}_{l,j}^{(1,k)}(i_{n})\neq\varnothing.$ Then, we can find $b\in(b_{*},b^{*})$ such that 
		$$\big|\omega(b,\varepsilon)\cdot l+\mu_{j,k,n}^{(\infty)}(b,\varepsilon)\big|\leqslant\tfrac{\gamma_{n+1}|j|}{\langle l\rangle^{\tau_1}}\cdot$$
		Applying the triangle and Cauchy-Schwarz inequalities together with \eqref{in gamn} and \eqref{rigidity gam-N0} yields
		\begin{align}\label{maj mujknfty}
			\big|\mu_{j,k,n}^{(\infty)}(b,\varepsilon)\big|&\leqslant \gamma_{n+1}|j|\langle l\rangle^{-\tau_1}+|{\omega}(b,\varepsilon)\cdot l|\nonumber\\
			&\leqslant  2\varepsilon^a|j|+C\langle l\rangle.
		\end{align}
		Now coming back to the structure of the eigenvalues in \eqref{def mujknfty}, then using the triangle inequality, we infer
		\begin{align}\label{low mjknfty}
			\big|\mu_{j,k,n}^{(\infty)}(b,\varepsilon)\big|&\geqslant |\Omega_{j,k}(b)|-|j|\big|r_{k,n}^{(0)}(b,\varepsilon)\big|-\big|r_{j,k,n}^{(\infty)}(b,\varepsilon)\big|.
		\end{align}
		Notice that \eqref{e-rjfty} implies
		\begin{align}\label{e-uni rjknfty}
			\forall q'\in\llbracket 0,q\rrbracket,\quad\sup_{n\in\mathbb{N}}\sup_{b\in(b_{*},b^{*})}\sup_{j\in\overline{\mathbb{S}}_{0,k}}\big|\partial_{b}^{q'}r_{j,k,n}^{(\infty)}(b,\varepsilon)\big|&\leqslant\gamma^{-q'}\sup_{n\in\mathbb{N}}\sup_{j\in\overline{\mathbb{S}}_{0,k}}\| r_{j,k,n}^{(\infty)}\|^{q,\gamma}\nonumber\\
			&\lesssim\varepsilon\gamma^{-q'-1}\nonumber\\
			&\lesssim\varepsilon^{1-a(q'+1)}.
		\end{align}
		Gathering \eqref{low mjknfty}, Lemma \ref{lem-asym}-3, \eqref{e-uni r0} and \eqref{e-uni rjknfty}, we obtain
		\begin{equation}\label{low mjknfty2}
			\big|\mu_{j,k,n}^{(\infty)}(b,\varepsilon)\big|\geqslant \Omega|j|-C\varepsilon^{1-a}|j|.
		\end{equation}
		Inserting \eqref{low mjknfty2} into \eqref{maj mujknfty} yields
		\begin{align*}
			\big( \Omega-C\varepsilon^{1-a}-2\varepsilon^a\big)|j|
			&\leqslant  C\langle l\rangle.
		\end{align*}
		Hence, taking $\varepsilon$ small enough we obtain $|j|\leqslant C_0\langle  l\rangle,$ for some $C_{0}>0.$\\
		\noindent\textbf{(iii)} Notice that for $j=j_0$ we have $\mathcal{R}_{l,j_0,j_{0}}^{k}(i_{n})=\mathcal{R}^{(0,k)}_{l,0}(i_{n})$. Hence this situation has already been studied in the first point. Therefore, we shall consider $j\neq j_0.$ Let us assume that $\mathcal{R}_{l,j,j_{0}}(i_{n})\neq\varnothing.$ We can find $b\in(b_{*},b^{*})$ such that 
		$$\big|\omega(b,\varepsilon)\cdot l+\mu_{j,k,n}^{(\infty)}(b,\varepsilon)-\mu_{j_0,k,n}^{(\infty)}(b,\varepsilon)\big|\leqslant\tfrac{2\gamma_{n+1}|j-j_0|}{\langle l\rangle^{\tau_2}}\cdot$$
		Using one more time the triangle and Cauchy-Schwarz inequalities together with \eqref{in gamn} and \eqref{rigidity gam-N0}, we infer
		\begin{align*}
			\big|\mu_{j,k,n}^{(\infty)}(b,\varepsilon)-\mu_{j_{0},k,n}^{(\infty)}(b,\varepsilon)\big|&\leqslant 2\gamma_{n+1}|j-j_{0}|\langle l\rangle^{-\tau_{2}}+|{\omega}(b,\varepsilon)\cdot l|\\
			&\leqslant 2\gamma_{n+1}|j-j_{0}|+C\langle l\rangle\\
			&\leqslant 4\varepsilon^a|j-j_{0}|+C\langle l\rangle.
		\end{align*}
		On the other hand, the triangle inequality, Lemma \ref{lem-asym}-5, \eqref{e-uni r0} and \eqref{e-uni rjknfty} give for $\varepsilon$ small enough
		\begin{align*}
			\big|\mu_{j,k,n}^{(\infty)}(b,\varepsilon)-\mu_{j_{0},k,n}^{(\infty)}(b,\varepsilon)| & \geqslant  |\Omega_{j,k}(b)-\Omega_{j_{0},k}(b)|-\big|r_{k,n}^{(0)}(b,\varepsilon)\big||j-j_{0}|-\big|r_{j,k,n}^{(\infty)}(b,\varepsilon)\big|-\big|r_{j_{0},k,n}^{(\infty)}(b,\varepsilon)\big|\\
			& \geqslant  \big(c-C\varepsilon^{1-a}\big)|j-j_{0}|\\
			& \geqslant \tfrac{c}{2}|j-j_{0}|.
		\end{align*}
		Putting together the foregoing inequalities yields
		$$\big(\tfrac{c}{2}-4\varepsilon^{a}\big)|j-j_0|\leqslant C\langle l\rangle.$$
		Thus, for $\varepsilon$ sufficiently small, we get $|j-j_{0}|\leqslant C_{0}\langle l\rangle,$ for some $C_{0}>0.$\\
		\noindent \textbf{(iv)} Observe that the case $j=j_0$ is trivial, so we may restrict our discussion to the case $j\neq j_0.$ Now, according to the symmetry property $\mu_{-j,k,n}^{(\infty)}(b,\varepsilon)=-\mu_{j,k,n}^{(\infty)}(b,\varepsilon)$ we can assume, without loss of generality that $0<j<j_0.$ Take $b\in\mathcal{R}_{l,j,j_{0}}^{k}(i_{n}).$ Then by definition, we have
		\begin{equation}\label{complmt Russ2}
			\big|{\omega}(b,\varepsilon)\cdot l+\mu_{j,k,n}^{(\infty)}(b,\varepsilon)\pm\mu_{j_{0},k,n}^{(\infty)}(b,\varepsilon)\big|\leqslant\tfrac{2\gamma_{n+1}\langle j\pm j_{0}\rangle}{\langle l\rangle^{\tau_{2}}}\cdot
		\end{equation}
		Recall from \eqref{def V10 V20} and \eqref{ASYFR1+} the decompositions for $j>0,$
		\begin{align*}
			\mathtt{v}_{k}(b)&=\Omega+(2-k)\frac{1-b^2}{2},\\
			\Omega_{j,k}(b)&=j\mathtt{v}_{k}(b)+\frac{(-1)^{k}}{2}+(-1)^{k+1}\mathtt{r}_{j}(b).
		\end{align*}
		Therefore, by the triangle inequality, we get
		\begin{align}
			\big|\omega(b,\varepsilon)\cdot l+(j\pm j_{0})\mathtt{c}_{k,n}(b,\varepsilon)\big|  &\leqslant  \big|\omega(b,\varepsilon)\cdot l+\mu_{j,k,n}^{(\infty)}(b,\varepsilon)\pm\mu_{j_{0},k,n}^{(\infty)}(b,\varepsilon)\big|+\tfrac{1}{2}(1\pm 1)\nonumber\\
			&\quad+\big|\mathtt{r}_{j}(b)\pm\mathtt{r}_{j_0}(b)\big|+\big|r_{j,k,n}^{(\infty)}(b,\varepsilon)\pm r_{j_{0},k,n}^{(\infty)}(b,\varepsilon)\big|.\label{link R0K and Rkk}
		\end{align}
		First, it is obvious that
		\begin{equation}\label{trive 1pm1}
			1\pm 1\leqslant \tfrac{\langle j\pm j_0\rangle}{j}\cdot
		\end{equation}
		Second, the estimate \eqref{ASYFR1-} implies in particular
		\begin{align}\label{e-ttrj pm ttrj0}
			\big|\mathtt{r}_{j}(b)\pm\mathtt{r}_{j_0}(b)\big|&\leqslant C\big(|j|^{-1}+|j_0|^{-1}\big)\nonumber\\
			&\leqslant C\tfrac{\langle j\pm j_0\rangle}{j}\cdot
		\end{align}
		Third, the estimate \eqref{e-rjfty} together with \eqref{rigidity gam-N0} give
		\begin{align}\label{e-rjfty pm rj0fty}
			\big|r_{j,k,n}^{(\infty)}(b,\varepsilon)\pm r_{j_{0},k,n}^{(\infty)}(b,\varepsilon)\big|  \leqslant &C \varepsilon^{1-a}\big(|j|^{-1}+|j_0|^{-1}\big)\nonumber\\
			\leqslant & C \tfrac{\langle j\pm j_{0}\rangle}{j}\cdot
		\end{align}
		Gathering \eqref{complmt Russ2}, \eqref{link R0K and Rkk}, \eqref{trive 1pm1}, \eqref{e-ttrj pm ttrj0} and \eqref{e-rjfty pm rj0fty} yields
		\begin{align*}
			\nonumber \big|{\omega}(b,\varepsilon)\cdot l+(j\pm j_{0})\mathtt{c}_{k,n}(b,\varepsilon)\big|  \leqslant & \tfrac{2\gamma_{n+1}\langle j\pm j_{0}\rangle}{\langle l\rangle^{\tau_{2}}}+ C \tfrac{\langle j\pm j_{0}\rangle}{j}\cdot
		\end{align*}
		Hence, using the assumptions $\displaystyle j\geqslant \tfrac{C}{2} \gamma_{n+1}^{-\upsilon}\langle l\rangle^{\tau_{1}}$  and $\tau_{2}>\tau_{1}$, we infer
		$$\big|{\omega}(b,\varepsilon)\cdot l+(j\pm j_{0})\mathtt{c}_{k,n}^{}(b,\varepsilon)\big| \leqslant  \tfrac{4\gamma_{n+1}^{\upsilon}\langle j\pm j_{0}\rangle}{\langle l\rangle^{\tau_{1}}}\cdot$$
		This gives the desired result and ends the proof of Lemma \ref{some cantor set are empty}, up to defining $c_{2}\triangleq \frac{C}{2}\cdot$
	\end{proof}
	The next and last lemma is concerned by the transversality property of the perturbed frequency vector $\omega(b,\varepsilon)$ given by Corollary \ref{Corollary NM}. 
	\begin{lem}\label{lem Ru-pert}
		Let $q_{0}$, $C_{0}$ and $\rho_{0}$ as in Lemma $\ref{lemma transversalityE}$. There exist $\varepsilon_{0}>0$ small enough such that for any   $\varepsilon\in[0,\varepsilon_{0}]$ the following assertions hold true.
		\begin{enumerate}[label=(\roman*)]
			\item For all $l\in\mathbb{Z}^{d}\setminus\{0\}, $ we have 
			$$\inf_{b\in[b_*,b^*]}\max_{q'\in\llbracket 0,q_{0}\rrbracket}\big|\partial_{b}^{q'}\left(\omega(b,\varepsilon)\cdot l\right)\big|\geqslant\tfrac{\rho_{0}\langle l\rangle}{2}\cdot$$
			\item For all $(l,j)\in\mathbb{Z}^{d}\times\mathbb{Z}_{\mathtt{m}}\setminus\{(0,0)\}$ such that $|j|\leqslant C_{0}\langle l\rangle,$ we have
			$$\forall n\in\mathbb{N},\quad\inf_{b\in[b_*,b^*]}\max_{q'\in\llbracket 0,q_{0}\rrbracket}\left|\partial_{b}^{q'}\big(\omega(b,\varepsilon)\cdot l+jc_{k,n}(b,\varepsilon)\big)\right|\geqslant\tfrac{\rho_{0}\langle l\rangle}{2}\cdot$$
			\item For all $(l,j)\in\mathbb{Z}^{d}\times(\mathbb{Z}_{\mathtt{m}}\setminus\overline{\mathbb{S}}_{0,k})$ such that $|j|\leqslant C_{0}\langle l\rangle,$ we have
			$$\forall n\in\mathbb{N},\quad\inf_{b\in[b_*,b^*]}\max_{q'\in\llbracket 0,q_{0}\rrbracket}\left|\partial_{b}^{q'}\left(\omega(b,\varepsilon)\cdot l+\mu_{j,k,n}^{(\infty)}(b,\varepsilon)\right)\right|\geqslant\tfrac{\rho_{0}\langle l\rangle}{2}\cdot$$
			\item For all $(l,j,j_{0})\in\mathbb{Z}^{d}\times(\mathbb{Z}_{\mathtt{m}}\setminus\overline{\mathbb{S}}_{0,k})^{2}$ such that $|j-j_0|\leqslant C_0\langle l\rangle$, we have
			$$\forall n\in\mathbb{N},\quad\inf_{b\in[b_*,b^*]}\max_{q'\in\llbracket 0,q_{0}\rrbracket}\left|\partial_{b}^{q'}\left(\omega(b,\varepsilon)\cdot l+\mu_{j,k,n}^{(\infty)}(b,\varepsilon)-\mu_{j_{0},k,n}^{(\infty)}(b,\varepsilon)\right)\right|\geqslant\tfrac{\rho_{0}\langle l\rangle}{2}\cdot$$
			\item For all $(l,j,j_{0})\in\mathbb{Z}^{d}\times(\mathbb{Z}_{\mathtt{m}}\setminus\overline{\mathbb{S}}_{0,1})\times(\mathbb{Z}_{\mathtt{m}}\setminus\overline{\mathbb{S}}_{0,2})$, we have
			$$\forall n\in\mathbb{N},\quad\inf_{b\in[b_*,b^*]}\max_{q'\in\llbracket 0,q_{0}\rrbracket}\left|\partial_{b}^{q'}\left(\omega(b,\varepsilon)\cdot l+\mu_{j,1,n}^{(\infty)}(b,\varepsilon)-\mu_{j_{0},2,n}^{(\infty)}(b,\varepsilon)\right)\right|\geqslant\tfrac{\rho_{0}\langle l,j,j_0\rangle}{2}\cdot$$
		\end{enumerate}
	\end{lem}
	\begin{proof} The points (i), (ii), (iii) and (iv) are obtained following closely \cite[Lem. 7.3]{HR21}. The estimates are obtained by a perturbative argument from the equilibrium transversality conditions proved in Lemma \ref{lemma transversalityE}-1-2-3-4. Therefore, it remains to prove the point (v).\\
		\textbf{(v)} Using the decompositions \eqref{mu0 r0}-\eqref{def mu lim}-\eqref{pert freq} together with \eqref{e-uni r0}, \eqref{e-uni rjknfty} and Lemma \ref{lemma transversalityE}-5, we get for $\varepsilon$ sufficiently small
		\begin{align*}
			\max_{q'\in\llbracket 0,q_{0}\rrbracket}\left|\partial_{b}^{q'}\Big(\omega(b,\varepsilon)\cdot l\right.&\left.\left.+\mu_{j,1,n}^{(\infty)}(b,\varepsilon)-\mu_{j_{0},2,n}^{(\infty)}(b,\varepsilon)\right)\right|
			\geqslant\max_{q'\in\llbracket 0,q_{0}\rrbracket}\left|\partial_{b}^{q'}\Big(\omega_{\textnormal{Eq}}(b)\cdot l+\Omega_{j,1}(b)-\Omega_{j_{0},2}(b)\Big)\right|\\
			&-\max_{q'\in\llbracket 0,q\rrbracket}\left|\partial_{b}^{q'}\left(\overline{\mathrm{r}}_{\varepsilon}(b)\cdot l+jr_{1,n}^{(0)}(b,\varepsilon)-j_{0}r_{2,n}^{(0)}(b,\varepsilon)+r_{j,1,n}^{(\infty)}(b,\varepsilon)-r_{j_{0},2,n}^{(\infty)}(b,\varepsilon)\right)\right|\\
			&\geqslant\rho_{0}\langle l,j,j_0\rangle-C\varepsilon^{1-a(1+q+q\overline{a})}\langle l,j,j_0\rangle-C\varepsilon^{1-a(1+q)}\langle l,j,j_0\rangle\\
			&\geqslant\tfrac{\rho_{0}\langle l,j,j_0\rangle}{2}\cdot
		\end{align*}
		This ends the proof of Lemma \ref{lem Ru-pert}.
	\end{proof}

	\noindent NYUAD Research Institute, New York University Abu Dhabi, PO Box 129188, Abu Dhabi, United Arab Emirates.\\
	Email address: zh14@nyu.edu.\\
	
	\noindent NYUAD Research Institute, New York University Abu Dhabi, PO Box 129188, Abu Dhabi, United Arab Emirates. Univ Rennes, CNRS, IRMAR – UMR 6625, F-35000 Rennes, France.\\
	Email address: thmidi@univ-rennes1.fr.\\
	
	\noindent SISSA International School for Advanced Studies, Via Bonomea 265, 34136, Trieste, Italy.\\
	E-mail address : eroulley@sissa.it.

\begin{thebibliography}{9999}
		\bibitem{AB11} T. Alazard, P. Baldi, {\it Gravity capillary standing water waves}, Archive for Rational Mechanics and Analysis 217 (2015), no. 3, 741--830.
		
		\bibitem{ADPMW20} W. Ao, J. Davila, M. Del Pino, M. Musso, J. Wei, \textit{Traveling and rotating solutions to the generalized inviscid surface quasi-geostrophic equation,} Transactions of the American Mathematical Society 374 (2021), no. 9, 6665--6689.
		
		
		\bibitem{A63} V. I. Arnold, {\it Small denominators and problems of stability of motion in classical mechanics and celestial mechanics,} Uspekhi Matematicheskikh Nauk 18 (1963), 91--192.
		
		\bibitem{BBMH18} P. Baldi, M. Berti, E. Haus, R. Montalto, {\it Time quasi-periodic gravity water waves in finite depth,} Inventiones Mathematicae 214 (2018), no. 2, 739--911.
		
		\bibitem{BBM14} P. Baldi, M. Berti, R. Montalto, {\it KAM for quasi-linear and fully nonlinear forced perturbations of Airy equation,} Mathematische Annalen 359 (2014), no. 1-2, 471--536.
		
		\bibitem{BBM16} P. Baldi, M. Berti, R. Montalto, {\it KAM for autonomous quasi-linear perturbations of KdV,} 
		Annales de l'Institut Henri Poincaré Analyse Non-Linéaire 33 (2016), no. 6, 1589--1638. 
		
		\bibitem{BBM16-1} P. Baldi, M. Berti, R. Montalto, {\it KAM for autonomous quasi-linear perturbations of mKdV,} Bolletino dell'Unione Matematica Italiana 9 (2016), 143--188.
		
		\bibitem{BM21} P. Baldi, R. Montalto, \textit{Quasi-periodic incompressible Euler flows in 3D,} Advances in Mathematics 384 (2021), 107730.
		
		\bibitem{BBM11} D. Bambusi, M. Berti, E. Magistrelli, {\it Degenerate KAM theory for partial differential equations}, Journal of Differential Equations, 250 (2011), no. 8, 3379--3397.
		
		\bibitem{B19} M. Berti, \textit{KAM theory for partial differential equations}, Analysis in Theory and Applications, 35 (2019), no. 3, 235--267.
		
		\bibitem{BBP13} M. Berti, L. Biasco, M. Procesi, {\it KAM theory for the Hamiltonian derivative wave equation,} Annales Scientifiques de l'École Normale Supérieure Volume 4, 46 (2013), no. 2,  301--373.
		
		\bibitem{BBP14} M. Berti, L. Biasco, M. Procesi, {\it KAM for Reversible Derivative Wave Equations,} Archive for Rational Mechanics and Analysis 212 (2014), no. 3, 905--955.
		
		\bibitem{BB12} M. Berti, P. Bolle, \textit{Sobolev quasi periodic solutions of multidimensional wave equations with a multiplicative potential,} Nonlinearity 25 (2012), no. 9, 2579--2613.
		
		\bibitem{BB15} M. Berti, P. Bolle, {\it A Nash-Moser approach to KAM theory}, Fields Institute Communications, special volume ``Hamiltonian PDEs and Applications'', (2015), 255--284.
		
		\bibitem{BFM21} M. Berti, L. Franzoi, A. Maspero, \textit{Traveling quasi-periodic water waves with constant vorticity}, Archive for Rational Mechanics and Analysis, 240 (2021), 99--202.
		
		\bibitem{BFM21-1} M. Berti, L. Franzoi, A. Maspero, \textit{Pure gravity traveling quasi-periodic water waves with constant vorticity}, arXiv:2101.12006.
		
		\bibitem{BHM21} M. Berti, Z. Hassainia, N. Masmoudi, {\it Time quasi-periodic vortex patches}, arXiv:2202.06215.
		
		\bibitem{BKM21} M. Berti, T. Kappeler, R. Montalto, \emph{Large KAM tori for quasi-linear perturbations of KdV,} Archive for Rational Mechanics and Analysis, 239 (2021), 1395--1500.
		
		\bibitem{BM18} M. Berti, R. Montalto, {\it Quasi-periodic standing wave solutions of gravity-capillary water waves}, MEMO, Volume 263, 1273, Memoirs of the American Mathematical Society, ISSN 0065-9266, (2020).
		
		\bibitem{BC93} A. L. Bertozzi, P. Constantin, {\it  Global regularity for vortex patches,} Communications in Mathematical Physics, 152 (1993), no. 1, 9--28.
				
		\bibitem{B05}  J. Bourgain,  {\it Green's function estimates for lattice Schr\"odinger operators and applications,} Annals of Mathematics Studies 158 (2005), Princeton University Press, Princeton.
		
		\bibitem{B94}  J.  Bourgain, {\it Construction of quasi-periodic solutions for Hamiltonian perturbations of linear equations and applications to nonlinear PDE,} International Mathematics Research Notices (1994), no. 11, 21 pages.
		
		\bibitem{B82} J. Burbea, {\it Motions of vortex patches,} Letters in Mathematical Physics 6 (1982), no. 1,  1--16.
		
		\bibitem{CLZ21} D. Cao, S. Lai, W. Zhan, \textit{Traveling vortex pairs for 2D incompressible Euler equations,} Calculus of Variations and Partial Differential Equations 60 (2021), no. 190.
		
		\bibitem{CQZZ21} D. Cao, G. Qin, W. Zhan, C. Zou, \textit{Existence and regularity of co-rotating and traveling-wave vortex solutions for the generalized SQG equation,} Journal of Differential Equations 299 (2021), 429--462.
		
		\bibitem{CWWZ21} D. Cao, J. Wan, G. Wang, W. Zhan, \textit{Rotating vortex patches for the planar Euler equations in a disk,} Journal of Differential Equations 275 (2021), 509--532.
		
		\bibitem{CCG16} A. Castro, D. C\'ordoba, J. G\'omez-Serrano, \textit{Uniformly rotating analytic global patch solutions for active scalars,} Annals of PDE, 2 (2016), no. 1, 1--34.
		
		\bibitem{C-C-G16}  A. Castro, D. C\'ordoba, J. G\'omez-Serrano, {\it Existence and regularity of rotating global solutions for the generalized surface quasi-geostrophic equations,} Duke Mathematical Journal 165 (2016), no. 5, 935--984.
		
		\bibitem{C93} J. Y. Chemin, \textit{Persistance de structures géométriques dans les fluides incompressibles bidimensionnels,} Annales de l'Ecole Normale Supérieure 26 (1993), 517--542.
		
		\bibitem{C95} J. Y. Chemin, {\it Fluides parfaits incompressibles,} Ast\'{e}risque 230, Soci\'{e}t\'{e} Math\'{e}matique de France, (1995).
		
		\bibitem{CY00} L. Chierchia, J. You, {\it KAM tori for 1D nonlinear wave equations with periodic boundary conditions,} Communications in Mathematical Physics 211 (2000), 497--525.
		
		\bibitem{CJ22} K. Choi, I. Jeong, \textit{Stability and instability of Kelvin waves,}
Calculus of Variations and Partial Differential Equations 61 (2022), no. 6.
		
		\bibitem{CW93} W. Craig, C. E. Wayne, {\it Newton’s method and periodic solutions of nonlinear wave equation,} Communications on Pure and Applied Mathematics 46 (1993), 1409--1498.
	
		\bibitem{CF13} N. Crouseilles, E. Faou, \textit{Quasi-periodic solutions of the 2D Euler equations,} Asymptotic analysis 81 (2013), no.1, 31--34.
		
		\bibitem{EJ19} T. Elgindi, I. Jeong, \textit{On singular vortex patches, I: Well-posedness issues,}  Memoirs of the American Mathematical Society 283 (2023), no. 1400, 1--102.
		
		\bibitem{EJ20} T. Elgindi, I. Jeong,  \textit{On singular vortex patches, II: long-time dynamics,} Transactions of the American Mathematical Society 373 (2020), no. 9, 6757--6775.
		
		\bibitem{HHH16} F. de la Hoz, Z. Hassainia, T. Hmidi, \textit{doubly connected V-states for the generalized surface quasi-geostrophic equations}, Archive for Rational Mechanics and Analysis, 220 (2016) 1209--1281.
		
		\bibitem{HHHM15} F. de la Hoz, Z. Hassainia, T. Hmidi, J. Mateu, \textit{An analytical and numerical study of steady patches in the disc}, Analysis and PDE, 9 (2015), no. 10.
		
		\bibitem{LS19} R. de la Llave, Y. Sire, \textit{An a posteriori KAM theorem for whiskered tori in Hamiltonian partial differential equations with applications to some ill-posed equations,} Archive for Rational Mechanics and Analysis 231 (2019), no. 2, 971--1044. 
		
		\bibitem{DZ78} G. S. Deem, N. J. Zabusky, {\it Vortex waves : Stationary "V-states", Interactions, Recurrence, and Breaking,} Physical Review Letters  40  (1978), no. 13, 859--862.
		
		\bibitem{DHR19} D. G. Dritschel, T. Hmidi, C. Renault, \textit{Imperfect bifurcation for the shallow-water quasi-geostrophic equations}, Archive for Rational Mechanics and Analysis 231 (2019), no. 3, 1853--1915.
		
		\bibitem{EGK16}  L. H. Eliasson, B. Grébert, S. Kuksin, \textit{KAM for the nonlinear beam equation,} Geometric and Functional Analysis 26 (2016), no. 6, 1588--1715.
		
		\bibitem{EK10} L. H. Eliasson, S. B. Kuksin, {\it  KAM for the nonlinear Schr\"odinger equation,} Annals of Mathematics 172 (2010), 371--435.
		
		\bibitem{EPT22} A. Enciso, D. Peralta-Salas, F. Torres de Lizaur, \textit{Quasi-periodic solutions to the incompressible Euler equations in dimensions two and higher,} arXiv:2209.09812.
		
		\bibitem{FGMP19} R. Feola, F. Giuliani, R. Montalto, M. Procesi, \textit{Reducibility of first order linear operators on tori via Moser's theorem}, Journal of Functional Analysis 276 (2019), no. 3, 932--970.
		
		\bibitem{FP14} R. Feola, M. Procesi, \textit{Quasi-periodic solutions for fully nonlinear forced reversible Schrödinger equations,} Journal of Differential Equations 259 (2015), no. 7, 3389--3447.
		
		\bibitem{F00} E. Fraenkel, \textit{An introduction to maximum principles and symmetry in elliptic problems,} Cambridge University Press 128 (2000).
		
		\bibitem{GYZ11} M. Gao, X. Yuan, J. Zhang, {\it KAM tori for reversible partial differential equations,} Nonlinearity 24 (2011), 1189--1228.
		
		\bibitem{G20} C. Garc\'{i}a, \textit{K\'{a}rm\'{a}n vortex street in incompressible fluid models}, Nonlinearity, 33 (2020), no. 4, 1625--1676.
		
		\bibitem{G21} C. Garc\'{i}a, \textit{Vortex patches choreography for active scalar equations}, Journal of Nonlinear Science, 31 (2021), no. 75, 1432--1467.
		
		\bibitem{GH22} C. Garc\'ia, S. V. Haziot, \textit{Global bifurcation for corotating and counter-rotating vortex pairs,} arXiv:2204.11327.
		
		\bibitem{GHS20} C. Garc\`ia, T. Hmidi, J. Soler, {\it Non uniform rotating vortices and periodic orbits for the two-dimensional Euler Equations.} Archive for Rational Mechanics and Analysis 238 (2020), 929--1085.
		
		\bibitem{GPSY20} J. G\'omez-Serrano, J. Park, J. Shi, Y. Yao, {\it Symmetry in stationary and uniformly-rotating solutions of active scalar equations,} Duke Mathematical Journal 170 (2021), no. 13, 2957--3038.
		
		\bibitem{GGS20} L. Godard-Cadillac, P. Gravejat, D. Smets, \textit{Co-rotating vortices with N fold symmetry for the inviscid surface quasi-geostrophic equation,} arXiv:2010.08194.
		
		\bibitem{G74} S. M. Graff, \textit{On the continuation of hyperbolic invariant tori for Hamiltonian systems,} Journal of Differential Equations 15 (1974), 1--69.
		
		\bibitem{GS19} P. Gravejat, D. Smets, \textit{Smooth travelling-wave solutions to the inviscid surface quasi-geostrophic equation,} International Mathematics Research Notices (2019), no. 6, 1744--1757.
		
		\bibitem{HH15}  Z. Hassainia, T. Hmidi, {\it On the V-States for the generalized quasi-geostrophic equations,} Communications in Mathematical Physics 337 (2015), no. 1, 321--377.
		
		\bibitem{HH21} Z. Hassainia, T. Hmidi, \textit{Steady asymmetric vortex pairs for Euler equations}, American Institut of Mathematical Science, 41 (2021), no. 4, 1939--1969.
		
		\bibitem{HHM21} Z. Hassainia, T. Hmidi, N. Masmoudi, \textit{KAM theory for active scalar equations,} arXiv:2110.08615.
		
		\bibitem{HMW20} Z. Hassainia, N. Masmoudi, M. H. Wheeler, {\it Global bifurcation of rotating vortex patches},  Communications on Pure and Applied Mathematics LXXIII (2020), 1933--1980.
		
		\bibitem{HR21-1} Z. Hassainia, E. Roulley, \textit{Boundary effects on the existence of quasi-periodic solutions for Euler equations,} arXiv:2202.10053.
		
		\bibitem{HW21} Z. Hassainia, M. Wheeler, \textit{Multipole vortex patch equilibria for active scalar equations,} SIAM Journal on Mathematical Analysis, 54 (2022), no. 6, 6054--6095.
		
		\bibitem{H15}  T. Hmidi, {\it On the trivial solutions for the rotating patch model,} Journal of Evolution Equations 15 (2015), no. 4, 801--816.
		
		\bibitem{HHMV16} T. Hmidi, F. de la Hoz, J.  Mateu, J.  Verdera, {\it  doubly connected V-states for the planar Euler equations,} SIAM Journal on Mathematical Analysis 48 (2016), no. 3, 1892--1928. 
		
		\bibitem{HM16} T. Hmidi, J. Mateu, {\it Bifurcation of rotating patches from Kirchhoff vortices,} Discrete and Continuous Dynamical Systems 36 (2016), no. 10, 5401--5422. 
		
		\bibitem{HM16-2} T. Hmidi, J. Mateu, {\it Degenerate bifurcation of the rotating patches,} Advances in Mathematics 302 (2016), 799--850. 
		
		\bibitem{HM17} T. Hmidi,  J.  Mateu, {\it  Existence of corotating and counter-rotating vortex pairs for active scalar equations,} Communications in Mathematical Physics 350 (2017), no. 2, 699--747.   
		
		\bibitem{HMV13} T. Hmidi, J.  Mateu, J.  Verdera, {\it Boundary Regularity of Rotating Vortex Patches,} Archive for Rational Mechanics and Analysis 209 (2013), no. 1, 171--208.
				
		\bibitem{HR21} T. Hmidi, E. Roulley, \textit{Time quasi-periodic vortex patches for quasi-geostrophic shallow-water equations}, arXiv:2110.13751.
		
		\bibitem{HXX22} T. Hmidi, L. Xue, Z. Xue, \textit{Emergence of time periodic solutions for the generalized surface quasi-geostrophic equation in the disc,} arXiv:2210.08760.
		
		 \bibitem{ILN03} D. Iftimie, M. C. Lopes Filho M. C. and H. J. Nussenzveig Lopes, \textit{ On the large-time behavior of two-dimensional vortex dynamics,} Physica D: Nonlinear Phenomena 179 (2003) 153--160
		
                 \bibitem{ISG99} D. Iftimie, T. C. Sideris and P. Gamblin,  \textit{On the evolution of compactly supported planar vorticity,} Communications in Partial Differential Equations 24 (1999), no. 9--10, 1709--1730.
		
		\bibitem{IPT05} G. Iooss, P. Plotnikov, J. Toland, \textit{Standing waves on an infinitely deep perfect fluid under gravity}, Archive for Rational Mechanics and Analysis 177 (2005), no. 3, 367--478.
		
		\bibitem{KP03} T. Kappeler, J. P\"oschel, {\it KdV and KAM,} Springer, Berlin, (2003). 
		
		\bibitem{K74} G. Kirchhoff, \textit{Vorlesungen uber mathematische Physik}, Leipzig, (1874).
		
		\bibitem{K87} S. Kuksin, {\it Hamiltonian perturbations of in finite-dimensional linear systems with imaginary spectrum,} Funktsional'nyi Analiz i ego Prilozheniya 21 (1987), no. 3, 22--37, 95.
		
		\bibitem{K98} S. Kuksin, {\it A KAM theorem for equations of the Korteweg-de Vries type,} Reviews in Mathematical Physics 10 (1998), no. 3, 1--64.
		
		\bibitem{K00} S. Kuksin, {\it Analysis of Hamiltonian PDEs,} Oxford Lecture Series in Mathematics and its Applications, volume 19, Oxford University Press, Oxford, (2000).
		
		\bibitem{K54} A. N. Kolmogorov, {\it On the persistence of conditionally periodic motions under a small change of the hamiltonian function,} Doklady Akademii Nauk SSSR 98 (1954), 527--530.
		
		\bibitem{LY11} J. Liu, X. Yuan, {\it A KAM theorem for Hamiltonian partial differential equations with unbounded perturbations,} Communications in Mathematical Physics 307 (2011), 629--673.
		
		\bibitem{M94} C. Marchioro, \textit{Bounds on the growth of the support of a vortex patch,} Communications in  Mathematical Physics 164 (1994), no. 3, 507--524.
		
		
		\bibitem{M62} J. Moser, \textit{On invariant curves of area-preserving mappings of an annulus}, Nachrichten von der Gesellschaft der Wissenschaften zu Göttingen, Mathematisch-Physikalische Klasse (1962), 1--20. 
		
		\bibitem{M67} J. Moser, \textit{Convergent series expansions for quasi-periodic motions,} Mathematische Annalen 169 (1967), 136--176.
		
		\bibitem{N54} J. Nash, \textit{C1-isometric imbeddings}, Annals of Mathematics 60 (1954), no. 3, 383--396.
		
		\bibitem{PT01} P. Plotnikov, J. Toland, \textit{Nash-Moser theory for standing water waves}, Archive for Rational Mechanics and Analysis 159 (2001), no. 1, 1--83.
		
		\bibitem{P89} J. Pöschel, \textit{On elliptic lower dimensional tori in hamiltonian systems,} Mathematische Zeitschrift 202 (1989), 559--608.
		
		\bibitem{P96} J. P\"oschel, {\it A KAM-Theorem for some nonlinear PDEs,} Annali della Scuola Normale Superiore di Pisa 23 (1996), 119--148.
		
		\bibitem{P96-1} J. P\"oschel, {\it Quasi-periodic solutions for a nonlinear wave equation}, Commentarii Mathematici Helvetici  71 (1996), no. 2, 269--296.
		
		\bibitem{PP15} C. Procesi, M. Procesi, \textit{A KAM algorithm for the resonant non-linear Schrödinger equation,} Advances in Mathematics (2015), 399--470.
		
		\bibitem{R17} C. Renault, \textit{Relative equilibria with holes for the surface quasi-geostrophic equations,} Journal of Differential Equations 263 (2017), no. 1, 567--614.
		
		\bibitem{R21} E. Roulley, \textit{Vortex rigid motion in quasi-geostrophic shallow-water equations,} Asymptotic Analysis (2022).
		
		\bibitem{R22} E. Roulley, \textit{Periodic and quasi-periodic Euler-$\alpha$ flows close to Rankine vortices,} arXiv:2208.13109.
		
		\bibitem{R01} H. R\"ussmann, {\it Invariant tori in non-degenerate nearly integrable Hamiltonian systems,} Regular and Chaotic Dynamics 6 (2001), no. 2, 119--204.
		
		\bibitem{T85} B. Turkington, \textit{Corotating steady vortex flows with n-fold symmetry,} Nonlinear Analysis, Theory, Methods and Applications 9 (1985), no. 4, 351--369.
		
		\bibitem{WXZ22} Y. Wang, X. Xu, M. Zhou, \textit{Degenerate bifurcation of two-fold doubly-connected vortex patches,} arXiv:2212.01869.
		
		\bibitem{W90} C. E. Wayne, {\it Periodic and quasi-periodic solutions of nonlinear wave equations via KAM theory,} Communications in Mathematical Physics 127 (1990), no. 3, 479--528.
		
		\bibitem{Y63} Y. Yudovich, {\it Nonstationary flow of an ideal incompressible liquid,} USSR Computational Mathematics and Mathematical Physics 3 (1963), 1032--1066.
		
		\bibitem{Z75} E. Zehnder, \textit{Generalized implicit function theorem with applications to some small divisors I,} Communications on Pure and Applied Mathematics 28 (1975) 91–140.
		
		\bibitem{Z76} E. Zehnder, \textit{Generalized implicit function theorem with applications to some small divisors II,} Communications on Pure and Applied Mathematics 29 (1976), 49–111.
	\end{thebibliography}
\end{document}